\documentclass[10pt,francais,leqno,makeidx]{smfart}

\usepackage[T1]{fontenc}
\usepackage[francais]{babel}

\usepackage{amssymb,url,xspace,smfthm}
\usepackage{mathrsfs,euscript,color}

\usepackage[all]{xy}

\newcommand{\BibTeX}{{\scshape Bib}\kern-.08em\TeX}
\newcommand{\T}{\S\kern .15em\relax }
\newcommand{\AMS}{$\mathcal{A}$\kern-.1667em\lower.5ex\hbox
        {$\mathcal{M}$}\kern-.125em$\mathcal{S}$}

\theoremstyle{plain}

\newtheorem*{montheo}{\textsc{Théorème}}
\newtheorem*{mapropo}{\textsc{Proposition}}
\newtheorem*{monhypo}{\textsc{Hypothèse}}

\newtheorem*{monlem}{\textsc{Lemme}}
\newtheorem*{monlem1}{\textsc{Lemme 1}}
\newtheorem*{monlem2}{\textsc{Lemme 2}}

\newtheorem*{moncoro}{\textsc{Corollaire}}

\newtheorem*{marema}{\textsc{Remarque}}
\newtheorem*{marema1}{\textsc{Remarque 1}}
\newtheorem*{marema2}{\textsc{Remarque 2}}
\newtheorem*{marema3}{\textsc{Remarque 3}}

\newtheorem*{madefi}{\textsc{Définition}}

\newtheorem*{notation}{\textsc{Notation}}
\newtheorem*{notations}{\textsc{Notations}}

\newtheorem*{exemple}{\textsc{Exemple}}
\newtheorem*{variante}{\textsc{Variante}}

\catcode`\@=11
\catcode`\Ž=13
\catcode`\=13
\catcode`\ˆ=13
\catcode`\=13
\catcode`\=13
\catcode`\‰=13
\catcode`\™=13
\catcode`\"=13
\catcode`\ž=13
\catcode`\=13
\catcode`\'=13
\def Ž{\'e}
\def {\`e}
\def ˆ{\`a}
\def {\`u}
\def {\^e}
\def ‰{\^a}
\def ™{\^o}
\def "{\^{\i}}
\def ž{\^u}
\def {\c c}
\def '{\"e}
\catcode`\:=13
\def :{~\string:}
\catcode`\;=13
\def ;{~\string;}
\catcode`\!=13
\def !{~\string!}

\def\cad{c'est--\`a--dire\ }
\def\ni{\noindent}

\def\v#1{\vskip#1mm}

\def\ni{\noindent}

\def\H{\EuScript{H}}
\def\Hn{\EuScript{H}^\natural}
\def\G{\EuScript{G}}
\def\P{\EuScript{P}}
\def\F{\EuScript{F}}
\def\A{\EuScript{A}}
\def\N{\EuScript{N}}

\def\C{\EuScript{C}}
\def\E{\EuScript{E}}

\def\O{\EuScript{O}}

\def\sP{\scriptscriptstyle{P}}
\def\sQ{\scriptscriptstyle{Q}}
\def\sK{\scriptscriptstyle{K}}
\def\sJ{\scriptscriptstyle{J}}
\def\sR{\scriptscriptstyle{R}}

\date{\today}

\title[LA TH\'EOR\`EME DE PALEY-WIENER TORDU]
{LA TRANSFORM\'EE DE FOURIER POUR LES ESPACES TORDUS SUR UN GROUPE R\'EDUCTIF $\mathfrak{p}$-ADIQUE\\ 
I. LE TH\'EOR\`EME DE PALEY--WIENER}

\author{Guy Henniart}
\address{Universit\'e  Paris--Sud\\
Laboratoire de Math\'ematiques d'Orsay\\
CNRS\\
F--91405 Orsay Cedex}
\email{Guy.Henniart@math.u-psud.fr}

\author{Bertrand Lemaire}
\address{Institut de Math\'ematiques de Luminy et
UMR 6206 du CNRS\\
Universit\'e Aix--Marseille II, Case Postale 907\\
F--13288 Marseille Cedex 9}
\email{lemaire@iml.univ-mrs.fr}

\makeindex
\begin{document}
\def\smfbyname{}
\small

\begin{abstract}
Soit ${\boldsymbol{G}}$ un groupe rŽductif connexe d\'efini sur un corps local 
non archimŽdien $F$. On pose $G={\boldsymbol{G}}(F)$. Soit aussi
$\theta$ un $F$--automorphisme de ${\boldsymbol{G}}$, et $\omega$ un caractre lisse de $G$. 
On s'intŽresse aux repr\'esentations complexes lisses $\pi$ de $G$ telles que $\pi^\theta=\pi\circ\theta$ 
est isomorphe \`a $\omega\pi=\omega\otimes\pi$. Si $\pi$ est admissible, en particulier irrŽductible, 
le choix d'un isomorphisme $A$ de $\omega\pi$ sur $\pi^\theta$ (et d'une mesure de Haar sur $G$) 
dŽfinit une distribution $\Theta_\pi^A={\rm tr}(\pi\circ A)$ sur $G$. La transform\'ee de Fourier tordue 
associe \`a une fonction $f$ sur $G$ localement constante et \`a support compact, la fonction 
$(\pi,A)\mapsto \Theta_\pi^A(f)$ sur un groupe de Grothendieck ad\'equat. On d\'ecrit ici son image 
(thŽorme de Paley--Wiener), et l'on rŽduit la description de son noyau (th\'eor\`eme de densit\'e spectrale) ˆ un ŽnoncŽ 
sur la partie discrte de la thŽorie.
\end{abstract}

\begin{altabstract}
Let ${\boldsymbol{G}}$ be a connected reductive group defined over a non--Archimedean local field $F$. Put 
$G={\boldsymbol{G}}(F)$. Let $\theta$ be an $F$--automorphism of ${\boldsymbol{G}}$, and let $\omega$ be a smooth character of $G$. 
This paper is concerned with the smooth complex representations $\pi$ of $G$ such that $\pi^\theta=\pi\circ\theta$ 
is isomorphic to $\omega\pi=\omega\otimes\pi$. If $\pi$ is admissible, in particular irreducible, the choice of an 
isomorphism $A$ from $\omega\pi$ to $\pi^\theta$ (and of a Haar measure on $G$) defines a distribution 
$\Theta_\pi^A={\rm tr}(\pi\circ A)$ on $G$. The twisted Fourier transform associates to a compactly supported locally constant 
function $f$ on $G$, the function $(\pi,A)\mapsto \Theta_\pi^A(f)$ on a suitable Grothendieck group. Here we describe 
its image (Paley--Wiener theorem), and we reduce the description of its kernel (spectral density theorem) to a result on the discrete part of the theory.
\end{altabstract}

\subjclass{22E50}

\keywords{corps local non archimŽdien, espace tordu, caractre tordu, transformŽe de Fourier , thŽorme de Paley--Wiener, thŽorme de densitŽ spectrale}

\altkeywords{non--Archimedean local field,
twisted space, twisted character, Fourier transform, Paley--Wiener theorem, spectral density theorem}
\maketitle
\tableofcontents

\section{Introduction}

\subsection{}Soit $F$ un corps commutatif localement compact non archimŽdien, 
et soit ${\boldsymbol{G}}$ un groupe rŽductif connexe dŽfini sur $F$. Le groupe $G={\boldsymbol{G}}(F)$ 
des points $F$--rationnels de ${\boldsymbol{G}}$, muni de la topologie 
donnŽe par $F$, est localement profini --- en particulier localement compact --- et unimodulaire. 
On appelle {\it reprŽsentation de $G$}, ou {\it $G$--module}, une reprŽsentation lisse de $G$ 
ˆ valeurs dans le groupe des automorphismes d'un espace 
vectoriel sur ${\Bbb C}$. Le choix d'une mesure de Haar $dg$ sur 
$G$ permet de 
dŽfinir, pour toute reprŽsentation admissible $\pi$ de $G$, une distribution $\Theta_\pi$ sur $G$, \cad 
une forme lin\'eaire sur l'espace $\H(G)$ des fonctions localement constantes et \`a support compact sur $G$: pour 
$f\in \H(G)$, l'opŽrateur 
$\pi(f)=\int_Gf(g)\pi(g)dg$ sur l'espace de $\pi$ est de rang fini, et l'on pose
$$
\Theta_\pi(f)={\rm tr}(\pi(f)).
$$
Cette distribution $\Theta_\pi$ ne dŽpend que de la classe d'isomorphisme de $\pi$ (et aussi du choix de $dg$). 
Notons $\G(G)$ le groupe de Grotendieck des reprŽsentations de longueur finie de $G$. Tout ŽlŽment $\pi$ de $\G(G)$ 
dŽfinit par linŽaritŽ une distribution $\Theta_\pi$ sur $G$. 

Le thŽorme de Paley--Wiener (scalaire) prouvŽ dans \cite{BDK} caractŽrise les applications 
${\Bbb Z}$--linŽaires de $\G(G)$ vers ${\Bbb C}$ qui sont de la forme $\pi\mapsto \Theta_\pi(f)$ pour une fonction $f\in \H(G)$. 
L'espace $\H(G)$ est muni d'un produit de convolution $*$, donn\'e par
$$
f*h(x)= \int_Gf(g)h(g^{-1}x)dg.
$$ 
Le noyau de l'application $f\mapsto (\pi\mapsto \Theta_\pi(f))$ contient le sous--espace $[\H(G),\H(G)]$ 
de $\H(G)$ engendr\'e par les commutateurs $f*h-h*f$. Le th\'eor\`eme de densit\'e spectrale affirme que 
ce noyau est \'egal \`a $[\H(G),\H(G)]$. Il a \'et\'e d\'emontr\'e par Kazhdan dans \cite[appendix]{K1}, via un 
argument local--global utilisant la formule des traces, donc valable seulement 
si $F$ est de caractŽristique nulle. Kazhdan a ensuite Žtendu son rŽsultat au cas o $F$ est de caractŽristique 
non nulle \cite{K2}, par la mŽthode des corps proches en supposant $G$ dŽployŽ. 
Notons que cette mŽthode est certainement valable sous des hypothses moins restrictives, par exemple en 
supposant la caractŽristique rŽsiduelle grande par rapport au rang de $G$ --- voir les travaux rŽcents de J.--L.~Waldspurger 
sur le lemme fondamental ---, mais cela reste ˆ rŽdiger. 

\subsection{}
On s'intŽresse ici ˆ la version \og tordue \fg des rŽsultats prŽcŽdents. La torsion 
en question est donnŽe par un $F$--automorphisme de $G$, disons $\theta$. On fixe aussi 
un caractre $\omega$ de $G$, \cad un homomorphisme continu dans ${\Bbb C}^\times$. Pour $f\in \H(G)$ 
et $x\in G$, on note ${^x\!f}$ la fonction 
$g\mapsto \omega(x)f(x^{-1}g\theta(x))$ sur $G$. La thŽorie de l'endoscopie tordue 
Žtudie les distributions $D$ sur $G$ qui, pour tout $f\in \H(G)$ et tout $x\in G$, 
vŽrifient $D({^x\!f})=D(f)$. 

Soit $\pi$ une reprŽsentation irrŽductible de $G$ telle que $\pi^\theta=\pi\circ \theta $ est 
isomorphe ˆ $\omega\pi =\omega\otimes \pi$. Le choix d'un isomorphisme $A$ de $\omega\pi$ sur $\pi^\theta$ 
dŽfinit comme plus haut une distribution $\Theta_\pi^A={\rm tr}(\pi\circ A)$ sur $G$. Cette distribution 
$\Theta_\pi^A$ dŽpend bien s\^ur du choix de $A$ (et aussi de celui de $dg$), et elle vŽrifie
$$
\Theta_\pi^A({^x\!f})=\Theta_\pi^A(f).
$$

Pour dŽcrire l'image et le noyau de l'application $f\mapsto (\pi\mapsto \Theta_\pi^A(f))$ comme dans le cas non tordu, il faut commencer 
par la dŽfinir! On peut le faire de diverses mani\`eres, l'une d'elle \'etant la suivante. 
Soit $\G_{\Bbb C}(G,\theta,\omega)$ le ${\Bbb C}$--espace vectoriel engendr\'e par les paires $(\pi,A)$ 
o\`u $\pi$ est une repr\'esentation de $G$ de longueur finie telle que $\pi^\theta\simeq \omega\pi$ et $A$ 
est un isomorphisme de $\omega\pi$ sur $\pi^\theta$, modulo les relations:

\begin{itemize}
\item pour toute suite exacte $0\rightarrow (\pi_1,A_1)\rightarrow (\pi_2,A_2)\rightarrow (\pi_3,A_3)\rightarrow 0$, i.e. 
une suite exacte de $G$-modules qui commute aux $A_i$, on a 
$(\pi_3,A_3)=(\pi_2,A_2)-(\pi_1,A_1)$;
\item pour tout $\lambda\in {\Bbb C}^\times$, on a $(\pi,\lambda A)=\lambda (\pi,A)$;
\item pour tout entier $k>1$ et toute paire $(\rho,B)$ form\'ee d'une repr\'esentation de longueur finie 
$\rho$ de $G$ telle que $\rho(k)\simeq \rho$ et d'un isomorphisme $B$ de $\rho$ sur 
$\rho(k)$, on a $\iota_k(\rho,B)=0$.
\end{itemize}

\ni Ci--dessus, $\iota_k(\rho,B)$ est la paire $(\pi,A)$ dŽfinie par 
$$
\pi=\rho\oplus \rho(1)\oplus \cdots \oplus \rho(k-1),\quad 
A(v_0,v_1,\ldots ,v_{k-1})=(v_1,\ldots ,v_{k-1},B(v_0)),
$$ o\`u 
l'on a pos\'e $\rho(i)=\omega_i^{-1}\rho^{\theta^i}$, $\omega_i$ d\'esignant 
le caract\`ere $g\mapsto w(g\theta(g)\cdots \theta^{i-1}(g))$ de $G$. Par construction, $\pi(1)=\omega^{-1}\pi^\theta$ 
est isomorphe ˆ $\pi$ et $A$ est un isomorphisme de $\pi$ sur $\pi(1)$.

L'application $f\mapsto ((\pi,A)\mapsto \Theta_\pi^A(f))$ d\'efinit un morphisme ${\Bbb C}$--lin\'eaire
$$
\H(G)\rightarrow \G_{\Bbb C}(G,\theta,\omega)^*={\rm Hom}_{\Bbb C}(\G_{\Bbb C}(G,\theta,\omega),{\Bbb C}).
$$
C'est ce morphisme que l'on \'etudie dans cet article.

\subsection{}Plut\^ot que de fixer le $F$--automorphisme $\theta$ de ${\boldsymbol{G}}$, il convient de travailler avec un 
${\boldsymbol{G}}$--espace algŽbrique tordu ${\boldsymbol{G}}^\natural$ tel que l'ensemble $G^\natural={\boldsymbol{G}}^\natural(F)$ de ses points 
$F$--rationnels est non vide. Le choix d'un point--base $\delta_1\in G^\natural$ dŽfinit 
un $F$--automorphisme $\theta={\rm Int}_{{\boldsymbol{G}}^\natural}(\delta_1)$ de ${\boldsymbol{G}}$ qui permet d'identifier $G^\natural$ au 
$G$-espace topologique tordu $G\theta$ (cf. \ref{donnŽes}). On appelle {\it $\omega$--reprŽsentation de $G^\natural$}, ou 
{\it $(G^\natural,\omega)$--module}, la donn\'ee d'une paire $(\pi,A)$ form\'ee d'une reprŽsentation 
$\pi$ de $G$ telle que $\pi(1)\simeq \pi$ et d'un isomorphisme $A$ de $\pi$ sur $\pi(1)$ 
(on renvoie ˆ \ref{rappelsrep} pour une dŽfinition plus intrinsque). On note $\Pi$ la paire $(\pi,A)$, et l'on pose $\Pi^\circ=\pi$. 
Les $\omega$--reprŽsentations de $G^\natural$ s'organisent 
naturellement en une catŽgorie abŽlienne, et $\G_{\Bbb C}(G^\natural,\omega)=\G_{\Bbb C}(G,\theta,\omega)$ est un 
quotient du groupe de Grothendieck des $\omega$--repr\'esentations $\Pi$ de $G^\natural$ telles que la reprŽsentation 
$\Pi^\circ$ de $G$ sous--jacente est de longueur finie.

Toute $\omega$--repr\'esentation $\Pi$ de $G^\natural$ telle que $\Pi^\circ$ est admissible d\'efinit comme plus haut 
une distribution $\Theta_\Pi$ sur $G^\natural$, \cad une forme linŽaire sur l'espace $\H(G^\natural)$ des fonctions 
localement constantes et ˆ support compact sur $G^\natural$: pour $\phi\in \H(G^\natural)$, on pose 
$$
\Theta_\Pi(\phi)={\rm tr}(\Pi(\phi)),
$$
o $\Pi(\phi)$ est l'opŽrateur $\int_{G^\natural}\phi(\delta)\Pi(\delta)d\delta$ sur l'espace de $\Pi$ 
(il est de rang fini). Ici $d\delta$ est la mesure $G$--invariante sur $G^\natural$ image de $dg$ par 
l'homŽomorphisme $g\mapsto g\cdot \delta_1$. On a donc
$$
\Theta_\Pi(\phi)= \Theta_{\Pi^\circ}^{\Pi(\delta_1)}(\phi^\circ),
$$
o $\phi^\circ$ est la fonction $g\mapsto \phi(g\cdot \delta_1)$ sur $G$. 
Traduite en ces termes, la transform\'ee de Fourier pour $(G^\natural,\omega)$ est le morphisme ${\Bbb C}$--lin\'eaire
$$
\H(G^\natural)\rightarrow \G_{\Bbb C}(G^\natural,\omega)^*={\rm Hom}_{\Bbb C}(\G_{\Bbb C}(G^\natural,\omega),{\Bbb C})
$$
d\'eduit par linŽaritŽ de l'application $\phi \mapsto (\Pi\mapsto \Theta_\Pi(\phi))$. 

Notre rŽsultat principal (\'enoncŽ en \ref{ŽnoncŽ}) est une description de ce morphisme: le thŽorme 
de \og Paley--Wiener tordu\fg dŽcrit son image, et le thŽorme de \og densitŽ spectrale tordue\fg son noyau. 
En fait la densitŽ spectrale en question est plut\^ot une consŽquence de la description du noyau: l'espace $\H(G^\natural)$ 
est naturellement muni d'une structure de $\H(G)$--bimodule, et le sous--espace $[\H(G^\natural),\H(G)]_\omega$ de 
$\H(G^\natural)$ engendrŽ par les fonctions $\phi*f -\omega f * \phi$ est clairement contenu dans le noyau. Le thŽorme de densitŽ spectrale 
tordue dit que cette inclusion est une ŽgalitŽ: si une fonction $\phi$ annule toutes les traces $\Theta_\Pi$, o $\Pi$ parcourt  les 
$\omega$--reprŽsentations de $G^\natural$ telles que $\Pi^\circ$ est irrŽductible, alors elle est dans $[\H(G^\natural),\H(G)]_\omega$. 
Cela implique en particulier qu'elle annule toutes les distributions $\mathfrak{D}$ sur $G^\natural$ telles que $\mathfrak{D}({^x\phi'})=
\mathfrak{D}(\phi')$ pour tout $\phi'\in \H(G^\natural)$ et tout $x\in G$, 
o\`u l'on a pos\'e ${^x\phi'}(\delta)= \omega(x)\phi'(x^{-1}\cdot \delta\cdot x)$. 

Dans cet article, on dŽmontre la surjectivitŽ dans le thŽorme de \ref{ŽnoncŽ} (thŽorme de Paley--Wiener tordu). 
La preuve de l'injectivitŽ (thŽorme de densitŽ spectrale tordue), beaucoup plus longue que prŽvue, est seulement ŽbauchŽe ici, 
et sera terminŽe ailleurs \cite{HL}.  

\subsection{}Le thŽorme de Paley--Wiener tordu a ŽtŽ dŽmontrŽ par Rogawski dans \cite{R}, pour 
$\theta$ d'ordre fini et $\omega=1$. La preuve est essentiellement celle de \cite{BDK}, adaptŽe au cas tordu. 
Sous les m\^emes hypothses ($\theta^l={\rm id}$ et $\omega=1$), Flicker a dŽcrit dans \cite{F} une preuve 
locale du thŽorme de densitŽ spectrale, 
utilisant la mŽthode de \og dŽvissage \fg de Bernstein. \`A notre connaissance, cette mŽthode n'a 
jamais ŽtŽ rŽdigŽe par Bernstein. Elle est particulirement bien expliquŽe par Dat (dans le cas non tordu) 
dans son article sur le $K_0$ \cite{D}. Ce dernier permet d'ailleurs de reconstruire les arguments 
manquants dans \cite{F}. 

La dŽmonstration proposŽe ici et dans \cite{HL} est entirement locale, et aussi entirement spectrale puisqu'aucun 
recours aux intŽgrales orbitales n'est nŽcessaire. Comme dans \cite{F}, on traite de fa\c{c}on semblable 
la surjectivitŽ (Paley--Wiener) et l'injectivitŽ (densitŽ spectrale). Le th\'eor\`eme de Paley--Wiener est 
dŽmontrŽ ici en adaptant au cas tordu les arguments de \cite{BDK}. Le th\'eor\`eme de densit\'e spectrale sera (compl\`etement) 
dŽmontrŽ dans \cite{HL} gr\^ace la mŽthode de dŽvissage de Bernstein.

Par induction parabolique et rŽcurrence sur la 
dimension de $G$, on ramne l'Žtude de la transform\'ee de Fourier ˆ la partie  
\og discr\`ete \fg de $(G^\natural,\omega)$. Notant $\smash{\overline{\H}}^{\rm dis}(G^\natural,\omega)$ le sous--espace 
de $\smash{\overline{\H}}(G^\natural,\omega)= \H(G^\natural)/[\H(G^\natural),\H(G)]_\omega$ engendrŽ par les fonctions 
\og $\omega$--cuspidales \fg, et $\G^{\rm dis}_{\Bbb C}(G^\natural,\omega)^*$ l'espace des formes linŽaires 
\og discrtes \fg sur $\G_{\Bbb C}(G^\natural, \omega)$ --- ces notions 
sont les avatars tordus des notions habituelles, cf. \ref{dŽcrivons les points-clŽs} ---, 
la transformŽe de Fourier pour 
$(G^\natural,\omega)$ se restreint en un morphisme 
$$
\smash{\overline{\H}}^{\rm dis}(G^\natural,\omega)\rightarrow\G_{\Bbb C}^{\rm dis}(G^\natural,\omega)^*.
$$
Une bonne partie du prŽsent article est consacrŽe ˆ l'Žtude de ce morphisme, prŽcisŽment ˆ la description de 
son image, son injectivitŽ Žtant dŽmontrŽe dans \cite{HL}. 
La description (image et noyau) de la transformŽe de Fourier sur l'espace $\G_{\Bbb C}(G^\natural,\omega)$ tout entier 
s'en dŽduit ensuite aisŽment. Notons que si le centre de $G^\natural$ est compact --- cas particulier auquel il 
est en principe toujours possible de se ramener en fixant le caractre central --- le morphisme ci--dessus est un isomorphisme. 
Notons aussi que dans le cas non tordu, cet isomorphisme a dŽjˆ ŽtŽ Žtabli en 
caractŽristique nulle par Kazhdan \cite[theorem B]{K1}.

\subsection{}Le thŽorme de Paley--Wiener dŽmontrŽ ici a dŽjˆ ŽtŽ utilisŽ par 
J.-L. Waldspurger pour Žtablir la formule des traces locale tordue en caractŽristique nulle \cite{W}. 
Notons que dans ce m\^eme papier, l'auteur dŽmontre --- toujours en caractŽristique nulle, et 
sous l'hypothse o la restriction de $\theta$ au centre $Z({\boldsymbol{G}})$ de ${\boldsymbol{G}}$ est d'ordre fini --- un thŽorme de densitŽ, 
appelŽ \og thŽorme $0$\fg de Kazhdan \cite[5.5]{W}, qui (en caractŽristique nulle) est Žquivalent au thŽorme de densitŽ spectrale 
Žtabli ici. PrŽcisŽment, Waldspurger (dans \cite{W}) commence par 
Žtablir une premire formule des traces locale tordue {\it non--invariante}, formule de laquelle il dŽduit le \og thŽorme $0$\fg de Kazhdan. 
Ensuite il utilise le thŽorme de Paley--Wiener --- en particulier l'existence de pseudo--coefficients --- pour rendre cette 
premire formule invariante. 

\subsection{}\label{dŽcrivons les points-clŽs}DŽcrivons brivement les points--clŽs de la dŽmonstration du thŽorme de Paley--Wiener tordu. 
On fixe une famille $\P(G^\natural)$ de sous--espaces paraboliques standard $P^\natural$ de $G^\natural$ 
munis d'une dŽcomposition de Levi standard $P^\natural=M_{\sP}^\natural\cdot U_{\sP}$; ici $P$ dŽsigne le 
sous--groupe parabolique de $G$ sous--jacent ˆ $P^\natural$, et $U_{\sP}$ son radical unipotent. 
Pour $P^\natural\in \P(G^\natural)$, les versions tordues des foncteurs induction parabolique 
et restriction de Jacquet normalisŽs  dŽfinissent des morphismes ${\Bbb C}$--linŽaires
$${^\omega{i}}_{P^\natural}^{G^\natural}:\G_{\Bbb C}(M_{\sP}^\natural,\omega)
\rightarrow \G_{\Bbb C}(G^\natural,\omega),\quad 
{^{\omega}r}_{G^\natural}^{P^\natural}:\G_{\Bbb C}(G^\natural,\omega)\rightarrow \G_{\Bbb C}(M_{\sP}^\natural,\omega).
$$
Notons $\G_{{\Bbb C},{\rm ind}}(G^\natural,\omega)$ le sous--espace de $\G_{\Bbb C}(G^\natural,\omega)$ 
engendrŽ par les ${^\omega{i}}_{P^\natural}^{G^\natural}(\G_{\Bbb C}(M_{\sP}^\natural,\omega))$ pour $P^\natural\in \P(G^\natural)$ 
distinct de $G^\natural$, et posons
$$
\G^{\rm dis}_{\Bbb C}(G^\natural,\omega)=\G_{\Bbb C}(G^\natural,\omega)/\G_{{\Bbb C},{\rm ind}}(G^\natural,\omega).
$$
Une forme linŽaire sur $\G_{\Bbb C}(G^\natural,\omega)$ est dite \og discrte \fg si elle s'annule sur $\G_{{\Bbb C},{\rm ind}}(G^\natural,\omega)$. 
Une $\omega$--reprŽsentation $\Pi$ de $G^\natural$ telle que $\Pi^\circ$ est irrŽductible est dite 
\og discrte \fg si son image dans $\G_{\Bbb C}^{\rm dis}(G^\natural,\omega)$ n'est pas nulle. On note 
${\rm Irr}^{\rm dis}_{\Bbb C}(G^\natural,\omega)$ le sous--ensemble de $\G_{\Bbb C}(G^\natural,\omega)$ formŽ des 
$\omega$--reprŽsentations discrtes de $G^\natural$. Notons que la dŽcomposition
$$
\G_{\Bbb C}(G^\natural,\omega)=\langle \Pi:\Pi\in {\rm Irr}^{\rm dis}_{\Bbb C}(G^\natural,\omega)\rangle + \G_{{\Bbb C},{\rm ind}}(G^\natural,\omega)
$$
n'est en gŽnŽral pas une somme directe.

Pour $\phi\in \H(G^\natural)$ et $P^\natural\in \P(G^\natural)$, on a dŽfini dans \cite[5.9]{L2} le terme constant 
tordu ${^\omega\phi_{P^\natural,K}}\in \H(M_{\sP}^\natural)$ relatif ˆ un sous--groupe compact maximal spŽcial 
$K$ de $G$ choisi de manire convenable (en bonne position par rapport aux sous--groupes paraboliques standard de $G$, cf. 
\ref{terme constant}). On dit que $\phi$ est \og $\omega$--cuspidale \fg si pour tout $P^\natural\in \P(G^\natural)$ distinct de 
$G^\natural$, l'image de ${^\omega\phi_{P^\natural,K}}$ dans $\overline{\H}(M_{\sP}^\natural,\omega)$ est nulle. 
D'apr\`es l'analogue tordu de la formule de Van Dijk pour les traces des reprŽsentations induites \cite[thŽo.]{L2}, si 
$\phi$ est $\omega$--cuspidale alors $\Theta_\Pi(\phi)=0$ pour tout $\Pi\in \G_{{\Bbb C},{\rm ind}}(G^\natural,\omega)$. D'ailleurs 
si l'on admet le thŽorme de densitŽ spectrale tordue pour tous les sous--espaces de Levi propres de $G^\natural$, 
la rŽciproque est vraie aussi. La transform\'ee de Fourier induit donc bien un morphisme
$$
\smash{\overline{\H}}^{\rm dis}(G^\natural,\omega)\rightarrow \G_{\Bbb C}^{\rm dis}(G^\natural,\omega)^*.
$$
On dŽmontre en \ref{rŽduction} que le thŽorme principal (\ref{ŽnoncŽ}) se ramne ˆ un ŽnoncŽ analogue sur la partie discrte 
de la thŽorie (\ref{ŽnoncŽ discret}), \cad ˆ la description de l'image et ˆ l'injectivitŽ du morphisme ci--dessus.

\subsection{}DŽcrivons l'image du morphisme prŽcŽdent. 
L'application $(k,\pi)\mapsto \pi(k)$ induit une action de ${\Bbb Z}$ sur la plupart des objets reliŽs ˆ la thŽorie 
des reprŽsentations de $G$:
\begin{itemize}
\item l'ensemble ${\rm Irr}(G)$ des classes d'isomorphisme de 
reprŽsentations irrŽductibles de $G$;
\item l'ensemble $\Theta(G)=\coprod_\mathfrak{s}\Theta(\mathfrak{s})$ des classes 
de $G$--conjugaison de paires cuspidales de $G$, o $\frak{s}$ parcourt l'ensemble des classes d'Žquivalence inertielle 
de paires cuspidales de $G$ et $\Theta(\mathfrak{s})$ dŽsigne la variŽtŽ complexe associŽe ˆ $\mathfrak{s}$ (cf. \ref{support inertiel}); 
\item le centre $\mathfrak{Z}(G)$ de la catŽgorie des $G$--modules;
\item etc. 
\end{itemize}
L'application caractre infinitŽsimal
$$
\theta_G: {\rm Irr}(G)\rightarrow \Theta(G)
$$
est ainsi ${\Bbb Z}$--Žquivariante. Pour chaque classe d'Žquivalence inertielle $\mathfrak{s}$, on note 
$\Theta^{\rm dis}_{G^\natural,\omega}(\mathfrak{s})$ le sous--ensemble de $\Theta(\mathfrak{s})$ 
form\'e des $\theta_G(\Pi^\circ)$ pour une $\omega$--reprŽsentation 
discrte $\Pi$ de $G^\natural$. Remarquons que pour que 
$\Theta^{\rm dis}_{G^\natural,\omega}(\mathfrak{s})$ soit non vide, 
il faut que la variŽtŽ $\Theta(\mathfrak{s})$ soit ${\Bbb Z}$--stable. 
Comme dans \cite{BDK}, on montre (en \ref{dŽcomposition}--\ref{constructible}) la

\begin{mapropo}
L'ensemble $\Theta^{\rm dis}_{G^\natural,\omega}(\mathfrak{s})$ est une 
partie constructible de $\Theta(\mathfrak{s})$.
\end{mapropo}

Notons $\mathfrak{P}(G^\natural)$ le groupe --- algŽbrique, diagonalisable sur 
${\Bbb C}$ --- des caractres 
non ramifiŽs de $G$ qui sont $\theta$--stables (il ne dŽpend pas du choix de $\delta_1\in G^\natural$), et posons 
$d(G^\natural)= \dim \mathfrak{P}(G^\natural)$. 
Comme dans loc.~cit., on en dŽduit (en \ref{finitude}) le

\begin{moncoro}
L'ensemble $\Theta^{\rm dis}_{G^\natural,\omega}(\mathfrak{s})$ 
est union finie de $\mathfrak{P}(G^\natural)$--orbites.
\end{moncoro}

Soit aussi $\mathfrak{P}_{\Bbb C}(G^\natural)$ l'ensemble des ($\omega=1$)--reprŽsentations du $G/G^1$--espace tordu 
$G^\natural/G^1$, o $G^1\subset G$ dŽsigne le groupe engendrŽ par les sous--groupes compacts de $G$. 
L'ensemble $\mathfrak{P}_{\Bbb C}(G^\natural)$, identifi\'e ˆ un ensemble de fonctions $G^\natural \rightarrow {\Bbb C}^\times$, 
est muni d'une structure de groupe, qui en fait une extension (alg\'ebrique, scindŽe) de 
$\mathfrak{P}(G^\natural)$ par ${\Bbb C}^\times$. On en dŽduit la description de l'image du 
morphisme $\smash{\overline{\H}}^{\rm dis}(G^\natural,\omega)\rightarrow \G_{\Bbb C}^{\rm dis}(G^\natural,\omega)^*$. C'est 
l'espace, disons $\F^{\rm dis}(G^\natural,\omega)$, des formes linŽaires $\varphi$ sur $\G_{\Bbb C}^{\rm dis}(G^\natural,\omega)$ 
vŽrifiant: 
 
 \begin{itemize}
 \item il existe un ensemble {\it fini} $\mathfrak{S}$ de classes d'Žquivalence inertielle $\mathfrak{s}$ tel que 
 pour tout $\Pi\in {\rm Irr}^{\rm dis}_{\Bbb C}(G^\natural,\omega)$, on a 
 $\varphi(\Pi)=0$ si $\theta_G(\Pi^\circ)\in \Theta(G)\smallsetminus \coprod_{\frak{s}\in \frak{S}}\Theta(\frak{s})$;
 \item pour tout $\Pi\in {\rm Irr}^{\rm dis}_{\Bbb C}(G^\natural,\omega)$,  l'application 
 $\mathfrak{P}_{\Bbb C}(G^\natural)\rightarrow {\Bbb C},\, \Psi\mapsto \varphi(\Psi\Pi)$ est une fonction {\it rŽgulire} 
 sur la variŽtŽ $\mathfrak{P}_{\Bbb C}(G^\natural)$.
 \end{itemize}

\subsection{}L'article s'organise comme suit. 

Dans la section 2, on reprend la thŽorie des reprŽsentations de $G$  
dans le cas tordu. Il s'agit essentiellement de suivre l'action de ${\Bbb Z}$ 
--- donnŽe par $(k,\pi)\mapsto \pi(k)$ --- sur les principaux objets de la thŽorie. 
Le cadre choisi est celui des $\omega$--reprŽsentations de $G^\natural$, qui 
sont reliŽes aux reprŽsentations tordues de $G$ via le foncteur d'oubli $\Pi\mapsto \Pi^\circ$. 
Parmi les rŽsultats obtenus dans ce cadre tordu, signalons: le lemme gŽomŽtrique, le 
thŽorme du quotient de Langlands, la description du centre de Berstein.

Le rŽsultat principal est ŽnoncŽ dans la section 3. On montre aussi 
comment la description de l'image de la transformŽe de Fourier implique 
la \og variante tempŽrŽe \fg du thŽorme de Paley--Wiener, c'est--ˆ--dire la version en termes 
des $\omega_{\rm u}$--reprŽsentations tempŽrŽes des sous--espaces de Levi de $G^\natural$; o $\omega_{\rm u}$ est le 
caractre unitaire $\omega \vert \omega\vert^{-1}$ de $G$.

Dans la section 4, on ramne l'Žtude de la transformŽe de Fourier ˆ celle de sa restriction ˆ la 
partie \og discrte \fg des reprŽsentations.

Dans la section 5, on dŽmontre le thŽorme de Paley--Wiener dans le cas discret. Pour cela on 
adapte au cas tordu les techniques de \cite{BDK}. Comme dans loc.~cit., le point--clŽ consiste ˆ 
montrer que pour toute classe d'Žquivalence inertielle $\mathfrak{s}$ dans $G$, 
l'ensemble $\Theta_{G^\natural,\omega}^{\rm dis}(\mathfrak{s})$ est une partie constructible de $\Theta(\mathfrak{s})$. 

\subsection{}Signalons brivement certaines hypothses admises au cours de l'article. 

En \ref{donnŽes}, on fixe un point--base $\delta_1\in G^\natural$, et l'on pose $\theta={\rm Int}_{\boldsymbol{G}^\natural}(\delta_1)$. 

En \ref{H modules}, on fixe une mesure de Haar $dg$ sur $G$ et l'on note $d\delta$ la mesure de Haar sur $G^\natural$ image de $dg$ par 
l'isomorphisme $G\rightarrow G^\natural,\, g\mapsto g\cdot \delta$ pour un (resp. pour tout) $\delta\in G^\natural$.

En \ref{ip et rj}, on fixe un sous--espace parabolique minimal $P^\natural_\circ$ de $G^\natural$, et une dŽcomposition de Levi 
$P^\natural_\circ=M^\natural_\circ\cdot U_\circ$. Le groupe $P_\circ$ sous--jacent ˆ $P^\natural_\circ$ est un sous--groupe parabolique 
minimal de $G$, et $P_\circ=M_\circ U_\circ$ (dŽcomposition de Levi ) o $M_\circ$ est le groupe 
sous--jacent ˆ $M^\natural_\circ$. 

\`A partir de \ref{ip et rj}, on suppose que $\delta_1$ appartient ˆ $M^\natural_\circ$. La paire parabolique minimale 
$(P_\circ,M_\circ)$ de $G$ est donc $\theta$--stable.

\`A partir de \ref{bons sgoc}, on suppose que $\theta$ stablise un sous--groupe d'Iwahori de $G$ en bonne position par rapport ˆ $(P_\circ,M_\circ)$. 

En \ref{terme constant}, on fixe un sous--groupe compact maximal spŽcial $K_\circ$ de $G$ en bonne position par rapport ˆ toute paire 
parabolique de $G$ contenant $(P_\circ,M_\circ)$. 
Ce groupe $K_\circ$ n'est pas supposŽ $\theta$--stable. \`A partir de \ref{terme constant}, on 
suppose que toutes les mesures de Haar utilisŽes sont celles normalisŽes par $K_\circ$.

\vfill\eject
\section{ReprŽsentations des espaces tordus}

\subsection{Conventions}\label{conventions}Pour Žviter de tomber dans les piges habituels, on fixe un 
{\it univers} de Grothendieck assez grand $\mathfrak{U}$, cf. \cite[chap.~I]{G}. Toutes les catŽgories considŽrŽes 
dans cet article sont implicitement des $\mathfrak{U}$--catŽgories: les objets d'une catŽgorie $\C$ sont les ŽlŽments 
d'un ensemble qui appartient ˆ l'univers $\mathfrak{U}$, notŽ ${\rm Ob}(\C)$, et pour $M,\,N\in {\rm Ob}(\C)$, 
les fl\`eches $M\rightarrow N$ dans $\C$ sont les ŽlŽments d'un ensemble qui appartient lui aussi ˆ l'univers 
$\mathfrak{U}$, notŽ ${\rm Hom}_\C(M,N)$. En particulier, on appelle simplement \og ensemble \fg un 
ensemble qui appartient ˆ l'univers $\mathfrak{U}$. Toutes les conventions de loc.~cit. sont adoptŽes 
ici. Par exemple, quand on parle de  {\it systme inductif} (resp. {\it projectif}) 
d'objets d'une catŽgorie, on suppose implicitement que ce systme est indexŽ par un ensemble 
appartenant ˆ $\mathfrak{U}$; idem pour les sommes directes et les produits directs. 

Sauf mention expresse du contraire, 
les modules sur un anneau $A$ sont des modules {\it ˆ gauche}. Rappelons que tout anneau 
$A$ possde une unitŽ, disons $1_A$, et que tout $A$--module $X$ vŽrifie 
$1_A\cdot x = x$, $x\in X$.

\subsection{Les donnŽes}\label{donnŽes}
\index{$F$ (corps de base), $\mathfrak{o}$, $\mathfrak{p}$, $\kappa$}
Soit $F$ un corps commutatif localement compact non archimŽdien 
(de caractŽristique quelconque). On note $\mathfrak{o}$ 
l'anneau des entiers de $F$, $\mathfrak{p}$ l'idŽal maximal de $F$, et $\kappa$ le corps rŽsiduel $\mathfrak{o}/\mathfrak{p}$.

Soit ${\boldsymbol{G}}$\index{$\boldsymbol{G}$, $\boldsymbol{G}^\natural$} un groupe r\'eductif connexe d\'efini sur $F$, et soit ${\boldsymbol{G}}^\natural$ 
un ${\boldsymbol{G}}$--espace alg\'ebrique tordu (au sens de J.-P.~Labesse) lui aussi d\'efini sur $F$. 
Rappelons que ${\boldsymbol{G}}^\natural$ est une vari\'et\'e alg\'ebrique affine d\'efinie sur $F$, munie:

\begin{itemize}
\item d'une action alg\'ebrique de ${\boldsymbol{G}}$ ˆ gauche d\'efinie sur $F$
$$
{\boldsymbol{G}}\times {\boldsymbol{G}}^\natural\rightarrow {\boldsymbol{G}}^\natural,\, (g,\delta)\mapsto g\cdot \delta
$$ 
telle que pour tout $\delta\in {\boldsymbol{G}}^\natural$, l'application ${\boldsymbol{G}}\rightarrow {\boldsymbol{G}}^\natural,\, 
g\mapsto g\cdot \delta$ est un isomorphisme de vari\'et\'es alg\'ebriques;
\item d'une application
$$
{\rm Int}_{{\boldsymbol{G}}^\natural}:{\boldsymbol{G}}^\natural\rightarrow {\rm Aut}({\boldsymbol{G}})
$$
o ${\rm Aut}({\boldsymbol{G}})$ d\'esigne le groupe 
des automorphismes alg\'ebriques de ${\boldsymbol{G}}$, 
telle que pour tout $g\in {\boldsymbol{G}}$ et tout $\delta\in {\boldsymbol{G}}^\natural$, on a
$${\rm Int}_{{\boldsymbol{G}}^\natural}(g\cdot \delta)= 
{\rm Int}_{\boldsymbol{G}}(g)\circ {\rm Int}_{{\boldsymbol{G}}^\natural}(\delta).
$$
\end{itemize}

\ni Cela munit ${\boldsymbol{G}}^\natural$ d'une action algŽbrique de ${\boldsymbol{G}}$ ˆ droite dŽfinie sur $F$, donnŽe par
$$
{\boldsymbol{G}}^\natural\times {\boldsymbol{G}}\rightarrow {\boldsymbol{G}}^\natural,\,(\delta,g)\mapsto \delta\cdot g= {\rm Int}_{{\boldsymbol{G}}^\natural}(\delta)(g)\cdot \delta.
$$

On suppose que l'ensemble $G^\natural={\boldsymbol{G}}^\natural(F)$\index{$G=\boldsymbol{G}(F)$, $G^\natural=\boldsymbol{G}^\natural(F)$, 
$\G_k=\boldsymbol{G}_k(F)$} des points $F$--rationnels de 
${\boldsymbol{G}}^\natural$ est non vide, et l'on munit $G=\boldsymbol{G}(F)$ et $G^\natural$ de la topologie $\mathfrak{p}$--adique, 
ce qui fait de $G^\natural$ un $G$--espace topologique tordu, cf. \cite[2.4]{L2}. La donnŽe de l'espace topologique 
tordu $(G,G^\natural)$ Žquivaut ˆ celle d'un groupe topologique ${\Bbb Z}$--graduŽ $\mathcal{G}=\coprod_{k\in {\Bbb Z}}\mathcal{G}_k$ 
tel que $\mathcal{G}_0=G$ et $\mathcal{G}_1=G^\natural$. Pour $k\in {\Bbb Z}$, $\mathcal{G}_k$ est un $G$--espace topologique 
tordu, et le groupe des points $F$--rationnels d'un ${\boldsymbol{G}}$--espace algŽbrique tordu ${\boldsymbol{G}}_k$ dŽfini sur $F$. Les ${\boldsymbol{G}}_k$ 
sont reliŽs entre eux par des $F$--morphismes de transition
$$
\varphi_{k,k'}:{\boldsymbol{G}}_k\times {\boldsymbol{G}}_{k'}\rightarrow {\boldsymbol{G}}_{k+k'}
$$
vŽrifiant
$$
\varphi_k(g\cdot \gamma,g'\cdot \gamma')= g{\rm Int}_{{\boldsymbol{G}}_k}(\gamma)(g')\cdot \varphi_{k,k'}(\gamma,\gamma')
$$
et
$$
{\rm Int}_{{\boldsymbol{G}}_{k+k'}}(\varphi_{k,k'}(\gamma,\gamma'))= {\rm Int}_{{\boldsymbol{G}}_k}(\gamma)\circ {\rm Int}_{{\boldsymbol{G}}_{k'}}(\gamma').
$$
Pour $k'=-k$, on dispose d'un $F$--morphisme \og inverse \fg
$$
{\boldsymbol{G}}_k\rightarrow {\boldsymbol{G}}_{-k},\, \gamma \mapsto \gamma^{-1}
$$
donnŽ par
$$
\varphi_{k,-k}(\gamma,\gamma^{-1})=1_G.
$$
On a donc ${\rm Int}_{{\boldsymbol{G}}_{-k}}(\gamma^{-1})={\rm Int}_{{\boldsymbol{G}}_k}(\gamma)^{-1}$. 
Ces donnŽes dŽfinissent un $F$--schŽma en groupes lisse de composantes 
connexes les ${\boldsymbol{G}}_k$, dont $\mathcal{G}$ est le groupe des points $F$--rationnels.

Fixons\index{$\delta_1$, $\theta={\rm Int}_{\boldsymbol{G}^\natural}(\delta_1)$, $\omega$} un point--base $\delta_1\in G^\natural$, 
et notons $\theta$ le $F$--automorphisme ${\rm Int}_{{\boldsymbol{G}}^\natural}(\delta_1)$ de ${\boldsymbol{G}}^\natural$. Le sous--ensemble 
$G\theta=G\rtimes \theta$ de $G\rtimes {\rm Aut}_F(\boldsymbol{G})$ est naturellement muni d'une structure de $G$--espace topologique tordu (cf. la remarque 1 
de \cite[3.4]{L2}); ici ${\rm Aut}_F(\boldsymbol{G})$ dŽsigne le groupe des $F$--automorphismes algŽbriques de $G$. On peut bien s\^ur 
identifier $G^\natural$ ˆ $G\theta$ via l'application 
$g\cdot \delta_1\mapsto g\theta$, mais on pr\'ef\`ere ne pas le faire car cette identification n'est 
pas canonique (en gŽnŽral elle dŽpend du choix de $\delta_1$, cf. la remarque 3 de loc.~cit.). 

On fixe aussi 
un caract\`ere $\omega$ de $G={\boldsymbol{G}}(F)$, \cad un morphisme continu de $G$ dans ${\Bbb C}^\times$. 
Notons $C({\boldsymbol{G}})$\index{$Z({\boldsymbol{G}})$, $C({\boldsymbol{G}})$} 
la composante neutre du centre $Z({\boldsymbol{G}})$ de ${\boldsymbol{G}}$. C'est un tore 
d\'efini sur $F$. 

\begin{marema}{\rm La restriction de $\theta$ \`a $C({\boldsymbol{G}})$ ne dŽpend pas du choix de $\delta_1$, 
et on ne suppose pas qu'elle est d'ordre fini. 
En d'autres termes, on ne suppose pas que ${\boldsymbol{G}}^\natural$ est isomorphe \`a une composante connexe 
d'un groupe alg\'ebrique affine. On 
ne suppose pas non plus que $\omega$ est unitaire.\hfill $\blacksquare$ }
\end{marema}

\subsection{$\omega$--reprŽsentations de $G^\natural$}\label{rappelsrep}
Pour un groupe topologique totalement discontinu $H$, on appelle {\it reprŽsentation de $H$}, ou {\it $H$--module}, 
une reprŽsentation lisse de $H$ ˆ valeurs dans le groupe des 
automorphismes d'un espace vectoriel complexe. Les reprŽsentations de $H$ forment une catŽgorie 
abŽlienne, notŽe $\mathfrak{R}(H)$. On note ${\rm Irr}(H)$ l'ensemble des classes d'isomorphisme de reprŽsentations 
irrŽductibles de $H$.\index{$\mathfrak{R}(H)$, ${\rm Irr}(H)$}

On s'int\'eresse aux repr\'esentations 
$\pi$ de $G$ telles que $\omega\pi = \omega\otimes \pi$ 
est isomorphe \`a $\pi^\theta =\pi\circ\theta$. Si $\pi$ est irr\'eductible, alors d'apr\`es le lemme de Schur,  
l'espace ${\rm Hom}_G(\omega\pi,\pi^\theta)$ des op\'erateurs d'entrelacement entre $\omega\pi$ et $\pi^\theta$ 
est de dimension $1$, mais en g\'en\'eral il n'y a pas de vecteur privil\'egi\'e dans cet espace --- sauf si le 
groupe ${\boldsymbol{G}}$ est quasi--d\'eploy\'e sur $F$, mais m\^eme dans ce cas il faut faire des choix. 
On a donc int\'er\^et \`a travailler dans 
une cat\'egorie de repr\'esentations englobant cet espace, par exemple celle des $\omega$--repr\'esentations 
(lisses) de $G^\natural$ introduite dans \cite[2.6]{L2}.

Une {\it $\omega$--repr\'esentation de $G^\natural$} --- ou {\it $(G^\natural,\omega)$--module} ---, est la donn\'ee d'une repr\'esentation $(\pi,V)$ de $G$ et d'une application 
$\Pi:G^\natural\rightarrow {\rm Aut}_{\Bbb C}(V)$ telle que, pour tout $\delta\in G^\natural$ et tous  $x,\,y\in G$, on a
$$
\Pi(x\cdot \delta\cdot y)=\omega(y)\pi(x)\circ\Pi(\delta)\circ \pi(y).
$$
Pour $x\in G$ et $\delta\in G^\natural$, on a
$$
\pi(x)=\Pi(x\cdot\delta)\circ \Pi(\delta)^{-1}= \omega(x)^{-1}\Pi(\delta)^{-1}\circ \Pi(\delta\cdot x).
$$
La repr\'esentation $\pi$ est d\'etermin\'ee par $\Pi$, et not\'ee $\Pi^\circ$ comme dans loc.~cit. 
Remarquons que 
l'op\'erateur $A=\Pi(\delta_1)$ est un isomorphisme de $\omega\pi$ sur $\pi^\theta$. 
L'espace d'une $\omega$--reprŽsentation $\Pi$ de $G^\natural$, \cad celui 
de la reprŽsentation $\Pi^\circ$ de $G$ sous--jacente, est notŽ $V_\Pi=V_{\Pi^\circ}$.

Les $\omega$--repr\'esentations de $G^\natural$ s'organisent naturellement en une cat\'egorie, notŽe  
$\mathfrak{R}(G^\natural,\omega)$.\index{$\mathfrak{R}(G^\natural,\omega)$} Un morphisme $u$ entre 
deux $\omega$--reprŽsentations $\Pi$ et $\Pi'$ de $G^\natural$ est une application ${\Bbb C}$--linŽaire 
$u:V_\Pi\rightarrow V_{\Pi'}$ telle que $u\circ \Pi(\delta)=\Pi'(\delta)\circ u$ pour tout $\delta\in G^\natural$ --- de manire Žquivalente, $u$ 
est un morphisme entre $\Pi^\circ$ et $\Pi'^\circ$ tel que $u\circ \Pi(\delta_1)=\Pi'(\delta_1)\circ u$. 
L'application $\Pi\mapsto \Pi^\circ$ d\'efinit un foncteur d'oubli 
de $\mathfrak{R}(G^\natural,\omega)$ dans $\mathfrak{R}(G)$, et ce foncteur est fid\`ele. 
Notons que s'il existe une $\omega$--reprŽsentation $\Pi$ de $G^\natural$ telle que $\Pi^\circ$ est irrŽductible, alors 
le caract\`ere 
$\omega$ est trivial sur le centre $Z^\natural= Z(G^\natural)$\index{$Z^\natural= Z(G^\natural)$} de $G^\natural$, d\'efini par
$$
Z^\natural=\{z\in Z(G):\theta(z)=z\}.
$$
En d'autres termes, si $\omega\vert_{Z^\natural}\neq 1$ alors la thŽorie qui nous intŽresse ici est vide.

On a des notions Žvidentes de sous--$\omega$--repr\'esentation (resp. de $\omega$--repr\'esentation quotient) d'une $\omega$--repr\'esentation 
de $G^\natural$, et de suite exacte courte de $\omega$--repr\'esentations de $G^\natural$ (cf. loc.~cit.). Si $u$ est 
un morphisme entre deux $\omega$--reprŽsentations $\Pi$ et $\Pi'$ de $G^\natural$, 
le noyau $\ker u$ et l'image ${\rm Im}\,u$ sont des sous--$\omega$--reprŽsentations de $\Pi$ et $\Pi'$ respectivement, et 
l'on a la suite exacte courte de $\omega$--reprŽsentations de $G^\natural$:
$$
0\rightarrow \ker u \rightarrow \Pi \rightarrow \Pi'/{\rm Im}\,u\rightarrow 0.
$$
Cela fait de $\mathfrak{R}(G^\natural,\omega)$ une catŽgorie ab\'elienne. 

Une $\omega$--repr\'esentation $\Pi$ de $G^\natural$ est dite {\it irr\'eductible} si $V_\Pi$ est l'unique sous--espace non nul $G^\natural$--stable de $V_\Pi$, et 
{\it $G$--irrŽductible} (ou {\it fortement irr\'eductible} \cite{L2}) si la repr\'esentation $\Pi^\circ$ de $G$ est irr\'eductible. On note 
${\rm Irr}(G^\natural,\omega)$\index{${\rm Irr}(G^\natural,\omega)$, ${\rm Irr}_0(G^\natural,\omega)$, 
${\rm Irr}_{G^\natural,\omega}(G)$} l'ensemble des classes d'isomorphisme de $\omega$--repr\'esentations irr\'eductibles de $G$, 
et ${\rm Irr}_0(G^\natural,\omega)$ le sous--ensemble de ${\rm Irr}(G^\natural,\omega)$ form\'e des $\omega$--repr\'esentations 
qui sont $G$--irr\'eductibles. 

On a une action naturelle de ${\Bbb C}^\times$ sur ${\rm Irr}(G^\natural,\omega)$, not\'ee
$$
{\Bbb C}^\times \times {\rm Irr}(G^\natural,\omega)\rightarrow {\rm Irr}(G^\natural,\omega),\,(\lambda,\Pi)\mapsto \lambda\cdot \Pi.
$$
Cette action stabilise ${\rm Irr}_0(G^\natural,\omega)$, et le 
foncteur d'oubli $\Pi\mapsto \Pi^\circ$ induit une application injective
$$
{\rm Irr}_0(G^\natural,\omega)/{\Bbb C}^\times\hookrightarrow {\rm Irr}(G)
$$
d'image le 
sous--ensemble ${\rm Irr}_{G^\natural,\omega}(G)$ de ${\rm Irr}(G)$ form\'e des $\pi$ tels que 
$\omega^{-1}\pi^\theta= \pi$.

\subsection{Les reprŽsentations $\pi(k)$ pour $k\in {\Bbb Z}$}\label{pi(k)}

Si $\pi$ est une repr\'esentation de $G$, pour chaque entier $k\geq 1$, 
on note $\omega_k$ le caract\`ere de $G$ d\'efini par\index{$\omega_k$, $ \N_{\theta,k}$, $\pi(k)$}
$$
\omega_k= \omega\circ \N_{\theta,k},\quad
\N_{\theta,k}(x)=
x\theta(x)\cdots \theta^{k-1}(x),\quad
x\in G,
$$
et $\pi(k)$ la repr\'esentation de $G$ d\'efinie par
$$
\pi(k)=\omega_k^{-1}\pi^{\theta^k}.
$$ 
Le caract\`ere $\omega_k$ ne d\'epend pas du choix du 
choix du point--base $\delta_1$, alors que (contrairement ˆ ce que la notation pourrait faire croire) 
la repr\'esentation $\pi(k)$ 
en d\'epend. Pour $k,\,k'\geq 1$, on a $\pi(k)(k')=\pi(k+k')$. Pour chaque entier $k\geq 1$, notons 
$\pi(-k)$ la repr\'esentation de $G$ telle que $\pi(-k)(k)=\pi$. 
PrŽcisŽment, on a $\pi(-k)= \omega_{-k}^{-1}\pi^{\theta^{-k}}$ o $\omega_{-k}$ est le caractre de 
$G$ dŽfini par
$$
\omega_{-k}= \omega_k^{-1}\circ \theta^{-k}=\omega^{-1}\circ \theta^{-1}\circ \N_{\theta^{-1},k}.
$$
On vŽrifie que $\pi(k)(-k)=\pi$. Par suite posant $\pi(0)=\pi$, on a
$$
\pi(k)(k')=\pi(k+k'),\quad k,\,k'\in {\Bbb Z}.
$$

On l'a dit plus haut, la repr\'esentation $\pi(k)$ dŽpend du choix du point--base 
$\delta_1$. En effet, remplacer $\delta_1$ par $\delta'_1=x\cdot \delta_1$ pour un $x\in G$ revient \`a remplacer $\theta$ 
par le $F$--automorphisme $\theta'={\rm Int}_{\boldsymbol{G}}(x)\circ \theta$ de ${\boldsymbol{G}}$, et pour chaque 
entier $k \geq 1$, \`a remplacer $\pi(k)$ par
$$
\omega_k^{-1}\pi^{\theta'^k}= \omega_k^{-1}\otimes (\pi\circ {\rm Int}_{\boldsymbol{G}}(\N_{\theta,k}(x))\circ \theta^k),
$$ 
qui est isomorphe \`a $\pi(k)$, et $\pi(-k)$ par
$$
\omega_{-k}^{-1}\pi^{\theta'^{-k}}= \omega_{-k}^{-1}\otimes (\pi\circ \theta^{-k}\circ {\rm Int}_{\boldsymbol{G}}(\N_{\theta,k}(x)^{-1})),
$$
qui est isomorphe \`a $\pi(-k)$. Pour $k\in {\Bbb Z}$, $\pi(k)$ ne d\'epend donc \`a isomorphisme pr\`es 
que de $G^\natural$, de $\omega$, et 
de la classe d'isomorphisme de 
$\pi$.

\subsection{Le foncteur $\iota_k$ pour $k\geq 1$}\label{iota}\index{$\iota^k$}
Soit un entier $k\geq 1$. Pour $\delta\in {\boldsymbol{G}}^\natural$, on dŽfinit comme en \ref{pi(k)} une 
application $\N_{\delta,k}=\N_{\tau,k}:{\boldsymbol{G}}\rightarrow {\boldsymbol{G}}$, $\tau={\rm Int}_{\boldsymbol{G}}(\delta)$: 
pour $x\in {\boldsymbol{G}}$, on pose
$$
\N_{\delta,k}(x)= x \tau(x)\cdots \tau^{k-1}(x).
$$
L'application $\N_{\delta,k}$ ainsi dŽfinie est un 
morphisme de variŽtŽs algŽbriques, et si $\tau$ est dŽfini sur $F$ (e.g. si $\delta\in G^\natural$) 
alors $\N_{\delta,k}$ l'est aussi.

Le ${\boldsymbol{G}}$--espace algŽbrique tordu ${\boldsymbol{G}}^\natural$ est muni d'un $F$--morphisme de variŽtŽs algŽbriques 
$$
{\boldsymbol{G}}^\natural\rightarrow {\boldsymbol{G}}_k,\, \delta \mapsto \delta^k
$$
dŽfini comme suit: on pose $\delta^1=\delta$ et $\delta^k= \varphi_{1,k-1}(\delta,\delta^{k-1})$ si $k>1$. 
Pour $\delta\in {\boldsymbol{G}}^\natural$ et $g\in {\boldsymbol{G}}$, on a
$$
(g\cdot \delta)^k= \N_{\delta,k}(g)\cdot \delta^k,
$$
et si $k>1$, on a
$$
\delta^{k-1} = \varphi_{k,-1}(\delta^k,\delta^{-1})= \varphi_{-1,k}(\delta^{-1},\delta^k).
$$
Le choix d'un point--base $\delta_1$ de $G^\natural$ fournit un point--base $\delta_k$ de 
$\mathcal{G}_k$: on pose\index{$\delta_k=(\delta_1)^k$}
$$
\delta_k=(\delta_1)^k.
$$

On d\'efinit comme suit un foncteur
$$
\iota_k:\mathfrak{R}(\mathcal{G}_k,\omega_k)\rightarrow \mathfrak{R}(G^\natural,\omega).
$$
Pour une $\omega_k$--repr\'esentation $\Sigma$ de $\mathcal{G}_k$, on note $\iota_k(\Sigma)=\Pi$ 
la $\omega$--repr\'esentation de $G^\natural$ d\'efinie par:

\begin{itemize}
\item la repr\'esentation $\Pi^\circ$ de $G$ 
sous--jacente \`a $\Pi$ est $\oplus_{i=0}^{k-1}\Sigma^\circ(i)$,
\item $\Pi(\delta_1)(v_0,\ldots ,v_{k-1})=(v_1,\ldots ,v_{k-1},\Sigma(\delta_k)(v_0))$.
\end{itemize}
\ni Pour un morphisme $u$ entre deux $\omega_k$--repr\'esentations $\Sigma$ et $\Sigma'$ de 
$\mathcal{G}_k$, on note $\iota_k(u)$ le morphisme $u\times\cdots \times u$ entre $\iota_k(\Sigma)$ 
et $\iota_k(\Sigma')$. 

Pour $k=1$, on a $\mathcal{G}_1=G^\natural$ et $\iota_1$ est le foncteur identique de $\mathfrak{R}(G^\natural,\omega)$. 
Notons que pour $k>1$, la $\omega$--reprŽsentation $\iota_k(\Sigma)$ dŽpend du choix de $\delta_1$, mais sa classe 
d'isomorphisme n'en dŽpend pas.

\begin{marema}\label{iota(k,k')}\index{$\iota_k^{k'}$}
{\rm Pour des entiers $k,\,k'\geq 1$ tels que $k'$ divise $k$, on dŽfinit de la m\^eme manire 
un $F$--morphisme de variŽtŽs algŽbriques
$$
{\boldsymbol{G}}_{k'}\rightarrow {\boldsymbol{G}}_k,\,\delta\mapsto \delta^{k/ k'}
$$
et --- gr\^ace aux points--base $\delta_{k'}$ de $\mathcal{G}_{k'}$ et $\delta_k = (\delta_{k'})^{k/k'}$ de $\mathcal{G}_k$ 
--- un foncteur
$$
\iota_k^{k'}:\mathfrak{R}(\mathcal{G}_k,\omega_k)\rightarrow \mathfrak{R}(\mathcal{G}_{k'},\omega_{k'}).
$$
Pour $k,\,k'\!,\, k''\geq 1$ tels que $k'$ divise $k$ et $k''$ divise $k'$, on a
$$
\iota_{k'}^{k''}\circ\iota_k^{k'}= \iota_k^{k''}.
$$
En particulier pour $k''=1$, on a $\iota_{k'}\circ \iota_k^{k'}= \iota_k$.
\hfill $\blacksquare$
}
\end{marema}

\subsection{L'invariant $s(\Pi)$}\label{s(Pi)}\index{$s(\pi)=s_{G^\natural,\omega}(\pi)$, $s(\Pi)$}
Si $\pi$ est une reprŽsentation irrŽductible de $G$, on lui associe 
comme suit un invariant $s(\pi)\in {\Bbb Z}_{\geq 1}\cup \{+\infty\}$. S'il existe un plus petit entier 
$k_0\geq 1$ tel que $\pi(k_0)$ est isomorphe \`a $\pi$, on pose $s(\pi)=k_0$, sinon 
on pose $s(\pi)=+\infty$. Notons que cet invariant $s(\pi)$ ne dŽpend que $G^\natural$, de 
$\omega$, et de la classe d'isomorphisme de $\pi$. Si nŽcessaire, on le notera aussi $s_{G^\natural,\omega}(\pi)$. 
Pour $k\in {\Bbb Z}$, on a
$$s(\pi(k))=s(\pi).
$$

Si $\Pi$ est une $\omega$--repr\'esentation irr\'eductible de $G^\natural$, on lui associe 
comme dans \cite[8.4]{L2} un invariant $s(\Pi)\in {\Bbb Z}_{\geq 1}\cup\{+\infty\}$. 
Rappelons la construction. On choisit une sous--repr\'esentation irr\'eductible $\pi_0$ de $\Pi^\circ$, 
et l'on pose
$$s(\Pi)=s(\pi_0).
$$ L'invariant $s(\Pi)$ est bien d\'efini (i.e. il ne d\'epend 
pas du choix de $\pi_0$), et il d\'epend seulement de 
la classe d'isomorphisme de $\Pi$. 
On a
$$
\Pi^\circ \simeq \left\{
\begin{array}{ll}
\oplus_{k\in {\Bbb Z}}\,\pi_0(k) & \mbox{si $s(\pi_0)=+\infty$}\\
\oplus_{k=0}^{s(\pi_0)-1}\pi_0(k)&\mbox{sinon}
\end{array}
\right..
$$
En particulier la reprŽsentation 
$\Pi^\circ$ est semisimple, et elle est de type fini si et seulement si 
$s(\Pi)<+\infty$, auquel cas elle est de longueur finie. 
Si $s=s(\Pi)<+\infty$, alors d'apr\`es loc.~cit., il existe une $\omega_s$--repr\'esentation 
$G$--irr\'eductible $\Sigma$ de $\mathcal{G}_s$ telle que $\Sigma^\circ =\pi_0$ et $\iota_s(\Sigma)$ est isomorphe \`a $\Pi$.

Pour $k\in {\Bbb Z}_{\geq 1}$, notons ${\rm Irr}_{k-1}(G^\natural,\omega)$\index{${\rm Irr}_{k-1}(G^\natural,\omega)$, 
${\rm Irr}'_0(\mathcal{G}_k,\omega_k)$} le sous--ensemble 
de ${\rm Irr}(G^\natural,\omega)$ form\'e des $\Pi$ tels que $s(\Pi)=k$ (pour $k=1$, les notations sont coh\'erentes), 
et ${\rm Irr}'_0(\mathcal{G}_k,\omega_k)$ le 
sous--ensemble de ${\rm Irr}_0(\mathcal{G}_k,\omega_k)$ 
form\'e des $\Sigma$ tels que $s(\Sigma^\circ)=k$. Ici ${\rm Irr}_0(\mathcal{G}_k,\omega_k)$ est l'ensemble des classes 
d'isomorphisme de $\omega_k$--reprŽsentations $G$--irrŽductibles de $\mathcal{G}_k$, et $s(\Sigma^\circ)=s_{G^\natural,\omega}(\Sigma^\circ)$. 
Les $\omega_k$--reprŽsentations $G$--irrŽductibles de $\mathcal{G}_k$ dont 
la classe d'isomorphisme appartient ˆ ${\rm Irr}'_0(\mathcal{G}_k,\omega_k)$ sont 
dites {\it $(G^\natural,\omega)$--rŽgulires}. 

\begin{marema}\label{s(Sigma)}
{\rm On peut, pour toute repr\'esentation 
irr\'eductible $\sigma$ de $G$, d\'efinir l'invariant $s_k(\sigma)=s_{\mathcal{G}_k,\omega_k}(\sigma)\in {\Bbb Z}_{\geq 1}\cup\{+\infty\}$\index{$s(\sigma)=s_{\G_k,\omega_k}(\sigma)$} 
en rempla\c{c}ant dans la d\'efinition de $s(\sigma)$ la paire $(G^\natural,\omega)$ par la paire $(\mathcal{G}_k,\omega_k)$. 
On a $s_k(\sigma)=+\infty$ si $s(\sigma)=+\infty$, et
$s_k(\sigma)=\inf \{i\in {\Bbb Z}_{\geq 1}: \sigma(ki)\simeq \sigma(k)\}$ sinon. En d'autre termes, on a
$$
s_k(\sigma)={1\over k}{\rm ppcm}(k,s(\sigma))
$$
Pour une $\omega_k$--repr\'esentation irr\'eductible $\Sigma$ de $\mathcal{G}_k$, l'invariant 
$s(\Sigma)$ associ\'e comme plus haut \`a $\Sigma$ est donn\'e par
$$
s(\Sigma)=s_k(\sigma_0)
$$
pour une (resp. pour toute) sous--repr\'esentation irr\'eductible $\sigma_0$ de $\Sigma^\circ$. Ainsi 
$\Sigma$ est $G$--irr\'eductible si et seulement si $s(\Sigma)=1$.\hfill $\blacksquare$}
\end{marema}

On dŽfinit comme suit une action de ${\Bbb Z}_k={\Bbb Z}/k{\Bbb Z}$\index{${\Bbb Z}_k={\Bbb Z}/k{\Bbb Z}$} 
sur ${\rm Irr}(\mathcal{G}_k,\omega_k)$. 
Rappelons que l'on a posŽ $\delta_k=(\delta_1)^k\in \mathcal{G}_k$.
Pour 
une $\omega_k$--repr\'esentation $\Sigma$ de $\mathcal{G}_k$ et un entier $i\geq 1$, on note $\Sigma(i)$ la 
$\omega_k$--repr\'esentation de $\mathcal{G}_k$ donn\'ee par
$$
\Sigma(i)(g\cdot \delta_k)=\Sigma^\circ(i)(g)\circ \Sigma(\delta_k),\quad g\in G.
$$
La repr\'esentation de $G$ sous--jacente est $\Sigma(i)^\circ=\Sigma^\circ(i)$, et \`a isomorphisme pr\`es, $\Sigma(i)$ 
ne d\'epend pas du choix de $\delta_1$. Comme ${\rm Int}_{{\boldsymbol{G}}_k}(\delta_k)=\theta^k$, la 
$\omega_k$--repr\'esentation $\Sigma(k)$ de $\mathcal{G}_k$ est isomorphe \`a $\Sigma$. 
On obtient ainsi une action de ${\Bbb Z}_k$ sur ${\rm Irr}(\mathcal{G}_k,\omega_k)$ 
qui stabilise ${\rm Irr}_0(\mathcal{G}_k,\omega_k)$, et 
${\rm Irr}'_0(\mathcal{G}_k,\omega_k)$ est le sous--ensemble de ${\rm Irr}_0(\mathcal{G}_k,\omega_k)$ 
form\'e des $\Sigma$ dont le stabilisateur sous ${\Bbb Z}_k$ est trivial. Le lemme suivant 
est une simple gŽnŽralisation de \cite[lemma 2.1]{R}.

\begin{monlem}\label{bijection iota}
Le foncteur $\iota_k$ induit une application bijective
$$
{\rm Irr}'_0(\mathcal{G}_k,\omega_k)/{\Bbb Z}_k\rightarrow {\rm Irr}_{k-1}(G^\natural,\omega).
$$
\end{monlem}

\begin{proof} Soit $\Pi$ une $\omega$--reprŽsentation irrŽductible de $G^\natural$ d'invariant $s(\Pi)=k$. 
On a vu que pour toute sous--reprŽsentation irrŽductible $\pi_0$ de $\Pi^\circ$, il 
existe une $\omega_k$--reprŽsentation $G$--irrŽductible $\Sigma$ de $\mathcal{G}_k$ telle que 
$\Sigma^\circ =\pi_0$ et $\Pi \simeq \iota_k(\Sigma)$. Par dŽfinition de $s(\Pi)$, les 
reprŽsentations $\pi_0(i)$ de $G$, $i=0,\dots ,k-1$, sont deux--ˆ--deux non Žquivalentes. Par suite les 
$\omega_k$--reprŽsentations $\Sigma(i)$ de $\mathcal{G}_k$, $i=0,\ldots ,k-1$, sont deux--ˆ--deux 
non Žquivalentes. Elles dŽfinissent donc un ŽlŽment de ${\rm Irr}'_0(\mathcal{G}_k,\omega_k)/{\Bbb Z}_k$. 
RŽciproquement, soit $\Sigma'$ une $\omega_k$--reprŽsentation $G$--irrŽductible de $\mathcal{G}_k$ 
dont le stabilisateur sous ${\Bbb Z}_k$ est trivial, telle que $\iota_k(\Sigma')\simeq \Pi$. Puisque 
$\iota_k(\Sigma')^\circ \simeq \oplus_{i=0}^{k-1}\Sigma'(i)^\circ$ et $\Pi^\circ\simeq \oplus_{i=0}^{k-1}\Sigma(i)^\circ$, 
il existe un indice $j\in \{0,\ldots ,k-1\}$ tel que $\Sigma'^\circ \simeq \Sigma(j)^\circ$. On en dŽduit que $\Sigma'$ 
est isomorphe ˆ $\lambda\cdot \Sigma(j)$ pour un nombre complexe non nul $\lambda$, mais comme 
$\iota_k(\lambda\cdot \Sigma(j))=\lambda\cdot \iota_k(\Sigma(j))\simeq \lambda \cdot \iota_k(\Sigma)$, ce 
$\lambda$ vaut $1$. D'o le lemme. \end{proof}

\subsection{L'espace  $\G_{\Bbb C}(G^\natural,\omega)$.}\label{groth}L'action de ${\Bbb C}^\times$ sur 
${\rm Irr}(G^\natural,\omega)$ provient d'une action fonctorielle sur $\mathfrak{R}(G^\natural,\omega)$, 
triviale sur les fl\`eches, 
encore notŽe $(\lambda,\Pi)\mapsto \lambda\cdot \Pi$. Soit $\G_{\Bbb C}(G^\natural,\omega)$\index{$\G_{\Bbb C}(G^\natural,\omega)$} 
le ${\Bbb C}$--espace vectoriel engendrŽ (sur ${\Bbb C}$) 
par les $\omega$--reprŽsentations $\Pi$ de $G^\natural$ telles que $\Pi^\circ$ est de longueur finie, 
modulo les relations:

\begin{itemize}
\item pour toute suite exacte $0\rightarrow \Pi_1\rightarrow\Pi_2\rightarrow \Pi_3\rightarrow 0$ de $\omega$--reprŽsentations 
de $G^\natural$ (telles que les $\Pi_i^\circ$ sont de longueur finie), 
on a $\Pi_3= \Pi_2 - \Pi_1$;
\item pour tout $\lambda\in {\Bbb C}Ñ^\times$, on a $\lambda\cdot \Pi=\lambda\Pi$;
\item pour tout entier $k>1$, on a $\iota_k(\mathcal{G}_k,\omega_k)=0$.
\end{itemize}

\ni Pour une $\omega$--reprŽsentation $\Pi$ de $G^\natural$ et un nombre complexe non nul $\lambda$, on a 
$(\lambda\cdot\Pi)^\circ=\Pi^\circ$ mais $\lambda\cdot \Pi\not\simeq \Pi$ si $\lambda\neq 1$. 
La deuxime relation signifie que si $\lambda_1,\ldots ,\lambda_n$ sont des nombres complexes non nuls 
tels que $\sum_{i=1}^n\lambda_i =0$, alors pour toute $\omega$--reprŽsentation $\Pi$ de $G^\natural$ 
telle que $\Pi^\circ$ est de longueur finie, on a $\sum_{i=1}^n\lambda_i\Pi=0$ dans $\G_{\Bbb C}(G^\natural,\omega)$. 

\v1
Notons ${\rm Irr}_{<+\infty}(G^\natural,\omega)$\index{${\rm Irr}_{<+\infty}(G^\natural,\omega)$} 
le sous--ensemble des ${\rm Irr}(G^\natural,\omega)$ formŽ des 
$\Pi$ tels que $s(\Pi)<+\infty$. On a donc
$$
{\rm Irr}_{<+\infty}(G^\natural,\omega)=\coprod_{k\geq 0}{\rm Irr}_k(G^\natural,\omega),
$$
et l'action de ${\Bbb C}^\times$ sur ${\rm Irr}(G^\natural,\omega)$ stabilise chacun des espaces 
${\rm Irr}_k(G^\natural,\omega)$. D'autre part on a une action de ${\Bbb Z}$ sur 
${\rm Irr}(G)$, donn\'ee par $(k,\pi)\mapsto \pi(k)$. D'apr\`es \ref{s(Pi)}, l'application 
$\Pi\mapsto \pi_0$, o\`u $\pi_0$ est sous--repr\'esentation irr\'eductible de $\Pi^\circ$, 
induit une application injective
$$
{\rm Irr}_{<+\infty}(G^\natural,\omega)/{\Bbb C}^\times\rightarrow {\rm Irr}(G)/{\Bbb Z}
$$
d'image l'ensemble des ${\Bbb Z}$--orbites des $\pi$ dans ${\rm Irr}(G)$ tels que $s_{G^\natural,\omega}(\pi)<+\infty$. 

On note:\index{$\EuScript{G}(G^\natural,\omega)$, $\EuScript{G}_0(G^\natural,\omega)$, $\EuScript{G}_{>0}(G^\natural,\omega)$}

\begin{itemize}
\item $\EuScript{G}(G^\natural,\omega)$ 
le ${\Bbb Z}$--module libre de base ${\rm Irr}^\omega_{<+\infty}(G^\natural)$, 
\item $\EuScript{G}_0(G^\natural,\omega)$ 
le sous--groupe de $\EuScript{G}(G^\natural,\omega)$ engendr\'e par ${\rm Irr}_0(G^\natural,\omega)$, 
\item $\EuScript{G}_{>0}(G^\natural,\omega)$ le sous--groupe de $\EuScript{G}(G^\natural,\omega)$ engendr\'e par 
$\coprod_{k\geq 1}{\rm Irr}_k(G^\natural,\omega)$.
\end{itemize}
On a donc la dŽcomposition
$$
\G(G^\natural,\omega)=\G_0(G^\natural,\omega)\oplus \G_{>0}(G^\natural,\omega)
$$
Soit aussi $\EuScript{G}(G)$ le ${\Bbb Z}$--module libre de base ${\rm Irr}(G)$. Le foncteur d'oubli $\Pi\mapsto \Pi^\circ$ 
induit un morphisme de groupes
$$
\EuScript{G}(G^\natural,\omega)\rightarrow \EuScript{G}(G),
$$
encore notŽ $\Pi\mapsto \Pi^\circ$.

On peut aussi, pour chaque entier $k\geq 1$, remplacer la paire $(G^\natural,\omega)$ par la paire 
$(\mathcal{G}_k,\omega_k)$ dans les d\'efinitions ci--dessus (cf. \ref{s(Sigma)}, remarque). Le foncteur 
$\iota_k^\circ:\mathfrak{R}(G)\rightarrow \mathfrak{R}(G)$ sous--jacent \`a $\iota_k$ 
envoie repr\'esentation de longueur finie sur repr\'esentation de longueur finie, par cons\'equent $\iota_k$ induit 
un morphisme de groupes
$$
\G(\mathcal{G}_k,\omega_k)\rightarrow \G(G^\natural,\omega),
$$
encore not\'e $\iota_k$. 

Le quotient $\G_0(G^\natural,\omega)=\G(G^\natural,\omega)/\G_{>0}(G^\natural,\omega)$ est encore 
trop gros: il contient des \'el\'ements qui ne contribuent en rien \`a l'affaire 
qui nous int\'eresse (cf. \ref{carac}). Soit donc $\G_{0^+}(G^\natural,\omega)$\index{$\G_{0^+}(G^\natural,\omega)$} 
le sous--groupe de $\G(G^\natural,\omega)$ engendr\'e 
par $\G_{>0}(G^\natural,\omega)$ et par les \'el\'ements de la forme $
\sum_{i=1}^{n}\lambda_i\cdot \Pi$ pour un ŽlŽment $\Pi$ de ${\rm Irr}_0(G^\natural,\omega)$, un entier $n>1$, 
et des nombres complexes non nuls $\lambda_1,\ldots ,\lambda_n$ tels que $\sum_{i=1}^n\lambda_i=0$.

\begin{monlem}\label{inclusion iota +}Pour tout entier $k >1$, on a l'inclusion
$$
\iota_k(\G(\mathcal{G}_k,\omega_k))\subset \G_{0^+}(G^\natural,\omega).
$$
\end{monlem}

\begin{proof}Il suffit de montrer que pour toute $\omega_k$--reprŽsentation 
irrŽductible $\Sigma$ de $\mathcal{G}_k$, la classe d'isomorphisme de $\iota_k(\Sigma)$ 
appartient ˆ $\G_{0^+}(G^\natural,\omega)$. 
D'aprs la remarque de \ref{iota} et le lemme de \ref{s(Sigma)}, il existe un entier $a\geq 1$ et une 
$\omega_{ka}$--reprŽsentation $G$--irrŽductible $\Sigma'$ de $\mathcal{G}_{ka}$ tels que 
$\Sigma$ est isomorphe ˆ $\iota_{ka}^k(\Sigma')$. Par suite $\iota_k(\Sigma)$ est isomorphe ˆ 
$\iota_k(\iota_{ka}^k(\Sigma'))= \iota_{ka}(\Sigma')$, et quitte ˆ remplacer $k$ par $ka$ et $\Sigma$ 
par $\Sigma'$, on peut supposer que $\Sigma$ est $G$--irrŽductible. Soit alors $\sigma=\Sigma^\circ$ 
et $s=s(\sigma)$. Puisque $\sigma(k)\simeq \sigma$, $s$ divise $k$. 
Si $s=1$, alors d'aprs le lemme de \ref{s(Sigma)}, $\iota_k(\Sigma)$ est une $\omega$--reprŽsentation irrŽductible 
de $G$ d'invariant $s(\iota_k(\Sigma))=k$, et son image dans $\G(G^\natural,\omega)$ appartient 
ˆ $\G_{>0}(G^\natural,\omega)$. On peut donc supposer $s>1$. Posons $\Delta=\iota_k^s(\Sigma)$. 
C'est une $\omega_s$--reprŽsentation de $\mathcal{G}_s$, telle que
$$
\Delta^\circ =\sigma \oplus \sigma(s)\oplus \cdots \oplus \sigma((k'-1)s),\, k'=k/s.
$$
Choisissons un isomorphisme $\widetilde{B}$ de $\sigma$ sur $\sigma(s)$. Alors $\widetilde{B}^{k'}$ 
est un isomorphisme de $\sigma$ sur $\sigma(k)$, et l'on peut choisir $\widetilde{B}$ de telle manire 
que $\widetilde{B}^{k'}= \Sigma(\delta_k)$. Notons $\widetilde{\Sigma}$ la 
$\omega_s$--reprŽsentation ($G$--irrŽductible) de $\mathcal{G}_s$ dŽfinie par $\widetilde{\Sigma}^\circ =\sigma$ et 
$\widetilde{\Sigma}(\delta_s)= \widetilde{B}$. Soit $\mu$ une racine primitive $k'$--ime de l'unitŽ 
(dans ${\Bbb C}^\times$). Posons $\Delta'=\oplus_{i=0}^{k'-1}\mu^i\cdot \widetilde{\Sigma}$. C'est une 
$\omega_s$--reprŽsentation de $\mathcal{G}_s$, telle que $\Delta'^\circ = \oplus_{i=0}^{k'-1}\sigma$. 
Pour $j=0,\ldots ,k'-1$, notons $V_{\Delta,j}$ le sous--espace vectoriel de $V_\Delta=\oplus_{i=0}^{k'-1} V_\sigma$ formŽ 
des vecteurs de la forme
$$
(v,\mu^j\widetilde{B}(v),\mu^{2j}\widetilde{B}^2(v),\ldots , \mu^{(k'-1)j}\widetilde{B}^{k'-1}(v)),\quad v\in V_\sigma.
$$
Il est stable sous l'action de $G$ (via $\Delta^\circ$) et sous celle de $\Delta(\delta_s)$, 
donc dŽfinit une sous--$\omega_{\rm s}$--reprŽsentation 
de $\Delta$, que l'on note $\Delta_j$. L'application
$$
V_\sigma\rightarrow V_{\Delta,j},\, v\mapsto (v,\mu^j\widetilde{B}(v),\mu^{2j}\widetilde{B}^2(v),\ldots , \mu^{(k'-1)j}\widetilde{B}^{k'-1}(v))
$$
est un isomorphisme de $\mu^j\cdot \widetilde{\Sigma}$ sur $\Delta_j$. 
On en dŽduit que $\Delta=\oplus_{j=0}^{k'-1}\Delta_j$ est isomorphe ˆ $\Delta'$. Par consŽquent $\iota_k(\Sigma)=\iota_s(\Delta)$ est isomorphe 
ˆ $\oplus_{i=0}^{k'-1}\mu^i\cdot \iota_s(\widetilde{\Sigma})$, dont la classe d'isomorphisme appartient ˆ 
$\G_{0^+}(G^\natural,\omega)$.  
\end{proof}
 
D'aprs le lemme, l'espace $\G_{\Bbb C}(G^\natural,\omega)$ introduit au dŽbut de ce $n^\circ$ s'identifie canoniquement 
au quotient $\G(G^\natural,\omega)/\G_{0^+}(G^\natural,\omega)$. De plus le dual algŽbrique
$$
\G_{\Bbb C}(G^\natural,\omega)^*={\rm Hom}_{\Bbb C}(\G_{\Bbb C}(G^\natural,\omega),{\Bbb C})
$$
co\"{\i}ncide avec l'espace des formes ${\Bbb Z}$--lin\'eaires 
$\Phi$ sur $\G_0(G^\natural,\omega)$ 
vŽrifiant
$$
\Phi(\lambda\cdot\Pi)=\lambda \Phi(\Pi)
$$
pour tout $\Pi\in {\rm Irr}_0(G^\natural,\omega)$ et tout $\lambda\in {\Bbb C}^\times$. 

\begin{notations}\label{nota}
{\rm La projection canonique $\G(G^\natural,\omega)\rightarrow \G_{\Bbb C}(G^\natural,\omega)$ identifie 
${\rm Irr}_0(G^\natural,\omega)$ ˆ un sous--ensemble de $\G_{\Bbb C}(G^\natural,\omega)$, que l'on note aussi 
${\rm Irr}_{\Bbb C}(G^\natural,\omega)$\index{${\rm Irr}_{\Bbb C}(G^\natural,\omega)$} --- il engendre $\G_{\Bbb C}(G^\natural,\omega)$ mais n'est pas 
une base sur ${\Bbb C}$.
D'autre part, l'action de ${\Bbb C}^\times$ sur $\G(G^\natural,\omega)$ 
est notŽe avec un \og $\cdot$ \fg, que l'on aura tendance ˆ supprimer aprs projection sur $\G_{\Bbb C}(G^\natural,\omega)$.}
\end{notations}

\subsection{$(\Hn,\omega)$--modules et $(\Hn_{\sK},\omega)$--modules}\label{H modules}
Pour un espace topologique 
totalement discontinu $X$, on note $\EuScript{H}(X)$ l'espace des fonctions complexes sur $X$ qui 
sont localement constantes et ˆ support compact. 
On pose $\EuScript{H}=\EuScript{H}(G)$ et $\Hn=\EuScript{H}(G^\natural)$.
\index{$\H(X)$, $\H=\H(G)$, $\H^\natural=\H(G^\natural)$}

Fixons une mesure de Haar $dg$\index{$dg$, $d\delta$} sur $G$, et notons $d\delta$ la mesure de Haar sur 
$G^\natural$ au sens de \cite[2.5]{L2} image de $dg$ par l'isomorphisme $G\rightarrow G^\natural,
\,g\mapsto g\cdot \delta$ pour un (resp. pour tout) 
$\delta\in G^\natural$. La mesure $dg$ munit l'espace $\H$ d'un produit de convolution, et 
l'espace $\Hn$ d'une structure de $\H$--bimodule \cite[8.2]{L2}: pour $f,\,f'\in \H$, 
$\phi\in\Hn$ et $\delta\in G^\natural$, on pose
$$
(f*\phi)(\delta)=\int_Gf(g)\phi(g^{-1}\cdot\delta)dg,\quad (\phi*f )(\delta)=\int_G\phi(\delta\cdot g)f(g^{-1})dg.
$$
Comme dans loc.~cit. on appelle $(\Hn,\omega)$--module un 
$\H$--module $V$ muni d'une application 
$\Hn\rightarrow {\rm End}_{\Bbb C}(V),\,\phi\mapsto (v\mapsto \phi\cdot v)$ 
telle que pour tout $\phi\in \Hn$, tous $f,\,f'\in \H$ et tout $v\in V$, on a
$$
(f*\phi*f')\cdot v= f\cdot (\phi\cdot (\omega f'\cdot v)).
$$
Les $(\Hn,\omega)$--modules forment une sous--catŽgorie (non pleine) de la 
catŽgorie des $\H$--modules ˆ gauche: un morphisme entre deux $(\Hn,\omega)$--modules $V_1$ et $V_2$ est 
une application ${\Bbb C}$--linŽaire $u:V_1\rightarrow V_2$ telle que 
$u(\phi\cdot v)=\phi\cdot u(v)$ pour tout $\phi\in \Hn$ et tout $v\in V_1$ 
(une telle application est automatiquement $\H$--linŽaire).

\begin{variante}
{\rm Soit $\H^\natural_\omega = \H(G^\natural,\omega)$\index{$\H^\natural_\omega = \H(G^\natural,\omega)$} 
l'espace vectoriel $\H^\natural$ muni 
de la structure de $\H$--bimodule 
donnŽe par (pour $\phi\in \H^\natural$ et $f\in \H$):
$$
f\cdot \phi = f* \phi,\quad \phi\cdot f = \phi * \omega^{-1}f.
$$
Par dŽfinition, la notion de $(\Hn,\omega)$--module Žquivaut ˆ celle de $\Hn_\omega$--module (\cad de 
$(\Hn_\omega,\xi=1)$--module).\hfill $\blacksquare$}
\end{variante}

\begin{exemple}
{\rm
L'application $\H\rightarrow \H^\natural_\omega,\, f\mapsto u(f)=u_{\delta_1}(f)$ dŽfinie par $u(f)(g\cdot \delta_1)=f(g)$ 
($g\in G$) est un isomorphisme ${\Bbb C}$--linŽaire vŽrifiant $u(f*h*f')= f\cdot u(h) \cdot \omega f'^\theta$.
Pour $\phi\in \H^\natural$ et $f\in \H$, posons
$$
\phi \bullet  f = u^{-1}(\phi\cdot f)\in \H.
$$ 
Pour $f,\,f',\,h\in \H$ et $\phi\in \H^\natural$, on a
$$
(f*\phi * f')\bullet h=  f * (\phi\bullet (\omega f' * h)).
$$
En d'autres termes, $\H$ est muni d'une structure de $(\H^\natural,\omega)$--module. \`A isomorphisme 
prs, cette structure ne dŽpend pas du choix de $\delta_1\in G^\natural$: remplacer $\delta_1$ par $\delta'_1= x\cdot \delta$ 
revient ˆ remplacer $u$ par $u'= \delta_x\circ u$, o l'on a posŽ $\delta_x(f)(g)=f(gx)$, $f\in \H$, $g\in G$.\hfill 
$\blacksquare$}
\end{exemple}

Pour une $\omega$--reprŽsentation $\Pi$ de $G^\natural$ et une fonction $\phi\in \Hn$, on note 
$\Pi(\phi)$ le ${\Bbb C}$-endo\-morphisme de l'espace $V$ de $\Pi$ donnŽ par
$$
\Pi(\phi)(v)=\int_{G^\natural}\phi(\delta)\Pi(\delta)(v)d\delta,\quad v\in V.
$$
Puisque $\phi$ est localement constante et ˆ support compact, et que $\Pi^\circ$ est lisse, l'intŽgrale 
est absolument convergente (c'est m\^eme une 
somme finie). Cela munit $V$ d'une structure de $(\Hn,\omega)$--module {\it non dŽgŽnŽrŽ}, 
\cad tel que $\Hn\cdot V= V$, et l'application 
$(\Pi,V)\mapsto V$ est un isomorphisme entre $\mathfrak{R}(G^\natural,\omega)$ et la catŽgorie 
des $(\Hn,\omega)$--modules non dŽgŽnŽrŽs --- une sous--catŽgorie pleine de celle des $(\Hn,\omega)$--modules.

Pour un sous-groupe ouvert compact $K$ de $G$, on note 
$\H_K=\H_K(G)$\index{$\H_K=\H_K(G)$, $e_K$} la sous-algbre de $\H$ formŽe des fonctions qui sont bi--invariantes par $K$. 
On note $e_K$ l'ŽlŽment unitŽ de $\H_K$, \cad la fonction caractŽristique de $K$ divisŽe par ${\rm vol}(K,dg)$, et 
l'on pose\index{$\Hn_{\sK}=\H_K(G^\natural)$, $e_{K^\natural}$}
$$
\Hn_{\sK}=\H_K(G^\natural)=e_K*\Hn*e_K.
$$
C'est le sous--$\H_K$--bimodule de $\Hn$ formŽ des fonctions qui sont bi--invariantes par $K$. 
Si de plus $\omega$ est trivial sur $K$, on dŽfinit comme ci--dessus les notions de 
$(\Hn_{\sK},\omega)$--module et de $(\Hn_{\sK},\omega)$-module non dŽgŽnŽrŽ, ainsi que les catŽgories 
correspondantes. Alors pour toute $\omega$--reprŽsentation $(\Pi,V)$ de $G^\natural$, 
le sous--espace $V^K=e_K\cdot V$ de $V$ (formŽ des vecteurs qui sont fixŽs par $K$) 
est naturellement muni d'une structure de $(\Hn_{\sK},\omega)$--module.

Soit $K^\natural$ un sous--espace tordu ouvert compact de $G^\natural$, \cad un sous--ensemble 
de la forme $K^\natural =K\cdot \delta$ pour un sous--groupe ouvert compact $K$ de $G$ 
et un ŽlŽment $\delta$ de $G^\natural$ tel que ${\rm Int}_{{\boldsymbol{G}}^\natural}(\delta)(K)=K$. 
Le $\H_K$--bimodule $\Hn_{\sK}$ est un 
$\H_K$--module ˆ gauche (resp. ˆ droite) libre de rang $1$: notant $e_{K^\natural}$ la fonction 
caractŽristique de $K^\natural$ divisŽe par ${\rm vol}(K^\natural,d\delta)={\rm vol}(K,dg)$, on a
$$
\Hn_{\sK}= \H_K*e_{K^\natural}= e_{K^\natural}*\H_K.
$$
Si de plus $\omega$ est trivial sur $K$, alors pour toute $\omega$--reprŽsentation 
$(\Pi,V)$ de $G^\natural$, le $(\H^\natural_K,\omega)$--module $V^K$ est automatiquement non dŽgŽnŽrŽ: on a
$$
V^K = \Hn_{\sK}\cdot V^K\;(= \Hn_{\sK}\cdot V).
$$
En particulier il co\"{\i}ncide avec le sous--espace $V^{K^\natural}=e_{K^\natural}\cdot V$ de $V$.

Un $(\Hn,\omega)$--module non dŽgŽnŽrŽ $V$ est dit {\it simple} s'il est non nul 
et si le seul sous--espace non nul $\Hn$-stable de $V$ est $V$ lui--m\^eme, et il est dit {\it $\H$--simple} 
s'il est simple comme $\H$--module. 
On dŽfinit de la m\^eme manire les notions de $(\Hn_{\sK},\omega)$--module (non dŽgŽnŽrŽ) simple et 
$\H_K$--simple. 

D'aprs \cite[8.3, 8.6]{L2}, une $\omega$--reprŽsentation non nulle $(\Pi,V)$ de $G^\natural$ est irrŽductible 
(resp. $G$--irrŽductible) si et seulement si pour tout sous--espace tordu ouvert compact $K^\natural$ de $G^\natural$ 
tel que $\omega$ est trivial sur $K$, le 
$(\Hn_{\sK},\omega)$--module (resp. le $\H_K$--module) $V^K$ est nul ou simple; o 
$K$ est le sous--groupe de $G$ sous--jacent ˆ $K^\natural$.
De plus pour un tel $K^\natural$, l'application $(\Pi,V)\mapsto V^{K^\natural}=V^K$ induit (d'aprs loc.~cit.) une bijection 
entre:

\begin{itemize}
\item l'ensemble des classes 
d'isomorphisme de $\omega$--reprŽsenta\-tions irrŽductibles $(\Pi,V)$ de $G^\natural$ 
telles que $V^K\neq 0$, 
\item et l'ensemble des classes d'isomorphisme de $(\Hn_{\sK},\omega)$--modules simples;
\end{itemize}
elle induit aussi une bijection entre:
\begin{itemize}
\item l'ensemble des classes 
d'isomorphisme de $\omega$--reprŽsentations $G$--irrŽductibles $(\Pi,V)$ de $G^\natural$ 
telles que $V^K\neq 0$,
\item et l'ensemble des classes d'isomorphisme de $(\Hn_{\sK},\omega)$--modules $\H_K$--simples.
\end{itemize}

\begin{marema1}
{\rm Les rŽsultats ci--dessus sont vrais pour tout espace topologique tordu 
localement profini $(G,G^\natural)$ vŽrifiant la propriŽtŽ $({\rm P}_2)$ de \cite[8.3]{L2}, \cad tel qu'il 
existe une base de voisinages de $1$ dans $G$ formŽe de sous--groupes ouverts compacts et 
un ŽlŽment $\delta\in G^\natural$ normalisant chacun des ŽlŽments de la base. Dans le cas qui nous intŽresse ici, on a vŽrifiŽ \cite[8.6]{L2} 
que si $I^\natural$ est un sous--espace d'Iwahori de $G^\natural$, \cad un sous--espace tordu de 
la forme $I^\natural=I\cdot \delta$ pour un sous--groupe d'Iwahori $I$ de $G$, alors 
tous les sous-groupes de congruence de $I$ sont normalis\'es par $\delta$ (voir \ref{bons sgoc}). Notons que puisque 
les sous--groupes d'Iwahori de $G$ sont tous conjuguŽs dans $G$, il existe un sous--espace d'Iwahori 
de $G^\natural$.\hfill $\blacksquare$
}
\end{marema1}

\begin{marema2}
{\rm Soit $K^\natural$ un sous--espace tordu ouvert compact de $G^\natural$ tel que $\omega$ est trivial sur le groupe 
$K$ sous--jacent ˆ $K^\natural$. 
\`A tout $(\H^\natural_{\sK},\omega)$--module simple $W$ est associŽ comme en 
\ref{s(Pi)} un invariant $s(W)\in {\Bbb Z}_{\geq 1}\cup \{+\infty\}$. Puisque 
$\H^\natural_{\sK}= \H_K* e_{K^\natural}=e_{K^\natural}*\H_K$, l'application $x \mapsto e_{K^\natural}\cdot x$ 
est un ${\Bbb C}$--automorphisme de $W$. Choisissons 
un sous--$\H_K$--module simple $W_0$ de $W$, et pour chaque entier $k\geq 1$, 
notons $W_k$ et $W_{-k}$ les sous--$\H_K$--modules de $W$ dŽfinis par $W_k= e_{K^\natural}\cdot W_{k-1}$ et 
$e_{K^\natural}\cdot W_{-k}= W_{-k+1}$. Pour $k\in {\Bbb Z}$, le $\H_K$--module $W_k$ est simple. 
On distingue 
deux cas: ou bien $\dim_{\Bbb C}(W)=+\infty$, auquel cas on pose 
$s(W)= +\infty$, et l'on a $W = \oplus_{k\in {\Bbb Z}}W_k$; ou bien $\dim_{\Bbb C}(W)<+\infty$, auquel cas il existe un plus petit entier 
$s=s(W)\geq 1$ tel que $W_s = W_0$, et l'on a $W=\oplus_{k=0}^{s-1}W_k$. Bien s\^ur si $(\Pi,V)$ est une 
$\omega$--reprŽsentation irrŽductible de $G^\natural$ telle que le $(\H^\natural_{\sK},\omega)$--module 
$V^K$ est isomorphe ˆ $W$, on a $s(\Pi)=s(V^K)$. \hfill $\blacksquare$
}\end{marema2}

\begin{marema3}
{\rm Puisque $G$ est dŽnombrable ˆ l'infini, la dŽmonstration du lemme de Schur donnŽe 
dans \cite[2.11]{BZ} est valable ici: pour toute $\omega$--reprŽsentation irrŽductible $\Pi$ de 
$G^\natural$, l'espace des endomorphismes de $\Pi$ est de dimension $1$. En particulier, $\Pi$ 
possde un caractre central $\omega_\Pi: Z(G^\natural) \rightarrow {\Bbb C}^\times$. Pour 
toute sous--reprŽsentation irrŽductible $\pi_0$ de $\Pi^\circ$, la restriction ˆ $Z(G^\natural)$ 
du caractre central $\omega_{\pi_0}:Z(G)\rightarrow {\Bbb C}^\times$ de $\pi_0$ 
co\"{\i}ncide avec $\omega_{\Pi}$. Si de plus $\Pi$ est $G$--irrŽductible, on a 
$\omega^{-1}(\omega_{\Pi^\circ})^\theta = \omega_{\Pi^\circ}$. \hfill $\blacksquare$}
\end{marema3}

\subsection{Les caractres $\Theta_\Pi$}\label{carac}
Pour toute $\omega$--repr\'esentation $\Pi$ de $G^\natural$ 
telle que $\Pi^\circ$ est admissible, on dŽfinit comme suit une distribution $\Theta_\Pi$\index{$\Theta_\Pi$} sur $G^\natural$, appel\'ee 
{\it caractre--distribution} ou simplement {\it caractre}, de $\Pi$: pour $\phi\in \Hn$, 
l'opŽrateur $\Pi(\phi)$ sur l'espace de $\Pi$ est de rang fini, et l'on pose
$$
\Theta_\Pi(\phi)={\rm trace}(\Pi(\phi)).
$$
La distribution $\Theta_\Pi$ ne dŽpend que de la classe d'isomorphisme de $\Pi$ (et bien s\^ur du choix de $d\delta$), et v\'erifie
$$
\Theta_{\lambda\cdot\Pi}=\lambda\Theta_\Pi,\quad \lambda\in {\Bbb C}^\times.
$$
Pour tout ŽlŽment $\Pi$ de $\EuScript{G}(G^\natural,\omega)$, on d\'efinit par linŽaritŽ 
une distribution $\Theta_\Pi$ sur $\H^\natural$, qui v\'erifie

\begin{itemize}
\item  $\Theta_{\lambda\cdot \Pi}=\lambda\Theta_\Pi$ pour tout $\lambda\in {\Bbb C}^\times$,
\item $\Theta_\Pi=0$ si $\Pi\in \iota_k(\G(\mathcal{G}_k,\omega_k))$ pour un entier $k>1$.  
\end{itemize}

\ni On en dŽduit que pour $\phi\in \Hn$, l'application
$$
\EuScript{G}(G^\natural,\omega)\rightarrow {\Bbb C},\, \Pi\mapsto \Theta_\Pi(\phi)
$$
se factorise ˆ travers $\G_{\Bbb C}(G^\natural,\omega)$. C'est donc 
un \'el\'ement du dual algŽbrique $\G_{\Bbb C}(G^\natural,\omega)^*$\index{$\G_{\Bbb C}(G^\natural,\omega)^*$} 
de $\G_{\Bbb C}(G^\natural,\omega)$, que l'on note $\Phi_\phi$. Notre thŽorme principal 
--- cf. \ref{ŽnoncŽ} pour un ŽnoncŽ prŽcis --- est une 
description de ce morphisme ${\Bbb C}$--linŽaire
$$\H^\natural\rightarrow \G_{\Bbb C}(G^\natural,\omega)^*,\,\phi\mapsto \Phi_\phi.
$$
Le thŽorme de Paley--Wiener 
dŽcrit son image, et le thŽorme de densitŽ spectrale son noyau.

\subsection{Induction parabolique et restriction de Jacquet}\label{ip et rj} 
Soit $P^\natural$ un sous--espace parabolique de 
$G^\natural$, muni d'une d\'ecomposition de Levi\index{$P^\natural =M^\natural\cdot U$} 
$$P^\natural =M^\natural\cdot U.$$
On note avec les m\^emes lettres sans l'exposant \og $\natural$ \fg
les groupes topologiques sous--jacents \`a $P^\natural$ et $M^\natural$:  
$P$ est un sous--groupe parabolique de $G$ de radical unipotent $U$, 
et $M$ est une composante de Levi 
de $P$. Soit\index{$i_P^G$, $r_G^P$}
$$
i_P^G:\mathfrak{R}(M)\rightarrow \mathfrak{R}(G)
$$
et
$$
r_G^P:\mathfrak{R}(G)\rightarrow \mathfrak{R}(M)
$$
les foncteurs induction parabolique et restriction de Jacquet normalisŽs. On 
considre $\omega$ comme un caractre de $M$ par restriction. 
Dans \cite[5.9 et 5.10]{L2} sont d\'efinis les foncteurs induction parabolique normalis\'ee
\index{${^\omega{i}}_{P^\natural}^{G^\natural}$, ${^\omega{r}}_{G^\natural}^{P^\natural}$}
$$
{^\omega{i}}_{P^\natural}^{G^\natural}:\mathfrak{R}(M^\natural,\omega)\rightarrow \mathfrak{R}(G^\natural,\omega)
$$
et restriction de Jacquet normalis\'ee
$$
{^\omega{r}}_{G^\natural}^{P^\natural}:\mathfrak{R}(G^\natural,\omega)\rightarrow \mathfrak{R}(M^\natural,\omega).
$$
Ils vŽrifient $({^\omega{i}}_{P^\natural}^{G^\natural})^\circ =
\iota_P^G$ et $({^\omega{r}}_{G^\natural}^{P^\natural})^\circ = r_G^P$, et commutent aux foncteurs 
$\iota_k$. Comme 
$i_P^G$ et $r_G^P$ prŽservent la propriŽtŽ d'\^etre de longueur finie, on obtient des 
morphismes de groupes
$$
{^\omega{i}}_{P^\natural}^{G^\natural}:\G(M^\natural,\omega)\rightarrow \G(G^\natural,\omega)
$$
et 
$$
{^\omega{r}}_{G^\natural}^{P^\natural}:\G(G^\natural,\omega)\rightarrow \EuScript{G}(M^\natural,\omega).
$$
Par passage aux quotients, ces derniers induisent des morphismes ${\Bbb C}$--linŽaires
$$
{^\omega{i}}_{P^\natural}^{G^\natural}:\G_{\Bbb C}(M^\natural,\omega)\rightarrow \G_{\Bbb C}(G^\natural,\omega)
$$
et 
$$
{^\omega{r}}_{G^\natural}^{P^\natural}:\G_{\Bbb C}(G^\natural,\omega)\rightarrow \EuScript{G}_{\Bbb C}(M^\natural,\omega).
$$

\begin{marema}
{\rm L'espace tordu $M^\natural$ et le groupe $M$ ne sont pas spŽcifiŽs dans les notations, 
mais cette ambigu\"{\i}tŽ dispara\^{\i}tra plus loin puisque nous n'aurons ˆ considŽrer que des 
sous--groupes paraboliques \og standard \fg, \cad contenant un sous--groupe parabolique 
minimal de $G$ fix\'e une fois pour toutes. Notons aussi que pour que l'espace 
$\G_{\Bbb C}(M^\natural,\omega)$ soit non nul, il faut que $\omega$ soit trivial sur 
le centre $Z(M^\natural)$ de $M^\natural$ \hfill$\blacksquare$}
\end{marema}

Fixons un sous--espace parabolique 
minimal $P^\natural_\circ$ de $G^\natural$, et une d\'ecomposition de Levi\index{$P^\natural_\circ=
M^\natural_\circ\cdot U_\circ$}
$$P^\natural_\circ=
M^\natural_\circ\cdot U_\circ.
$$ 
Notons $\EuScript{P}(G^\natural)$\index{$\EuScript{P}(G^\natural)$, $\P(G^\natural,\omega)$} 
l'ensemble des sous--espaces paraboliques de $G^\natural$ contenant 
$P^\natural_\circ$ --- on les qualifie de \og standard \fg ---, et $\P(G^\natural,\omega)$ le 
sous--ensemble de $\P(G^\natural)$ formŽ des $P^\natural$ tels que $\omega$ est trivial sur $M^\natural_{\sP}$. 
Pour $P^\natural \in \EuScript{P}(G^\natural)$, on note 
$M_{\sP}^\natural$ l'unique composante de Levi de $P^\natural$ 
contenant $M^\natural_\circ$, $P$ et $M_P$ les groupes topologiques sous--jacents 
\`a $P^\natural$ et $M_{\sP}^\natural$, et $U_P$ le radical unipotent de $P$.\index{$M_P$, $P^\natural=M^\natural_{\sP}\cdot U_P$} On a la d\'ecomposition 
de Levi
$$
P^\natural=M_{\sP}^\natural\cdot U_P.
$$
Pour $P^\natural,\, Q^\natural\in \P(G^\natural)$ tels que $Q^\natural\subset P^\natural$, on note 
${^{\omega}i_{Q^\natural}^{P^\natural}}$ le foncteur\index{${^{\omega}i_{Q^\natural}^{P^\natural}}$, 
${^{\omega}r_{P^\natural}^{Q^\natural}}$, ${^{\omega}i_{Q^\natural,{\Bbb C}}^{P^\natural}}$, ${^{\omega}r_{P^\natural,{\Bbb C}}^{Q^\natural}}$}
$$
{^{\omega}i_{Q^\natural\cap M_{\sP}^\natural}^{M_{\sP}^\natural}}:\mathfrak{R}(M_{\sQ}^\natural,\omega)
\rightarrow \mathfrak{R}(M_{\sP}^\natural,\omega)
$$
et ${^{\omega}r_{P^\natural}^{Q^\natural}}$ le foncteur
$$
{^{\omega}r_{M_{\sP}^\natural}^{Q^\natural\cap M_{\sP}^\natural}}:\mathfrak{R}(M_{\sP}^\natural,\omega)
\rightarrow \mathfrak{R}(M_{\sQ}^\natural,\omega).
$$
On obtient comme plus haut 
des morphismes de groupes
$$
{^{\omega}i_{Q^\natural}^{P^\natural}}:\G(M_{\sQ}^\natural,\omega)\rightarrow \G(M_{\sP}^\natural,\omega),
\quad
{^{\omega}r_{P^\natural}^{Q^\natural}}:\G(M_{\sP}^\natural,\omega)\rightarrow \G(M_{\sQ}^\natural,\omega),
$$
et des morphismes ${\Bbb C}$--linŽaires
$$
{^{\omega}i_{Q^\natural}^{P^\natural}}:\G_{\Bbb C}(M_{\sQ}^\natural,\omega)\rightarrow \G_{\Bbb C}(M_{\sP}^\natural,\omega),
\quad
{^{\omega}r_{P^\natural}^{Q^\natural}}:\G_{\Bbb C}(M_{\sP}^\natural,\omega)\rightarrow \G_{\Bbb C}(M_{\sQ}^\natural,\omega),
$$
que l'on notera parfois aussi ${^{\omega}i_{Q^\natural,{\Bbb C}}^{P^\natural}}$ et ${^{\omega}r_{P^\natural,{\Bbb C}}^{Q^\natural}}$ pour 
Žviter toute ambigu\"{\i}tŽ.

\begin{monlem}Soit $P^\natural,\,Q^\natural,\,R^\natural\in \P(G^\natural)$ tels que $R^\natural\subset Q^\natural\subset P^\natural$. 
Soit $\Sigma$ une $\omega$--reprŽsentation de $M^\natural_{\sR}$, et soit $\Pi$ une $\omega$--reprŽsentation de $M^\natural_{\sP}$.
\begin{enumerate}
\item[(1)]
On a un isomorphisme naturel, fonctoriel en $\Sigma$,
$$
{^\omega{i}_{Q^\natural}^{P^\natural}}\circ {^\omega{i}_{R^\natural}^{Q^\natural}}(\Sigma)\simeq {^\omega{i}_{R^\natural}^{P^\natural}}(\Sigma).
$$
\item[(2)]
On a un isomorphisme naturel, fonctoriel en $\Pi$,
$$
{^\omega{r}_{Q^\natural}^{R^\natural}}\circ {^\omega{r}_{P^\natural}^{Q^\natural}}(\Pi)\simeq
{^\omega{r}_{P^\natural}^{R^\natural}}(\Pi).
$$
\item[(3)] On a un isomorphisme naturel, fonctoriel en $\Sigma$ et $\Pi$,
$$
{\rm Hom}_{\smash{M^\natural_{\sR}}}({^\omega{r}_{P^\natural}^{R^\natural}}(\Pi),\Sigma)\simeq
{\rm Hom}_{\smash{M^\natural_{\sP}}}(\Pi,{^\omega{i}_{R^\natural}^{P^\natural}}(\Sigma)).
$$
\end{enumerate}
\end{monlem}

\begin{proof} Les propriŽtŽs de transitivitŽ des foncteurs induction parabolique et 
restriction de Jacquet normalisŽs sont consŽquences directes des dŽfinitions. 
On Žpargne au lecteur leurs vŽrifications. Quant au point (3), on sait que le foncteur $r_P^R =({^\omega{r}_{P^\natural}^{R^\natural}})^\circ$ 
est un adjoint ˆ gauche du foncteur $i_R^P= ({^\omega{i}_{R^\natural}^{P^\natural}})^\circ$: on a un isomorphisme naturel, 
fonctoriel en $\Sigma^\circ$ et $\Pi^\circ$,
$$
{\rm Hom}_{M_R}(r_P^R(\Pi^\circ),\Sigma^\circ)\simeq {\rm Hom}_{M_P}(\Pi^\circ, i_R^P(\Sigma^\circ)).
$$
Il induit par restriction un isomorphisme 
de ${\rm Hom}_{\smash{M^\natural_{\sR}}}({^\omega{r}_{P^\natural}^{R^\natural}}(\Pi),\Sigma)$ 
sur ${\rm Hom}_{\smash{M^\natural_{\sP}}}(\Pi,{^\omega{i}_{R^\natural}^{P^\natural}}(\Sigma))$, 
fonctoriel en $\Sigma$ et $\Pi$.
\end{proof}

On note $\EuScript{L}(G^\natural)$\index{$\EuScript{L}(G^\natural)$, $\EuScript{L}(G^\natural,\omega)$} l'ensemble $\{M_{\sP}^\natural:P^\natural\in \EuScript{P}(G^\natural)\}$, 
et $\EuScript{L}(G^\natural,\omega)$ le sous--ensemble de $\EuScript{L}(G^\natural)$ formŽ 
des $M_{\sP}^\natural$ tels que $\omega$ est trivial sur le centre $Z(M^\natural_{\sP})=Z(M_P)^\theta$ de $M^\natural_{\sP}$. 
Pour $P\in \EuScript{P}(G^\natural)$, on a $P^\natural =M_{\sP}^\natural\cdot U_\circ$, par consŽquent 
l'application $\EuScript{P}(G^\natural)\rightarrow \EuScript{L}(G^\natural),\,
P^\natural\mapsto M_{\sP}^\natural$ est bijective, et elle induit une bijection $\P(G^\natural,\omega)\rightarrow \EuScript{L}(G^\natural,\omega)$. 
  
Le groupe $P_\circ$ est un sous--groupe parabolique minimal de $G$, et l'on d\'efinit 
de la m\^eme manire $\P(G)$, $M_P$ pour 
$P\in \P(G)$, $i_Q^P$ et $r_P^Q$ pour $P,\,Q\in \P(G)$ tels que $Q\subset P$, et 
$\EuScript{L}(G)$. L'application 
$\P(G)\rightarrow \EuScript{L}(G),\, P\mapsto M_P$ est bijective.\index{$\P(G)$, $\EuScript{L}(G)$}\index{$i_Q^P$, $r_P^Q$}

\begin{monhypo}
On suppose dŽsormais que le point--base 
$\delta_1$ est choisi dans $M^\natural_\circ$.
\end{monhypo}
\v1

Puisque $\theta={\rm Int}_{{\boldsymbol{G}}^\natural}(\delta_1)$ stabilise $P_\circ$, 
il op\`ere sur $\P(G)$ et l'application $P^\natural\mapsto P$ est une bijection de $\EuScript{P}(G^\natural)$ sur 
le sous--ensemble $\P(G)^\theta$\index{$\P(G)^\theta$, $\EuScript{L}(G)^\theta$} de $\EuScript{P}(G)$ form\'e des $P$ qui sont 
$\theta$--stables. Pour $P\in \EuScript{P}(G)$, on a $\theta(M_P)=M_{\theta(P)}$ et $\theta(U_P)=U_{\theta(P)}$. 
En particulier $\theta$ opre aussi sur $\EuScript{L}(G)$. 
On note $\EuScript{L}(G)^\theta$ le sous--ensemble de $\EuScript{L}(G)$ formŽ des 
$M$ qui sont $\theta$-stables. 
Pour $P\in \EuScript{P}(G)$, puisque $P=M_PP_\circ$, on a
$$\theta(P)=\theta(M_P)P_\circ.
$$ 
On en d\'eduit que l'application $\EuScript{P}(G)^\theta\rightarrow \EuScript{L}(G)^\theta,\,P\mapsto M_P$ 
est elle aussi bijective.

Notons que les opŽrations de $\theta$ sur $\P(G)$ et sur $\EuScript{L}(G)$ dŽpendent 
du choix de la composante de Levi $M_\circ^\natural$ de $P_\circ^\natural$, mais elles 
ne dŽpendent pas du choix de $\delta_1$ dans $M_\circ^\natural$.

\subsection{Caractres non ramifiŽs}\label{carac nr}
Soit $M$ un groupe topologique localement profini, et soit $M^\natural$ un $M$--espace topologique tordu. 
On note $M^1$ le sous-groupe de $M$ engendrŽ par ses sous--groupes ouverts compacts, 
et l'on pose\index{$M^1$, $\mathfrak{P}(M)={\rm Hom}(M/M^1,{\Bbb C}^\times)$}
$$\mathfrak{P}(M)={\rm Hom}(M/M^1,{\Bbb C}^\times).
$$
Les ŽlŽments de $\mathfrak{P}(M)$ sont appelŽs  {\it caractres non ramifiŽs de $M$}.

\begin{exemple}
{\rm Soit $A$ le groupe des points $F$--rationnels d'un tore dŽployŽ et dŽfini 
sur $F$, et soit $X={\rm X}^*(A)$\index{${\rm X}^*(A)$, $A^\varpi$, $\mathfrak{P}(A)=
{\rm Hom}(A^\varpi,{\Bbb C}^\times)$} le groupe des caractres algŽbriques de $A$. On a  $A={\rm Hom}(X,F^\times)$ 
et $A^1$ est le sous--groupe compact maximal ${\rm Hom}(X,\mathfrak{o}^\times)$ de $A$. Le 
choix d'une uniformisante $\varpi$ de $F^\times$ dŽfinit un sous--groupe fermŽ (discret) $A^\varpi = 
{\rm Hom}(X,\langle \varpi \rangle)$ de $A$. L'application produit $A^\varpi \times A^1\rightarrow A$ est un 
isomorphisme de groupes topologiques; d'o un isomorphisme $A^\varpi\simeq A/A^1$, qui 
identifie ${\rm Hom}(A^\varpi,{\Bbb C}^\times)$ au groupe $\mathfrak{P}(A)$.}
\end{exemple}

Pour tout caract\`ere $\chi$ de $M$, le caractre $\chi^\tau=\chi\circ\tau$ de $M$, o $\tau={\rm Int}_{M^\natural}(\delta)$ 
pour un $\delta\in M^\natural$, ne dŽpend pas du choix de $\delta\in M^\natural$. On note 
$\mathfrak{P}(M^\natural)$\index{$\mathfrak{P}(M^\natural)$} le sous--groupe de $\mathfrak{P}(M)$ formŽ des 
$\psi$ tels que $\psi\circ{\rm Int}_{G^\natural}(\delta)=\psi$ pour un (i.e. pour tout) 
$\delta\in M^\natural$. Les ŽlŽments de $\mathfrak{P}(M^\natural)$ sont appelŽs {\it caractres non ramifiŽs de $M^\natural$}. 
Pour $\delta\in M^\natural$, on pose aussi $\mathfrak{P}(M)^\tau= \mathfrak{P}(M^\natural)$, 
$\tau={\rm Int}_{M^\natural}(\delta)$.

\begin{marema1}
{\rm Supposons de plus que $M^\natural$ 
est un sous--espace tordu de $G^\natural$, par exemple une composante de Levi d'un sous--espace parabolique 
de $G^\natural$. On peut alors dŽfinir un 
autre sous--groupe de $\mathfrak{P}(M)$, qui contient $\mathfrak{P}(M^\natural)$. 
Soit $\mathfrak{P}(M,G^\natural)$ le sous--groupe de $\mathfrak{P}(M)$ formŽ des 
$\psi$ vŽrifiant la propriŽtŽ suivante: il 
existe un ŽlŽment $w$ de $N_G(M)/M$ tel que pour un (i.e. pour tout) $\delta\in M^\natural$, posant 
$\tau={\rm Int}_{M^\natural}(\delta)$, on a 
$$
\psi^\tau = {^w\psi},\quad {^w\psi}= \psi\circ {\rm Int}\,w^{-1};
$$
o\`u $N_G(M)$ dŽsigne le normalisateur de $M$ dans $G$. 
En d'autres termes, $\mathfrak{P}(M,G^\natural)$ est le sous--groupe de $\mathfrak{P}(M)$ form\'e des $\psi$ tels que 
$\psi\circ {\rm Int}_{G^\natural}(\delta)\vert_M=\psi$ pour un $\delta\in G^\natural$ normalisant $M$. Ce groupe 
$\mathfrak{P}(M,G^\natural)$ dŽpend bien s\^ur de $M$ et de $G^\natural$, mais il ne dŽpend pas de $M^\natural$. On a 
$\mathfrak{P}(M^\natural)=\mathfrak{P}(M,M^\natural)$.
\hfill$\blacksquare$}
\end{marema1}

Soit $\mathfrak{P}_{\Bbb C}(M^\natural)$\index{$\mathfrak{P}_{\Bbb C}(M^\natural)$} 
le sous--ensemble de ${\rm Irr}_{\Bbb C}(M^\natural)={\rm Irr}_{\Bbb C}(M^\natural,\omega=1)$ 
form\'e des $\Pi$ tels que $\Pi^\circ$ est un caract\`ere non ramifi\'e de $M$. Ce caract\`ere est 
alors un \'el\'ement de $\mathfrak{P}(M^\natural)$. D'ailleurs le foncteur d'oubli $\Pi\mapsto \Pi^\circ$ induit une application 
bijective
$$
\mathfrak{P}_{\Bbb C}(M^\natural)/{\Bbb C}^\times \rightarrow \mathfrak{P}(M^\natural).
$$
Pour $\Psi\in \mathfrak{P}_{\Bbb C}(M^\natural)$ et 
$\delta\in M^\natural$, $\Psi(\delta)$ est un ${\Bbb C}$-automorphisme d'un espace vectoriel complexe de dimension $1$, 
\cad la multiplication par un nombre complexe non nul, et l'on identifie $\Psi(\delta)$ \`a 
ce nombre. Cela munit l'ensemble $\mathfrak{P}_{\Bbb C}(M^\natural)$ 
d'une structure de groupe: pour $\Psi,\,\Psi'\in \mathfrak{P}_{\Bbb C}(M^\natural)$ et $\delta\in M^\natural$, on pose
$$(\Psi \Psi')(\delta)=\Psi(\delta)\Psi'(\delta).
$$
Tout ŽlŽment $\delta$ de $M^\natural$ dŽfinit un 
scindage du morphisme $\mathfrak{P}_{\Bbb C}(M^\natural)\rightarrow \mathfrak{P}(M^\natural),\,\Psi\mapsto \Psi^\circ$: 
pour $\psi\in \mathfrak{P}(M^\natural)$, on note $\psi^\delta$ l'ŽlŽment de $\mathfrak{P}_{\Bbb C}(M^\natural_{\sP})$ 
dŽfini par $\psi^\delta(g\cdot \delta)= \psi(g)$, $g\in G$. 
Pour tout caractre $\eta$ de $M$, le groupe $\mathfrak{P}_{\Bbb C}(M^\natural)$ op\`ere sur l'ensemble  
des $\eta$-reprŽsentations de $M^\natural$: pour 
$\Psi\in \mathfrak{P}_{\Bbb C}(M^\natural)$, $\Pi$ une $\eta$-reprŽsentation de $M^\natural$ et $\delta\in M^\natural$, on pose
$$
(\Psi\cdot\Pi)(\delta)= \Psi(\delta)\Pi(\delta),\quad \delta\in M^\natural.
$$
On a
$$
(\Psi\cdot\Pi)^\circ =\Psi^\circ\Pi^\circ.
$$

\begin{marema2}
{\rm 
Pour $P\in \P(G^\natural)$, le groupe $\mathfrak{P}(M_P)$ est un tore complexe (de 
groupe des caractres algŽbriques $M_P/M_P^1$), 
et $\mathfrak{P}(M^\natural_{\sP})=\mathfrak{P}(M_P)^\theta$. Par suite les groupes 
$\mathfrak{P}(M^\natural_{\sP})$ et $\mathfrak{P}_{\Bbb C}(M^\natural_{\sP})$ sont 
des variŽtŽs algŽbriques affines complexes, en fait des groupes algŽbriques affines diagonalisables sur ${\Bbb C}$. 
Le morphisme $\mathfrak{P}_{\Bbb C}(M^\natural_{\sP})\rightarrow \mathfrak{P}(M^\natural_{\sP}),\, \Psi\mapsto \Psi^\circ$ 
est algŽbrique, et pour $\delta\in M^\natural_{\sP}$, le morphisme $\mathfrak{P}(M^\natural_{\sP})\rightarrow \mathfrak{P}_{\Bbb C}(M^\natural_{\sP}),\, 
\psi\mapsto \psi^\delta$ est lui aussi algŽbrique. On note $\mathfrak{P}^0(M^\natural_{\sP})$ et $\mathfrak{P}^0_{\Bbb C}(M^\natural_{\sP})$
\index{$\mathfrak{P}^0(M^\natural_{\sP})$, $\mathfrak{P}^0_{\Bbb C}(M^\natural_{\sP})$} 
les composantes neutres de $\mathfrak{P}(M^\natural_{\sP})$ et $\mathfrak{P}_{\Bbb C}(M^\natural_{\sP})$. 
Ce sont des tores complexes, et le morphisme 
$\mathfrak{P}_{\Bbb C}(M^\natural_{\sP})\rightarrow \mathfrak{P}(M^\natural_{\sP})$ 
induit par restriction un morphisme surjectif $\mathfrak{P}^0_{\Bbb C}(M^\natural_{\sP})\rightarrow \mathfrak{P}^0(M^\natural_{\sP})$ 
de noyau ${\Bbb C}^\times$.\hfill $\blacksquare$
}
\end{marema2}

On aura aussi besoin de la variante suivante des constructions prŽcŽdentes. 
Pour un caractre non ramifi\'e $\xi$ de $M$, on note $\mathfrak{P}(M^\natural,\xi)$ le sous--groupe 
de $\mathfrak{P}(M)$ formŽ des $\psi$ tels que $\psi\circ {\rm Int}_{G^\natural}(\delta)=\xi \psi$ pour 
un (i.e. pour tout) $\delta\in M^\natural$, et l'on note $\mathfrak{P}_{\Bbb C}(M^\natural,\xi)$\index{$\mathfrak{P}(M^\natural,\xi)$, 
$\mathfrak{P}_{\Bbb C}(M^\natural,\xi)$} le 
sous--ensemble de ${\rm Irr}_{\Bbb C}(M^\natural,\xi)$ form\'e des $\Pi$ tels que $\Pi^\circ$ est un 
caractre non ramifiŽ de $M$. Ce caractre appartient ˆ $\mathfrak{P}(M^\natural,\xi)$, et le foncteur d'oubli $\Pi\mapsto \Pi^\circ$ 
induit une application bijective
$$
\mathfrak{P}_{\Bbb C}(M^\natural,\xi)/{\Bbb C}^\times \rightarrow \mathfrak{P}(M^\natural,\xi).
$$
L'ensemble 
$\mathfrak{P}_{\Bbb C}(M^\natural,\xi)$ est un espace principal homogne sous $\mathfrak{P}_{\Bbb C}(M^\natural)$, 
et pour $\Psi\in \mathfrak{P}_{\Bbb C}(M^\natural,\xi)$ et $\delta\in M^\natural$, on peut comme plus haut 
identifier $\Psi(\delta)$ ˆ un nombre complexe non nul. Enfin pour tout caractre $\eta$ de $M$, 
toute $\eta$-reprŽsentation $\Pi$ de $M^\natural$ 
et tout $\Psi\in \mathfrak{P}_{\Bbb C}(M^\natural,\xi)$, on note $\Psi\cdot \Pi$ la $\xi\eta$--reprŽsentation 
de $M^\natural$ dŽfinie comme ci--dessus. On a encore
$$
(\Psi\cdot\Pi)^\circ =\Psi^\circ\Pi^\circ.
$$

\subsection{Quotient de Langlands}\label{langlands}Soit $\vert \omega \vert $ le caractre non 
ramifiŽ de $G$ donnŽ par
$$\vert \omega \vert (g)=\vert \omega (g)\vert,\quad g\in G,
$$
et soit $\omega_{\rm u}$ le caract\`ere 
unitaire de $G$ tel que\index{$\omega=\omega_{\rm u}\vert\omega\vert$}
$$\omega=\omega_{\rm u}\vert\omega\vert.
$$

Soit $\Pi$ une $\omega$--repr\'esentation $G$--irr\'eductible de $G^\natural$. 
D'apr\`es la classification de Langlands, il existe un triplet $(P,\sigma,\xi)$\index{$(P,\sigma,\xi)$} form\'e d'un 
sous-groupe parabolique standard $P$ de $G$, d'une repr\'esentation irr\'eductible temp\'er\'ee 
$\sigma$ de $M_P$ et d'un caract\`ere non ramifi\'e $\xi$ de $M_P$ qui est positif par rapport 
\`a $U_P$, tels que $\Sigma^\circ$ est isomorphe \`a l'unique quotient irr\'eductible de 
${i}_P^G(\xi\sigma)$. L'unicit\'e de la d\'ecomposition de Langlands implique que $\theta(P)=P$ et
que $\omega^{-1}(\xi\sigma)^\theta$ est isomorphe \`a $\sigma$. Comme la repr\'esentation $\omega_{\rm u}^{-1}\sigma^\theta$ 
et le caract\`ere $\vert \omega\vert^{-1} \xi^\theta$ de $M_P$ sont respectivement temp\'er\'ee et 
positif par rapport \`a $U_P$, $\omega_{\rm u}^{-1}\sigma^\theta $ est isomorphe ˆ $\sigma$ et $\vert \omega\vert^{-1}\xi^\theta =\xi$. 
On en d\'eduit qu'il existe une $\omega_{\rm u}$-repr\'esentation $M_P$--irr\'eductible $\Sigma$ de $M_{\sP}^\natural$ 
et un ŽlŽment $\Xi$ de $\mathfrak{P}_{\Bbb C}(M_{\sP}^\natural,\vert \omega\vert)$ tels que $\Sigma^\circ = \sigma$, $\Xi^\circ = \xi$, 
et $\Pi$ est l'unique quotient irr\'eductible de 
${^\omega{i}}_{P^\natural}^{G^\natural}(\Xi\cdot\Sigma)$. 

\begin{madefi}
{\rm Une $\omega$--reprŽsentation $G$--irrŽductible $\Pi$ de $G^\natural$ 
est dite {\it unitaire} s'il existe un produit scalaire hermitien $G^\natural$--invariant 
sur l'espace de $\Pi$, ce qui n'est possible que si le caractre $\omega$ est 
lui--m\^eme unitaire. Un tel produit, s'il existe, est a fortiori $G$--invariant, et la reprŽsentation 
$\Pi^\circ$ de $G$ sous--jacente \`a $\Pi$ est unitaire.
}
\end{madefi}

Supposons que $\pi= \Pi^\circ$ est unitaire, et fixons un 
produit scalaire hermitien $G$--invariant sur l'espace $V$ de $\pi$. Posons
$$
(v,v')^\natural = (\Pi(\delta)(v),\Pi(\delta)(v')),\quad v,\,v'\in V,
$$
pour un (i.e. pour tout) $\delta\in G^\natural$. Si $\omega$ est unitaire, c'est un autre produit scalaire hermitien 
$G$--invariant sur $V$, par suite $(\cdot,\cdot)^\natural = \lambda (\cdot,\cdot)$ pour 
un nombre r\'eel $\lambda>0$, et $\Pi$ est unitaire si et seulement 
si $\lambda=1$. 

\begin{marema}
{\rm Si $\Pi$ est une $\omega$--repr\'esentation de $G^\natural$, 
sa contragr\'ediente $\check{\Pi}$\index{$(\check{\Pi},\check{V})$ (contragrŽdiente)} est dŽfinie comme suit. Notant $V$ l'espace de $\pi=\Pi^\circ$ et $\check{V}$ celui de 
la contragrŽdiente $\check{\pi}$ de $\pi$, on pose
$$
\langle v, \check{\Pi}(\delta)(\check{v})\rangle = \langle \Pi(\delta)^{-1}(v),\check{v}\rangle,\quad (v,\check{v})\in V\times \check{V}.
$$
Cela d\'efinit une $\omega^{-1}$--repr\'esentation $\check{\Pi}$ de $G^\natural$, telle que 
$\check{\Pi}^\circ = \check{\pi}$. Si de plus $\pi$ est irr\'eductible et unitaire, alors 
notant $(\overline{\Pi},\overline{V})$ 
la conjugu\'ee complexe de $\Pi$\index{$(\overline{\Pi},\overline{V})$ (conjuguŽe complexe)} --- c'est une $\overline{\omega}$--repr\'esentation 
$G$--irr\'eductible de $G^\natural$ ---, l'op\'erateur $v\mapsto A_v=\langle v,\cdot\rangle$ de 
$\overline{V}$ dans $\check{V}$ est un isomorphisme de $\overline{\pi}$ 
sur $\check{\pi}$, et c'est un isomorphisme de $\overline{\Pi}$ sur $\check{\Pi}$ 
si et seulement si $\Pi$ est unitaire (auquel cas on a forcŽment $\overline{\omega}=\omega^{-1}$).\hfill$\blacksquare$
}
\end{marema}

D'apr\`es ce qui pr\'ec\`ede, \`a toute $\omega$--repr\'esentation $G$--irr\'eductible $\Pi$ de $G^\natural$ est associ\'ee un 
triplet $(P^\natural,\Sigma, \Xi)$\index{$(P^\natural,\Sigma, \Xi)$} o\`u:

\begin{itemize}
\item $P^\natural$ est un sous--espace parabolique standard de $G^\natural$,
\item $\Sigma$ est une $\omega_{\rm u}$--repr\'esentation $M_P$--irr\'eductible {\it temp\'er\'ee} de $M_{\sP}^\natural$, 
\cad unitaire et telle que $\Sigma^\circ$ est temp\'er\'ee,
\item $\Xi$ est un \'el\'ement de $\mathfrak{P}_{\Bbb C}(M_{\sP}^\natural,\vert \omega \vert)$ tel que 
$\Xi^\circ$ est positif par rapport \`a $U_P$.
\end{itemize}
\goodbreak

\ni Un tel triplet est appelŽ {\it triplet de Langlands pour $(G^\natural,\omega)$}, et $\Pi$ est l'unique 
quotient $G$--irrŽductible, donc aussi l'unique quotient irrŽductible, de l'induite parabolique 
${^\omega{i}_{P^\natural}^{G^\natural}}(\Xi\cdot\Pi)$.

R\'eciproquement, \`a tout triplet de Langlands $\mu=(P^\natural,\Sigma,\Xi)$ pour $(G^\natural,\omega)$ est associ\'ee 
une $\omega$--reprŽsentation $G$--irrŽductible $\Pi=\Pi_\mu$ de $G^\natural$: c'est l'unique quotient 
irrŽductible de $\widetilde{\Pi}_\mu= {^\omega{i}}_{P^\natural}^{G^\natural}(\Xi\cdot \Sigma)$\index{$\widetilde{\Pi}_\mu$, $\Pi_\mu$}. 
En effet, puisque $\mu^\circ=(P,\Sigma^\circ,\Xi^\circ)$ est 
un triplet de Langlands pour $G$, la reprŽsentation 
$$
{i}_P^G(\Xi^\circ\Sigma^\circ)= \widetilde{\Pi}_\mu^\circ
$$
de $G$ 
a un unique quotient irr\'eductible, disons $\pi$. Comme la reprŽsentation 
$${i}_P^G(\Xi^\circ\Sigma^\circ)(1)= \omega^{-1}{i}_P^G(\Xi^\circ\Sigma^\circ)^\theta$$ de $G$ 
est isomorphe \`a ${i}_P^G(\Xi^\circ\Sigma^\circ)$, on a nŽcessairement
$$
\pi(1)\simeq \pi.
$$
On en dŽduit qu'il existe un quotient $\Pi$ de $\widetilde{\Pi}_\mu$ tel que $\Pi^\circ=\pi$. Ce quotient est 
l'unique quotient irrŽductible de $\widetilde{\Pi}_\mu$. 
 
Deux tels triplets de Langlands $(P^\natural_1,\Sigma_1,\Xi_1)$, $(P^\natural_2,\Sigma_2,\Xi_2)$ pour $(G^\natural,\omega)$ 
sont dits \'equivalents si $P_1^\natural =P_2^\natural$ 
et s'il existe un nombre complexe $\lambda$ de module $1$ tels que $\lambda\cdot\Sigma_1\simeq \Sigma_2$ et 
$\lambda^{-1}\Xi_1=\Xi_2$. Comme dans le cas non tordu, l'application
$$
\mu\mapsto \Pi_\mu
$$
induit une bijection de l'ensemble des classes d'Žquivalence de triplets de Langlands pour $(G^\natural,\omega)$ 
sur ${\rm Irr}_0(G^\natural,\omega)$.

\subsection{\og DŽcomposition \fg de Langlands}\label{base langlands}
Notons 
${\rm Irr}_{0,{\rm t}}(G^\natural,\omega_{\rm u})$\index{${\rm Irr}_{0,{\rm t}}(G^\natural,\omega_{\rm u})$} le sous--ensemble de 
${\rm Irr}_0(G^\natural,\omega_{\rm u})$ formŽ des ŽlŽments tempŽrŽs. Il s'identifie ˆ un sous--ensemble 
de ${\rm Irr}_{\Bbb C}(G^\natural, \omega_{\rm u})$, que l'on note ${\rm Irr}_{{\Bbb C},{\rm t}}(G^\natural,\omega_{\rm u})$ --- cf. 
\ref{nota} (notations). D'aprs \ref{langlands}, le ${\Bbb Z}$--module $\G_0(G^\natural,\omega)$ est somme directe 
des ${\Bbb Z}$--modules ${\Bbb Z}\Pi_\mu$ pour $\mu$ parcourant les classes d'Žquivalence de triplets de 
Langlands pour $(G^\natural,\omega)$. 
On en dŽduit la version tordue suivante du thŽorme de la base de Langlands (qu'il convient 
d'appeler \og dŽcomposition \fg de Langlands, cf. la remarque ci--dessous).

\begin{monlem1}
On a la dŽcomposition dans $\G(G^\natural,\omega)$
$$
\G(G^\natural,\omega)=\sum_\mu {\Bbb Z}\widetilde{\Pi}_\mu + \G_{>0}(G^\natural,\omega)
$$
o\`u $\mu$ parcourt les classes d'Žquivalence de triplets de Langlands pour $(G^\natural,\omega)$.
\end{monlem1}

\begin{proof} Puisque $\G(G^\natural,\omega)=\G_0(G^\natural,\omega)\oplus \G_{>0}(G^\natural,\omega)$, 
il suffit de montrer que $\G_0(G^\natural,\omega)$ est contenu dans l'expression ˆ droite 
de l'ŽgalitŽ dans l'ŽnoncŽ. Soit donc $\Pi$ une $\omega$--reprŽsentation 
$G$--irrŽductible de $G^\natural$, et soit 
$(P^\natural,\Sigma,\Xi)$ un triplet de Langlands pour $(G^\natural,\omega)$ associ\'e \`a $\Pi$. 
Posons $\sigma=\Sigma^\circ$, $\xi=\Xi^\circ$, et soit $(M_{P'},\sigma')$ 
une paire cuspidale standard de $G$ telle que $P'\subset P$ et $\xi\sigma$ 
est isomorphe \`a un sous--quotient de l'induite parabolique ${i}_{P'}^P(\sigma')$. Chaque composant --- 
i.e. classe d'isomorphisme d'un sous--quotient irr\'eductible --- de ${i}_P^G(\xi\sigma)={^\omega{{i}_{P^\natural}^{G^\natural}(\Xi\cdot\Sigma)}}^\circ$ est aussi un composant de ${i}_{P'}^G(\sigma')$. Dans $\EuScript{G}(G^\natural,\omega)$, \'ecrivons
$$
\Pi \equiv {^\omega{{i}_{P^\natural}^{G^\natural}}}(\Xi\cdot \Sigma) - \sum_{i=1}^n \Pi_i\quad({\rm mod}\; \EuScript{G}_{>0}(G^\natural,\omega))
$$
o\`u les $\Pi_i$ sont des \'el\'ements de ${\rm Irr}_0(G^\natural,\omega)$ tels que $\theta_G(\Pi_i^\circ)=[M_{P'},\sigma']$. 
On refait ensuite la m\^eme chose avec chacun des $\Pi_i$:
$$
\Pi_i\equiv {^\omega{i}_{P_i^\natural}^{G^\natural}}(\Xi_i\cdot \Sigma_i) - \sum_{j=1}^{n_i}\Pi_{i,j}\quad ({\rm mod}\;\EuScript{G}_{>0}(G^\natural,\omega)).
$$
Puisque l'application $\theta_G:{\rm Irr}(G)\rightarrow \Theta(G)$ est ˆ fibres finies, 
d'aprs le r\'esultat bien connu de Borel--Wallach \cite[4.3]{BW}, le processus de d\'ecomposition s'arr\^ete au bout d'un 
nombre fini d'\'etapes. On obtient donc une d\'ecomposition de la forme
$$
\Pi \equiv \sum_{i=1}^m\widetilde{\Pi}_{\mu_i}\quad ({\rm mod}\;\EuScript{G}_{>0}(G^\natural,\omega)),
$$
o les $\mu_i$ sont des triplets de Langlands pour $(G^\natural,\omega)$ (qui ne sont 
pas forcŽment deux--ˆ--deux non Žquivalents). \end{proof}

\begin{marema}
{\rm 
D'aprs le thŽorme de la base de Langlands pour $G$, la somme 
$\sum_\mu {\Bbb Z}\widetilde{\Pi}_\mu$ est directe, mais sa projection 
sur $\G_0(G^\natural,\omega)=\G(G^\natural,\omega)/\G_{>0}(G^\natural,\omega)$ 
ne l'est en gŽnŽral pas. En effet si 
$\mu$ est un triplet de Langlands pour $(G^\natural,\omega)$ tel 
que la $\omega$--reprŽsentation $\widetilde{\Pi}_\mu$ de $G^\natural$ est rŽductible, 
les composants irrŽductibles de $\widetilde{\Pi}_\mu$ ne sont pas forcŽment 
$G$--irrŽductibles.\hfill $\blacksquare$ }
\end{marema}

Un triplet de Langlands $(P^\natural,\Sigma,\Xi)$ pour $(G^\natural,\omega)$ est dit {\it induit} si 
$P^\natural\neq G^\natural$. Une $\omega$--reprŽsentation $G$--irrŽductible $\Pi$ de $G^\natural$ est dite 
{\it essentiellement tempŽrŽe} s'il existe un $\Psi\in \mathfrak{P}_{\Bbb C}(G^\natural,\vert \omega \vert^{-1})$ tel que 
la $\omega_{\rm u}$--reprŽsenta\-tion $\Psi\cdot\Pi$ de $G^\natural$ est tempŽrŽe, \cad si $\Pi$ est 
isomorphe \`a $\Pi_\mu=\widetilde{\Pi}_\mu$ pour un $\mu$ non induit. De manire Žquivalente, $\Pi$ est essentiellement 
tempŽrŽe si $\Pi^\circ$ est une reprŽsentation essentiellement tempŽrŽe de $G$. On note 
${\rm Irr}_{0,{\rm e.t}}(G^\natural,\omega)$\index{${\rm Irr}_{0,{\rm e.t}}(G^\natural,\omega)$} 
le sous--ensemble de ${\rm Irr}_0(G^\natural,\omega)$ formŽ
des ŽlŽments essentiellement tempŽrŽs. On note aussi:

\begin{itemize}
\item $\G_{\rm e.t}(G^\natural,\omega)$ le sous--groupe de $\G(G^\natural,\omega)$ 
engendrŽ par les $\iota_k({\rm Irr}_{0,{\rm e.t}}(\mathcal{G}_k,\omega_k))$ pour un entier 
$k\geq 1$ --- il est stable sous ${\Bbb C}^\times$;
\item $\G_{L-{\rm ind}}(G^\natural,\omega)$ le sous--groupe de $\G(G^\natural,\omega)$ 
engendrŽ par les $\iota_k(\widetilde{\Pi}_\mu)$ pour un entier $k\geq 1$ et un triplet de Langlands induit 
$\mu$ pour $(G^\natural_k,\omega_k)$ --- lui aussi est stable sous ${\Bbb C}^\times$;
\item $\G_{{\Bbb C},{\rm e.t}}(G^\natural,\omega)$ et $\G_{{\Bbb C},{L-{\rm ind}}}(G^\natural,\omega)$ leurs projections 
respectives dans $\G_{\Bbb C}(G^\natural,\omega)$ --- deux sous--espaces vectoriels de $\G_{\Bbb C}(G^\natural,\omega)$.
\end{itemize}
\index{$\G_{\rm e.t}(G^\natural,\omega)$, $\G_{L-{\rm ind}}(G^\natural,\omega)$}
\index{$\G_{{\Bbb C},{\rm e.t}}(G^\natural,\omega)$, $\G_{{\Bbb C},{L-{\rm ind}}}(G^\natural,\omega)$}

\ni D'aprs le lemme de \ref{inclusion iota +}, $\G_{{\Bbb C},{\rm e.t}}(G^\natural,\omega)$ est 
le sous--espace vectoriel de $\G_{\Bbb C}(G^\natural,\omega)$ 
engendrŽ par ${\rm Irr}_{{\Bbb C},{\rm e.t}}(G^\natural,\omega)={\rm Irr}_{0,{\rm e.t}}(G^\natural,\omega)$, 
et $\G_{{\Bbb C},{L-{\rm ind}}}(G^\natural,\omega)$ est le sous--espace vectoriel de $\G_{\Bbb C}(G^\natural,\omega)$ 
engendrŽ par les projections des $\widetilde{\Pi}_\mu$ dans $\G_{\Bbb C}(G^\natural,\omega)$, o $\mu$ 
parcourt les triplets de Langlands induits pour $(G^\natural,\omega)$. 

D'aprs le lemme 1 et le thŽorme de la base de Langlands pour $G$, on a la dŽcom\-position en somme 
directe
$$
\G(G^\natural,\omega)= \G_{\rm e.t}(G^\natural,\omega)\oplus
\G_{L-{\rm ind}}(G^\natural,\omega).
$$
On en dŽduit le

\begin{monlem2}
On a la dŽcomposition en somme directe
$$
\G_{\Bbb C}(G^\natural,\omega)= \G_{{\Bbb C},{\rm e.t}}(G^\natural,\omega)\oplus
\G_{{\Bbb C},{L-{\rm ind}}}(G^\natural,\omega).
$$
\end{monlem2}

\begin{proof} Pour allŽger l'Žcriture, posons $\G_{0^+}=\G_{0^+}(G^\natural,\omega)$, 
$\G_{\rm e.t}=\G_{\rm e.t}(G^\natural,\omega)$, etc. Il s'agit d'Žtablir l'inclusion
$$
\G_{0^+}\subset (\G_{\rm e.t}\cap \G_{0^+})+
(\G_{L-{\rm ind}}\cap \G_{0^+}).\leqno{(*)}
$$

Soit un ŽlŽment $\Pi$ de ${\rm Irr}_{k-1}(G^\natural,\omega)$ pour un entier $k> 1$, et soit $\Sigma\in {\rm Irr}_0(\mathcal{G}_k,\omega_k)$ 
tel que $\Pi=\iota_k(\Sigma)$. Dans $\G(\mathcal{G}_k,\omega_k)$, $\Sigma$ se dŽcompose en $\Sigma=\Sigma_1+\Sigma_2$, 
o $\Sigma_1\in \G_{\rm e.t}(\mathcal{G}_k,\omega_k)$ et $\Sigma_2\in \G_{L-{\rm ind}}(\mathcal{G}_k,\omega_k)$. 
Alors $\Pi=\iota_k(\Sigma_1)+\iota_k(\Sigma_2)$, et d'aprs le lemme de \ref{inclusion iota +}, $\Pi$ appartient ˆ la somme 
ˆ droite de l'inclusion $(*)$.

Soit un ŽlŽment $\Pi'$ de ${\rm Irr}_0(G^\natural,\omega)$ et des nombres complexes non nuls $\lambda_1,\ldots ,\lambda_n$, 
$n>1$, tels que $\sum_{i=1}^n\lambda_i=0$. 
\'Ecrivons $\Pi'=\Pi'_1+\Pi'_2$, o $\Pi'_1\in \G_{\rm e.t}$ et $\Pi'_2\in \G_{L-{\rm ind}}$. 
Alors $\Pi = \sum_{i=1}^n\lambda_i\cdot \Pi'$ s'Žcrit $\Pi=\Pi_1+\Pi_2$, o $\Pi_1=\sum_{i=1}^n\lambda_i\cdot \Pi'_1$ et 
$\Pi_2=\sum_{i=1}^n\lambda_i\cdot \Pi'_2$, par consŽquent $\Pi$ appartient ˆ la somme ˆ droite de l'inclusion $(*)$.

L'inclusion $(*)$ Žtant vŽrifiŽe, le lemme est dŽmontrŽ.
\end{proof}

\subsection{Support cuspidal et caractres infinitŽsimaux}\label{supp cusp}
Soit $A_\circ$ le tore dŽployŽ maximal (du centre) de $M_\circ$. On a $Z_G(A_\circ)=M_\circ$, 
et le normalisateur $N_G(A_\circ)$ de $A_\circ$ dans $G$ co\"{\i}ncide avec $N_G(M_\circ)$. 
On note $W_G$ le groupe de Weyl de $G$ dŽfini par\index{$W_G=N_G(A_\circ)/M_\circ$}
$$
W_G=N_G(A_\circ)/M_\circ.
$$

Soit $\pi$ une repr\'esentation irr\'eductible $\pi$ de $G$. 
\`A $\pi$ est associŽe une {\it paire cuspidale standard} $(M_P,\rho)$ de $G$, o\`u $P¬\in \EuScript{P}(G)$ 
et $\rho$ est une repr\'esentation irr\'eductible cuspidale de $M_P$, telle que $\pi$ est isomorphe \`a un  
sous--quotient de $i_P^G(\rho)$. Si $(M_{P'},\rho')$ est une autre paire cuspidale standard de $G$ 
telle que $\pi$ est isomorphe \`a un sous--quotient de $i_{P'}^G(\rho')$, alors il existe un ŽlŽment $w$ 
de $W_G$ tel que $w(M_P)=M_{P'}$ et ${^{n_w\!}\rho}\simeq \rho'$, o\`u:\index{$w(M_P)$, $n_w$}

\begin{itemize}
\item $n_w$ est un repr\'esentant de $w$ dans $N_G(A_\circ)$,
\item ${^{n_w\!}\rho}$ est la repr\'esentation $\rho\circ {\rm Int}_{\boldsymbol{G}}(n_w^{-1})$ de $w(M_P)= {\rm Int}_{\boldsymbol{G}}(n_w)(M_P)$.
\end{itemize}

\ni Pour $X\subset G$ et $g\in G$, on pose ${^gX}={\rm Int}_{\boldsymbol{G}}(g)(X)$. 
Deux paires cuspidales --- standard ou non --- $(M,\rho)$ et $(M'\!,\rho')$ de $G$ telles que $M'={^gM}$ et 
$\rho'\simeq {^g\rho}$ pour un $g\in G$ sont dites {\it $G$--\'equivalentes} ou simplement {\it Žquivalentes}, et l'on note 
$[M,\rho]=[M,\rho]_G$\index{$[M,\rho]=[M,\rho]_G$} la classe d'\'equivalence de $(M,\rho)$. Toute paire cuspidale de 
$G$ est Žquivalente ˆ une paire cuspidale standard. 
La classe d'\'equivalence d'une paire cuspidale 
standard $(M_P,\rho)$ de $G$ 
telle que $\pi$ est isomorphe \`a un sous--quotient 
de $i_P^G(\rho)$ est appelŽe {\it support cuspidal de $\pi$}, et not\'ee $\theta_G(\pi)$.

Pour une paire cuspidale $(M,\rho)$ de $G$, on note $\rho(1)$ la repr\'esentation 
$\omega^{-1}\rho^\theta$ de $\theta^{-1}(M)$. On obtient ainsi une autre paire cuspidale 
$$(M,\rho)(1)=(\theta^{-1}(M),\rho(1))$$ 
de $G$, dont la classe 
de $G$--\'equivalence, notŽe $[M,\rho](1)$, ne dŽpend pas du choix de $\delta_1\in M^\natural_\circ$. 
Notons que si $(M,\rho)$ est standard, i.e. si $M=M_P$ pour un $P\in \EuScript{P}(G)$, alors 
$(M,\rho)(1)$\index{$(M,\rho)(1)$, $[M,\rho](1)$} est aussi standard, puisque $\theta^{-1}(M_P)=M_{\theta^{-1}(M_P)P_1}$.

Soit $\Pi$ une $\omega$--repr\'esentation $G$--irr\'eductible de $G^\natural$. Posons 
$\pi=\Pi^\circ$, et soit $(M_P,\rho)$ une paire cuspidale standard de $G$ telle que 
$\theta_G(\pi)=[M_P,\rho]$.  Puisque $\pi$ est isomorphe ˆ $\pi(1)$, 
c'est un sous--quotient de 
$i_{\theta^{-1}(P)}^G(\rho(1))$, et aussi un sous--quotient de ${i}_{\theta^{-1}(M_P)P_1}^G(\rho(1))$. 
Par suite les paires cuspidales standard $(M_P,\rho)$ et $(M_P,\rho)(1)$ de $G$ 
sont \'equivalentes: 
il existe un ŽlŽment $w\in W_G$ tel que $w(M_P)=\theta^{-1}(M_P)$ et 
${^{n_w}\rho}\simeq \rho(1)$. En d'autres termes, posant $\tau=\theta\circ {\rm Int}_{\boldsymbol{G}}(n_w)$, on a 
$\tau(M_P)=M_P$ et $\omega^{-1}\rho^\tau\simeq \rho$.

On note $\Theta(G)$ l'ensemble des classes d'\'equivalence de paires cuspidales de 
$G$, et $\Theta_1(G)$ le sous--ensemble de $\Theta(G)$ form\'e des $[M,\rho]$ tels que 
$[M,\rho](1)=[M,\rho]$. L'application support cuspidal\index{$\Theta(G)$, $\Theta_1(G)$, $\theta_G$}
$$
\theta_G:{\rm Irr}(G)\rightarrow \Theta(G)
$$
induit une application
$$
{\rm Irr}_0(G^\natural,\omega)/{\Bbb C}^\times\rightarrow \Theta_1(G),\,\Pi\mapsto \theta_{G^\natural}(\Pi)= \theta_G(\Pi^\circ),
$$
dont l'image est notŽe $\Theta_{G^\natural,\omega}(G)$. 
Notons que pour qu'une paire cuspidale standard $(M_P,\rho)$ de $G$ telle que 
$[M_P,\rho](1)=[M_P,\rho]$ soit dans l'image de $\theta_{G^\natural}$, il faut et il suffit que l'induite 
parabolique $\pi=i_P^G(\rho)$ --- qui est isomorphe ˆ $\pi(1)$ --- possde un sous--quotient irrŽductible $\pi'$ tel que 
$\pi'(1)\simeq \pi'$. 

L'application $[M,\rho]\mapsto [M,\rho](1)$ dŽfinit comme en \ref{s(Pi)} 
une action de ${\Bbb Z}$ sur $\Theta(G)$. On note $\Theta(G)/{\Bbb Z}$ l'ensemble 
des orbites de ${\Bbb Z}$ pour cette action, et l'on identifie $\Theta_1(G)$ 
ˆ un sous--ensemble de $\Theta(G)/{\Bbb Z}$. 
L'application $\theta_G$ est ${\Bbb Z}$-Žquivariante, et l'application 
$\Pi\mapsto \theta_{G^\natural}(\Pi)$ 
se prolonge en une application surjective\index{$\Theta(G)/{\Bbb Z}$, $\theta_{G^\natural}$}
$$
\theta_{G^\natural}:{\rm Irr}(G^\natural,\omega)/{\Bbb C}^\times\rightarrow \Theta(G)/{\Bbb Z},
$$
dŽfinie comme suit. Pour une $\omega$--reprŽsentation irrŽductible $\Pi$ de $G^\natural$, on choisit 
une sous--reprŽsentation irrŽductible $\pi_0$ de $\Pi^\circ$, et l'on note $\theta_{G^\natural}(\Pi)$ la 
${\Bbb Z}$--orbite de $\theta_G(\pi_0)$ dans $\Theta(G)$. Cette ${\Bbb Z}$--orbite est bien dŽfinie --- i.e. 
elle ne dŽpend pas du choix de $\pi_0$ --- et dŽpend seulement de la classe d'isomorphisme de $\Pi$.

\subsection{Support inertiel}\label{support inertiel}
Deux paires cuspidales $(M,\rho)$ et $(M',\rho')$ de 
$G$ sont dites {\it inertiellement Žquivalentes} s'il existe un caractre non ramifiŽ $\psi'$ de 
$M'$ tel que ${^gM}=M'$ et ${^g\rho}\simeq \psi' \rho'$ pour un $g\in G$. 

Soit $\mathfrak{B}(G)$\index{$\mathfrak{B}(G)$} l'ensemble des classes d'\'equivalence inertielle de paires cuspidales de $G$. 
Pour $\mathfrak{s}\in \mathfrak{B}(G)$, on note $\Theta(\mathfrak{s})\subset \Theta(G)$\index{$\Theta(\mathfrak{s})$} la fibre au-dessus 
de $\mathfrak{s}$. Si $\mathfrak{s}$ est la classe d'\'equivalence inertielle d'une 
paire cuspidale $(M,\rho)$ de $G$, alors $\Theta(\mathfrak{s})$ est un espace homog\`ene sous 
le tore complexe $\mathfrak{P}(M)$ --- pour $\psi\in \mathfrak{P}(M)$, on pose $\psi\cdot[M,\rho]=[M,\psi\rho]$ ---, 
et en particulier une vari\'et\'e alg\'ebrique affine complexe. PrŽcisŽment, notons 
$\mathfrak{P}(M)(\rho)$\index{$\mathfrak{P}(M)(\rho)$} la $\mathfrak{P}(M)$--orbite de $\rho$ dans ${\rm Irr}(M)$, \cad 
l'ensemble des classes d'isomorphisme de reprŽsentations de $M$ de la forme $\psi\rho$ pour un 
$\psi\in \mathfrak{P}(M)$. C'est une variŽtŽ algŽbrique affine complexe 
(le quotient de $\mathfrak{P}(M)$ par un groupe fini). 
Posons\index{$W_G(M)=N_G(M)/M$, $W_\mathfrak{s}$}
$$
W_G(M)=N_G(M)/M
$$
et
$$
W_\mathfrak{s}= \{w\in W_G(M): {^{n_w}\rho}\simeq \psi \rho,\;\exists \psi\in \mathfrak{P}(M)\},
$$
o $n_w$ dŽsigne un reprŽsentant de 
$w$ dans $N_G(M)$. Le groupe 
$W_{\mathfrak{s}}$ opre sur la variŽtŽ $\mathfrak{P}(M)(\rho)$, et l'application 
$\psi\rho\mapsto [M,\psi\rho]$ est un isomorphisme de la variŽtŽ quotient 
$\mathfrak{P}(M)(\rho)/W_{\mathfrak{s}}$ sur $\Theta(\mathfrak{s})$. 

L'action de ${\Bbb Z}$ sur $\Theta(G)$ prŽserve les fibres de l'application $\Theta(G)\rightarrow \mathfrak{B}(G)$, 
donc induit une action sur $\mathfrak{B}(G)$, notŽe $(k,\mathfrak{s})\mapsto \mathfrak{s}(k)$: si $\mathfrak{s}$ est la 
classe inertielle d'une paire cuspidale 
$(M,\rho)$ de $G$, alors $\mathfrak{s}(k)$\index{$\mathfrak{s}(k)$} est la classe inertielle de $(\theta^{-k}(M),\rho(k))$. 
On note $\mathfrak{B}_1(G)$\index{$\mathfrak{B}_1(G)$, $\Theta_1(G)$} le sous--ensemble de $\mathfrak{B}(G)$ formŽ des $\mathfrak{s}$ tels que 
$\mathfrak{s}(1)=\mathfrak{s}$, i.e. tels que la fibre $\Theta(\mathfrak{s})$ au-dessus de $\mathfrak{s}$ est ${\Bbb Z}$--stable.

\begin{monlem}
Pour $\mathfrak{s}\in \mathfrak{B}_1(G)$, l'ensemble
$$
\Theta_1(\mathfrak{s})=\Theta(\mathfrak{s})\cap \Theta_1(G)
$$
est une sous--vari\'et\'e algŽbrique ferm\'ee (Žventuellement vide) de $\Theta(\mathfrak{s})$.
\end{monlem}

\begin{proof}
On peut supposer $\Theta_1(\mathfrak{s})$ non vide sinon 
il n'y a rien ˆ dŽmontrer. Soit $(M_P,\rho)$ une paire cuspidale standard 
de $G$ telle que $[M_P,\rho]\in \Theta_1(\mathfrak{s})$, et soit $w\in W_G$ tel 
que $\rho(1)\simeq {^{n_w}\rho}$. D'apr\`es \ref{supp cusp}, posant $\tau = \theta\circ {\rm Int}_{\boldsymbol{G}}(n_w)$, 
on a $\tau(M_P)=M_P$ et $\omega^{-1}\rho^\tau \simeq \rho$. De m\^eme pour $\psi\in \mathfrak{P}(M_P)$, 
la classe $[M_P,\psi \rho]$ appartient \`a $\Theta_1(\mathfrak{s})$ si et seulement s'il 
existe un ŽlŽment $w_\psi$ de $W_G$ tel que, posant $\tau_\psi=\theta\circ {\rm Int}_{\boldsymbol{G}}(n_{w_\psi})$, on a 
$\tau_\psi(M_P)=M_P$ et $\omega^{-1}(\psi \rho)^{\tau_\psi} \simeq \psi\rho$. Puisque
$$w(M_P)=\theta^{-1}(M_P)=w_\psi(M_P),$$ il existe un $s_\psi\in N_G(M_P)$ tel que 
$n_{w_\psi}=n_w s_\psi$. On a donc $\tau_\psi= \tau \circ {\rm Int}_{\boldsymbol{G}}(s_\psi)$, 
et comme
$$
\omega^{-1}(\psi \rho)^{\tau_\psi}
=(\psi^\tau \omega^{-1}\rho^\tau)^{{\rm Int}_{\boldsymbol{G}}(s_\psi)}
\simeq (\psi^\tau \rho)^{{\rm Int}_{\boldsymbol{G}}(s_\psi)}
$$
est isomorphe \`a $\psi \rho$, on a
$$
{^{s_\psi}\rho}\simeq ({^{s_\psi}\psi})^{-1}\psi^\tau \rho.
$$
En particulier, $\bar{s}_\psi = s_\psi\;({\rm mod}\,M_P)$ appartient au sous--groupe $W_\mathfrak{s}$ de $W_G(M_P)$, 
et $({^{\bar{s}_\psi}\psi})^{-1}\psi^\tau$ stabilise $[M_P,\rho]$. 
R\'eciproquement, supposons qu'il existe un $s\in N_G(M_P)$ tel que ${^s\rho}$ est isomorphe \`a 
$({^s\psi})^{-1}\psi^\tau  \rho$. Alors
$${^s(\psi\rho)}\simeq \psi^\tau \rho\simeq \psi^\tau(\omega^{-1}\rho^\tau)= \omega^{-1}(\psi \rho)^\tau,
$$
et $[M_P,\psi\rho]$ appartient \`a $\Theta_1(\mathfrak{s})$. En dŽfinitive, on a montrŽ qu'un ŽlŽment 
$[M_P,\psi\rho]$ de $\Theta(\frak{s})$ appartient ˆ $\Theta_1(\frak{s})$ si et seulement s'il existe un $s\in N_G(M_P)$ tel que 
${^s(\psi\rho)}\simeq \psi^\tau\rho$. Puisque le groupe $W_G(M_P)$ est fini, cette condition dŽfinit une sous--variŽtŽ algŽbrique fermŽe 
de $\Psi(M_P)(\rho)$. Comme le morphisme $\Psi(M)(\rho)\rightarrow \Theta(\mathfrak{s}),\, \psi\rho \mapsto [M_P,\psi\rho]$ est 
quasi--fini, il est fini (par homogŽnŽitŽ), donc fermŽ. D'o\`u le lemme.
\end{proof}

\v2
L'application support inertiel
$$\beta_G:{\rm Irr}(G)\rightarrow \mathfrak{B}(G)$$
induit une application\index{$\beta_G$, $\mathfrak{B}_{G^\natural,\omega}(G)$}
$$
{\rm Irr}_0(G^\natural,\omega)/{\Bbb C}^\times\rightarrow \mathfrak{B}_1(G),\,\Pi\mapsto 
\beta_{G^\natural}(\Pi)=\beta_G(\Pi^\circ),
$$
dont l'image est notŽe $\mathfrak{B}_{G^\natural,\omega}(G)$.

\begin{marema}
{\rm Soit $\mathfrak{s}\in \mathfrak{B}_1(G)$, et soit $(M_P,\rho)$ une paire 
cuspidale standard de $G$ telle que $[M_P,\rho]\in \Theta(\mathfrak{s})$. On sait qu'il existe un sous--ensemble Zariski--dense 
$\EuScript{V}$ de $\Theta(\mathfrak{s})$, que l'on peut supposer ${\Bbb Z}$--stable, 
tel pour tout $[M_P,\rho']\in \EuScript{V}$, la reprŽsentation ${i}_P^G(\rho')$ 
--- dŽfinie seulement ˆ isomorphisme prs --- est irrŽductible. Si l'intersection 
$\EuScript{V}\cap \Theta_1(G)$ est non vide, alors $\mathfrak{s}$ est dans 
l'image de $\beta_{G^\natural}$.\hfill $\blacksquare$}
\end{marema}

Comme en \ref{supp cusp}, on note $\mathfrak{B}(G)/{\Bbb Z}$ l'ensemble des orbites de 
${\Bbb Z}$ dans $\mathfrak{B}(G)$, 
et l'on identifie $\mathfrak{B}_1(G)$ ˆ un sous--ensemble de $\mathfrak{B}(G)/{\Bbb Z}$. 
L'application $\beta_G$ est 
${\Bbb Z}$--Žquivariante, et l'application $\Pi\mapsto \beta_{G^\natural}(\Pi)$ se prolonge 
comme en loc.~cit. en une application surjective\index{$\beta_{G^\natural}$, $\mathfrak{B}(G)/{\Bbb Z}$}
$$
\beta_{G^\natural}:{\rm Irr}(G^\natural,\omega)/{\Bbb C}^\times\rightarrow \mathfrak{B}(G)/{\Bbb Z}.
$$

\subsection{Le \og centre \fg (rappels, cas non tordu)}\label{centre}Le centre d'une catŽgorie abŽlienne 
$\A$ est par dŽfinition l'anneau des endomorphisme du foncteur identique de $\A$. On le 
note $\mathfrak{Z}(\A)$.\index{$\mathfrak{Z}(\A)$, $\mathfrak{Z}(G)=\mathfrak{Z}(\mathfrak{R}(G))$} 
Un ŽlŽment de $z$ de $\mathfrak{Z}(\A)$ est la 
donnŽe, pour chaque objet $E$ de $\A$, d'une fl\`eche $z_E:E\rightarrow E$ dans $\A$, 
de sorte que pour toute fl\`eche $u:E_1\rightarrow E_2$ dans $\A$, on ait
$$
u\circ z_{E_1}=z_{E_2}\circ u.
$$

Notons $\mathfrak{Z}(G)$ le centre de la catŽgorie $\mathfrak{R}(G)$. 
Rappelons que $\H=\H(G)$ est une ${\Bbb C}$--algbre {\it ˆ idempotent}  
(cf. \cite[1.1]{BD}), et que l'application $(\pi,V)\mapsto V$ est un isomorphisme 
de $\mathfrak{R}(G)$ sur la catŽgorie des $\H$-modules (ˆ gauche) non dŽgŽnŽrŽs. 
On peut voir $\H$ comme un $\H$--module non dŽgŽnŽrŽ pour 
la multiplication ˆ gauche. Pour chaque ŽlŽment $z\in \mathfrak{Z}(G)$, on a donc 
un ŽlŽment $z_\H \in {\rm End}_G(\H)$. D'apr\`es \cite[1.5]{BD}, l'application $z\mapsto z_\H$ 
est un isomorphisme de $\mathfrak{Z}(G)$ sur le commutant ${\rm End}_{\H\times \H^{\rm op}}(\H)$ 
dans ${\rm End}_{\Bbb Z}(\H)$ des multiplications ˆ gauche et ˆ droite. Pour $z\in \mathfrak{Z}(G)$ et $f\in \H$, 
on pose $z\cdot f= z_\H(f)$.

Pour $\mathfrak{s}\in \mathfrak{B}(G)$, notons $\mathfrak{R}_{\mathfrak{s}}(G)$ la 
sous-catŽgorie pleine de $\mathfrak{R}(G)$ formŽe des reprŽsentations 
$\pi$ telles que $\beta_G(\pi')=\mathfrak{s}$ pour tout sous-quotient irrŽductible 
$\pi'$ de $\pi$. D'aprs \cite[2.10]{BD}, 
$\mathfrak{R}(G)$ est le produit des catŽgories $\mathfrak{R}_{\mathfrak{s}}(G)$ pour 
$\mathfrak{s}$ parcourant les ŽlŽments de $\mathfrak{B}(G)$. En d'autres termes, toute 
reprŽsentation $\pi$ de $G$ s'Žcrit comme une somme directe 
$\pi=\oplus_{\mathfrak{s}}\pi_{\mathfrak{s}}$ o $\mathfrak{s}$ parcourt les 
ŽlŽments de $\mathfrak{B}(G)$ et $\pi_{\mathfrak{s}}$ est un objet de 
$\mathfrak{R}_{\mathfrak{s}}(G)$, et si $\pi'=\oplus_{\mathfrak{s}}\pi'_{\mathfrak{s}}$ 
est une autre reprŽsentation de $G$, on a 
$$
{\rm Hom}_G(\pi,\pi')=\oplus_{\mathfrak{s}}{\rm Hom}_G(\pi_{\mathfrak{s}},\pi'_{\mathfrak{s}}).
$$
On note $\mathfrak{Z}_\mathfrak{s}=\mathfrak{Z}_\mathfrak{s}(G)$ 
le centre de la catŽgorie $\mathfrak{R}_{\mathfrak{s}}(G)$. On a la dŽcomposition en produit d'anneaux
$$\mathfrak{Z}(G)=\prod_{\mathfrak{s}}\mathfrak{Z}_\mathfrak{s}.
$$
Pour toute partie $\mathfrak{S}$ de $\mathfrak{B}(G)$, on pose\index{$\mathfrak{R}_{\mathfrak{s}}(G)$, $\mathfrak{R}_{\mathfrak{S}}(G)$}
\index{$\mathfrak{Z}_\mathfrak{s}=\mathfrak{Z}_\mathfrak{s}(G)$, $\mathfrak{Z}_\mathfrak{S}(G)$}
$$
\mathfrak{R}_{\mathfrak{S}}(G)=\prod_{\mathfrak{s}\in \mathfrak{S}}\mathfrak{R}_{\mathfrak{s}}(G), \quad 
\mathfrak{Z}_\mathfrak{S}(G)=\prod_{\mathfrak{s}\in \mathfrak{S}}\mathfrak{Z}_\mathfrak{s}.
$$

Fixons un ŽlŽment $\mathfrak{s}\in \mathfrak{B}(G)$, et choisissons une paire 
cuspidale standard $(M_P,\rho)$ de $G$ telle que $[M_P,\rho]\in \Theta(\mathfrak{s})$. 
Rappelons que $M_P^1$ est le sous--groupe de 
$M_P$ engendrŽ par ses sous--groupes ouverts compacts. Posons $\Lambda_P=M_P/M_P^1$.
\index{$\Lambda_P=M_P/M_P^1$, $B_P={\Bbb C}[\Lambda_P]$, $\varphi_P$, $\varphi_{B_P}$} 
C'est un ${\Bbb Z}$--module libre de type fini, 
et l'on a $\mathfrak{P}(M_P)={\rm Hom}_{\Bbb Z}(\Lambda_P,{\Bbb C}^\times)$. En d'autres termes, 
$\Lambda_P$ est le groupe des caractres algŽbriques du 
tore complexe $\mathfrak{P}(M_P)$. Soit $B_P={\Bbb C}[\Lambda_P]$ l'algbre affine de $\mathfrak{P}(M_P)$, 
et soit $\varphi_P: M_P\rightarrow B_P$ le \og caractre universel \fg donnŽ par l'Žvaluation: 
$$\varphi_P(m)(\psi)= \psi(m),\quad 
m\in M_P,\, \psi\in \mathfrak{P}(M_P).
$$
C'est aussi le composŽ de la projection canonique $M_P\rightarrow \Lambda_P$ et 
de l'inclusion $\Lambda_P\hookrightarrow B_P$.
Soit $\rho_{B_P}$ la reprŽsentation de $M_P$ dŽfinie comme suit: l'espace de $\rho_{B_P}$ est $W=V_\rho\otimes_{\Bbb C}B_P$, 
et pour $m\in M_P$, $v\in V_\rho$ et $b\in B_P$, on pose
$$
\rho_{B_P}(m)(v\otimes b)=\rho(m)(v)\otimes \varphi_P(m)b.
$$
Notons $\pi$ la reprŽsentation ${i}_P^G(\rho_{B_P})$ de $G$. 
L'anneau $B_P$ opre naturellement sur l'espace $V$ de $\pi$, et pour 
$\psi\in \mathfrak{P}(M_P)$ correspondant ˆ $u\in {\rm Hom}_{\Bbb C}(B_P,{\Bbb C})$ 
--- i.e. tels que $\psi =u\circ \varphi_P$ --- 
la localisation $\pi_\psi$ de $\pi$ en $\psi$, \cad la reprŽsentation de $G$ dŽduite de $\pi$ sur 
l'espace $V_\psi = V\otimes_{B_P,u}{\Bbb C}$, est isomorphe \`a ${i}_P^G(\psi\rho)$.\index{$(\pi_\psi,V_\psi)$} 
Soit maintenant $z$ un \'elŽment de $\mathfrak{Z}_\mathfrak{s}$. D'aprs \cite[1.17]{BD}, 
l'endomorphisme $z_\pi$ de $\pi$ est la multiplication par un ŽlŽment de $B_P$, disons $b$, 
et pour chaque $\psi\in \mathfrak{P}(M_P)$, l'endomorphisme $z_{\pi_\psi}$ de $\pi_\psi$ 
est la multiplication par $b(\psi)$. De plus (loc. cit.), si $\psi,\,\psi'\in \mathfrak{P}(M_P)$ sont tels que 
${^{n_w}(\psi\rho)}\simeq \psi'\rho$ pour un $w\in W_G(M_P)$, alors $b(\psi)=b(\psi')$. Par consŽquent 
la fonction $\psi\mapsto b(\psi)$ sur $\mathfrak{P}(M_P)$ se descend en une 
fonction rŽgulire sur la variŽtŽ $\Theta(\mathfrak{s})$, disons $f_z$. 
Le thŽorme 2.13 de \cite{BD} dit que l'application\index{$z\mapsto f_z$}
$$
\mathfrak{Z}_\mathfrak{s}\rightarrow {\Bbb C}[\Theta(\mathfrak{s})],\,z\mapsto f_z
$$
est un isomorphisme d'anneaux. 

\begin{marema}
{\rm 
Soit $z\in \mathfrak{Z}(G)$, dŽcomposŽ en 
$z=\prod_{\mathfrak{s}}z_{\mathfrak{s}}$, $z_{\mathfrak{s}}\in \mathfrak{Z}_\mathfrak{s}$. 
Pour toute reprŽsentation irrŽductible $\pi$ de $G$ telle que $\beta_G(\pi)=\mathfrak{s}$, on a
$$
z_\pi= f_{z_{\mathfrak{s}}}(\theta_G(\pi)){\rm id}_{V_\pi}.\eqno{\blacksquare}
$$
}
\end{marema}

\subsection{L'anneau $\mathfrak{Z}(G^\natural,\omega)$}\label{l'anneau Z}
On l'a dit plus haut, l'action de ${\Bbb C}^\times$ sur ${\rm Irr}(G^\natural,\omega)$ provient d'une 
action fonctorielle sur $\mathfrak{R}(G^\natural,\omega)$, triviale sur les fl\`eches. Par dualitŽ on 
obtient une action sur l'anneau des endomorphismes du foncteur identique de $\mathfrak{R}(G^\natural,\omega)$. 
Soit $\mathfrak{Z}(G^\natural,\omega)$\index{$\mathfrak{Z}(G^\natural,\omega)$} l'anneau des endomorphismes {\it ${\Bbb C}^\times$--invariants} 
du foncteur identique de $\mathfrak{R}(G^\natural,\omega)$.  
Un ŽlŽment $\EuScript{Z}$ de $\mathfrak{Z}(G^\natural,\omega)$ 
est par dŽfinition la donnŽe, pour chaque $\omega$--reprŽsentation $\Pi$ de $G^\natural$, 
d'un endomorphisme $\EuScript{Z}_\Pi$ de $\Pi$ tel que $\EuScript{Z}_{\lambda\cdot \Pi}=\EuScript{Z}_\Pi$ 
pour tout $\lambda\in {\Bbb C}^\times$, de sorte que pour tout morphisme $u:\Pi\rightarrow \Pi'$ entre deux $\omega$--reprŽsentations 
$\Pi$ et $\Pi'$ de $G^\natural$, on ait
$$
u\circ \EuScript{Z}_\Pi =\EuScript{Z}_{\Pi'}\circ u.
$$

Soit $(\Pi,V)$ une $\omega$--reprŽsentation de $G^\natural$. DŽcomposons $\Pi^\circ$ en 
$$
\Pi^\circ =\oplus_{\mathfrak{S}}(\Pi^\circ)_{\mathfrak{S}}
$$
o $\mathfrak{S}$ parcourt les ŽlŽments de $\mathfrak{B}(G)/{\Bbb Z}$, 
et $(\Pi^\circ)_{\mathfrak{S}}$ 
est un objet de $\mathfrak{R}_{\mathfrak{S}}(G)$. Pour chaque ${\Bbb Z}$-orbite 
$\mathfrak{S}$ dans $\mathfrak{B}(G)$, l'opŽrateur $\Pi(\delta_1)$ envoie $\Pi^\circ_{\mathfrak{S}}$ sur 
$\Pi^\circ_{\mathfrak{S}}(1)=\Pi^\circ_{\mathfrak{S}}$, par consŽquent $\Pi$ se dŽcompose en
$$
\Pi=\oplus_{\mathfrak{S}}\Pi_{\mathfrak{S}},\quad \Pi_{\mathfrak{S}}^\circ =(\Pi_{\mathfrak{S}})^\circ,
$$
et si $\Pi'=\oplus_{\mathfrak{S}}\Pi'_{\mathfrak{S}}$ est une autre $\omega$--reprŽsentation de $G^\natural$, 
on a
$$
{\rm Hom}_{G^\natural}(\Pi,\Pi')=\oplus_{\mathfrak{S}}{\rm Hom}_{G^\natural}(\Pi_{\mathfrak{S}},\Pi'_{\mathfrak{S}}).
$$
Pour une partie $\mathfrak{S}$ de $\mathfrak{B}(G)/{\Bbb Z}$, on note $\mathfrak{R}_{\mathfrak{S}}(G^\natural,\omega)$
\index{$\mathfrak{R}_{\mathfrak{S}}(G^\natural,\omega)$, $\mathfrak{Z}_{\mathfrak{S}}(G^\natural,\omega)$} 
la sous--catŽgorie pleine de $\mathfrak{R}(G^\natural,\omega)$ formŽe des $\Pi$ tels que $\Pi^\circ$ est un 
objet de $\mathfrak{R}_{\mathfrak{S}}(G)$, ou --- ce qui revient au m\^eme --- tels que $\beta_{G^\natural}(\Pi')\in \mathfrak{S}$ 
pour tout sous--quotient irrŽductible $\Pi'$ de $\Pi$. Elle est stable sous l'action de ${\Bbb C}^\times$. 
On note $\mathfrak{Z}_{\mathfrak{S}}(G^\natural,\omega)$ le sous--anneau de $\mathfrak{Z}(G^\natural,\omega)$ formŽ 
des endomorphismes ${\Bbb C}^\times$--invariants du foncteur identique de $\mathfrak{R}_{\mathfrak{S}}(G^\natural,\omega)$. 
D'aprs ce qui prŽcde, $\mathfrak{R}(G^\natural,\omega)$ est le produit des catŽgories $\mathfrak{R}_{\mathfrak{S}}(G^\natural,\omega)$ 
pour $\mathfrak{S}$ parcourant les ŽlŽments de $\mathfrak{B}(G)/{\Bbb Z}$. Par suite on a la dŽcomposition en produit d'anneaux
$$
\mathfrak{Z}(G^\natural,\omega)=\prod_{\mathfrak{S}\in \mathfrak{B}(G)/{\Bbb Z}}\mathfrak{Z}_{\mathfrak{S}}(G^\natural,\omega).
$$

Posons\index{$\mathfrak{R}_0(G^\natural,\omega)$, $\mathfrak{Z}_0(G^\natural,\omega)$}
$$\mathfrak{R}_0(G^\natural,\omega)
=\mathfrak{R}_{\mathfrak{S}_0}(G^\natural,\omega),\quad \mathfrak{S}_0=\mathfrak{B}_{G^\natural,\omega}(G).
$$
Toute $\omega$--reprŽsentation 
$G$--irrŽductible de $G^\natural$ est un objet de $\mathfrak{R}_0(G^\natural,\omega)$, et 
$\mathfrak{R}_0(G^\natural,\omega)$ est le produit des catŽgories $\mathfrak{R}_{\mathfrak{s}}(G^\natural,\omega)$ 
pour $\mathfrak{s}$ parcourant les ŽlŽments de $\mathfrak{B}_{G^\natural,\omega}(G)$. 
Ainsi, pour dŽcrire l'anneau
$$\mathfrak{Z}_0(G^\natural,\omega)=\prod_{\mathfrak{s}\in \mathfrak{S}_0}\mathfrak{Z}_\mathfrak{s}(G^\natural,\omega)
$$
des endomorphismes ${\Bbb C}^\times$--invariants du foncteur identique de $\mathfrak{R}_0(G^\natural,\omega)$, il 
suffit de dŽcrire chacun des anneaux $\mathfrak{Z}_{\mathfrak{s}}(G^\natural,\omega)$.

\subsection{Action de ${\Bbb Z}$ sur le \og centre \fg}\label{action sur le centre}
L'application $(k,\pi)\mapsto \pi(k)$ dŽfinit une action fonctorielle de ${\Bbb Z}$ sur $\mathfrak{R}(G)$, 
triviale sur les fl\`eches. On en dŽduit une action $(k,z)\mapsto z(k)$\index{$(k,z)\mapsto z(k)$} 
de ${\Bbb Z}$ sur $\mathfrak{Z}(G)$: pour $z\in \mathfrak{Z}(G)$ et $k\in {\Bbb Z}$, $z(k)$ est l'ŽlŽment 
de $\mathfrak{Z}(G)$ donnŽ par
$$z(k)_\pi=z_{\pi(k)}.$$
Pour $k\in {\Bbb Z}$, l'application $\mathfrak{Z}(G)\rightarrow \mathfrak{Z}(G),\,z\mapsto z(k)$ est un automorphisme 
d'anneau (et m\^eme de ${\Bbb C}$--algbre). On note $\mathfrak{Z}_1(G)$\index{$\mathfrak{Z}_1(G)$} le sous--anneau de $\mathfrak{Z}(G)$ formŽ des $z$ tels que $z(1)=z$. 
Pour $z\in \mathfrak{Z}(G)$ et $\Pi$ une $\omega$--reprŽsentation de $G^\natural$, 
on a
$$
z_{\Pi^\circ(1)}\circ\Pi(\delta_1)= \Pi(\delta_1)\circ z_{\Pi^\circ}.
$$
En particulier si $z(1)=z$, alors $z_{\Pi^\circ(1)}=z_{\Pi^\circ}$ est un endomorphisme de $\Pi$.

\begin{marema1}
{\rm L'application $\H\rightarrow \H,\,f\mapsto \omega f^\theta$ est un automorphisme 
d'anneau, disons $\tau$. Pour toute reprŽsentation $\pi$ de $\mathfrak{R}(G)$, on a $\pi(k)(f)= \pi(\tau^{-k}(f))$, $k\in {\Bbb Z}$, $f\in \H$. 
On en dŽduit que l'action de ${\Bbb Z}$ sur $\mathfrak{Z}(G)$, identifiŽ ˆ ${\rm End}_{\H,\H^{\rm op}}(\H)$ via 
l'isomorphisme $z\mapsto z_\H$, est donnŽe par $z(k)\cdot f= z\cdot \tau^{-k}(f)$. Soit $\H(\tau-1)$\index{$\H(\tau-1)$, $\overline{\H}_1=\overline{\H}_1(G)$} 
le sous--espace vectoriel de $\H$ engendrŽ par les fonctions $f* (\tau(h)-h)$ pour $f,\,h\in \H$. 
C'est un idŽal bilatre $\tau$--stable de $\H$. Puisque $z(1)_\H = z_\H\circ \tau^{-1}$, 
un ŽlŽment $z$ de $\frak{Z}(G)$ appartient ˆ $\frak{Z}_1(G)$ 
si et seulement s'il s'annule sur $\H(\tau-1)$. Notons $\overline{\H}_1=\overline{\H}_1(G)$ 
le $\H$--bimodule quotient $\H/\H(\tau-1)$. L'isomorphisme $z\mapsto z_\H$ 
identifie donc $\mathfrak{Z}_1(G)$ au sous--anneau ${\rm End}_{\H\times \H^{\rm op}}(\overline{\H}_1)$ de 
${\rm End}_{\H\times \H^{\rm op}}(\H)$. \hfill $\blacksquare$}
\end{marema1}

Soit $\mathfrak{s}$ un ŽlŽment de $\mathfrak{B}_1(G)$. L'anneau $\mathfrak{Z}_\mathfrak{s}$ est 
${\Bbb Z}$--stable, et l'on note $\mathfrak{Z}_{\mathfrak{s},1}=\mathfrak{Z}_{\mathfrak{s},1}(G)$
\index{$\mathfrak{Z}_{\mathfrak{s},1}=\mathfrak{Z}_{\mathfrak{s},1}(G)$} le sous--anneau de 
$\mathfrak{Z}_\mathfrak{s}$ formŽ des $z$ tels que $z(1)=z$. La variŽtŽ $\Theta(\mathfrak{s})$ est 
elle--aussi ${\Bbb Z}$--stable, et l'isomorphisme 
$\mathfrak{Z}_\mathfrak{s}\rightarrow {\Bbb C}[\Theta(\mathfrak{s})],\,z\mapsto f_z$ se restreint en un isomorphisme 
d'anneaux
$$
\mathfrak{Z}_{\mathfrak{s},1}\rightarrow {\Bbb C}[\Theta(\mathfrak{s})]^{\Bbb{Z}},
$$
o $\Bbb{C}[\Theta(\frak{s})]^{\Bbb{Z}}$ dŽsigne le sous--anneau de $\Bbb{C}[\Theta(\frak{s}]$ formŽ des fonctions $\Bbb{Z}$--invariantes.

On dŽfinit comme suit une application
$$
\mathfrak{Z}_{\mathfrak{s},1}\rightarrow \mathfrak{Z}_{\mathfrak{s}}(G^\natural,\omega),
\,z\mapsto \iota(z).
$$
Pour $z\in \mathfrak{Z}_{\mathfrak{s},1}$ et $\Pi$ un objet de $\mathfrak{R}_{\mathfrak{s}}(G^\natural,\omega)$, 
l'endomorphisme $z_{\Pi^\circ}$ de $\Pi^\circ$ 
est en fait un endomorphisme de $\Pi$, et puisque $(\lambda\cdot\Pi)^\circ=\Pi^\circ$ pour tout $\lambda\in {\Bbb C}^\times$, c'est un 
ŽlŽment de $\mathfrak{Z}_{\mathfrak{s}}(G^\natural,\omega)$. On pose 
$\iota(z)_\Pi= z_{\Pi^\circ}$. L'application $\iota$ ainsi 
dŽfinie est un morphisme d'anneaux, et il est bijectif. En effet, on dŽfinit comme suit 
son inverse $\EuScript{Z}\mapsto \EuScript{Z}^\circ$. Pour $\EuScript{Z}\in \mathfrak{Z}_{\mathfrak{s}}(G^\natural,\omega)$ et $\pi$ 
un objet irrŽductible de $\mathfrak{R}_{\mathfrak{s}}(G)$, on pose $\widetilde{\pi}=\pi\oplus \pi(1)\oplus \cdots \oplus \pi(s-1)$ si $s=s(\pi)<+\infty$ 
et $\widetilde{\pi}=\oplus_{k\in {\Bbb Z}}\pi(k)$ sinon. La reprŽsentation $\widetilde{\pi}(1)$ de $G$ est 
isomorphe ˆ $\widetilde{\pi}$, et il existe une $\omega$--reprŽsentation $\Pi$ de $G^\natural$ 
telle que $\Pi^\circ=\widetilde{\pi}$. Cette reprŽsentation est irrŽductible, et d'aprs le lemme de Schur, 
l'endomorphisme 
$\EuScript{Z}_\Pi$ de $\Pi$ est la multiplication par une constante $\mu\in {\Bbb C}^\times$. On pose 
$(\EuScript{Z}^\circ)_\pi=\mu {\rm id}_{V_\pi}$. Par construction, on a $(\EuScript{Z}^\circ)_{\pi(k)}=(\EuScript{Z}^\circ)_\pi$ 
pour tout $k\in {\Bbb Z}$. D'aprs  \cite[1.8]{BD}, cela dŽfinit un ŽlŽment $\EuScript{Z}^\circ$ de $\mathfrak{Z}(G)$, 
qui vŽrifie $\EuScript{Z}^\circ(1)=\EuScript{Z}^\circ$. L'application
$$\mathfrak{Z}_{\mathfrak{s}}(G^\natural,\omega)
\rightarrow  \mathfrak{Z}_{\mathfrak{s},1},\,\EuScript{Z}\mapsto \EuScript{Z}^\circ
$$
ainsi dŽfinie est un morphisme d'anneaux, et pour $z\in \mathfrak{Z}_{\mathfrak{s},1}$ et 
$\EuScript{Z}\in \mathfrak{Z}_{\mathfrak{s}}(G^\natural,\omega)$, on a bien
$$
\iota(z)^\circ=z,\quad \iota(\EuScript{Z}^\circ)=\EuScript{Z}.
$$
En composant l'isomorphisme $\mathfrak{Z}_{\mathfrak{s}}(G^\natural,\omega)
\rightarrow  \mathfrak{Z}_{\mathfrak{s},1},\,\EuScript{Z}\mapsto \EuScript{Z}^\circ$ 
avec l'isomorphisme
$$\mathfrak{Z}_{\mathfrak{s},1}\rightarrow {\Bbb C}[\Theta(\mathfrak{s})]^{\Bbb Z},\,z\mapsto f_z,
$$ 
on obtient un isomorphisme d'anneaux
$$
\mathfrak{Z}_{\mathfrak{s}}(G^\natural,\omega)\rightarrow  {\Bbb C}[\Theta(\mathfrak{s})]^{\Bbb Z},\,\EuScript{Z}\mapsto f_{\EuScript{Z}}.
$$

\begin{marema2}
{\rm La description de l'anneau $\mathfrak{Z}_{\mathfrak{s}}(G^\natural,\omega)$ 
ci--dessus est valable m\^eme 
si la variŽtŽ $\Theta_1(\mathfrak{s})$ est vide, et a fortiori m\^eme 
si aucun objet de $\mathfrak{R}_{\mathfrak{s}}(G^\natural,\omega)$ 
n'est $G$--irrŽductible (on n'a pas supposŽ que $\mathfrak{s}$ appartient ˆ 
$\mathfrak{B}_{G^\natural,\omega}(G)$).\hfill $\blacksquare$}
\end{marema2}

\begin{marema3}
{\rm Si l'action de $\Bbb{Z}$ sur 
$\Theta(\frak{s})$ se factorise ˆ travers un quotient fini de $\Bbb{Z}$, alors l'ensemble $\Theta(\frak{s})/\Bbb{Z}$ des 
$\Bbb{Z}$--orbites dans $\Theta(\frak{s})$ est un quotient gŽomŽtrique de $\Theta(\frak{s})$. En particulier c'est une variŽtŽ algŽbrique 
affine, d'algbre affine $\Bbb{C}[\Theta(\frak{s})/\Bbb{Z}]= \Bbb{C}[\Theta(\frak{s})]^{\Bbb{Z}}$.

En gŽnŽral, cela n'est pas vrai: l'ensemble $\Theta(\frak{s})/\Bbb{Z}$ des $\Bbb{Z}$--orbites dans $\Theta(\frak{s})$ 
n'est pas un quotient gŽomŽtrique de $\Theta(\frak{s})$. Par exemple si $\boldsymbol{G}=\Bbb{G}_{\rm m}$, 
$\theta={\rm id}$ et $\omega$ est un caractres non ramifiŽ unitaire de $G=F^\times$ dont les puissances sont denses dans le tore 
$\mathfrak{P}(F^\times)\simeq {\Bbb C}^\times$, alors les $\Bbb{Z}$--orbites dans $\frak{P}(F^\times)$ 
ne sont pas fermŽes, et toute fonction rŽgulire $\Bbb{Z}$--invariante sur $\frak{P}(F^\times)$ est constante. \hfill $\blacksquare$}
\end{marema3}

\subsection{\og Bons \fg sous--groupes ouverts compacts de $G$}\label{bons sgoc} Si $J$ est un sous--groupe 
ouvert compact de $G$, on note ${\rm Irr}_J(G)$ le sous--ensemble de ${\rm Irr}(G)$ formŽ des 
$\pi$ tels que $V_\pi^J\neq 0$. D'aprs \cite[3.7 et 3.9]{BD}, $\beta_G({\rm Irr}_J(G))$ est un sous--ensemble {\it fini} 
de $\mathfrak{B}(G)$. On dit que $J$ est un \og bon \fg sous--groupe ouvert compact de $G$ 
si
$$
\beta_G^{-1}(\beta_G({\rm Irr}_J(G)))={\rm Irr}_J(G),
$$
auquel cas on pose $\mathfrak{S}(J)=\beta_G({\rm Irr}_J(G))$. De manire Žquivalente, $J$ est un bon sous--groupe 
ouvert compact de $G$ si et seulement si toute reprŽsentation $\pi$ de $G$ admet une dŽcomposition en 
somme directe\index{$\pi=\pi_J\oplus \pi_J^\perp$}
$$
\pi=\pi_J\oplus \pi_J^\perp
$$
o $\pi_J$ dŽsigne la sous--reprŽsentation de $\pi$ d'espace $\pi(G)(V^J)$ de $V$, 
et $\pi_J^\perp$ la sous--reprŽsen\-tation de $\pi$ d'espace $\pi(G)(\ker \pi(e_J))$.
En ce cas, on note 
$\mathfrak{R}_J(G)$ la sous--catŽgorie pleine 
de $\mathfrak{R}(G)$ formŽe des reprŽsentations $\pi$ telles que $\pi_J=\pi$. Elle est stable 
par sous--quotient, et co\"{\i}ncide avec $\mathfrak{R}_{\mathfrak{S}(J)}(G)$ \cite[3.9]{BD}. 
De plus (loc.~cit.), le foncteur 
$V\mapsto V^J$ est une Žquivalence entre $\mathfrak{R}_J(G)$\index{$\mathfrak{R}_J(G)$} et la catŽgorie des $\H_J$--modules 
(ˆ gauche), 
de quasi--inverse le foncteur qui a un $\H_J$--module $W$ associe le 
$\H$--module non dŽgŽnŽrŽ $(\H*e_J)\otimes_{\H_J}W$.

D'aprs \cite[2.9]{BD}, pour qu'un sous--groupe ouvert compact $J$ de $G$ soit bon, il suffit qu'il vŽrifie les conditions (a) et (b) suivantes --- notŽes (3.7.1) et (3.7.2) dans \cite[3.7]{BD} ---, pour tout sous--groupe de Levi $M$ de $G$ et 
tout sous--groupe parabolique $P$ de $G$ de composante de Levi $M$:

\begin{enumerate}
\item[{\bf (a)}] pour tout conjuguŽ $J'$ de $J$ dans $G$, 
la classe de conjugaison de $J'_P= (J'\cap P)/ (J'\cap U_P)$ dans $M$ ne dŽpend pas de $P$ ni de $J'$; 
o $J'_P$ est identifiŽ ˆ un sous--groupe de $M$ via l'isomorphisme canonique $M\rightarrow P/U_P$.
\item[{\bf (b)}] pour toute reprŽsentation $(\pi,V)$ de $G$, la projection canonique $V\rightarrow V_P=V/V(U_P)$ 
induit une application surjective $V^J \rightarrow (V_P)^{J_P}$; o $V(U_P)$ est le sous--espace de $V$ 
engendrŽ par les vecteurs $\pi(u)(v)-v$, $u\in U_P$, $v\in V$.
\end{enumerate}

Soit $I$ un sous--groupe d'Iwahori de $G$, \cad le fixateur {\it connexe} d'une chambre de 
l'immeuble (affine) Žtendu de $G$. Rappelons que $I$ est le groupe des points 
$\mathfrak{o}$--rationnels d'un $\mathfrak{o}$--schŽma en groupes affine lisse connexe 
$\mathfrak{I}$ de fibre gŽnŽrique ${\boldsymbol{G}}$. Pour chaque entier $n\geq 1$, on note $I^n$ 
le $n$--ime sous--groupe de congruence de $I$, \cad le noyau de la projection canonique 
(rŽduction modulo $\mathfrak{p}^n$) 
$\mathfrak{J}(\mathfrak{o})\rightarrow \mathfrak{J}(\mathfrak{o}/\mathfrak{p}^n)$. On 
pose aussi $I^0=I$.\index{$I=I^0$, $I^n$}

\begin{mapropo}
Pour $n\geq 0$, $I^n$ est un bon sous--groupe ouvert compact de $G$.
\end{mapropo}

\proof Le rŽsultat est connu mais nous n'en avons trouvŽ aucune dŽmonstration 
dans la littŽrature. Puisque 
$G$ opre transitivement sur les chambres de son immeuble Žtendu, quitte ˆ remplacer $I$ par l'un de ses conjuguŽs dans $G$, on peut supposer que la chambre fixŽe par $I$ est contenue dans l'appartement associŽ ˆ $A_\circ$. 
En ce cas $I$ est en \og bonne position \fg par rapport ˆ la paire $(P_\circ,M_\circ)$, 
\cad qu'il vŽrifie la dŽcomposition triangulaire\index{$\overline{U}_\circ$, $\overline{U}$}
$$
I= (I\cap \overline{U}_\circ)(I\cap M_\circ)(I\cap U_\circ),
$$
o $\overline{U}_\circ$ est le radical unipotent du sous--groupe parabolique de $G$ opposŽ ˆ 
$P_\circ$ par rapport 
ˆ $M_\circ$. D'ailleurs (cf. \cite{BT}) cette dŽcomposition est une \og dŽcomposition schŽmatique\fg au sens o 
il existe des $\mathfrak{o}$--schŽmas en groupes affines lisses connexes $\mathfrak{M}_\circ$, $\mathfrak{U}_\circ$, 
$\overline{\mathfrak{U}}_\circ$ de fibres gŽnŽriques $\boldsymbol{M}_\circ$, $\boldsymbol{U}_\circ$, 
$\overline{\boldsymbol{U}}_\circ$ et de groupes des points $\mathfrak{o}$--rationnels $I\cap M_\circ$, $I\cap U_\circ$, 
$I\cap \overline{U}_\circ$
tels que l'application produit $\overline{\mathfrak{U}}_\circ\times_\mathfrak{o} \mathfrak{M}_\circ\times_\mathfrak{o} \mathfrak{U}_\circ
\rightarrow \mathfrak{I}$ est un isomorphisme de $\mathfrak{o}$--schŽmas. Ici $\boldsymbol{M}_\circ$, 
$\boldsymbol{U}_\circ$, $\overline{\boldsymbol{U}}_\circ$ dŽsignent les groupes algŽbriques dŽfinis sur $F$ dont $M_\circ$, $U_\circ$, 
$\overline{U}_\circ$ sont les groupes des points $F$--rationnels. On en dŽduit que pour chaque entier 
$n\geq 1$, $I^n$ est en bonne position par rapport ˆ $(P_\circ,M_\circ)$.

Fixons un entier $n\geq 0$ et posons $J=I^n$. Fixons aussi un sous--groupe de Levi $M$ de $G$ et un 
sous--groupe parabolique $P$ de $G$ de composante de Levi $M$.  Notons $U$ le radical unipotent de $P$, et 
$\overline{U}$ le radical unipotent du sous--groupe parabolique de $G$ opposŽ ˆ $P$ par rapport ˆ $M$. 

Supposons pour commencer que $M$ contient $M_\circ$. Soit $g\in G$ tel que ${^gP}=gPg^{-1}$ contient $P_\circ$. 
Rappelons que $A_\circ$ est un tore dŽployŽ maximal de $G$. 
Puisque $A_\circ$ et ${^{g^{-1}}\!A_\circ}$ sont contenus dans $P$, il existe un $p\in P$ tel que ${^{g^{-1}}\!A_\circ}={^p\!A_\circ}$. 
Alors $x=gp$ appartient ˆ $\N_\circ=N_G(A_\circ)$,\index{$\N_\circ=N_G(A_\circ)$} et l'on a ${^gP}={^xP}$. Puisque $x$ stabilise l'appartement de l'immeuble Žtendu 
de $G$ associŽ ˆ $A_\circ$, la chambre fixŽe par ${^xI}$ est encore dans cet appartement, par consŽquent ${^x\!J}$ est en bonne position par 
rapport ˆ $(P_\circ,M_\circ)$, donc aussi par rapport ˆ $(P',M_{P'})$ pour tout $P'\in \P(G)$. 
Puisque $P'={^xP}\in \P(G)$ et ${^xM}\supset {^xM_\circ}=M_\circ$, on a $M_{P'}={^xM}$, d'o
$$
{^x\!J}=({^x\!J}\cap {^x\overline{U}})({^x\!J}\cap {^xM})({^x\!J}\cap {^xU}).
$$
En conjuguant par $x^{-1}$, on obient la 
dŽcomposition triangulaire
$$
J=(J\cap \overline{U})(J\cap M)(J\cap U).\leqno{(*)}
$$
D'aprs \cite[3.5.2]{BD}, la propriŽtŽ $(*)$ implique la condition {\bf (b)} pour la paire $(P,M)$. Quant 
ˆ la condition $({\boldsymbol{a}})$, soit $J'={^{g'}\!\!J}$ pour un $g'\in G$. Puisque $P_\circ
\subset {^xP}$ pour un $x\in \N_\circ$, la dŽcomposition d'Iwasawa $G=P_\circ\N_\circ I$ implique la dŽcomposition $G=P\N_\circ I$. 
Comme par ailleurs $I$ normalise $J$, on peut supposer que $g'=py$ pour un $p\in P$ et un $y\in \N_\circ$. 
Ainsi $J'\cap P= {^p(^y\!J\cap P)}$ et $J'_P$ est conjuguŽ dans $M$ ˆ $(^y\!J)_P$. Or 
le groupe ${^y\!J}$ est en bonne position par rapport ˆ $(P_\circ,M_\circ)$, par consŽquent 
lui aussi vŽrifie $(*)$. En particulier on a la dŽcomposition
$$
{^y\!J}\cap P= (^y\!J\cap M)(^y\!J\cap U),
$$
laquelle implique l'ŽgalitŽ $(^y\!J)_P ={^y\!J} \cap M$. Par rŽcurrence sur la longueur des ŽlŽments de $W_G=\N_\circ/M_\circ$, on 
en dŽduit comme dans la dŽmonstration du lemme 1.3.2 de \cite{L1} qu'il existe un ŽlŽment $y_M\in N_M(A_\circ)$ tel 
que ${^y\!J}\cap M= {^{y_M}(J\cap M)}$. Cela prouve que la condition ${\bf (a)}$ est vŽrifiŽe. 

Passons au cas gŽnŽral: on ne suppose plus que $M$ contient $M_\circ$. On procde alors exactement 
comme dans les dŽmonstrations des lemmes 1.3.2 et 1.3.3 de \cite{L1}.\endproof

\begin{marema}
{\rm On suppose toujours que le groupe $I$ est 
en bonne position par rapport ˆ $(P_\circ,M_\circ)$. Le groupe $\theta(I)$ est un autre 
sous--groupe d'Iwahori de $G$, et puisque $\delta_1$ appartient ˆ $M_\circ^\natural$, il est lui aussi 
en bonne position par rapport ˆ $(P_\circ,M_\circ)$: on a la dŽcomposition triangulaire
$$
\theta(I)=(\theta(I)\cap \overline{U}_\circ)(\theta(I)\cap M_\circ)(\theta(I)\cap U_\circ).
$$
Par suite il existe un Žlement $m_1\in M_1$ tel que $\theta(I)=m_1^{-1}Im_1$. 
Quitte ˆ remplacer $\delta_1$ par $m_1\cdot \delta_1\in M_\circ^\natural$, on peut supposer 
$\theta(I)=I$. Alors $\theta(I^n)=I^n$ pour chaque entier $n\geq 0$.\hfill $\blacksquare$
}
\end{marema}

D'aprs la remarque, on peut supposer vŽrifiŽe l'htpothse suivante. 

\begin{monhypo}
On suppose de plus que le point--base $\delta_1\in M_\circ^\natural$ est choisi de 
telle manire que le $F$--automorphisme $\theta={\rm Int}_{{\boldsymbol{G}}^\natural}(\delta_1)$ de 
${\boldsymbol{G}}$ stabilise un sous--groupe d'Iwahori de $G$ en bonne position par rapport ˆ $(P_\circ,M_\circ)$. 
\end{monhypo}

\v2 D'aprs la proposition et la remarque ci--dessus, il existe une base de voisinages de $1$ dans $G$ 
formŽe de sous--groupes ouverts compacts qui sont tous bons, $\theta$--stables et en 
bonne position par rapport ˆ $(P_\circ,M_\circ)$. Cette propriŽtŽ sera trs 
utile pour la suite.

\subsection{\og Bons \fg sous--espaces tordus ouverts compacts de $G^\natural$}\label{bons seoc}
Un sous--espace tordu ouvert compact $J^\natural$ de $G^\natural$ est dit \og bon \fg 
si le sous--groupe ouvert compact de $G$ sous--jacent ˆ $J^\natural$ est bon. 
D'aprs \ref{bons sgoc}, il existe des bons sous--espaces tordus ouverts compacts 
de $G^\natural$ aussi petits que l'on veut. 

Soit $J^\natural=J\cdot \delta$ un bon sous--espace ouvert compact de $G^\natural$ tel que $\omega\vert_J=1$. 
Si $\Pi$ est une $\omega$--reprŽsentation de $G^\natural$, puisque
$$V^J = \pi(e_J)(V)= \Pi(e_{J^\natural})(V)=\Pi(\delta)(V^J),
$$ 
la dŽcomposition de $V$ en $V=V_J\oplus V_J^\perp$ (cf. \ref{bons sgoc}) est $G^\natural$--stable: pour 
$g\in G$ et $v\in V$, on a
$$
\Pi(g\cdot \delta)\circ \pi(e_J)(v)= \pi(g)\circ \Pi(\delta)\circ \pi(e_J)(v)= \pi(g)\circ \Pi(e_{J^\natural})(v).
$$ 
En d'autres termes, $\Pi$ se dŽcompose en\index{$\Pi=\Pi_J\oplus \Pi_J^\perp$}
$$
\Pi=\Pi_J\oplus \Pi_J^\perp,
$$
o $\Pi^J$ est la restriction de $\Pi$ sur $V_J$ et $\Pi_J^\perp$ la restriction de $\Pi$ sur $V_J^\perp$. 
On note $\mathfrak{R}_J(G^\natural,\omega)$ la sous--catŽgorie pleine de $\mathfrak{R}(G,\omega)$ formŽ 
des $\omega$--reprŽsentations $\Pi$ telles que $\Pi_J=\Pi$. Puisque le foncteur d'oubli $\Pi\mapsto \Pi^\circ$ envoie 
$\mathfrak{R}_J(G^\natural,\omega)$ dans $\mathfrak{R}_J(G)$, la catŽgorie $\mathfrak{R}_J(G^\natural,\omega)$ 
est stable par sous--quotients.

Posons $\mathfrak{S}=\mathfrak{S}(J)$. C'est une partie (finie) ${\Bbb Z}$--stable de $\mathfrak{B}(G)$, et 
la catŽgorie $\mathfrak{R}_J(G^\natural,\omega)$ co\"{\i}ncide avec $\mathfrak{R}_{\mathfrak{S}}(G^\natural,\omega)$. 
De plus, le foncteur $V\mapsto V^J$ est une Žquivalence entre $\mathfrak{R}_J(G^\natural,\omega)$ 
et la catŽgorie des $(\Hn_{\sJ},\omega)$--modules non dŽgŽnŽrŽs. En effet, pour un 
$(\Hn_{\sJ},\omega)$--module non dŽgŽnŽrŽ $W$, l'espace $V= (\H*e_J)\otimes_{\H_J}W$ est un 
$\H$-module non dŽgŽnŽrŽ, tel que $V^J=W$ et $V_J=V$. Il dŽfinit donc une reprŽsentation de $G$, 
disons $\pi$. Pour $g\in G$ et $v\in V$ de la forme $v= (f*e_J)\otimes w$, o $f\in \H$ et $w\in W$, on pose 
$$
\Pi(g\cdot \delta)(v)=\pi(g)\circ \Pi(\delta)(v),\quad \Pi(\delta)(v)= ({^\tau(\omega^{-1}f)}* e_J)\otimes w;
$$
o l'on a posŽ ${^\tau{f'}}= f'\circ \tau^{-1}$, $\tau={\rm Int}_{{\boldsymbol{G}}^\natural}(\delta)$. Notant 
${_gf}$ la fonction $x\mapsto f(g^{-1}x)$, on a
\begin{eqnarray*}
\Pi(\delta)\circ \pi(g)(v)&= & \Pi(\delta)({_gf}*e_J\otimes w)\\
&=&{^\tau(\omega^{-1}{_gf})}*e_J\otimes w\\
&=& \omega^{-1}(g) \pi(\tau(g))\circ \Pi(\delta)(v)= \omega^{-1}(g)\Pi(\delta\cdot g)(v).
\end{eqnarray*}
On obtient ainsi 
une $\omega$--reprŽsentation $(\Pi,V)$ de $G^\natural$, telle que $\Pi^\circ = \pi$. Puisque le foncteur d'oubli 
$\Pi\mapsto \Pi^\circ$ est fidle, cela dŽfinit un quasi--inverse du foncteur $V\mapsto V^J$ entre 
$\mathfrak{R}_J(G^\natural,\omega)$ et la catŽgorie des 
$(\Hn_{\sJ},\omega)$--modules non dŽgŽnŽrŽs.

\begin{notation}
{\rm Soit ${\boldsymbol{J}}(G^\natural,\omega)$\index{${\boldsymbol{J}}(G^\natural,\omega)$} 
l'ensemble des bons sous--espaces ouverts compacts $J^\natural$ de $G^\natural$ 
tels que $\omega$ est trivial sur $J$, et soit ${\boldsymbol{J}}_{G^\natural,\omega}(G)$ l'ensemble des sous--groupes 
ouverts compacts $J$ de $G$ tels que $J$ est le groupe sous--jacent ˆ un ŽlŽment $J^\natural$ de ${\boldsymbol{J}}(G^\natural,\omega)$.}
\end{notation}

D'aprs \ref{bons sgoc}, si $I$ est un sous--groupe d'Iwahori de $G^\natural$ normalisŽ par $\delta_1$, notant 
$n_0$ le plus petit entier $\geq 0$ tel que $\omega\vert_{I^{n_0}}=1$, on a l'inclusion
$$
\{I^n:n\geq n_0\}\subset {\boldsymbol{J}}_{G^\natural,\omega}(G).
$$
En particulier, ${\boldsymbol{J}}_{G^\natural,\omega}(G)$ est une base de voisinages de $1$ dans $G$.

\subsection{$(H^\natural,\omega,B)$--modules admissibles}\label{modules admissibles}
La reprŽsentation ${i}_P^G(\rho_{B_P})$ de $G$ introduite en \ref{centre} est un 
{\it $(G,B_P)$--module admissible} au 
sens de \cite{BD}, \cad un $B_P$-module $V$ muni d'une reprŽsentation $\pi:G\rightarrow {\rm Aut}_{\Bbb C}(V)$  
telle que l'action de $G$ sur $V$ commute ˆ celle de $B_P$, vŽrifiant la condition d'admissibilitŽ: 
pour tout sous-groupe ouvert compact $J$ de $G$, le $B_P$-module $V^J$ est 
projectif et de type fini. Plus gŽnŽralement, on a:

\begin{madefi}
{\rm Soit $H^\natural$ un espace topologique tordu de groupe sous-jacent $H$ localement profini, et soit 
$\mathfrak{X}$ une variŽtŽ algŽbrique affine sur ${\Bbb C}$ 
d'algbre affine $B={\Bbb C}[\mathfrak{X}]$. On appelle {\it $(H^\natural,\omega,B)$--module admissible}\index{$(H^\natural,\omega,B)$--module admissible} la donnŽe 
d'un $B$-module $V$ et d'une $\omega$--reprŽsentation de $H^\natural$ d'espace $V$ 
telle que l'action de $H^\natural$ sur $V$ commute ˆ celle de $B$, vŽrifiant la condition d'admissibilitŽ: 
pour tout sous-groupe ouvert compact $J$ de $H$, 
le $B$--module $V^{J}$ est projectif de type fini --- i.e. le $(H,B)$--module sous--jacent est admissible. 
}
\end{madefi}

\begin{marema}
{\rm Si $H^\natural$ vŽrifie la propriŽtŽ $({\rm P}_1)$ de \cite[8.3]{L2}, \cad s'il existe une famille de sous--espaces 
tordus ouverts compacts de $H^\natural$ telle que les sous--groupes de $H$ sous--jacents aux ŽlŽments de cette famille 
forment une base de voisinage de $1$ dans $H$, alors la condition d'admissibilitŽ est Žquivalente ˆ: 
pour tout sous--espace tordu ouvert compact $J^\natural$ de $H^\natural$, le $B$--module $V^{J^\natural}$ est projectif 
de type fini.\hfill $\blacksquare$}
\end{marema}

Soit $V$ un $(H^\natural,\omega,B)$--module admissible, o $B$ est l'algbre affine 
d'une variŽtŽ algŽbrique affine $\mathfrak{X}$ sur ${\Bbb C}$ et $H^\natural$ opre 
sur $V$ via une $\omega$--reprŽsentation $\Pi$. Tout morphisme de variŽtŽs algŽbriques 
$u:\mathfrak{X}'\rightarrow \mathfrak{X}$ dŽfinit comme suit un 
$(H^\natural,\omega,B')$--module admissible $V_u=V\otimes_{B,\tilde{u}}B'$, 
o $B'={\Bbb C}[\mathfrak{X}']$ et $\tilde{u}:B'\rightarrow B$ 
est le morphisme d'algbres correspondant ˆ $u$. En particulier pour tout point 
$x\in \mathfrak{X}$, vu comme un morphisme $x:{\rm Spec}({\Bbb C})\rightarrow\mathfrak{X}$, 
le localisŽ $V_x= V\otimes_{B,\tilde{x}}{\Bbb C}$ de $V$ en $x$ est une $\omega$--reprŽsentation 
de $H^\natural$, notŽe $\Pi_x$, et la reprŽsentation sous--jacente $\Pi_x^\circ$ de $H$ est admissible et de 
longueur finie.

Soit $P\in \EuScript{P}(G^\natural)$, et soit $\Sigma$ une $\omega$--reprŽsentation de $M_{\sP}^\natural$ 
telle que la reprŽsentation sous--jacente $\Sigma^\circ$ de $M_P$ est admissible. Rappelons 
que $\mathfrak{P}_{\Bbb C}(M_{\sP}^\natural)$ est un groupe algŽbrique affine, diagonalisable sur ${\Bbb C}$ 
(cf. la remarque 2 de \ref{carac nr}). Notons 
$B=B_{P^\natural}$\index{$B_{P^\natural}={\Bbb C}[\mathfrak{P}_{\Bbb C}(M_{\sP}^\natural)]$, $\varphi_{P^\natural}$} 
l'algbre affine ${\Bbb C}[\mathfrak{P}_{\Bbb C}(M_{\sP}^\natural)]$, 
et $\varphi_{P^\natural}:M_{\sP}^\natural\rightarrow B$ le 
\og caractre universel \fg donnŽ par l'Žvaluation:
$$\varphi_{P^\natural}(\delta)(\Xi)=\Xi(\delta),\quad \delta\in M_{\sP}^\natural,\, \Xi\in \mathfrak{P}_{\Bbb C}(M_{\sP}^\natural).
$$
On dŽfinit comme en \ref{centre} une $\omega$--reprŽsentation $\Sigma_B=\Sigma\otimes\varphi_{P^\natural}$ 
de $M_{\sP}^\natural$: l'espace de $\Sigma_B$ est $W=V_\Sigma\otimes_{\Bbb C}B$, et pour 
$\delta\in M_{\sP}^\natural$, $v\in W$ et $b\in B$, on pose
$$
\Sigma_B(\delta)(v\otimes b)=\Sigma(\delta)(v)\otimes \varphi_{P^\natural}(\delta)b.
$$
Posons
$$
(\Pi,V)= {^\omega{{i}}_{P^\natural}^{G^\natural}}(\Sigma_B,W).
$$
C'est une $\omega$--reprŽsentation de $G$. L'anneau $B$ opre naturellement sur l'espace $V$, 
ce qui le munit d'une structure de $(G^\natural,\omega,B)$--module admissible. Pour 
$\Xi\in \mathfrak{P}_{\Bbb C}(M_{\sP}^\natural)$ correspondant ˆ $u:B\rightarrow {\Bbb C}$, la $\omega$--reprŽsentation 
$\Pi_\Xi$ de $G^\natural$ sur le localisŽ $V_\Xi= V\otimes_{B,u}{\Bbb C}$ est isomorphe ˆ 
${^\omega{{i}}_{P^\natural}^{G^\natural}}(\Xi\cdot\Sigma)$. 

Soit une fonction $\phi\in \Hn$. Rappelons qu'on a notŽ $\Phi_\phi$ l'ŽlŽment de $\G_{\Bbb C}(G^\natural,\omega)^*$ 
dŽfini par\index{$\phi\mapsto \Phi_\phi$}
$$
\Phi_\phi(\Pi')=\Theta_{\Pi'}(\phi),\quad\Pi'\in {\rm Irr}_{\Bbb C}(G^\natural,\omega).
$$ 
Soit $J^\natural$ un sous--espace tordu ouvert compact de $G^\natural$, 
de groupe sous--jacent $J$, tel que $\omega\vert_J=1$ et $\phi\in \Hn_{\sJ}$. 
L'opŽrateur $\Pi(\phi)$ est un $B$--endomorphisme du sous--espace $V^J=V^{J^\natural}$ de 
$V$ formŽ des vecteurs fixŽs par $J$. Puisque $V^J$ est un $B$--module projectif de type fini, 
on dispose d'une application trace ${\rm tr}_B:{\rm End}_B(V^J)\rightarrow B$. 
On pose
$$
b={\rm tr}_B(\Pi(\phi))\in B.
$$
Pour $\Xi\in \mathfrak{P}_{\Bbb C}(M_{\sP}^\natural)$ correspondant ˆ $u:B\rightarrow {\Bbb C}$, l'endomorphisme $\Pi(\phi)\otimes_{B,u}{\Bbb C}$ de 
$$(V\otimes_{B,u}{\Bbb C})^J=V^J\otimes_{B,u}{\Bbb C}$$ est isomorphe ˆ ${^\omega{{i}}_{P^\natural}^{G^\natural}}(\Xi\cdot\Sigma)(\phi)$, par consŽquent 
$$
b(\Xi)= \Phi_\phi({^\omega{{i}}_{P^\natural}^{G^\natural}}(\Xi\cdot\Sigma)).
$$
En d'autres termes, l'application $\Xi\mapsto \Phi_\phi({^\omega{{i}}_{P^\natural}^{G^\natural}}(\Xi\cdot\Sigma))$ est 
une fonction rŽgulire sur la variŽtŽ $\mathfrak{P}_{\Bbb C}(M_{\sP}^\natural)$.

\section{\'EnoncŽ du rŽsultat}

\subsection{Le thŽorme principal}\label{ŽnoncŽ} 
Soit $\F(G^\natural,\omega)$\index{$\F(G^\natural,\omega)$, $\F_{\rm tr}(G^\natural,\omega)$} 
le sous--espace de $\G_{\Bbb C}(G^\natural,\omega)^*$ formŽ des 
formes linŽaires $\Phi$ qui vŽrifient les conditions (i) et (ii) suivantes:

\begin{enumerate}
\item[(i)] Il existe un ensemble {\rm fini} $\mathfrak{S}\subset\mathfrak{B}_{G^\natural,\omega}(G)$ tel que $\Phi(\Pi)=0$ 
pour tout $\Pi\in {\rm Irr}_{\Bbb C}(G^\natural,\omega)$ tel que $\beta_{G^\natural}(\Pi)\not\in \mathfrak{S}$.
\item[(ii)] Pour $P^\natural\in \EuScript{P}(G^\natural)$ et $\Sigma\in {\rm Irr}_{\Bbb C}(M_{\sP}^\natural,\omega)$, l'application 
$\Xi\mapsto \Phi({^\omega{i}}_{P^\natural}^{G^\natural}(\Xi \Sigma))$
est une fonction r\'eguli\`ere sur la variŽtŽ $\mathfrak{P}_{\Bbb C}(M_{\sP}^\natural)$.
\end{enumerate}

\begin{marema1}
{\rm D'aprs \cite[3.7 et 3.9]{BD}, la condition (i) est Žquivalente ˆ la condition (i') suivante: {\it 
il existe un sous--groupe ouvert compact $J$ de $G$ tel que $\Phi(\Pi)=0$ pour tout 
$\Pi\in {\rm Irr}_{\Bbb C}(G^\natural,\omega)$ tel que $V_\Pi^J=0$.} D'ailleurs d'aprs \ref{H modules} 
cette condition (i') est Žquivalente ˆ la condition (i'') suivante: 
{\it il existe un sous--espace tordu ouvert compact $J^\natural$ de $G^\natural$ tel que $\Phi(\Pi)=0$ pour tout 
$\Pi\in {\rm Irr}_{\Bbb C}(G^\natural,\omega)$ tel que $V_\Pi^{J^\natural}=0$.} D'autre part dans la condition (ii), on 
peut bien s\^ur remplacer le groupe algŽbrique affine $\mathfrak{P}_{\Bbb C}(M^\natural_{\sP})$ par sa composante neutre 
$\mathfrak{P}_{\Bbb C}^0(M^\natural_{\sP})$.
\hfill $\blacksquare$
}
\end{marema1}

Soit aussi $\F_{\rm tr}(G^\natural,\omega)$ le sous--espace de $\F(G^\natural,\omega)$ formŽ des 
$\Phi$ de la forme $\Phi_\phi$ pour une fonction $\phi\in \Hn$.

Soit enfin $[\Hn,\H]_\omega$\index{$[\Hn,\H]_\omega$, 
$\smash{\overline{\H}}^\natural_\omega=\overline{\H}(G^\natural,\omega)$} le sous--espace vectoriel de $\Hn$ engendrŽ par les fonctions de la forme $\phi*f -\omega f*\phi$ 
pour $\phi\in \Hn$ et $f\in\H$. On note $\smash{\overline{\H}}^\natural_\omega=\overline{\H}(G^\natural,\omega)$ 
l'espace quotient $\H^\natural/[\H^\natural,\H]_\omega$. En d'autres termes, 
$\smash{\overline{\H}}^\natural_\omega$ est le quotient de l'espace $\H^\natural_\omega= \H(G^\natural,\omega)$ 
introduit en \ref{H modules} (variante) par le sous--espace $[\H^\natural_\omega,\H]$ 
engendrŽ par les commutateurs $\phi\cdot f - f\cdot \phi$ pour $\phi\in \H^\natural_\omega$ et $f\in \H$.

\begin{montheo}
L'application $\H(G^\natural)\mapsto \G_{\Bbb C}(G^\natural,\omega)^*,\,\phi\mapsto \Phi_\phi $ induit un 
isomorphisme de ${\Bbb C}$--espaces vectoriels
$$
\overline{\H}(G^\natural,\omega) \rightarrow \F(G^\natural,\omega).
$$
\end{montheo}

D'aprs \ref{modules admissibles}, on a l'inclusion $\EuScript{F}_{\rm tr}(G^\natural,\omega)\subset 
\EuScript{F}(G^\natural,\omega)$. Si $\Pi$ est 
une $\omega$--reprŽsentation de $G^\natural$, pour $\phi\in \Hn$ et $f\in \H$, 
on a $\Pi(\phi* f)= \Pi(\omega f*\phi)$ et donc 
$\Pi(\phi*f -\omega f* \Pi)=0$. Ainsi, la transformŽe de Fourier $\Hn\mapsto \G_{\Bbb C}(G^\natural,\omega)^*,\,\phi\mapsto \Phi_f$ 
induit bien une application ${\Bbb C}$--linŽaire
$$
\overline{\H}(G^\natural,\omega) \rightarrow \EuScript{F}(G^\natural,\omega),
$$
et il s'agit de prouver qu'elle est surjective (thŽorme de Paley--Wiener) 
et injective (thŽorme de densitŽ spectrale). La suite de l'article et \cite{HL} sont consacrŽs ˆ la dŽmonstration de ces deux rŽsultats. 
Par rŽcurrence sur la dimension des sous--espaces paraboliques de $G^\natural$, on se ramne dans la section 4 
ˆ dŽmontrer un thŽorme analogue (\ref{ŽnoncŽ discret}) sur la partie \og discrte\fg de la thŽorie. 
La surjectivitŽ de l'application du thŽorme de \ref{ŽnoncŽ discret} est prouvŽe dans la section 5, tandis que son injectivitŽ 
est prouvŽe dans \cite{HL}. 

On peut voir $\H^\natural$ comme 
un $\H$--module non dŽgŽnŽrŽ (ˆ gauche). Pour chaque ŽlŽment $z$ de $\mathfrak{Z}(G)$, on a donc 
un ŽlŽment $z_{\H^\natural}\in {\rm End}_G(\H^\natural)$. Reprenons le ${\Bbb C}$--isomorphisme $u:\H\rightarrow \H^\natural_\omega$ 
de l'exemple de \ref{H modules} (rappelons que $\H^\natural_\omega=\H^\natural$ comme $\H$--module ˆ gauche). 
Pour $f,\,h,\,f'\in \H$, il vŽrifie $u(f*h*f')= f\cdot u(h)\cdot \tau(f')$, o l'on a posŽ 
$\tau(f')=\omega f'^\theta$. On a donc
$$u \circ z_\H= 
z_{\H^\natural}\circ u,\quad z\in \frak{Z}(G).
$$
En particulier pour $f,\, f'\in \H$ et $\phi\in \H^\natural_\omega$, posant $h=u^{-1}(\phi)$, on a
\begin{eqnarray*}
z_{\H^\natural}(f\cdot \phi\cdot f') &=& z_{\H^\natural}\circ u (f*h*\tau^{-1}(f'))\\
&=&u\circ z_\H(f*h*\tau^{-1}(f'))\\
&=&u(f*z_\H(h)*\tau^{-1}(f'))= f\cdot z_{\H^\natural}(\phi)\cdot f'.
\end{eqnarray*}
L'application $z\mapsto z_{\H^\natural}$ est donc un isomorphisme 
de $\mathfrak{Z}(G)$ sur ${\rm End}_{\H\times \H^{\rm op}}(\H^\natural_\omega)$. De plus 
$[\H^\natural_\omega,\H]$ est un sous--espace vectoriel $\frak{Z}(G)$--stable de $\H^\natural_\omega$, i.e. 
$[\H^\natural,\H]_\omega$ 
est un sous--espace vectoriel $\mathfrak{Z}(G)$--stable de $\H^\natural$.  

L'anneau $\mathfrak{Z}(G)$ opre sur l'espace $\G_{\Bbb C}(G^\natural,\omega)^*$: 
pour $z\in \mathfrak{Z}(G)$ et $\Phi\in \G_{\Bbb C}(G^\natural,\omega)^*$, on note $z\cdot \Phi$ 
l'ŽlŽment de $\G_{\Bbb C}(G^\natural,\omega)^*$ dŽfini par
$$
(z\cdot \Phi)(\Pi)= f_z(\theta_{G^\natural}(\Pi))\Phi(\Pi),\quad \Pi\in {\rm Irr}_{\Bbb C}(G^\natural,\omega).
$$

\begin{monlem}
La transformŽe de Fourier $\H(G^\natural)\mapsto \G_{\Bbb C}(G^\natural,\omega)^*,\,\phi\mapsto \Phi_\phi $ 
est un morphisme de $\mathfrak{Z}(G)$--modules.
\end{monlem}

\begin{proof}
D'aprs \cite[1.5]{BD}, le centre $\mathfrak{Z}(G)$ de $\mathfrak{R}(G)$ est la limite projective des centres 
$Z(e*\H*e)$ o $e$ parcourt les idempotents de $\H$, pour les morphismes de transitions 
$$Z(e'*\H*e')\rightarrow Z(e*\H*e),\, h\mapsto h*e\quad \hbox{si 
$e*\H*e\subset e'* \H*e'$.}
$$ En d'autres termes, 
tout \'el\'ement 
$z$ de $\mathfrak{Z}(G)$ est la donn\'ee, pour chaque idempotent $e$ de $\EuScript{H}$, d'un \'el\'ement 
$z(e)\in  Z(e*\EuScript{H}*e)$, avec la relation $z(e)=z(e')*e$ si $e = e*e'$. L'action de $\mathfrak{Z}(G)$ 
sur le $\H$--module (ˆ gauche) $\H$ est donnŽe par (pour $z\in \mathfrak{Z}(G)$ et $f\in \H$):
$$
z\cdot f =z(e)*f \quad \hbox{si $e*f =f$} 
$$
De m\^eme l'action de $\mathfrak{Z}(G)$ sur $\H^\natural$ est donnŽe par 
(pour $z\in \mathfrak{Z}(G)$ et $\phi\in \H^\natural$):
$$
z\cdot \phi = z(e)* \phi\quad \hbox{si $e*\phi =\phi$}.
$$
Si $\Pi$ est une $\omega$--repr\'esentation de $G^\natural$, pour $z\in \mathfrak{Z}(G)$ et $\phi\in \EuScript{H}^\natural$, 
on a
$$
\Pi(z\cdot \phi)=\Pi(z(e)*\phi)= \Pi^\circ(z(e))\circ \Pi(\phi)=z_{\Pi^\circ}\circ \Pi(\phi),
$$
o\`u $\Pi^\circ(z(e))$ est l'opŽrateur $\int_G z(e)(g)\Pi^\circ(g)dg$ sur l'espace de $\Pi$. 
Si de plus $\Pi$ est $G$--irr\'eductible, on a $z_{\Pi^\circ}=f_z(\theta_G(\Pi^\circ)){\rm id}_{V_\Pi}$. 
D'o le lemme.
\end{proof}

\begin{marema2}
{\rm 
Le lemme ci--dessus ne dŽpend pas du thŽorme. Joint au thŽorme, il implique en particulier que le sous--espace 
$\F(G^\natural,\omega)$ de $\G_{\Bbb C}(G^\natural,\omega)^*$ est 
$\mathfrak{Z}(G)$--stable.
\hfill $\blacksquare$
}\end{marema2}

\subsection{Variante \og tempŽrŽe \fg du thŽorme}\label{variante tempŽrŽe}
Pour $P^\natural\in \P(G^\natural)$, $\mathfrak{P}_{\Bbb C}(M_{\sP}^\natural,\vert \omega \vert)$ 
est un espace principal homogne sous $\mathfrak{P}_{\Bbb C}(M_{\sP}^\natural)$, donc en particulier une variŽtŽ algŽbrique affine complexe. 
Comme dans le cas non tordu, on peut remplacer la condition (ii) du thŽorme de \ref{ŽnoncŽ} 
par la condition (ii') suivante: {\it pour $P^\natural\in \EuScript{P}(G^\natural)$ et $\Sigma\in {\rm Irr}_{{\Bbb C},{\rm t}}(M_{\sP}^\natural,\omega_{\rm u})$, l'application 
$\Xi\mapsto \Phi({^\omega{i}}_{P^\natural}^{G^\natural}(\Xi \Sigma))$
est une fonction r\'eguli\`ere sur la variŽtŽ $\mathfrak{P}_{\Bbb C}(M_{\sP}^\natural,\vert\omega\vert)$.} 
Cette affirmation rŽsulte du lemme suivant.

\begin{monlem}
Soit $\Phi\in \F(G^\natural,\omega)$ tel que pour tout $P^\natural\in \P(G^\natural)$ et 
tout $\Sigma\in {\rm Irr}_{{\Bbb C},{\rm t}}(M_{\sP}^\natural,\omega_{\rm u})$, l'application $\Xi\mapsto 
\Phi({^\omega{{i}}_{P^\natural}^{G^\natural}}(\Xi\Sigma))$ est une fonction rŽgulire sur 
la variŽtŽ $\mathfrak{P}_{\Bbb C}(M_{\sP}^\natural,\vert\omega\vert)$. Alors pour tout $P^\natural\in \P(G^\natural)$ et 
tout $\Sigma\in {\rm Irr}_{\Bbb C}(M_{\sP}^\natural,\omega)$, 
l'application $\Xi\mapsto 
\Phi({^\omega{{i}}_{P^\natural}^{G^\natural}}(\Xi\Sigma))$ est une fonction rŽgulire sur la variŽtŽ 
$\mathfrak{P}_{\Bbb C}(M_{\sP}^\natural)$.
\end{monlem}

\begin{proof}D'aprs le lemme 1 de \ref{base langlands}, 
il suffit de montrer que pour tout $P^\natural\in \P(G^\natural)$ et tout triplet de Langlands $\mu'$ pour $(M^\natural_{\sP},\omega)$, 
notant $\widetilde{\Sigma}_{\mu'}$ la $\omega$--reprŽsentation de $M^\natural_{\sP}$ associŽe 
ˆ $\mu'$ et $\widetilde{\Sigma}_{{\Bbb C},{\mu'}}$ son image dans $\G_{{\Bbb C}}(M^\natural_{\sP},\omega)$, 
l'application $\Xi\mapsto \Phi({^\omega{i}_{P^\natural}^{G^\natural}}(\Xi\widetilde{\Sigma}_{{\Bbb C},\mu'}))$ 
est une fonction rŽgulire sur la variŽtŽ $\mathfrak{P}_{\Bbb C}(M_{\sP}^\natural)$. Le triplet $\mu'$ s'Žcrit 
$\mu'=(Q^\natural \cap M^\natural_{\sP}, \Sigma',\Xi')$ o $Q^\natural$ est un ŽlŽment de $\P(G^\natural)$ tel 
que $Q^\natural\subset P^\natural$, $\Sigma'$ est une $\omega_{\rm u}$--reprŽsentation $M_Q$--irrŽductible 
tempŽrŽe de $M^\natural_{{\sQ}}$, et $\Xi'$ est un ŽlŽment de $\mathfrak{P}_{\Bbb C}(M^\natural_{{\sQ}},\vert\omega\vert)$ 
tel que $\Xi'^\circ$ est positif par rapport ˆ $U_Q\cap M_P$. On a (par dŽfinition) $\widetilde{\Sigma}_{\mu'}=
{^\omega{i}_{Q^\natural}^{P^\natural}}(\Xi'\cdot \Sigma')$ et
$$
{^\omega{i}_{P^\natural}^{G^\natural}}(\Xi\cdot \widetilde{\Sigma}_{\mu'})= 
{^\omega{i}_{Q^\natural}^{G^\natural}}(\Xi\vert_{M^\natural_{\sQ}}\Xi'\cdot \Sigma'),\quad \Xi\in \mathfrak{P}_{\Bbb C}(M^\natural_{\sP}).
$$
D'o le rŽsultat, puisque l'application
$$
\mathfrak{P}_{\Bbb C}(M^\natural_{\sP})\rightarrow \mathfrak{P}_{\Bbb C}(M^\natural_{\sQ},\vert \omega\vert),\,
\Xi \mapsto \Xi\vert_{M^\natural_{\sQ}}\Xi'
$$
est un morphisme algŽbrique. \end{proof}

\begin{marema}
{\rm Supposons le caractre $\omega$ {\it unitaire} (i.e. $\vert \omega \vert=1$). Pour $P^\natural\in \P(G^\natural)$, notons $\frak{P}_{\Bbb{C},{\rm u}}^0(M^\natural_{\sP})$ 
le sous--groupe de $\frak{P}_{\Bbb{C}}^0(M^\natural_{\sP})$ formŽ des ŽlŽments unitaires. Rappelons que $\frak{P}_{\Bbb{C}}^0(M^\natural_{\sP})$ est un sous--ensemble 
de ${\rm Irr}_{\Bbb C}(M^\natural_{\sP})={\rm Irr}_{\Bbb C}(M^\natural_{\sP},\xi=1)$, et qu'une reprŽsentation de $M^\natural_{\sP}$ est dite unitaire si son espace est muni d'un produit scalaire hermitien $M^\natural_{\sP}$--invariant (cf. \ref{langlands}). 
Alors on peut, comme le fait Waldspurger \cite[6.1]{W}, remplacer la condition (ii) du thŽorme de \ref{ŽnoncŽ} 
par la condition (ii'') suivante: {\it pour $P^\natural\in \EuScript{P}(G^\natural)$ et $\Sigma\in {\rm Irr}_{{\Bbb C},{\rm t}}(M_{\sP}^\natural,\omega)$, l'application 
$\Xi\mapsto \Phi({^\omega{i}}_{P^\natural}^{G^\natural}(\Xi \Sigma))$
est une \og fonction de Paley--Wiener \fg sur $\mathfrak{P}_{\Bbb{C},{\rm u}}^0(M_{\sP}^\natural)$.} 
Pour la notion de fonction de Paley--Wiener, on renvoie ˆ 
\cite[2.6]{W} et \ref{les espaces bP}.
\hfill $\blacksquare$}
\end{marema}

\subsection{Variante \og finie \fg du thŽorme}\label{variante finie}
Pour un sous--espace tordu ouvert compact $K^\natural=K\cdot \delta$ de $G$ tel que $\omega$ est 
trivial sur $K$, on note 
$[\H^\natural_{\sK},\H_K]_\omega$\index{$[\H^\natural_{\sK},\H_K]_\omega$, 
$\smash{\overline{\H}}^\natural_{\sK,\omega}=\overline{\H}_K(G^\natural,\omega)$} le sous--espace vectoriel 
de $\H^\natural_{\sK}=\H_K(G^\natural)$ engendrŽ par 
les fonctions de la forme $\phi* f - \omega f * \phi$ pour $\phi\in \H^\natural_{\sK}$ et $f\in \H_K$, et l'on note 
$\smash{\overline{\H}}^\natural_{\sK,\omega}=\overline{\H}_K(G^\natural,\omega)$ l'espace quotient
$
\H^\natural_{\sK}/[\H^\natural_{\sK},\H_K]_\omega.
$
En d'autres termes, notant $\Hn_{\sK,\omega}=\H_K(G^\natural,\omega)$ le $\H_K$--bimodule dŽfini comme en \ref{H modules} (variante) 
en rempla\c{c}ant la paire $(\Hn,\H)$ par la paire $(\Hn_{\sK},\H_K)$, l'espace $\smash{\overline{\H}}^\natural_{\sK,\omega}$ est le quotient de 
$\Hn_{\sK,\omega}$ par le sous--espace $[\Hn_{\sK,\omega},\H_K]$ engendrŽ par les commutateurs $\phi\cdot f - f\cdot \phi$ pour 
$\phi\in \Hn_{\sK,\omega}$ et $f\in \H_K$. 

D'autre part on note $\F_K(G^\natural,\omega)$\index{$\F_K(G^\natural,\omega)$} le sous--espace vectoriel de $\F(G^\natural,\omega)$ formŽ des ŽlŽments $\Phi$ 
telles que $\Phi(\Pi)=0$ pour tout $\Pi\in {\rm Irr}_{\Bbb C}(G^\natural,\omega)$ tel que $\Pi^K=0$, o $\Pi^K$ dŽsigne la 
classe d'isomorphisme du $(\H^\natural_{\sK},\omega)$--module $\H_K$--simple associŽ ˆ $\Pi$ (cf. \ref{H modules}).

Rappelons que ${\boldsymbol{J}}_{G^\natural,\omega}(G)$ est l'ensemble des sous--groupes ouverts compacts $J$ de $G$ 
tels que $J$ est \og bon \fg et normalisŽ par un ŽlŽment de $G^\natural$, et $\omega\vert_J=1$

\begin{montheo}
Pour tout $J\in {\boldsymbol{J}}_{G^\natural,\omega}(G)$, l'application $\H(G^\natural,J)\rightarrow \G_{\Bbb C}(G^\natural,\omega)^*,\,\phi\mapsto \Phi_\phi$ induit 
un isomorphisme de ${\Bbb C}$--espaces vectoriels
$$
\smash{\overline{\H}}_J(G^\natural,\omega)\rightarrow \F_J(G^\natural,\omega).
$$
\end{montheo}

D'aprs la remarque 1 de \ref{ŽnoncŽ}, l'espace $\F(G^\natural,\omega)$ est la limite inductive 
des sous--espaces $\F_J(G^\natural,\omega)$ 
pour $J\in {\boldsymbol{J}}_{G^\natural,\omega}(G)$. D'autre part, pour $J,\,J'\in {\boldsymbol{J}}_{G^\natural,\omega}(G)$ 
tels que $J'\subset J$, on verra en \cite{HL} que l'inclusion $\H_{J}(G^\natural,\omega)\subset \H_{J'}(G^\natural,\omega)$ 
induit par passage 
aux quotients une application injective (\cite[lemma 3.2]{K2} dans le cas non tordu)
$$
\overline{\H}_{J}(G^\natural,\omega)\subset \overline{\H}_{J'}(G^\natural,\omega).
$$
En d'autres termes, on a l'ŽgalitŽ
$$
[\H^\natural_{\smash{\sJ'}\!,\omega},\H_{\smash{\sJ'}}] \cap \H^\natural_{\sJ}= [\H^\natural_{{\sJ},\omega},\H_J].\leqno{(*)}$$ 
D'ailleurs cette ŽgalitŽ est contenue dans le thŽorme ci--dessus. 
Par consŽquent si 
le thŽorme ci--dessus est vrai, l'espace $\overline{\H}(G^\natural,\omega)$ est la limite inductive des 
sous--espaces $\smash{\overline{\H}}_J(G^\natural,\omega)$ pour $J\in {\boldsymbol{J}}_{G^\natural,\omega}(G^\natural)$, 
et par passage aux limites inductives on en dŽduit 
thŽorme de \ref{ŽnoncŽ}.

\begin{marema}
{\rm Nous n'en aurons pas besoin par la suite, mais signalons quand m\^eme que la rŽciproque est vraie aussi: le thŽorme de \ref{ŽnoncŽ} implique le 
thŽorme ci--dessus. En effet d'aprs \ref{bons seoc}, pour $J\in {\boldsymbol{J}}_{G^\natural,\omega}(G)$, on a la dŽcomposition en produit 
de catŽgories abŽliennes\index{$\mathfrak{R}(G^\natural,\omega)=\mathfrak{R}_J(G^\natural,\omega)\times \mathfrak{R}_J^\perp G^\natural,\omega)$}
$$
\mathfrak{R}(G^\natural,\omega)=\mathfrak{R}_J(G^\natural,\omega)\times \mathfrak{R}_J^\perp G^\natural,\omega),
$$
o $\mathfrak{R}_J^\perp(G^\natural,\omega)$ est la sous--catŽgorie pleine de $\mathfrak{R}(G^\natural,\omega)$ 
engendrŽe par les $\omega$--reprŽsentations $\Pi_J^\perp$ de $G^\natural$ pour $\Pi$ parcourant 
les objets de $\mathfrak{R}(G^\natural,\omega)$. D'o la dŽcomposition
\index{$\F(G^\natural,\omega)=\F_J(G^\natural,\omega)\oplus \F_J^\perp(G^\natural,\omega)$}
$$
\F(G^\natural,\omega)=\F_J(G^\natural,\omega)\oplus \F_J^\perp(G^\natural,\omega),
$$
o $\F_J^\perp(G^\natural,\omega)$ est le sous--espace vectoriel de $\F(G^\natural,\omega)$ formŽ des 
ŽlŽments $\Phi$ tels que $\Phi(\Pi)=0$ pour tout $\Pi\in {\rm Irr}_{\Bbb C}(G^\natural,\omega)$ tel que $\Pi^J\neq 0$. 
On verra en \cite{HL} que 
l'espace vectoriel $\overline{\H}(G^\natural,\omega)$ admet lui aussi 
une dŽcomposition\index{$\overline{\H}(G^\natural,\omega)=\overline{\H}_J(G^\natural,\omega)\oplus \overline{\H}_J^\perp(G^\natural,\omega)$}
$$
\overline{\H}(G^\natural,\omega)=\overline{\H}_J(G^\natural,\omega)\oplus \overline{\H}_J^\perp(G^\natural,\omega)
$$
qui, par dŽfinition, vŽrifie $\Theta_\Pi\vert_{\overline{\H}_J(G^\natural,\omega)}=0$ pour tout 
$\Pi\in {\rm Irr}_{\Bbb C}(G^\natural,\omega)$ tel que $\Pi^J\neq 0$. On en dŽduit que si 
le thŽorme de \ref{ŽnoncŽ} est vrai alors le thŽorme ci--dessus l'est aussi.\hfill $\blacksquare$}
\end{marema}

\section{RŽduction ˆ la partie \og discrte \fg de la thŽorie}

\subsection{Le \og lemme gŽomŽtrique\fg}\label{lemme gŽo}Rappelons que l'on a pos\'e $W_G=N_G(A_\circ)/M_\circ$. Pour chaque 
$w\in W_G$, on fixe 
un reprŽsentant $n_w$ de $w$ dans $N_G(A_\circ)$. 
De la m\^eme manire, pour $P\in \P(G)$, on pose\index{$W_{M_P}=N_{M_P}(A_\circ)/M_\circ$}
$$
W_{M_P}=N_{M_P}(A_\circ)/M_\circ \subset W_G.
$$
Pour 
$P,\,Q\in \P(G)$, on pose\index{$W_G(P,Q)$, $W_G^{P,Q}$}
$$
W_G(P,Q)=\{w\in W_G: w(M_P)=M_Q\}
$$
et
$$
W_G^{P,Q}=\{w\in W_G: w(M_P\cap P_\circ)\subset P_\circ,\, w^{-1}(M_Q\cap P_\circ)\subset P_\circ\};
$$
o, pour toute partie $X$ de $G$ normalisŽe par $M_\circ$, on a posŽ $w(X)={\rm Int}_{\boldsymbol{G}}(n_w)(X)$. 
D'aprs \cite[2.11]{BZ}, $W_G^{P,Q}$ est un systme de reprŽsentants des doubles classes $W_{M_Q}\backslash W_G/W_{M_P}$  
dans $W_G$. 
Notons que $W_G(P,Q)\cap W_G^{P,Q}$ est un systme de reprŽsentants des classes de
$$
W_{M_Q}\backslash W_G(P,Q)= W_G(P,Q)/W_{M_P}.
$$
Pour $w\in W_G^{P,Q}$, on pose\index{$M_{P,\bar{w}}$, $M_{Q,w}$, $P_{\bar{w}}$, $Q_w$}
$$
M_{P,\bar{w}}= M_P\cap w^{-1}(M_Q),\quad M_{Q,w}=w(M_{P,{\bar{w}}})= w(M_P)\cap M_Q.
$$
Ces groupes sont les composantes de Levi standard des sous--groupes paraboliques standard $P_{\bar{w}}$ 
et $Q_w$ de $G$ 
dŽfinis par
$$
P_{\bar{w}}= M_{P,\bar{w}}U_\circ,\quad Q_w= M_{Q,w}U_\circ.
$$

Pour $P^\natural\in \P(G^\natural)$, le $F$--automorphisme $\theta={\rm Int}_{{\boldsymbol{G}}^\natural}(\delta_1)$ de ${\boldsymbol{G}}^\natural$ 
opre sur $W_{M_P}$. On pose\index{$W_{\smash{M^\natural_{\sP}}}$, $W_{G^\natural}(P,Q)$, $W_{G^\natural}^{P,Q}$}
$$
W_{\smash{M^\natural_{\sP}}}=\{w\in W_{M_P}:\theta(w)=w\}.
$$
C'est un sous--groupe de $W_{M_P}$ qui ne dŽpend pas du choix de $\delta_1\in M^\natural_\circ$. D'autre part, 
pour $P^\natural,\,Q^\natural\in \P(G^\natural)$, 
$\theta$ opre sur $W_G(P,Q)$ et sur $W_G^{P,Q}$, 
et l'on pose
$$
W_{G^\natural}(P,Q)=\{w\in W_G(P,Q):\theta(w)=w\},\quad W_{G^\natural}^{P,Q}=\{w\in W_G^{P,Q}: \theta(w)=w\}.
$$
On a donc
$$
W_{G^\natural}(P,Q)= W_{G^\natural}\cap W_G(P,Q),\quad W_{G^\natural}^{P,Q}= W_{G^\natural}\cap W_G^{P,Q},
$$
et tout comme $W_{G^\natural}$, ces groupes ne  
dŽpendent pas du choix de $\delta_1\in M^\natural_\circ$. De plus, $W_{G^\natural}^{P,Q}$ 
s'identifie au sous--ensemble $[W_{M_Q}\backslash W_G/W_{M_P}]^\theta$ de $W_{M_Q}\backslash W_G/W_{M_P}$ formŽ 
des doubles classes qui sont stabilisŽes par $\theta$. 
Pour $w\in W_{G^\natural}^{P,Q}$, les groupes $P_{{\bar{w}}}$ et $Q_w$ 
sont $\theta$--stables, donc dŽfinissent des sous--espaces paraboliques standard $P_{{\bar{w}}}\cdot \delta_1$ 
et $Q_w\cdot \delta_1$ de $G^\natural$, notŽs $P^\natural_{\bar{w}}$ et $Q^\natural_w$. Leurs 
composantes de Levi standard, $M^\natural_{\sP,\bar{w}}=M_{P,\bar{w}}\cdot \delta_1$ 
et $M^\natural_{\sQ,w}= M_{Q,w}\cdot \delta_1$, vŽrifient
$$
w(M^\natural_{\sP,\bar{w}})= M^\natural_{\sQ,w};
$$
o, pour toute partie $Y$ de $G^\natural$ normalisŽe par $M_\circ$, on a posŽ $w(Y)
=n_w\cdot Y\cdot n_w^{-1}$. Soit\index{${\boldsymbol{f}}(w)={^{\omega}{\boldsymbol{f}}_{P^\natural}^{Q^\natural}}(w)$, 
${\boldsymbol{f}}_{\Bbb C}(w)$} 
$${\boldsymbol{f}}(w)={^{\omega}{\boldsymbol{f}}_{P^\natural}^{Q^\natural}}(w):
\mathfrak{R}(M^\natural_{\sP},\omega)\rightarrow \mathfrak{R}(M^\natural_{\sQ},\omega)$$ le foncteur dŽfini par
$$
{\boldsymbol{f}}(w)= {^\omega{i}_{Q^\natural_w}^{Q^\natural}}\circ (\Sigma \rightarrow {^{n_w}\Sigma}) \circ {^\omega{r}_{P^\natural}^{P^\natural_{\bar{w}}}};
$$
o ${^{n_w}\Sigma}(\delta)=\Sigma(n_w^{-1}\cdot \delta \cdot n_w)$, $\delta \in M^\natural_{{\sP},\bar{w}}$. 
Il induit un morphisme ${\Bbb C}$--linŽaire
$$
{\boldsymbol{f}}_{\Bbb C}(w)={^{\omega}{\boldsymbol{f}}_{P^\natural,{\Bbb C}}^{Q^\natural}}(w):\G_{\Bbb C}(M^\natural_{\sP},\omega)\rightarrow \G_{\Bbb C}(M^\natural_{\sQ},\omega)
$$
qui ne dŽpend pas du choix du reprŽsentant $n_w$ de $w$ dans $N_G(A_\circ)$. 

Soit aussi\index{${\boldsymbol{h}}={^{\omega}{\boldsymbol{h}}_{P^\natural}^{Q^\natural}}$, ${\boldsymbol{h}}_{\Bbb C}$} 
$${\boldsymbol{h}}={^{\omega}{\boldsymbol{h}}_{P^\natural}^{Q^\natural}}:\mathfrak{R}(M^\natural_{\sP},\omega)\rightarrow \mathfrak{R}(M^\natural_{\sQ},\omega)$$
le foncteur dŽfini par
$$
{\boldsymbol{h}}= {^\omega{r}_{G^\natural}^{Q^\natural}}\circ {^\omega{i}_{P^\natural}^{G^\natural}}.
$$
Il induit un morphisme ${\Bbb C}$--linŽaire
$$
{\boldsymbol{h}}_{\Bbb C}={^{\omega}{\boldsymbol{h}}_{P^\natural,{\Bbb C}}^{Q^\natural}}:\G_{\Bbb C}(M^\natural_{\sP},\omega)\rightarrow \G_{\Bbb C}(M^\natural_{\sQ},\omega).
$$

\begin{mapropo}
Pour $\Sigma\in \G_{\Bbb C}(M^\natural_{\sP},\omega)$, on a l'ŽgalitŽ 
dans $\G_{\Bbb C}(M^\natural_{\sQ},\omega)$
$$
{\boldsymbol{h}}_{\Bbb C}(\Sigma)= 
\sum_w {\boldsymbol{f}}_{\Bbb C}(w)(\Sigma),
$$
o $w$ parcourt les ŽlŽments de $W_{G^\natural}^{P,Q}$.
\end{mapropo}

\begin{proof} Dans le cas non tordu --- i.e. pour $\theta={\rm id}$ et $\omega=1$ ---, on sait 
d'aprs \cite[2.12]{BZ} qu'il existe une filtration $0=h_0\subset h_1\subset \cdots \subset h_k = h$ 
du foncteur
$$h=r_G^Q\circ i_P^G:\mathfrak{R}(M_P)\rightarrow\mathfrak{R}(M_Q)
$$
telle que pour $i=1,\ldots ,k$, le foncteur quotient $h_i/h_{i-1}:\mathfrak{R}(M_P)\rightarrow\mathfrak{R}(M_Q)$ 
est isomorphe \`a
$$f(w_i)= i_{Q_{w_i}}^Q\circ (\sigma \rightarrow {^{n_{w_i}}\sigma})\circ r_P^{P_{\bar{w}_i}}$$
pour un $w_i\in W_G^{P,Q}$. Les $w_i$ sont deux--ˆ--deux distincts, et $W_G^{P,Q}=\{w_i:i=1,\ldots ,k\}$. 
PrŽcisŽment, soit 
$w_1,\ldots ,w_k$ les ŽlŽments de $W_G^{P,Q}$ ordonnŽs de telle manire que pour $i=1,\ldots ,k-1$, on ait 
$$
l(w_i)\geq l(w_{i+1}),
$$
o $l:W_G\rightarrow {\Bbb Z}_{>0}$ dŽsigne la fonction longueur. Alors pour $i=1,\ldots ,k$, 
$$
X_i= \textstyle{\coprod_{j=1}^i} Pn_{w_j}Q
$$
est ouvert dans $G$ (et $Q$--invariant ˆ droite). 
Rappelons que pour une reprŽsentation $\sigma$ de $M_P$, la reprŽsentation 
$\pi=i_P^G(\sigma)$ de $G$ opre par translations ˆ droite sur l'espace des fonctions $\varphi:G\rightarrow V_\sigma$ 
vŽrifiant:
\begin{itemize}
\item $\varphi(mug)= \delta_P^{1/ 2}(m)\sigma(m)\varphi(g)$ pour $m\in M_P$, $u\in U_P$, $g\in G$;
\item il existe un sous--groupe ouvert compact $K_\varphi$ de $G$ tel que $\varphi(gx)=\varphi(g)$ 
pour $g\in G$, $x\in K_\varphi$.
\end{itemize}
Ici $\delta_P:M_P\rightarrow {\Bbb R}_{>0}$ 
dŽsigne la fonction module habituelle. Pour 
$i=1,\ldots ,k$, on note $V_\pi(X_i)$ le sous--espace ($Q$--stable) de $V_\pi$ form\'e des fonctions $\varphi: G\rightarrow V_\sigma$ 
ˆ support dans $X_i$. Il dŽfinit une sous--reprŽsentation $\pi_i$ de $\pi\vert_Q$, dont on peut prendre la restriction 
de Jacquet normalis\'ee (d'espace le quotient de $V_\pi(X_i)$ par le sous--espace engendrŽ par les $\pi_i(u)(v)-v$ pour $v\in V_\pi(X_i)$ 
et $u\in U_Q$): c'est une reprŽsentation de $M_Q$, 
que l'on note $h_i(\sigma)$. La filtration $0=h_0\subset h_1\subset \cdots \subset h_k=h$ ainsi dŽfinie 
vŽrifie $h_i/h_{i-1}\simeq f(w_i)$ \cite[5.2]{BZ}. 

Passons au cas tordu. Notons $\Omega_1,\ldots ,\Omega_s$ les $\theta$--orbites dans 
$W_G^{P,Q}$. Puisque $l\circ \theta = l$, on peut supposer que les ŽlŽments de $W_G^{P,Q}$ 
ont ŽtŽ ordonnŽs de telle manire que
$$
\Omega_j =\{ w_{i_{j-1}+1}, w_{i_{j-1}+2},\ldots ,w_{i_j} \},\quad j=1,\ldots ,s\quad (i_0=0,\, i_s =k).
$$
Soit $\Sigma$ une $\omega$--reprŽsentation de $G^\natural$, et soit $\Pi= {^\omega{i}_{P^\natural}^{Q^\natural}}(\Sigma)$. 
Posons $\sigma=\Sigma^\circ$ et $\pi=\Pi^\circ$. On a $\pi=i_P^G(\sigma)$, 
et d'apr\`es \cite[2.7]{L2}, l'action de $\Pi(\delta_1)$ sur $V_\pi$ est donn\'ee par
$$
\Pi(\delta_1)(\varphi)(g)= \omega(\theta^{-1}(g))\Sigma(\delta_1)(\varphi(\theta^{-1}(g)),\quad \varphi\in V_\pi,\,g\in G.
$$
Pour $j=1,\ldots ,s$, puisque 
$\theta(X_{i_j})=X_{i_j}$, le sous--espace 
$V_j = V_\pi(X_{i_j})$ de $V_\pi$ est stable sous l'action de $\Pi(\delta_1)$, donc d\'efinit une sous--$\omega$--reprŽsentation 
$\Pi_j$ de $\Pi\vert_{Q^\natural}$ telle que $\Pi_j^\circ= \pi_{i_j}$. Comme dans le cas 
non tordu, on peut prendre la restriction de Jacquet normalis\'ee de $\Pi_j$ (cf. \cite[5.10]{L2}): c'est une $\omega$--reprŽsentation 
de $M^\natural_{\sQ}$, que l'on note ${\boldsymbol{h}}_j(\Sigma)$. 
On obtient ainsi une filtration $0={\boldsymbol{h}}_0\subset {\boldsymbol{h}}_1\subset\cdots \subset  {\boldsymbol{h}}_s = {\boldsymbol{h}}$ du foncteur ${\boldsymbol{h}}$ 
qui vŽrifie ${\boldsymbol{h}}_j^\circ = h_{i_j}$, $ j=1,\ldots ,s$.
Soit $j\in \{1,\ldots ,s\}$. Notons $\overline{\boldsymbol{h}}_j$ le foncteur quotient
$$
{\boldsymbol{h}}_j/{\boldsymbol{h}}_{j-1}: \mathfrak{R}(M^\natural_{\sP},\omega)
\rightarrow \mathfrak{R}(M^\natural_{\sQ},\omega).$$
Posons $k_j= i_j - i_{j-1}$. C'est le cardinal de $\Omega_j$. Pour $a=1,\ldots ,k_j$, posons 
$w_{j,a}= w_{i_{j-1}+a}$ et notons $V_j(a)$ 
le sous--espace de $V_j$ form\'e des fonctions $\varphi:G\rightarrow V_\sigma$ ˆ support 
dans
$$X_{j,a}=X_{i_{j-1}}\coprod Pn_{w_{j,a}}Q.
$$
Quitte ˆ rŽordonner l'orbite $\Omega_j$, on peut supposer que $w_{j,a}=\theta^{a-1}(w_{j,1})$ 
et $\theta^{k_j}(w_{j,1})= w_{j,1}$. 
L'automorphisme $\Pi_j(\delta_1)=\Pi(\delta_1)\vert_{V_j}$ de $V_j$ permute les $V_j(a)$: on a 
$$
\Pi_j(\delta_1)(V_j(a))=V_j(a+1),\quad a=1,\ldots k_{j-1}$$
et
$$
\Pi_j(\delta_1)(V_j(k_j))=V_j(1).
$$
On distingue deux cas: $k_j>1$ et $k_j=1$. Si $k_j>1$, 
la $\omega$--reprŽsentation $\Pi_j/\Pi_{j-1}$ de $Q^\natural$ est dans 
l'image du foncteur $\iota_{k_j}$ pour $Q^\natural$ (cf. \ref{iota}), par suite la $\omega$--reprŽsentation 
$\overline{\boldsymbol{h}}_j(\Sigma)$ de $M^\natural_{\sQ}$ est dans l'image du foncteur $\iota_{k_j}$ pour 
$M^\natural_{\sQ}$. Si $k_j=1$, i.e. si $\theta(w_{i_j})=w_{i_j}$, la $\omega$--reprŽsentation 
$\overline{\boldsymbol{h}}_j(\Sigma)$ de $M^\natural_{\sQ}$ vŽrifie
$$
\overline{\boldsymbol{h}}_j(\Sigma)^\circ = h_{i_j}(\sigma)/h_{i_j-1}(\sigma) \simeq f(w_{i_j})(\sigma)={\boldsymbol{f}}(w_{i_j})(\Sigma)^\circ.
$$
L'isomorphisme ci--dessus n'est pas vraiment canonique: il dŽpend du choix d'une mesure de Haar 
sur l'espace quotient $(U_Q\cap w(P))\backslash U_Q$, o l'on a posŽ $w=w_{i_j}$. Fixons une telle mesure, disons $d\bar{u}$. 
Notons $\alpha: V_\sigma \rightarrow r_P^{P_{\bar{w}}}(V_\sigma)$ la projection canonique, 
et $\beta:V_{\pi_{i_j}}\rightarrow V_{f(w)(\sigma)}$ l'opŽrateur dŽfini par
$$
\beta(\varphi)(m_Q)=\delta_Q^{1/2}(m)\int_{U_Q\cap w(P)\backslash U_Q} \alpha(\varphi(n_w^{-1}um_Q n_w))d\bar{u},\quad m_Q\in M_Q.
$$
D'aprs \cite[5.5]{BZ}, $\beta$ induit par passage au quotient l'isomorphisme cherchŽ
$$
\bar{\beta}:h_{i_j}(\sigma)/h_{i_j-1}(\sigma)\rightarrow f(w)(\sigma).
$$
Puisque $\bar{\beta}$ commute aux opŽrateurs $\overline{\boldsymbol{h}}_j(\Sigma)(\delta_1)$ 
et ${\boldsymbol{f}}(w)(\Sigma)(\delta_1)$ --- cf. l'action de $\Pi(\delta_1)$ rappelŽe plus haut ---, 
c'est un isomorphisme de $\overline{\boldsymbol{h}}_j(\Sigma)$ sur ${\boldsymbol{f}}(w)(\Sigma)$. 
En conclusion, pour $\Sigma \in \G(M^\natural_{\sP},\omega)$, si $\vert \Omega_j \vert>1$ 
on a $\overline{\boldsymbol{h}}_j(\Sigma)\in \G_{0^+}(M^\natural_{\sQ},\omega)$, et sinon on a l'ŽgalitŽ dans 
$\G(M^\natural_{\sQ},\omega)$
$$\overline{\boldsymbol{h}}_j(\Sigma) ={\boldsymbol{f}}(w_{i_j})(\Sigma).$$
La proposition est dŽmontrŽe.
\end{proof}

\subsection{Les espaces $\mathfrak{a}_P$, $\mathfrak{a}_Q^P$, $\mathfrak{a}_P^*$, 
etc.}\label{les espaces aP} Pour $P\in  \P(G)$, on note 
${\rm X}^*_F(M_P)$\index{${\rm X}^*_F(M_P)$, $A_P$, ${\rm X}^*(A_P)={\rm X}^*_F(A_P)$} 
le groupe des caractres algŽbriques de $M_P$ qui sont dŽfinis sur $F$. 
On note aussi $A_P$ le tore central dŽployŽ maximal 
de $M_P$, et ${\rm X}^*(A_P)={\rm X}^*_F(A_P)$ le groupe des caractres algŽbriques 
de $A_P$ --- ils sont tous dŽfinis sur $F$. L'application ${\rm X}^*_F(M_P)\rightarrow {\rm X}^*(A_P),\,\chi\mapsto \chi\vert_{A_P}$ est 
injective, de conoyau fini, et l'on pose\index{$\mathfrak{a}_P$, $H_P$, $\mathfrak{a}_{P,F}$, $\mathfrak{a}_{A_P,F}$}
$$
\mathfrak{a}_P = {\rm Hom}_{\Bbb Z}({\rm X}^*_F(M_P),{\Bbb R})={\rm Hom}_{\Bbb Z}({\rm X}^*(A_P),{\Bbb R}).
$$
C'est un espace vectoriel sur ${\Bbb R}$, de dimension finie $\dim A_P$. 
L'application $H_P:M_P\rightarrow \mathfrak{a}_P$ dŽfinie par
$$e^{\langle \chi, H_P(m)\rangle } = \vert \chi(m)\vert_F,\quad m\in M_P,\,\chi \in {\rm X}^*_F(M_P),
$$
a pour noyau $M_P^1$ et pour image un rŽseau de $\mathfrak{a}_P$, notŽ $\mathfrak{a}_{P,F}$. Ici 
$\vert\;\vert_F$ dŽsigne la valeur absolue normalisŽe sur $F$. Comme la restriction de $H_P$ 
ˆ $A_P$ co\"{\i}ncide avec $H_{A_P}$, on peut aussi poser $\mathfrak{a}_{A_P,F}=H_P(A_P)$. C'est 
encore un rŽseau de $\mathfrak{a}_P$, et un sous--groupe d'indice fini de $A_{P,F}$. 

Posons\index{$\mathfrak{a}_P^*$, $\mathfrak{a}_{P,{\Bbb C}}^*$}
$$
\mathfrak{a}_P^*={\rm Hom}_{\Bbb R}(\mathfrak{a}_P,{\Bbb R}), \quad 
\mathfrak{a}_{P,{\Bbb C}}^*={\rm Hom}_{\Bbb R}(\mathfrak{a}_P,{\Bbb C})= \mathfrak{a}_P^*\otimes_{\Bbb R}{\Bbb C}.
$$
Pour tout sous--groupe fermŽ $\Lambda$ de 
$\mathfrak{a}_P$, on note $\Lambda^\vee$ le sous--groupe de $\mathfrak{a}_P^*$ dŽfini par\index{$\Lambda \mapsto \Lambda^\vee$}
$$
\Lambda^\vee=\{\nu\in \mathfrak{a}_P^*:\langle \lambda, \nu\rangle \in 2\pi{\Bbb Z},\,\forall \lambda \in \Lambda\}.
$$
Pour $\nu\in \mathfrak{a}_{P,{\Bbb C}}^*$, on note $\psi_\nu$ le caractre non ramifiŽ de $M_P$ dŽfini par 
$$\psi_\nu(m) = e^{\langle H_P(m),\nu\rangle}.
$$
L'application
$$
\mathfrak{a}^*_{P,{\Bbb C}}/i \mathfrak{a}_{P,F}^\vee \rightarrow \mathfrak{P}(M_P),\, \nu \mapsto \psi_\nu
$$
est un isomorphisme de groupes. Il identifie $\mathfrak{a}_P^*$ au groupe des caractres non ramifiŽs {\it positifs} de $M_P$, 
et $i\mathfrak{a}_P^*/i\mathfrak{a}_{P,F}^\vee$ au groupe des caractres non ramifiŽs {\it unitaires} de $M_P$. On obtient de la 
m\^eme manire un isomorphisme de groupes 
$$
\mathfrak{a}_{P,{\Bbb C}}^*/i\mathfrak{a}_{A_P,F}^\vee\rightarrow \mathfrak{P}(A_P),\, \nu\mapsto \chi_\nu.
$$ 
Le groupe $\mathfrak{a}_{P,F}^\vee$ est un rŽseau de $\mathfrak{a}_P^*$ et un sous--groupe d'indice fini de 
$\mathfrak{a}_{A_P,F}^\vee$, et la projection canonique $\mathfrak{a}^*_{P,{\Bbb C}}/i \mathfrak{a}_{P,F}^\vee
\rightarrow \mathfrak{a}^*_{P,{\Bbb C}}/i \mathfrak{a}_{A_P,F}^\vee$ correspond, via les isomorphismes $\nu\mapsto \psi_\nu$ 
et $\nu\mapsto \chi_\nu$, ˆ la restriction des caractres de $\mathfrak{P}(M_P)\rightarrow \mathfrak{P}(A_P),\,\psi \mapsto \psi\vert_{A_P}$. 

Soit $P,\,Q\in \P(G)$ tels que $Q\subset P$. L'inclusion $M_Q\subset M_P$ induit une application injective (restriction 
des caractres) ${\rm X}^*_F(M_P)\rightarrow {\rm X}^*_F(M_Q)$, et donc aussi une application surjective\index{$\pi_Q^P:\mathfrak{a}_Q\rightarrow \mathfrak{a}_P$, 
$\mathfrak{a}_Q^P$, $(\mathfrak{a}_Q^P)^*$}
$$
\pi_Q^P:\mathfrak{a}_Q\rightarrow \mathfrak{a}_P,
$$
dont le noyau est notŽ $\mathfrak{a}_Q^P$. D'autre part l'inclusion $A_P\subset A_Q$ induit une application surjective 
(restriction des caractres) ${\rm X}(A_Q)\rightarrow {\rm X}^*(A_P)$, et donc aussi une application injective
$$
\mathfrak{a}_P\rightarrow \mathfrak{a}_Q,
$$
qui est une section de $\pi_Q^P$. On obtient les dŽcompositions\index{$\mathfrak{a}_Q=\mathfrak{a}_P\oplus \mathfrak{a}_Q^P$, 
$\mathfrak{a}_Q^*=\mathfrak{a}_P^*\oplus (\mathfrak{a}_Q^P)^*$}
$$
\mathfrak{a}_Q=\mathfrak{a}_P\oplus \mathfrak{a}_Q^P,\quad
\mathfrak{a}_Q^*=\mathfrak{a}_P^*\oplus (\mathfrak{a}_Q^P)^*,
$$
o l'on a posŽ $(\mathfrak{a}_Q^P)^*={\rm Hom}_{\Bbb R}(\mathfrak{a}_Q^P,{\Bbb R})$. 

Pour allŽger l'Žcriture, on remplace l'indice $P_\circ$ par un indice $\circ$ dans toutes les notations prŽcŽdentes. Ainsi pour 
$Q=P_\circ$ et $P=G$, on a les dŽcompositions
$$
\mathfrak{a}_\circ=\mathfrak{a}_G\oplus \mathfrak{a}_\circ^G,\quad \mathfrak{a}_\circ^*=\mathfrak{a}_G^*\oplus (\mathfrak{a}_\circ^G)^*.
$$

On fixe une 
forme quadratique $(\cdot ,\cdot )_\circ$\index{$(\cdot ,\cdot )_\circ$, $(\cdot ,\cdot )_P$, $(\cdot ,\cdot )_Q^P$} 
sur l'espace $\mathfrak{a}_\circ$, dŽfinie positive et invariante sous l'action du groupe de Weyl $W_G$. 
Cela munit $\mathfrak{a}_\circ$ d'une norme, donnŽe par $\vert \nu \vert =(\nu,\nu)^{1/2}$. Par dualitŽ 
l'espace $\mathfrak{a}_\circ^*$ est lui aussi muni d'une norme. Pour $P\in \P(G)$, la dŽcomposition $\mathfrak{a}_\circ=\mathfrak{a}_P\oplus \mathfrak{a}_\circ^P$ 
est orthogonale, et la forme quadratique $(\cdot,\cdot)_\circ$ induit par restriction des 
formes quadratiques dŽfinies positives sur les espaces $\mathfrak{a}_P$ et $\mathfrak{a}_\circ^P$, not\'ees $(\cdot,\cdot)_P$ et $(\cdot,\cdot)_\circ^P$. 
Plus gŽnŽralement, pour $P,\,Q\in \P(G)$ tels que $Q\subset P$, la dŽcomposition $\mathfrak{a}_Q=\mathfrak{a}_P\oplus \mathfrak{a}_Q^P$ est orthogonale, et la 
forme quadratique $(\cdot,\cdot)_Q$ induit par restriction une forme quadratique dŽfinie positive sur l'espace $\mathfrak{a}_Q^P$, 
notŽe $(\cdot,\cdot)_Q^P$. 

\begin{marema}
{\rm 
Les racines de $A_\circ$ dans ${\rm Lie}(G)$ sont des ŽlŽments de $\mathfrak{a}_\circ^*$, nuls sur $\mathfrak{a}_G$. 
Leurs restrictions ˆ $\mathfrak{a}_\circ^G$ forment un systme de racine, en gŽnŽral non rŽduit. 
L'ensemble des racines simples relativement au sous--groupe parabolique minimal $P_\circ$ de $G$, notŽ  
$\Delta_\circ^G$,\index{$\Delta_\circ^G$, $\Delta_\circ^P$, $\Delta_Q^P$} est une base 
de $(\mathfrak{a}_\circ^G)^*$, et l'ensemble des coracines $\check{\alpha}$ pour $\alpha\in \Delta_\circ^G$ est une base de 
$\mathfrak{a}_\circ^G$.

Pour $P\in \P(G)$, on dŽfinit de la m\^eme manire l'ensemble $\Delta_\circ^P$ des racines simples de $A_\circ$ dans ${\rm Lie}(M_P)$ 
relativement ˆ $P_\circ\cap M_P$; c'est une base de $(\mathfrak{a}_\circ^P)^*$. On peut aussi considŽrer les ŽlŽments de $\Delta_\circ^P$ 
comme des formes linŽaires sur $\mathfrak{a}_\circ$, ce qui fait de $\Delta_\circ^P$ un sous--ensemble de $\Delta_\circ^G$. Ainsi 
$\mathfrak{a}_P^G$ est le sous--espace de $\mathfrak{a}_\circ^G$ intersection des noyaux des $\alpha \in \Delta_\circ^P$. 

Pour $P,\,Q\in \P(G)$ tels que $Q\subset P$, soit $\Delta_Q^P$ l'ensemble des restrictions non nulles des ŽlŽments de $\Delta_\circ^P$ au sous--espace 
$\mathfrak{a}_Q^G$ de $\mathfrak{a}_\circ^G$. C'est une base de $(\mathfrak{a}_Q^P)^*$, et $\mathfrak{a}_P^G$ est le sous--espace de 
$\mathfrak{a}_Q^G$ intersection des noyaux des $\alpha\in \Delta_Q^P$. Notons  que $\Delta_Q^P$ n'est en gŽnŽral 
pas la base d'un systme de racines. 
\hfill $\blacksquare$
}
\end{marema} 

\subsection{Les espaces $\mathfrak{a}_{P^\natural}$, $\mathfrak{a}_{Q^\natural}^{P^\natural}$, $\mathfrak{b}_{P^\natural}$, 
$\mathfrak{b}_{P^\natural}^*$, etc.}\label{les espaces bP}
L'automorphisme $\theta={\rm Int}_{{\boldsymbol{G}}^\natural}(\delta_1)$ de $G$ induit par restriction un automorphisme 
de $A_\circ$, qui ne dŽpend pas du choix de $\delta_1\in M^\natural_\circ$. D'o un automorphisme 
de $\mathfrak{a}_\circ = {\rm Hom}_{\Bbb Z}({\rm X}^*(A_\circ),{\Bbb R})$,\index{$\mathfrak{a}_\circ = {\rm Hom}_{\Bbb Z}({\rm X}^*(A_\circ),{\Bbb R})$} 
dŽduit de $\theta\vert_{A_\circ}$ par fonctorialitŽ et que l'on note encore $\theta$, 
donnŽ par
$$ 
\langle \chi ,\theta(u)\rangle = \langle \chi \circ \theta\vert_{A_\circ} , u\rangle,\quad u\in \mathfrak{a}_\circ,\,\chi\in {\rm X}^*(A_\circ).
$$
On note encore $\theta$ l'automorphisme de $\mathfrak{a}_\circ^*$ dŽduit de $\theta$ par dualitŽ:
$$
\langle u, \theta(u^*)\rangle = \langle \theta^{-1}(u),u^*\rangle,\quad u\in \frak{a}_\circ,\, u^*\in \frak{a}_\circ^*.
$$
La dŽcomposition $\mathfrak{a}_\circ=\mathfrak{a}_G\oplus \mathfrak{a}_\circ^G$ est $W_G$--stable. Elle est aussi $\theta$--stable, 
et puisque $\theta$ induit une permutation de l'ensemble fini $\Delta_\circ^G$ (cf. la remarque de \ref{les espaces aP}), il 
induit un automorphisme d'ordre fini de l'espace $\mathfrak{a}_\circ^G$. On peut donc supposer --- i.e. on suppose! --- que la forme quadratique 
dŽfinie positive $W_G$--invariante $(\cdot ,\cdot)_\circ^G$ sur $\mathfrak{a}_\circ^G$, est aussi $\theta$--invariante. 

\begin{marema1}
{\rm 
Il n'est en gŽnŽral pas possible 
de choisir la forme quadratique dŽfinie positive $W_G$--invariante $(\cdot ,\cdot )$ sur $\mathfrak{a}_\circ$ 
de telle manire qu'elle soit $\theta$--invariante. C'est possible par exemple si la restriction de $\theta$ au centre $Z(G)$ 
de $G$ est d'ordre fini. 

Une hypothse moins restrictive consiste ˆ supposer que l'application naturelle 
$\mathfrak{a}_G^\theta\rightarrow \mathfrak{a}_{G,\theta}$ est un isomorphisme, o $\mathfrak{a}_G^\theta\subset \mathfrak{a}_G$ dŽsigne 
le sous--espace vectoriel des vecteurs $\theta$--invariants, et $\mathfrak{a}_{G,\theta}$ l'espace vectoriel 
des coinvariants de $\mathfrak{a}_G$ sous $\theta$, \cad le quotient 
de $\mathfrak{a}_G$ par le sous--espace vectoriel engendrŽ par les $u-\theta(u)$, $u\in \mathfrak{a}_G$. En ce cas le sous--espace vectoriel 
$(\mathfrak{a}_G^*)^\theta\subset \mathfrak{a}_G^*$ s'identifie au dual 
$(\mathfrak{a}_G^\theta)^*={\rm Hom}_{\Bbb R}(\mathfrak{a}_G^\theta,{\Bbb R})$ de $\mathfrak{a}_G^\theta$. \hfill $\blacksquare$
}
\end{marema1}

Pour $P\in \P(G^\natural)$, les dŽcompositions 
$$
\mathfrak{a}_\circ = \mathfrak{a}_P\oplus \mathfrak{a}_\circ^P,\quad 
\mathfrak{a}_P=\mathfrak{a}_G\oplus \mathfrak{a}_P^G,$$
sont $\theta$--stables. Soit 
$\mathfrak{a}_{P^\natural}=\mathfrak{a}_P^\theta$
\index{$A_{P^\natural}$, $\mathfrak{a}_{P^\natural}$, $\mathfrak{a}_{P^\natural}^{G^\natural}$, 
$\tilde{\mathfrak{a}}_G^{G^\natural}$, $\mathfrak{b}_{G^\natural}$, $\mathfrak{b}_{P^\natural}$} 
le sous--espace vectoriel de $\mathfrak{a}_P$ formŽ des vecteurs $\theta$--invariants. Il co\"{\i}ncide avec 
${\rm Hom}_{\Bbb Z}({\rm X}^*(A_{P^\natural}),{\Bbb R})$ o $A_{P^\natural}$ est le plus grand tore dŽployŽ du centre 
$Z(M^\natural_{\sP})$ de $M^\natural_{\sP}$ --- 
c'est un sous--tore de $A_P$, et l'inclusion $\mathfrak{a}_{P^\natural}\subset \mathfrak{a}_P$ est 
donnŽe par la restriction des caractres ${\rm X}^*(A_P)\rightarrow {\rm X}^*(A_{P^\natural}$). 
On note:
\begin{itemize}
\item $\mathfrak{a}_{P^\natural}^{G^\natural}= (\mathfrak{a}_P^G)^\theta$ le sous--espace vectoriel de $\mathfrak{a}_P^G$ formŽ des vecteurs 
$\theta$--invariants;
\item $\tilde{\mathfrak{a}}_G^{G^\natural}$ le sous--espace vectoriel de $\mathfrak{a}_G$ engendrŽ par les $u-\theta(u)$, $u\in \mathfrak{a}_G$;
\item $\mathfrak{b}_{G^\natural}= \mathfrak{a}_{G,\theta}$ l'espace vectoriel $\mathfrak{a}_G/\tilde{\mathfrak{a}}_G^{G^\natural}$ des coinvariants de 
$\mathfrak{a}_G$ sous $\theta$;
\item $\mathfrak{b}_{P^\natural}$ l'espace vectoriel produit $\mathfrak{b}_{G^\natural}\times \mathfrak{a}_{P^\natural}^{G^\natural}$.
\end{itemize}
L'espace dual $\mathfrak{b}_{G^\natural}^*={\rm Hom}_{\Bbb R}(\mathfrak{b}_{G^\natural},{\Bbb R})$ co\"{\i}ncide 
avec le sous--espace vectoriel $(\mathfrak{a}_G^*)^\theta$ de $\mathfrak{a}_G^*$ formŽ des vecteurs $\theta$--invariants. 
D'autre part, comme la forme quadratique $(\cdot,\cdot)_P^G$ sur 
$\mathfrak{a}_P^G$ est $\theta$--invariante, l'orthogonal de $\mathfrak{a}_{P^\natural}^{G^\natural}$ dans 
$\mathfrak{a}_P^G$ co\"{\i}ncide avec le sous--espace vectoriel de $\mathfrak{a}_P^G$ engendrŽ par les $u-\theta(u)$, $u\in \mathfrak{a}_P^G$. 
On en dŽduit que l'espace dual 
$(\mathfrak{a}_{P^\natural}^{G^\natural})^*={\rm Hom}_{\Bbb R}(\mathfrak{a}_{P^\natural}^{G^\natural},{\Bbb R})$
\index{$(\mathfrak{a}_{P^\natural}^{G^\natural})^*$, $\mathfrak{b}_{P^\natural}^*$, $\mathfrak{b}_{P^\natural,{\Bbb C}}^*$} co\"{\i}ncide 
avec le sous--espace vectoriel $(\mathfrak{a}_P^G)^{*,\theta}=((\mathfrak{a}_P^G)^*)^\theta$ de $(\mathfrak{a}_P^G)^*$ 
formŽ des vecteurs $\theta$--invariants. Par suite l'espace dual 
$\mathfrak{b}_{P^\natural}^*={\rm Hom}_{\Bbb R}(\mathfrak{b}_{P^\natural},{\Bbb R})$ 
co\"{\i}ncide avec le sous--espace 
$(\mathfrak{a}_P^*)^\theta= (\mathfrak{a}_G^*)^\theta\oplus (\mathfrak{a}_P^G)^{*,\theta}$ de $\mathfrak{a}_P^*$. 
On note\index{$\pi_P^{P^\natural}:\mathfrak{a}_P\rightarrow \mathfrak{b}_{P^\natural}$}
$$
\pi_P^{P^\natural}:\mathfrak{a}_P\rightarrow \mathfrak{b}_{P^\natural}
$$
la projection naturelle, produit de la projection canonique $\mathfrak{a}_G\rightarrow \mathfrak{b}_{G^\natural}$ et de la projection orthogonale 
$\mathfrak{a}_P^G\rightarrow \mathfrak{a}_{P^\natural}^{G^\natural}$. Enfin comme plus haut, on pose
$$
\mathfrak{b}_{{P^\natural},{\Bbb C}}^*={\rm Hom}_{\Bbb R}(\mathfrak{b}_{P^\natural},{\Bbb C})=
\mathfrak{b}_{P^\natural}^*\otimes_{\Bbb R}{\Bbb C}.
$$

Pour $P,\,Q\in \P(G^\natural)$ tels que $Q\subset P$, on dŽfinit de la m\^eme manire l'espace 
$\mathfrak{a}_{Q^\natural}^{P^\natural} =(\mathfrak{a}_Q^P)^\theta$\index{$\mathfrak{a}_{Q^\natural}^{P^\natural}$, $(\mathfrak{a}_{Q^\natural}^{P^\natural})^*$}. 
On a les dŽcompositions\index{$\mathfrak{b}_{Q^\natural}= \mathfrak{b}_{P^\natural}\oplus \mathfrak{a}_{Q^\natural}^{P^\natural}$, 
$\mathfrak{b}_{Q^\natural}^*= \mathfrak{b}_{P^\natural}^*\oplus (\mathfrak{a}_{Q^\natural}^{P^\natural})^*$}
$$
\mathfrak{b}_{Q^\natural}= \mathfrak{b}_{P^\natural}\oplus \mathfrak{a}_{Q^\natural}^{P^\natural},\quad 
\mathfrak{b}_{Q^\natural}^*= \mathfrak{b}_{P^\natural}^*\oplus (\mathfrak{a}_{Q^\natural}^{P^\natural})^*,
$$
o l'on a posŽ $(\mathfrak{a}_{Q^\natural}^{P^\natural})^*={\rm Hom}_{\Bbb R}(\mathfrak{a}_{Q^\natural}^{P^\natural},{\Bbb R})$. Soit
\index{$\pi_Q^{P^\natural}:\mathfrak{a}_Q\rightarrow \mathfrak{b}_{P^\natural}$, 
$\pi_{Q^\natural}^{P^\natural}: \mathfrak{a}_{Q^\natural}\rightarrow \mathfrak{b}_{P^\natural}$} 
$$
\pi_Q^{P^\natural}: \mathfrak{a}_Q\rightarrow \mathfrak{b}_{P^\natural}
$$
la projection naturelle, donnŽe par $\pi_Q^{P^\natural}= \pi_P^{P^\natural}\circ \pi_Q^P$. 
Notant
$$
\pi_{Q^\natural}^{P^\natural}: \mathfrak{b}_{Q^\natural}\rightarrow \mathfrak{b}_{P^\natural}
$$
la projection orthogonale, on a l'ŽgalitŽ
$$
\pi_Q^{P^\natural}= \pi_{Q^\natural}^{P^\natural}\circ \pi_Q^{Q^\natural}.
$$

\begin{marema2}
{\rm 
Soit $\tilde{\mathfrak{a}}_P^{P^\natural}$ le sous--espace de $\mathfrak{a}_P$ engendrŽ par les $u-\theta(u)$, $u\in \mathfrak{a}_P$, 
et $\tilde{a}_{P^\natural}$ son orthogonal dans $\frak{a}_P$. Soit aussi $\mathfrak{a}_P^{P^\natural}$ l'orthogonal de $\mathfrak{a}_{P^\natural}$ 
dans $\mathfrak{a}_P$. Les espaces $\tilde{\mathfrak{a}}_{P^\natural}$ et $\mathfrak{a}_P^{P^\natural}$ 
ne sont en gŽnŽral pas $\theta$--stables. Ils le sont par exemple si la forme quadratique $(\cdot,\cdot)_\circ$ sur $\mathfrak{a}_\circ$ 
est $\theta$--invariante, \cad si la forme quadratique $(\cdot,\cdot)_G$ sur $\mathfrak{a}_G$ est $\theta$--invariante, 
auquel cas on a les ŽgalitŽs $\tilde{\frak{a}}_{P^\natural}=\frak{a}_{P^\natural}$ et 
$\frak{a}_P^{P^\natural}=\tilde{\frak{a}}_P^{P^\natural}$, et la dŽcomposition $\theta$--invariante $\mathfrak{a}_P =\mathfrak{a}_{P^\natural}\oplus
\mathfrak{a}_P^{P^\natural}$. En ce cas $\mathfrak{b}_{P^\natural}=\mathfrak{a}_{P^\natural}$, 
et $\pi_P^{P^\natural}:\mathfrak{a}_P\rightarrow \mathfrak{a}_{P^\natural}$ est la projection orthogonale.\hfill $\blacksquare$
}
\end{marema2}

Pour $P\in \P(G^\natural)$, l'application $\mathfrak{a}_P^*\rightarrow\mathfrak{P}(M_P),\,\nu\mapsto \psi_\nu$ est $\theta$--\'equivariante. 
Elle identifie $\mathfrak{b}_{P^\natural}^*$ au groupe des caractres non ramifiŽs positifs de $M^\natural_{\sP}$. 
Pour tout sous--groupe fermŽ $\Lambda$ de $\mathfrak{b}_{P^\natural}$, on note $\Lambda^\vee$ le sous--groupe 
de $\mathfrak{b}_{P^\natural}^*$ dŽfini comme plus haut en rempla\c{c}ant la paire $(\mathfrak{a}_P,\mathfrak{a}_P^*)$ par 
la paire $(\mathfrak{b}_{P^\natural},\mathfrak{b}_{P^\natural}^*)$. Soit $\mathfrak{b}_{P^\natural,F}$\index{$\mathfrak{b}_{P^\natural,F}$, 
$\mathfrak{b}_{A_{\smash{P^\natural}},F}$} l'image de $\mathfrak{a}_{P,F}$ par 
l'application $\pi_P^{P^\natural}:\mathfrak{a}_P\rightarrow \mathfrak{b}_{P^\natural}$. On a l'ŽgalitŽ 
$$
\mathfrak{b}_{P^\natural,F}^\vee=\mathfrak{a}_{P,F}^\vee\cap \mathfrak{b}_{P^\natural}^*,
$$ 
et l'application $\nu\mapsto \psi_\nu$ induit par restriction un morphisme injectif
$$
\mathfrak{b}_{P^\natural,{\Bbb C}}^*/i\mathfrak{b}_{P^\natural,F}^\vee\rightarrow \mathfrak{P}(M^\natural_{\sP}).
$$
L'image de ce morphisme est un sous--tore (algŽbrique, complexe) de $\mathfrak{P}(M^\natural_{\sP})$, 
de dimension celle de $\mathfrak{P}(M^\natural_{\sP})$. 
Ce tore est la composante neutre $\mathfrak{P}^0(M^\natural_{\sP})$ de 
$\mathfrak{P}(M^\natural_{\sP})$, et le morphisme $\nu\mapsto \psi_\nu$ identifie $\mathfrak{b}_{P^\natural}^*$ 
(resp. $i\mathfrak{b}_{P^\natural}^*/i\mathfrak{b}_{P^\natural,F}^\vee$) au 
sous--groupe de $\mathfrak{P}^0(M^\natural_{\sP})$ formŽ des caractres positifs 
(resp. unitaires). Soit aussi $\mathfrak{b}_{A_{\smash{P^\natural}},F}$ l'image de $H_P(A_{P^\natural}(F))$ 
par l'application $\pi_P^{P^\natural}$. 
L'application $\nu\mapsto \chi_\nu$ induit par 
restriction un isomorphisme de groupes
$$
\mathfrak{b}_{P^\natural,{\Bbb C}}^*/i\mathfrak{b}_{A_{\smash{P^\natural}},F}^\vee \rightarrow\mathfrak{P}(A_{P^\natural}),
\quad \mathfrak{P}(A_{P^\natural})={\rm Hom}(A_{P^\natural}/A_{P^\natural}^1,{\Bbb C}^\times).
$$
Comme en \ref{les espaces aP}, la projection canonique 
$\mathfrak{b}_{P^\natural,{\Bbb C}}^*/i\mathfrak{b}_{P^\natural,F}^\vee\rightarrow 
\mathfrak{b}_{P^\natural,{\Bbb C}}^*/i\mathfrak{b}_{A_{\smash{P^\natural}},F}^\vee$ correspond, via les 
isomorphismes $\nu\mapsto \psi_\nu$ et $\nu\mapsto \chi_\nu$, ˆ la 
restriction des caractres $\mathfrak{P}^0(M^\natural_{\sP})\rightarrow \mathfrak{P}(A_{P^\natural})$.

\begin{marema3}
{\rm Posons $A^\natural_{\sP}= A_P\cdot \delta_1$.\index{$A^\natural_{\sP}$, $\mathfrak{b}_{\smash{A^\natural_{\sP}},F}$} 
C'est un sous--espace tordu de $M^\natural_{\sP}$ 
(isomorphe ˆ $A^\natural_{\sP}\cdot \delta$ pour tout $\delta\in M^\natural_{\sP}$), et 
le groupe $\mathfrak{P}(A^\natural_{\sP})$ des caractres non ramifiŽs de $A^\natural_{\sP}$ est par 
dŽfinition le sous--groupe de $\mathfrak{P}(A_P)$ formŽ des ŽlŽments $\theta$--invariants. Notons 
$\mathfrak{b}_{\smash{A^\natural_{\sP}},F}$ l'image de $\mathfrak{a}_{A_P,F}$ par l'application 
$\pi_P^{P^\natural}$. C'est un rŽseau de $\mathfrak{a}_{P^\natural}$ 
qui vŽrifie la double inclusion
$$\mathfrak{b}_{A_{\smash{P^\natural}},F}\subset \mathfrak{b}_{\smash{A^\natural_{\sP}},F}
\subset \mathfrak{b}_{P^\natural,F}.
$$
La restriction des caractres $\mathfrak{P}(M^\natural_{\sP})\rightarrow\mathfrak{P}(A^\natural_{\sP})$ 
est un morphisme de variŽtŽs algŽbriques, de noyau et de conoyau finis. D'aprs ce qui prŽcde, 
il induit un  morphisme surjectif 
$\mathfrak{P}^0(M^\natural_{\sP})\rightarrow\mathfrak{P}^0(A^\natural_{\sP})$ de noyau isomorphe au groupe 
(fini) $i\mathfrak{b}_{\smash{A^\natural_{\sP}},F}^\vee/i\mathfrak{b}_{P^\natural,F}^\vee$.\hfill $\blacksquare$}
\end{marema3}

\begin{notation}
{\rm 
Pour $P\in \P(G)$, on pose\index{$d(M_P)$, $d(M^\natural_{\sP})$}
$$
d(M_P)= \dim \mathfrak{P}(M_P)\;(=\dim_{\Bbb R} \mathfrak{a}_P),
$$
et pour $P^\natural\in \P(G^\natural)$, on pose
$$
d(M^\natural_{\sP})= \dim \mathfrak{P}(M^\natural_{\sP})\;(=
\dim \mathfrak{P}^0(M^\natural_{\sP})= \dim_{\Bbb R}\mathfrak{b}_{P^\natural})).
$$
}
\end{notation}

\subsection{Les morphismes ${^\omega{T}_{P^\natural,{\Bbb C}}}$}\label{les T}
Pour $P^\natural\in \P(G^\natural)$, on note ${^\omega{T}_{P^\natural}}$ le foncteur\index{${^\omega{T}_{P^\natural}}$, 
${^\omega{T}_{{P^\natural},{\Bbb C}}}$}
$$
{^\omega{i}_{P^\natural}^{G^\natural}}\circ {^\omega{r}_{G^\natural}^{P^\natural}}:\mathfrak{R}(G^\natural,\omega)\rightarrow 
\mathfrak{R}(G^\natural,\omega).
$$
Il induit un morphisme de ${\Bbb C}$--espaces vectoriels
$$
{^\omega{T}_{{P^\natural},{\Bbb C}}}:\G_{\Bbb C}(G^\natural,\omega)\rightarrow \G_{\Bbb C}(G^\natural,\omega).
$$

On note $\G_{{\Bbb C},{\rm ind}}(G^\natural,\omega)$\index{$\G_{{\Bbb C},{\rm ind}}(G^\natural,\omega)$, 
$\G^{\rm dis}_{\Bbb C}(G^\natural,\omega)$} le sous--espace de $\G_{\Bbb C}(G^\natural,\omega)$ 
engendrŽ par les ${^\omega{i}_{P^\natural}^{G^\natural}}(\G_{\Bbb C}(M^\natural_{\sP},\omega))$ pour 
$P^\natural\in \P(G^\natural)$, $P^\natural\neq G^\natural$, et l'on pose
$$
\G^{\rm dis}_{\Bbb C}(G^\natural,\omega)=
\G_{\Bbb C}(G^\natural,\omega)/\G_{{\Bbb C},{\rm ind}}(G^\natural,\omega).
$$

On dŽfinit comme dans le cas non tordu une filtration dŽcroissante 
$\{\G_{{\Bbb C},i}(G^\natural,\omega)\}$ de $\G_{\Bbb C}(G^\natural,\omega)$: pour chaque 
entier $i\geq -1$, on pose\index{$\G_{{\Bbb C},i}(G^\natural,{\Bbb C})$}
$$
\G_{{\Bbb C},i}(G^\natural,{\Bbb C})= \sum_{P^\natural,\,d(M_P)>i}{^\omega{i}_{P^\natural}^{G^\natural}}(\G_{\Bbb C}
(M^\natural_{\sP},\omega)),
$$
o $P^\natural$ parcourt les ŽlŽments de $\P(G^\natural)$. 
On a donc
$$
\G_{{\Bbb C},i}(G^\natural,\omega)=\G(G^\natural,\omega),\quad i<d(G),
$$ 
$$
\G_{{\Bbb C},i}(G^\natural,\omega)=0,\quad i\geq d(M_\circ),
$$
et
$$
\G_{{\Bbb C},d(G)}(G^\natural,\omega)= \G_{{\Bbb C},{\rm ind}}(G^\natural,\omega).
$$

\begin{mapropo}
Soit $ d$ un entier tel que $d(G)\leq d< d(M_\circ)$. Il existe une famille 
de nombres rationnels $\lambda_d = \{\lambda_d(P^\natural):P^\natural\in \P(G^\natural),\,d(M_P)>d\}$ 
telle que le ${\Bbb C}$--endomorphisme ${\boldsymbol{A}}_d= {\boldsymbol{A}}_{\lambda_d}$\index{${\boldsymbol{A}}_d= {\boldsymbol{A}}_{\lambda_d}$} 
de $\G_{\Bbb C}(G^\natural,\omega)$ 
dŽfini par
$${\boldsymbol{A}}_d= {\rm id} + \sum_{P^\natural,\,d(M_P)>d}\lambda_d(P^\natural){^\omega{T}_{P^\natural,{\Bbb C}}}$$
vŽrifie les propriŽtŽs:
$$\ker {\boldsymbol{A}}_d =\G_{{\Bbb C},d}(G^\natural,\omega),\quad 
{\boldsymbol{A}}_d\circ {\boldsymbol{A}}_d=
{\boldsymbol{A}}_d.
$$
\end{mapropo}

\begin{proof} 
Elle nŽcessite d'Žtablir quelques lemmes. Le lemme 1 Žtend au cas tordu la propriŽtŽ d'invariance du morphisme 
induction parabolique sous l'action de $W_G$ par conjugaison \cite[lemma 5.4]{BDK}. Le lemme 2 est la variante tordue de 
\cite[cor. 5.4]{BDK}. 

\begin{monlem1}
Soit $P^\natural,\,Q^\natural\in \P(G^\natural)$. Supposons que 
$M^\natural_{\sQ}=w(M^\natural_{\sP})\;(=n_w\cdot M^\natural_{\sP}\cdot n_w^{-1})$ 
pour un 
$w\in W_{G^\natural}$. Pour $\Sigma\in \G_{\Bbb C}(M^\natural_{\sP},\omega)$, on a l'Žgalit\'e dans 
$\G_{\Bbb C}(G^\natural,\omega)$
$$
{^\omega{i}_{Q^\natural}^{G^\natural}}(^w{\Sigma})= {^\omega{i}_{P^\natural}^{G^\natural}}(\Sigma).
$$
\end{monlem1}

\begin{proof}Par transport de structures, on a l'ŽgalitŽ
$${^\omega{i}_{w(P^\natural),{\Bbb C}}^{G^\natural}}(^w{\Sigma})= 
{^\omega{i}_{P^\natural,{\Bbb C}}^{G^\natural}}(\Sigma),$$ o ${^\omega{i}_{\omega(P^\natural),{\Bbb C}}^{G^\natural}}$ 
est le morphisme induction parabolique normalisŽ de $M^\natural_{\sQ}$ ˆ $G^\natural$ par 
rapport au sous--espace parabolique (non standard) $w(P^\natural)= M^\natural_{\sQ}\cdot U_{w(P)}$ de 
$G^\natural$. Rappelons que l'indice ${\Bbb C}$ indique que l'on est dans $\G_{\Bbb C}$, cf. \ref{ip et rj}. 
Il s'agit de voir que l'on peut remplacer $w(P^\natural)$ par $w(M^\natural_{\sP})\cdot U_\circ = Q^\natural$.

La double classe $W_{M_Q}wW_{M_P}=W_{M_Q}w = wW_{M_P}$ est $\theta$--stable. 
Quitte ˆ remplacer $w$ par un ŽlŽment de $W_{M_Q}w$, on peut supposer que 
$w\in W_{G^\natural}^{P,Q}$. Alors $w(M_P\cap P_\circ) = M_Q \cap P_\circ$, et pour tout 
sous--groupe parabolique standard $R'$ de $M_P$, $w(R')$ est un sous--groupe 
parabolique standard de $M_Q$. On en dŽduit que pour $R^\natural \in \P(G^\natural)$ tel que 
$R^\natural \subset P^\natural$, $w(R^\natural\cap M^\natural_{\sP})$ est un sous--espace 
parabolique standard de $M^\natural_{\sQ}$ de composante de Levi standard $\omega(M^\natural_{\sR})$.

D'aprs le lemme~1 de \ref{base langlands}, 
on peut supposer que $\Sigma = {^\omega{i}_{R^\natural,{\Bbb C}}^{P^\natural}}(\Xi'\Sigma')$ 
pour un $R\in \P(G^\natural)$ tel que $R^\natural\subset P^\natural$, une $\omega_{\rm u}$--reprŽsentation 
$\Sigma'\in {\rm Irr}_{{\Bbb C},{\rm t}}(M^\natural_{\sR},\omega_{\rm u})$ et un caractre $\Xi'\in \mathfrak{P}_{\Bbb C}(M^\natural_{\sR},
\vert\omega\vert)$ tel que $\Xi'^\circ$ est positif par rapport ˆ $U_R\cap M_P$. 
Par transitivitŽ du morphisme induction parabolique, on a l'ŽgalitŽ ${^\omega{i}_{P^\natural,{\Bbb C}}^{G^\natural}}(\Sigma)=
 {^\omega{i}_{R^\natural,{\Bbb C}}^{G^\natural}}(\Xi'\Sigma')$ et 
(d'aprs ce qui prŽcde) 
$${^\omega{i}_{Q^\natural,{\Bbb C}}^{G^\natural}}(^w{\Sigma})= 
{^\omega{i}^{G^\natural}_{w(R^\natural\cap M^\natural_{\sP})\cdot U_Q,{\Bbb C}}}({^w(\Xi'\Sigma')}).
$$
Quitte ˆ remplacer $P^\natural$ par $R^\natural$ et $\Sigma$ par $\Xi'\Sigma'$, on peut donc supposer que 
$\Sigma=\Xi'\Sigma'$ pour une $\omega_{\rm u}$--reprŽsentation $\Sigma'\in {\rm Irr}_{{\Bbb C},{\rm t}}(M^\natural_{\sP},\omega_{\rm u})$ et 
un caractre $\Xi'\in \mathfrak{P}_{\Bbb C}(M^\natural_{\sP},\vert\omega\vert)$ tel que $\Xi'^\circ >0$. 

Supposons qu'il existe un sous--ensemble Zariski--dense $\Omega$ de la sous--variŽtŽ 
irrŽductible $\mathfrak{X}=\{\Psi\Xi':\Psi\in \mathfrak{P}_{\Bbb C}^0(M^\natural_{\sP})\}$ de 
$\mathfrak{P}_{\Bbb C}(M^\natural_{\sP},\vert\omega\vert)$ tel 
que pour tout $\Psi\in  \Omega$, on a l'ŽgalitŽ
$$
{^\omega{i}_{P^\natural,{\Bbb C}}^{G^\natural}}(\Psi\Sigma')=
{^\omega{i}_{Q^\natural,{\Bbb C}}^{G^\natural}}({^w(\Psi\Sigma')}).
$$
Pour $\phi\in \H^\natural$, les fonctions $\Psi\mapsto \Phi_\phi({^\omega{i}_{P^\natural}^{G^\natural}}(\Psi\Sigma'))$ 
et $\Psi\mapsto \Phi_\phi({^\omega{i}_{Q^\natural}^{G^\natural}}({^w(\Psi\Sigma'))})$ sur 
$\mathfrak{P}_{\Bbb C}(M^\natural_{\sP},\vert\omega\vert)$
sont rŽgulires (\ref{modules admissibles}), d'o l'ŽgalitŽ
$$\Phi_\phi({^\omega{i}_{P^\natural,{\Bbb C}}^{G^\natural}}(\Sigma))
=\Phi_\phi({^\omega{i}_{Q^\natural,{\Bbb C}}^{G^\natural}}({^w\Sigma})).
$$
Enfin la propriŽtŽ d'indŽpendance linŽaire des caractres--distributions des 
$\omega$--reprŽsentations $G$--irrŽductibles de $G^\natural$ \cite[8.5, prop.]{L2} 
implique l'ŽgalitŽ
$$
{^\omega{i}_{P^\natural,{\Bbb C}}^{G^\natural}}(\Sigma)=
{^\omega{i}_{Q^\natural,{\Bbb C}}^{G^\natural}}({^w{\Sigma}}).
$$

Reste ˆ prouver l'existence de $\Omega$. Pour $\nu\in \mathfrak{b}_{{P^\natural},{\Bbb C}}^*$, notons 
$\Psi_\nu$ l'ŽlŽment de $\mathfrak{P}_{\Bbb C}^0(M^\natural_{\sP})$ dŽfini par 
$\Psi_\nu = \psi_\nu^{\delta_1}$ (\ref{les espaces bP}, et remarque 2 de \ref{carac nr}); on a donc 
$\Psi_\nu^\circ = \psi_\nu$ et $\psi_\nu\vert_{A_{P^\natural}}=\chi_\nu$. Soit $\boldsymbol{\Sigma}'$ une $\omega_{\rm u}$--reprŽsentation 
$M_P$--irrŽductible de $M^\natural_{\sP}$ dans la classe d'isomorphisme $\Sigma'$. Puisque $\boldsymbol{\Sigma}'$ est 
unitaire (\ref{langlands}, dŽfinition), 
son caractre central (cf. \ref{H modules}, remarque 3) l'est aussi. Choisissons une uniformisante 
$\varpi$ de $F$, et posons $A_{P^\natural}^\varpi={\rm Hom}({\rm X}^*(A_{P^\natural}),\langle \varpi \rangle)$. 
Via la dŽcomposition $A_{P^\natural}=A_{P^\natural}^1\times A_{P^\natural}^\varpi$, la restriction ˆ 
$A_{P^\natural}^\varpi$ du caractre central de $\Sigma'$ dŽfinit un caractre non ramifiŽ unitaire de $A_{P^\natural}$ 
(cf. \ref{carac nr}, exemple), i.e. de la forme $\chi_{i\mu_0}$ pour un ŽlŽment $\mu_0\in \mathfrak{b}_{P^\natural}^*$. 
Posons ${\boldsymbol{\Sigma}}'_0= \Psi_{-i\mu_0}\cdot 
\boldsymbol{\Sigma}'$. C'est une $\omega_{\rm u}$--reprŽsentation $M_P$--irrŽductible tempŽrŽe de $M^\natural_{\sP}$, 
de classe d'isomorphisme $\Sigma'_0=\Psi_{-i\mu_0}\Sigma'$. 
D'autre part, le caractre $\omega$ est trivial sur le centre $Z(M^\natural_{\sP})$ de $M^\natural_{\sP}$, par consŽquent 
la restriction de $\Xi'^\circ$ 
au tore $A_{P^\natural}$ est un caractre non ramifiŽ positif de $A_{P^\natural}$, i.e. de la forme $\chi_{\eta_0}$ pour un ŽlŽment 
$\eta_0\in \mathfrak{b}_{P^\natural}^*$ (cf. \ref{les espaces bP}). 
Posons $\Xi'_0= \Psi_{-\eta_0}\cdot \Xi'$. 
C'est  
un ŽlŽment de $\mathfrak{X}$ 
vŽrifiant $\Xi'_0(\delta_1)=\Xi'(\delta_1)$ et $\Xi'_0\vert_{A_{P^\natural}}=1$. En particulier, 
$\Xi'_0$ se relve en un ŽlŽment $\widetilde{\Xi}'_0\in \mathfrak{P}_{\Bbb C}(G^\natural,\vert\omega\vert)$. 
Posons $\boldsymbol{\Sigma}_0=\Xi'_0\cdot \boldsymbol{\Sigma}'_0$. C'est une 
$\omega$--reprŽsentation $M_P$--irrŽductible de $M^\natural_{\sP}$, de classe d'isomorphisme 
$\Sigma_0= \Xi'_0\Sigma'_0$. 
Pour $\nu\in \mathfrak{b}_{{P^\natural},{\Bbb C}}^*$, 
d'aprs le lemme de \ref{ip et rj} (rŽciprocitŽ de Frobenius), on a isomorphisme naturel
$$
{\rm Hom}_{\smash{G^\natural}}({^\omega{i}_{P^\natural}^{G^\natural}}(\Psi_\nu\cdot\boldsymbol{\Sigma}_0), 
{^\omega{i}_{Q^\natural}^{G^\natural}}({^{n_w}(\Psi_\nu\cdot\boldsymbol{\Sigma}_0)})\simeq
{\rm Hom}_{\smash{M^\natural_{\sQ}}}({\boldsymbol{h}}(\Psi_\nu\cdot\boldsymbol{\Sigma}_0),{^{n_w}(\Psi_\nu\cdot\boldsymbol{\Sigma}_0)});\leqno{(*)}
$$
o ${\boldsymbol{h}}$ dŽsigne le foncteur $
{^\omega{r}_{G^\natural}^{Q^\natural}}\circ{^\omega{i}_{P^\natural}^{G^\natural}}$.
D'aprs le lemme gŽomŽtrique (\ref{lemme gŽo}, proposition), l'image ${\boldsymbol{h}}_{\Bbb C}(\Psi_\nu\Sigma_0)$ de 
${\boldsymbol{h}}(\Psi_\nu\cdot{\boldsymbol{\Sigma}}_0)$ dans $\G_{\Bbb C}(M^\natural_{\sQ},\omega)$ est donnŽe par
$$
{\boldsymbol{h}}_{\Bbb C}(\Psi_\nu\Sigma_0)=\sum_s {\boldsymbol{f}}_{\Bbb C}(s)(\Psi_\nu\Sigma_0),\leqno{(**)}
$$
o $s$ parcourt les ŽlŽments de $W_{G^\natural}^{P,Q}$. 
Pour 
$s\in W_{G^\natural}^{P,Q}$, l'ŽlŽment ${\boldsymbol{f}}_{\Bbb C}(s)(\Psi_\nu\Sigma_0)$ de $\G_{\Bbb C}(M^\natural_{\sQ},\omega)$ 
a un caractre central, et la restriction ˆ $A_{Q^\natural}^\varpi\;(=w(A^\varpi_{P^\natural}))$ de ce caractre, identifiŽe comme plus haut 
ˆ un caractre non ramifiŽ de $A_{Q^\natural}$ via la dŽcomposition $A_{Q^\natural}=A_{Q^\natural}^\varpi\times A_{Q^\natural}^1$, 
est donnŽe par la projection orthogonale (cf. \ref{les espaces bP}) 
de $s(\nu) \in s(\mathfrak{b}_{P^\natural_{\bar{s}},{\Bbb C}}^*)= \mathfrak{b}_{Q^\natural_s,{\Bbb C}}^*$ sur 
$\mathfrak{b}_{Q^\natural,{\Bbb C}}^*$. De la m\^eme manire, la restriction ˆ $A_{Q^\natural}^\varpi$ du caractre central de 
${^w(\Psi_\nu\Sigma_0)}$ est donnŽe par $w(\nu)$. Comme $s$ et $w$ oprent isomŽtriquement sur les espaces en question, 
on en dŽduit que pour $\nu$ gŽnŽrique, seul l'ŽlŽment $s=w$ dans la somme $(**)$ peut donner 
une contribution non triviale ˆ l'espace ${\rm Hom}$ de droite dans l'isomorphisme $(*)$, et comme 
${^{n_w}(\Psi_\nu\cdot \boldsymbol{\Sigma}_0)}$ est un quotient de ${\boldsymbol{h}}(\Psi_\nu\cdot \boldsymbol{\Sigma}_0)$ --- cf. la dŽmonstration de 
la proposition de \ref{lemme gŽo} ---, il en donne effectivement une.  
On obtient que pour $\nu\in \mathfrak{b}_{P^\natural,{\Bbb C}}^*$ gŽnŽrique, on a 
$$
\dim_{\Bbb C}{\rm Hom}_{\smash{M^\natural_{\sQ}}}({^\omega{i}_{P^\natural,{\Bbb C}}^{G^\natural}}(\Psi_\nu\cdot \boldsymbol{\Sigma}_0), 
{^\omega{i}_{Q^\natural,{\Bbb C}}^{G^\natural}}({^{n_w}(\Psi_\nu\cdot \boldsymbol{\Sigma}_0)})=1.
$$
Le m\^eme raisonnement entra\^{\i}ne que pour $\nu\in \mathfrak{b}_{P^\natural,{\Bbb C}}^*$ gŽnŽrique, on a 
$$
\dim_{\Bbb  C}{\rm End}_{\smash{M^\natural_{\sP}}}({^\omega{i}_{P^\natural,{\Bbb C}}^{G^\natural}}(\Psi_\nu\cdot
\boldsymbol{\Sigma}_0))=1
$$
et
$$
\dim_{\Bbb  C}{\rm End}_{\smash{M^\natural_{\sQ}}}({^\omega{i}_{Q^\natural,{\Bbb C}}^{G^\natural}}({^w(\Psi_\nu
\cdot \boldsymbol{\Sigma}_0)}))=1
$$

D'aprs le raisonnement ci--dessus, on peut remplacer la condition \og $\nu\in\mathfrak{b}_{P^\natural,{\Bbb C}}^*$ gŽnŽrique \fg 
par \og $\nu=i\mu$, $\mu\in \mathfrak{b}_{P^\natural}^*$ gŽnŽrique \fg. Pour $\nu\in \mathfrak{b}_{P^\natural,{\Bbb C}}^*$, on a
$$
{^\omega{i}_{P^\natural}^{G^\natural}}(\Psi_\nu\cdot\boldsymbol{\Sigma}_0)= \widetilde{\Xi}'_0\cdot 
{^\omega{i}_{P^\natural}^{G^\natural}}(\Psi_\nu\cdot\boldsymbol{\Sigma}'_0).
$$
Pour $\mu\in \mathfrak{b}_{P^\natural}^*$, la reprŽsentation 
${^\omega{i}_{P^\natural}^{G^\natural}}(\Psi_{i\mu}\cdot\boldsymbol{\Sigma}'_0)^\circ=i_P^G(\psi_{i\mu}(\boldsymbol{\Sigma}'_0)^\circ)$ de $G$ est unitaire 
et donc semisimple, par suite la $\omega$--reprŽsentation ${^\omega{i}_{P^\natural}^{G^\natural}}(\Psi_{i\mu}\cdot\boldsymbol{\Sigma}_0)$ 
de $G^\natural$ est elle aussi semisimple (\cad somme directe de $\omega$--reprŽsentations irrŽductibles de $G^\natural$). 
De m\^eme, toujours pour $\mu\in \mathfrak{b}_{P^\natural}^*$, la 
$\omega$--reprŽsentation ${^\omega{i}_{Q^\natural}^{G^\natural}}({^{n_w}(\Psi_{i\mu}\cdot\boldsymbol{\Sigma}_0)})$ 
de $G^\natural$ est semisimple. Par suite pour $\mu\in \mathfrak{b}_{P^\natural}^*$ gŽnŽrique, les 
ŽlŽments ${^\omega{i}_{P^\natural,{\Bbb C}}^{G^\natural}}(\Psi_{i\mu}\Sigma_0)$ et 
${^\omega{i}_{P^\natural,{\Bbb C}}^{G^\natural}}({^w(\Psi_{i\mu}\Sigma_0)})$ de $\G_{\Bbb C}(G^\natural,\omega)$ 
appartiennent ˆ ${\rm Irr}_{\Bbb C}(G^\natural,\omega)$, et sont Žgaux.

Pour $\mu\in \mathfrak{b}_{P^\natural}^*$, on a $\Psi_{i\mu}\Sigma_0=
\Psi_{-\eta_0 + i(\mu-\mu_0)}\Xi'\Sigma'$. On conclut en remarquant que
$$
\Omega= \{\lambda\Psi_{-\eta_0 + i(\mu-\mu_0)} \Xi':\mbox{$\mu\in \mathfrak{b}_{P^\natural}^*$ gŽnŽrique, $\lambda\in {\Bbb C}^\times$}\}
$$
est un sous--ensemble Zariski--dense de $\mathfrak{X}$.
\end{proof}

\goodbreak
\begin{monlem2}Soit $P^\natural,\,Q^\natural\in \P(G^\natural)$.
\begin{enumerate}
\item[(1)]Pour $\Sigma\in \G_{\Bbb C}(G^\natural,\omega)$, on a l'ŽgalitŽ dans $\G_{\Bbb C}(G^\natural,\omega)$
$$
{^\omega{T}_{Q^\natural,{\Bbb C}}}\circ {^\omega{i}_{P^\natural}^{G^\natural}}(\Sigma)=
\sum_w {^\omega{i}_{P^\natural_{\bar{w}}}^{G^\natural}}\circ {^\omega{r}_{P^\natural}^{P^\natural_{\bar{w}}}}(\Sigma),
$$
o $w$ parcourt les ŽlŽments de $W_{G^\natural}^{P,Q}$.
\item[(2)]Pour $\Pi\in \G_{\Bbb C}(G^\natural,\omega)$, on a l'ŽgalitŽ dans $\G_{\Bbb C}(G^\natural,\omega)$
$$
{^\omega{T}_{Q^\natural,{\Bbb C}}}\circ {^\omega{T}_{P^\natural,{\Bbb C}}}(\Pi)= \sum_w {^{\omega}T_{P^\natural_{\bar{w}},{\Bbb C}}}(\Pi),
$$
o $w$ parcourt les ŽlŽments de $W_{G^\natural}^{P,Q}$.
\end{enumerate}
\end{monlem2}

\begin{proof}Lorsqu'il n'y a pas d'ambigu\"{\i}tŽ possible, pour $w\in W_{G^\natural}$, 
on note \og $w$ \fg le foncteur \og $\Sigma \rightarrow {^{n_w}\Sigma}$\fg. 
D'aprs la proposition de \ref{lemme gŽo}, on a l'ŽgalitŽ dans $\G_{\Bbb C}(G^\natural,\omega)$
$$
{^\omega{T}_{Q^\natural,{\Bbb C}}}\circ {^\omega{i}_{P^\natural}^{G^\natural}}(\Sigma)= \sum_w
{^\omega{i}_{Q^\natural}^{G^\natural}}\circ {^\omega{i}_{Q^\natural_w}^{Q^\natural}}\circ w \circ {^\omega{r}_{P^\natural}^{P^\natural_{\bar{w}}}}(\Sigma),
$$
o $w$ parcourt les ŽlŽments de $W_{G^\natural}^{P,Q}$. D'aprs le lemme 1, pour $w\in W_{P^\natural,Q^\natural}$, on a 
l'ŽgalitŽ dans $\G_{\Bbb C}(M^\natural_{\sQ},\omega)$
$$
{^\omega{i}_{Q^\natural_w}^{Q^\natural}}\circ w \circ {^\omega{r}_{P^\natural}^{P^\natural_{\bar{w}}}}(\Sigma)=
{^\omega{i}_{P^\natural_{\bar{w}}}^{Q^\natural}}\circ {^\omega{r}_{P^\natural}^{P^\natural_{\bar{w}}}}(\Sigma).
$$
D'o le point (1), par transitivitŽ du morphisme induction parabolique. 

Quant au point (2), par transitivitŽ du morphisme restriction de Jacquet, d'aprs la proposition de \ref{lemme gŽo}, on 
a l'ŽgalitŽ dans $\G_{\Bbb C}(M^\natural_{\sQ},\omega)$
$$
{^\omega{r}_{G^\natural}^{Q^\natural}}\circ {^\omega{T}_{P^\natural,{\Bbb C}}}(\Pi)=\sum_w{^\omega{i}_{\sQ^\natural_w}^{\sQ^\natural}}
\circ w \circ {^\omega{r}_{G^\natural}^{P^\natural_{\bar{w}}}}(\Pi),
$$
o $w$ parcourt les ŽlŽments de $W_{G^\natural}^{P,Q}$. On obtient le point (2) en appliquant le morphisme 
${^\omega{i}_{Q^\natural}^{G^\natural}}$ ˆ cette expression (gr\^ace au lemme 1, et par transitivitŽ du 
morphisme induction parabolique).
\end{proof}

On peut maintenant dŽmontrer la proposition. 
Pour allŽger l'Žcriture, on pose
$$\G_{{\Bbb C},i}=
\G_{{\Bbb C},i}(G^\natural,\omega),\quad i\geq -1.
$$ 
D'aprs le point (1) du lemme 2, pour $Q^\natural\in \P(G^\natural)$, l'opŽrateur ${^\omega{T}_{Q^\natural,{\Bbb C}}}$ 
prŽserve la filtration $\{\G_{{\Bbb C},i}\}$. Pour $P^\natural,\,Q^\natural\in \P(G^\natural)$, 
puisque $W_G(P,Q)\cap W_G^{P,Q}$ est un systme de reprŽsentants des classes de
$$
W_{M_Q}\backslash W_G(P,Q)=W_G(P,Q)/W_{M_P},
$$ l'ensemble 
$W_{G^\natural}(P,Q)\cap W_{G^\natural}^{P,Q}$ paramŽtrise le sous--ensemble
$$[W_{M_Q}\backslash W_G(P,Q)]^\theta\subset W_{M_Q}\backslash W_G(P,Q)
$$
formŽ des classes $\theta$--invariantes, 
et s'il est non vide alors il est de cardinal
$$
p_{Q^\natural}= \vert [W_{M_Q}\backslash W_G(Q,Q)]^\theta \vert. 
$$
D'aprs loc.~cit. et le lemme 1, on en dŽduit que pour 
et $\Sigma\in \G_{\Bbb C}(M^\natural_{\sP},\omega)$, on a 
$${^\omega{T}_{Q^\natural,{\Bbb C}}}\circ {^\omega{i}_{P^\natural}^{G^\natural}}(\Sigma)\equiv 
\left\{\begin{array}{cc} 
p_{Q^\natural}{^\omega{i}_{P^\natural}^{G^\natural}}(\Sigma)\;({\rm mod}\,\G_{{\Bbb C},d(M_Q)})&
\hbox{si $W_{G^\natural}(P,Q)\cap W_{G^\natural}^{P,Q}\neq \emptyset$}\\
0\;({\rm mod}\,\G_{{\Bbb C},d(M_Q)})\hfill &\hbox{sinon}\hfill \\
\end{array}\right..
$$

Pour $k> d$, choisissons un ordre $\{P_{k,1}^\natural,\ldots , P_{k,n(k)}^\natural\}$ 
sur l'ensemble des $P^\natural\in \P(G^\natural)$ tels que $d(M_P)=k$, et notons
$U_k:\G_{\Bbb C}(G^\natural,\omega)\rightarrow \G_{\Bbb C}(G^\natural,\omega)$ le morphisme 
dŽfini par
$$
U_k = (T_{k,n(k)} - 
p_{k,n(k)})\circ \cdots \circ (T_{k,1} - p_{k,1});
$$
o l'on a posŽ
$$
T_{k,i}={^\omega{T}_{P^\natural_{k,i},{\Bbb C}}},\quad p_{k,i}= p_{P^\natural_{k,i}}.
$$
D'aprs ce qui prŽcde, l'opŽrateur $U_k$ prŽserve la filtration $\{\G_{{\Bbb C},i}\}$ et 
annule $\G_{{\Bbb C},k-1}/\G_{{\Bbb C},k}$.
Posons
$$
\boldsymbol{A}'_d= U_{d(M_\circ)}\circ\cdots \circ  U_{d+1}.
$$
On a $\boldsymbol{A}'_d(\G_{{\Bbb C},d})\subset \G_{{\Bbb C},d(M_\circ)}=0$, 
et d'aprs le point (2) du lemme 2, il existe un $\mu\in {\Bbb Z}\smallsetminus \{0\}$ et 
des $\lambda_d(P^\natural)\in {\Bbb Q}$ pour $P^\natural\in \P(G^\natural)$, $d(M^\natural_{\sP})>d$, 
tels que
$$
\boldsymbol{A}'_d=\mu ({\rm id} + \sum_{P^\natural,\, d(M_P)>d}\lambda_d(P^\natural){^\omega{T}_{P^\natural,{\Bbb C}}}).
$$
L'opŽrateur ${\boldsymbol{A}}_d= \mu^{-1}\boldsymbol{A}'_d:\G_{\Bbb C}(G^\natural,\omega)\rightarrow \G_{\Bbb C}(G^\natural,\omega)$ vŽrifiant
$$
{\boldsymbol{A}}_d(\Pi)\equiv \Pi\;({\rm mod}\, \G_{{\Bbb C},d}(G^\natural,\omega)),
$$
on a bien
$$
\ker {\boldsymbol{A}}_d = \G_{{\Bbb C},d}(G^\natural,\omega),\quad {\boldsymbol{A}}_d\circ {\boldsymbol{A}}_d={\boldsymbol{A}}_d.
$$
Cela achve la dŽmonstration de la proposition. \end{proof}

\begin{variante}
{\rm On peut dŽfinir une autre filtration dŽcroisssante 
$\{\G_{\Bbb C}(G^\natural,\omega)_i\}$ de $\G_{\Bbb C}(G^\natural,\omega)$, en rempla\c{c}ant la 
condition $d(M_P)>i$ par la condition $d(M^\natural_{\sP})>i$: pour chaque entier 
$i\geq -1$, on pose\index{$\G_{{\Bbb C}}(G^\natural,{\Bbb C})_i$}
$$
\G_{{\Bbb C}}(G^\natural,{\Bbb C})_i= \sum_{P^\natural,\,d(M^\natural_{\sP})>i}{^\omega{i}_{P^\natural}^{G^\natural}}(\G_{\Bbb C}
(M^\natural_{\sP},\omega)),
$$
o $P^\natural$ parcourt les ŽlŽments de $\P(G^\natural)$. 
On a donc
$$
\G_{\Bbb C}(G^\natural,\omega)_i=\G(G^\natural,\omega),\quad i<d(G^\natural),
$$ 
$$
\G_{\Bbb C}(G^\natural,\omega)_i=0,\quad i\geq d(M^\natural_\circ).
$$
On peut aussi, gr\^ace au lemme 2, Žtablir la variante suivante de la proposition:  

{\it Soit $d'$ un entier tel que 
$d(G^\natural)\leq d'<d(M^\natural_\circ)$. 
Il existe une famille de nombres rationnels $\mu_{d' }= \{\mu_{d'}(P^\natural):P^\natural\in \P(G^\natural),\,d(M^\natural_{\sP})>d'\}$ 
telle que le ${\Bbb C}$--endomorphisme ${\boldsymbol{B}}_{d'}= {\boldsymbol{B}}_{\mu_{d'}}$\index{${\boldsymbol{B}}_{d'}= {\boldsymbol{B}}_{\mu_{d'}}$} 
de $\G_{\Bbb C}(G^\natural,\omega)$ 
dŽfini par
$${\boldsymbol{B}}_{d'}= {\rm id} + \sum_{P^\natural,\,d(M^\natural_{\sP})>d}\mu_{d'}(P^\natural){^\omega{T}_{P^\natural,{\Bbb C}}}$$
vŽrifie les propriŽtŽs:
$$\ker {\boldsymbol{B}}_{d'} =\G_{\Bbb C}(G^\natural,\omega)_{d'},\quad 
{\boldsymbol{B}}_{d'}\circ {\boldsymbol{B}}_{d'}=
{\boldsymbol{B}}_{d'}.\eqno{\blacksquare}
$$}
}
\end{variante}

\subsection{Actions duales de $\mathfrak{Z}(G)$ et de $\mathfrak{P}_{\Bbb C}(G^\natural)$.}\label{actions duales} 
Commen\c{c}ons par quelques rappels sur le \og centre \fg. Pour $P,\,Q\in \EuScript{P}(G)$ tels que $Q\subset P$, l'application naturelle
\index{${i}_{P,Q}$, ${i}_{P,Q}^*$} 
$$
{i}_{P,Q}:\Theta(M_Q)\rightarrow \Theta(M_P),
$$ qui a $[L,\rho]_{M_Q}$ associe $[L,\rho]_{M_P}$,  
est un morphisme {\it fini} de vari\'et\'es alg\'ebriques. Par dualit\'e il induit un morphisme d'anneaux
$$
{i}_{P,Q}^*:\mathfrak{Z}(M_P)\rightarrow \mathfrak{Z}(M_Q),
$$
donn\'e par 
$$
f_{{i}_{P,Q}^*(z)}(x)=f_z({i}_{P,Q}(x)),\quad z\in \mathfrak{Z}(M_P),\, x\in \Theta(M_Q).
$$
Ce morphisme $i_{P,Q}^*$ fait de $\mathfrak{Z}(M_Q)$ un $\mathfrak{Z}(M_P)$-module de type fini, et d'aprs 
\cite[2.13--2.16]{BD}, on a le

\begin{monlem1}Soit $P,\,Q\in \EuScript{P}(G)$ tels que $Q\subset P$. 
Soit $z\in \mathfrak{Z}(M_P)$ 
et $t={i}_{P,Q}^*(z)\in \mathfrak{Z}(M_Q)$.
\begin{enumerate} 
\item[(1)] Pour toute repr\'esentation $\sigma$ de $M_Q$, l'endomorphisme ${i}_Q^P(t)$ de ${i}_Q^P(\sigma)$ co\"{\i}ncide avec $z$.
\item[(2)] Pour toute repr\'esentation $\pi$ de $M_P$, l'endomorphisme $r_P^Q(z)$ de $r_P^Q(\pi)$ co\"{\i}ncide avec $t$.
\end{enumerate}
\end{monlem1}

Rappelons (\ref{ŽnoncŽ}) que l'anneau $\mathfrak{Z}(G)=\prod_{\mathfrak{s}}\mathfrak{Z}_\mathfrak{s}$ op\`ere sur 
l'espace $\G_{\Bbb C}(G^\natural,\omega)^*$: pour 
$z\in \EuScript{Z}(G)$ et $\Phi\in \G_{\Bbb C}(G^\natural,\omega)^*$, $z\cdot\Phi$ est l'ŽlŽment de 
$ \G_{\Bbb C}(G^\natural,\omega)^*$ est donnŽ par
$$
(z\cdot \Phi)(\Pi)=f_z(\theta_G(\Pi^\circ))\Phi(\Pi),\quad \Pi\in {\rm Irr}_{\Bbb C}(G^\natural,\omega).
$$
D'aprs le lemme 1, pour $P^\natural,\,Q^\natural\in \EuScript{P}(G^\natural)$ tels que $Q^\natural\subset P^\natural$, les morphismes $({^\omega{i}}_{P^\natural}^{Q^\natural})^*$ et $({^\omega{r}}_{P^\natural}^{Q^\natural})^*$ sont des morphismes de $\mathfrak{Z}(M_P)$-modules, i.e. on a
$$
({^\omega{i}}_{Q^\natural}^{P^\natural})^*(z\cdot\Phi)
={i}_{P,Q}^*(z)\cdot ({^\omega{i}}_{Q^\natural}^{P^\natural})^*(\Phi),\quad \Phi\in \G_{\Bbb C}(M_{\sP}^\natural,\omega)^*,\, z\in \mathfrak{Z}(M_P),
$$
et
$$
({^\omega{r}}_{P^\natural}^{Q^\natural})^*({i}_{P,Q}^*(z)\cdot\Phi)=z\cdot ({^\omega{r}}_{P^\natural}^{Q^\natural})^*(\Phi).
$$

On a aussi une action du groupe $\mathfrak{P}_{\Bbb C}(G^\natural)$ sur l'espace $\G_{\Bbb C}(G^\natural,\omega)^*$: pour 
$\Psi\in \mathfrak{P}_{\Bbb C}(G^\natural)$ 
et $\Phi\in \G_{\Bbb C}(G^\natural,\omega)^*$, on pose
$$
(\Psi\Phi)(\Pi)=\Phi(\Psi\Pi),\quad \Pi\in {\rm Irr}_{\Bbb C}(G^\natural,\omega).
$$

\begin{monlem2}
On a:
\begin{enumerate}
\item[(1)]$\mathfrak{Z}(G)\cdot \EuScript{F}(G^\natural,\omega)=\EuScript{F}(G^\natural,\omega)$ et 
$\mathfrak{P}_{\Bbb C}(G^\natural)\F(G^\natural,\omega)=\EuScript{F}(G^\natural,\omega)$.
\item[(2)]$\mathfrak{Z}(G) \cdot \EuScript{F}_{\rm tr}(G^\natural,\omega)=\EuScript{F}_{\rm tr}(G^\natural,\omega)$ et 
$\mathfrak{P}_{\Bbb C}(G^\natural) \EuScript{F}_{\rm tr}(G^\natural,\omega)=\EuScript{F}_{\rm tr}(G^\natural,\omega)$.
\end{enumerate}
\end{monlem2}

\begin{proof} 
Soit $P^\natural\in \EuScript{P}(G^\natural)$ et $\Phi\in \G_{\Bbb C}(G^\natural,\omega)^*$. Pour $z\in \mathfrak{Z}(G)$ et 
$\Sigma\in {\rm Irr}_{\Bbb C}(M_{\sP}^\natural,\omega)$, posant 
$t={i}^*_{G,P}(z)\in \mathfrak{Z}(M_P)$, on a

\begin{eqnarray*}
(z\cdot\Phi)({^\omega{i}_{P^\natural}^{G^\natural}}(\Sigma)) &=&
({^\omega{i}_{P^\natural}^{G^\natural}})^*(z\cdot \Phi)(\Sigma)\\
&=& (t\cdot ({^\omega{i}_{P^\natural}^{G^\natural}})^*(\Phi))(\Sigma)\\
&=& f_t(\theta_{M_P}(\Sigma^\circ))
({^\omega{i}_{P^\natural}^{G^\natural}})^*(\Phi)(\Sigma).
\end{eqnarray*}

Supposons que $\Phi$ appartient ˆ $\EuScript{F}(G^\natural,\omega)$. Alors 
$({^\omega{i}_{P^\natural}^{G^\natural}})^*(\Phi)$ appartient \`a $\EuScript{F}(M_{\sP}^\natural,\omega)$, 
d'apr\`es la propriŽtŽ de transitivitŽ du morphisme induction parabolique. 
Par suite l'application
$$
\Xi\mapsto (z\cdot\Phi)({^\omega{\iota_{P^\natural}^{G^\natural}}}(\Xi\Sigma))
$$
est une fonction r\'eguli\`ere sur $\mathfrak{P}_{\Bbb C}(M_{\sP}^\natural)$. Comme par ailleurs le \og support inertiel \fg
de $z\cdot \Phi$ --- \cad l'ensemble des $\mathfrak{s}\in \mathfrak{B}_{G^\natural,\omega}(G)$ tels que $\theta_{G^\natural}(\Pi)=\mathfrak{s}$ pour 
un $\Pi\in {\rm Irr}_{\Bbb C}(G^\natural,\omega)$ vŽrifiant $(z\cdot\Phi)(\Pi)\neq 0$ --- est contenu dans celui de $\Phi$, on a l'\'egalit\'e 
$\mathfrak{Z}(G)\cdot \EuScript{F}(G^\natural,\omega)=\EuScript{F}(G^\natural,\omega)$. La seconde \'egalit\'e du point (1) 
est claire, puisque pour $\Psi\in \mathfrak{P}_{\Bbb C}(G^\natural)$ on a
$$
\Psi\,{^\omega{{i}_{P^\natural}^{G^\natural}(\Sigma)}}= {^\omega{{i}_{P^\natural}^{G^\natural}(\Psi\vert_{M_{\sP}^\natural}\Sigma)}}.
$$

Quant au point (2), la premire ŽgalitŽ 
rŽsulte du fait que l'application $\phi\mapsto \Phi_\phi$ est un morphisme de $\mathfrak{Z}(G)$--module 
(lemme de \ref{ŽnoncŽ}), et la seconde 
du fait qu'elle est $\mathfrak{P}_{\Bbb C}(G^\natural)$--Žquivariante pour l'action naturelle de $\mathfrak{P}_{\Bbb C}(G^\natural)$ 
sur $\H^\natural$ (donnŽe par $(\Psi,\phi)\mapsto \Psi\phi$). 
\end{proof}

\subsection{Induction parabolique et restriction de Jacquet: morphismes duaux}\label{morphismes duaux}
Pour $P^\natural,\,Q^\natural\in \P(G^\natural)$ tels que $Q^\natural\subset P^\natural$, le 
morphisme ${^\omega{i}}_{Q^\natural}^{P^\natural}:\G_{\Bbb C}(M_Q^\natural,\omega)\rightarrow \G_{\Bbb C}(M_{\sP}^\natural,\omega)$ 
induit par dualitŽ un morphisme ${\Bbb C}$--linŽaire\index{$({^\omega{i}}_{Q^\natural}^{P^\natural})^*$, $({^\omega{r}}_{P^\natural}^{Q^\natural})^*$}
$$
({^\omega{i}}_{Q^\natural}^{P^\natural})^*: \G(M_{\sP}^\natural,\omega)^*\rightarrow \G_{\Bbb C}(M_Q^\natural,\omega)^*.
$$
PrŽcisŽment, on a
$$
({^\omega{i}}_{Q^\natural}^{P^\natural})^*(\Phi)(\Sigma)= \Phi({^\omega{i}}_{Q^\natural}^{P^\natural}(\Sigma)),\quad 
\Phi\in \G_{\Bbb C}(M_{\sP}^\natural,\omega),\,\Sigma\in \G_{\Bbb C}(M_Q^\natural,\omega).
$$
De la m\^eme mani\`ere, le morphisme 
${^\omega{r}}_{P^\natural}^{Q^\natural}:\G_{\Bbb C}(M_{\sP}^\natural,\omega)\rightarrow \G_{\Bbb C}(M_Q^\natural,\omega)$
induit par dualit\'e un morphisme ${\Bbb C}$--linŽaire
$$
({^\omega{r}}_{P^\natural}^{Q^\natural})^*:\G_{\Bbb C}(M_Q^\natural,\omega)^*\rightarrow \G_{\Bbb C}(M_{\sP}^\natural,\omega)^*.
$$

\begin{mapropo}\label{compat i-r et bo-tr}
Pour $P^\natural,\,Q^\natural\in \P(G^\natural)$ tels que $Q^\natural\subset P^\natural$, on a:
\begin{enumerate}
\item[(1)]$({^\omega{i}}_{Q^\natural}^{P^\natural})^*(\F(M_{\sP}^\natural,\omega))\subset \F(M_Q^\natural,\omega)$.
\item[(2)]$({^\omega{r}}_{P^\natural}^{Q^\natural})^*(\F_{\rm tr}(M_Q^\natural,\omega))\subset \F_{\rm tr}(M_{\sP}^\natural,\omega)$.
\end{enumerate}
\end{mapropo}

\begin{proof} Le point (1) --- dŽjˆ utilisŽ dans la preuve du lemme 2 de \ref{actions duales} --- 
r\'esulte de la propriŽtŽ de transitivit\'e du morphisme induction parabolique. 

Quant au point (2), on procde comme dans \cite[5.3]{BDK}, gr\^ace aux r\'esultats de Casselman reliant les 
caract\`eres--distributions aux foncteurs restriction de Jacquet \cite{C} (dans le cas tordu, voir aussi \cite[5.10]{L2}). On peut supposer 
$P^\natural=G^\natural$. Fixons un sous--groupe d'Iwahori $I$ de $G^\natural$ en bonne position par rapport 
ˆ $(P_\circ,M_\circ)$ et tel que $\theta(I)=I$, et notons $\mathfrak{I}$ l'ensemble $\{I^n: n\geq 0\}$ des sous--groupes 
de congruence de $\mathfrak{I}$. Pour $n\geq 0$, posons 
$\H_n(M^\natural_{\sQ})= \H_{I^n\cap {M_Q}}(M^\natural_{\sQ})$. Puisque $\H(M^\natural_{\sQ})=
\bigcup_{n\geq 0}\H_n(M^\natural_{\sQ})$, il suffit de montrer que pour chaque $n\geq 0$, le sous--espace 
$\H_n^\star(M^\natural_{\sQ})$ de $\H_n(M^\natural_{\sQ})$ formŽ des fonctions $\phi'$ telles 
que $({^\omega{r}}_{G^\natural}^{Q^\natural})^*(\Phi_{\phi'})\in \F_{\rm tr}(G^\natural,\omega)$, co\"{\i}ncide avec 
$\H_n(M^\natural_{\sQ})$; o $\Phi_{\phi'}$ est l'ŽlŽment de $\F_{\rm tr}(M^\natural_{\sQ},\omega)$ dŽfini comme en 
\ref{carac} en rempla\c{c}ant $G^\natural$ par $M^\natural_{\sQ}$. Posons $M=M_Q$, $U=U_Q$ et $\overline{U}=\overline{U}_Q$. 

Fixons $n\geq 0$, et posons $J=I^n$ et $J_\Gamma=J\cap \Gamma$ pour tout sous--groupe fermŽ $\Gamma$ de $G$. 
Pour $\delta\in G^\natural$, on note 
$\phi^J_\delta\in \H_J(G^\natural)$ la fonction caractŽristique de la double classe $J\cdot \delta \cdot J$ divisŽe par 
${\rm vol}(J\cdot \delta \cdot J, d\delta)$. Pour $\delta\in M^\natural=M^\natural_{\sQ}$, on dŽfinit de la m\^eme 
manire $\phi^{J_M}_\delta\in \H_{J_M}(M^\natural)$. Dans le cas non tordu (i.e. si $G^\natural=G$), 
on pose $f^J_g = \phi^J_g$ pour $g\in G$, 
et $f^{J_M}_m=\phi^{J_M}_m$ pour $m\in M$. Fixons un ŽlŽment $a\in A= A_Q$ tel que 
${\rm Int}_{\boldsymbol{G}}(a)$ contracte strictement $U$, \cad vŽrifiant
$$
\bigcap_{k\geq 1} {\rm Int}_{\boldsymbol{G}}(a)^k (U_J)=\{1\}.
$$

Soit $\Pi$ une $\omega$--reprŽsentation admissible de $G^\natural$. Posons 
$V=V_\Pi$ et $\pi=\Pi^\circ$. L'espace
$$V(U)=\langle \pi(u)(v)-v: v\in V,\,u\in U\rangle$$ co\"{\i}ncide avec l'ensemble des $v\in V$ 
tels que $\int_{\Omega_v}\pi(u)(v)du=0$ pour un sous--groupe ouvert compact $\Omega_v$ de $U$. Puisque $\Pi$ est 
admissible, l'espace $V(U)\cap V^J$ est dimension finie, et il existe un sous--groupe ouvert compact $\Omega$ de $U$ tel 
que $\int_\Omega \pi(u)(v)du=0$ pour tout $v\in V(U)\cap V^J$. Quitte ˆ remplacer $\Omega$ par un groupe plus gros, 
on peut supposer que $J_U$ est contenu dans $\Omega$. Soit $k_0\geq 1$ un entier tel que 
${\rm Int}_{{\boldsymbol{G}}^\natural}(a)^{k_0}(\Omega)\subset J_U$. 
Soit $\overline{\Pi}= \delta_{Q^\natural}^{1\over 2}{^\omega{r}_{G^\natural}^{Q^\natural}}$ 
la $\omega$--reprŽsentation de $M^\natural$ dŽduite de $\Pi$ par passage au quotient sur l'espace $\overline{V}=V/V(U)$. 
D'aprs la proposition 3.3 de \cite{C}, pour tout entier $k\geq k_0$, notant $V_a^{J,k}$ le sous--espace 
$\pi(f^J_a)^k(V)$ de $V^J$, on a:
\begin{itemize}
\item la projection canonique $p:V\rightarrow  \overline{V}$ induit par restriction 
un isomorphisme de ${\Bbb C}$--espaces vectoriels $V^{J,k}_a\rightarrow \smash{\overline{V}}^{J_M}$;
\item $\pi(f_a^J)(V_a^{J,k})=V^{J,k}_a$.
\end{itemize}
D'autre part, pour tout $\gamma\in M^\natural$ tel que ${\rm Int}_{{\boldsymbol{G}}^\natural}(\gamma)$ 
contracte strictement $U$, d'aprs la dŽmons\-tration 
du point (1) du lemme 2 de \cite[5.10]{L2}, on a
$$
p(\Pi(\phi_\gamma^J)(v))= \overline{\Pi}(\gamma)(p(v)),\quad v\in V^J.
$$
Pour un tel $\gamma$, et pour tout entier $k\geq k_0$, on a $\phi^J_{a^k\cdot \gamma}= (f^J_a)^{*k}* \phi^J_\gamma$, 
$(f^J_a)^{*k}=f^J_a*\cdots *f^J_a$ ($k$ fois); de m\^eme on a  et $\phi^{J_M}_{a^k\cdot\gamma}=(f^{J_M}_a)^{*k}*\phi^{J_M}_\gamma$. 
Par suite $\Pi(\phi^J_{a^k\cdot \gamma})(V)\subset V^{J,k}_a$, et puisque $p\circ \Pi(\delta_1)=\overline{\Pi}(\delta_1)\circ p$, on a
\begin{eqnarray*}
{\rm tr}(\Pi(\phi^J_{a^k\cdot \gamma}); V^J)&=& {\rm tr}(\Pi(\phi^J_{a^k\cdot \gamma}); V^{J,k}_a)\\
&=&{\rm tr}(\overline{\Pi}(a^k\cdot \gamma); \smash{\overline{V}}^{J_M})
={\rm tr}(\overline{\Pi}(\phi^{J_M}_{a^k\cdot\gamma});\smash{\overline{V}}^{J_M}).
\end{eqnarray*}
Posant $\phi=\phi^J_{a^k\cdot\gamma}$ et $\phi'= \phi^{J_M}_{a^k\cdot \gamma}$, 
on a $({^\omega{r}_{G^\natural}^{Q^\natural}})^*(\Phi_{\phi'})= \Phi_\phi$, par consŽquent $\phi'\in \H^\star_{J_M}(M^\natural)$. 

Pour un $\delta\in M^\natural$ arbitraire, il existe un entier $m\geq 1$ tel que ${\rm Int}_{{\boldsymbol{G}}^\natural}(a^m\cdot \delta)$ contracte strictement 
$U$, et d'aprs la discussion prŽcŽdente, pour tout entier $k\geq k_0$, la fonction
$$
\phi_{a^{k+m}\cdot \delta}^{J_M}= (f^{J_M}_a)^{*(k+m)}* \phi^{J_M}_\delta
$$
est dans l'espace $\H^*_{J_M}(M^\natural)$. Puisque les fonctions $\phi^{J_M}_\delta$ engendrent le 
${\Bbb C}$--espace vectoriel $\H_{J_M}(M^\natural)$, on obtient que pour toute fonction $\phi'\in \H_{J_M}(M^\natural)$, il existe un 
entier $l\geq 1$ tel que $(f^J_a)^{*l}*\phi'$ appartient ˆ $\H^\star_{J_M}(M^\natural)$. 

Rappelons que l'espace $\H_{J_M}(M)$, vu comme un $\H_{J_M}(M)$--module non dŽgŽnŽrŽ pour la multiplication 
ˆ gauche, est un $\mathfrak{Z}(M)$--module de type fini \cite[cor.~3.4]{BD}. D'aprs \ref{actions duales}, 
le morphisme 
$$i_{G,Q}^*:\mathfrak{Z}(G)\rightarrow \mathfrak{Z}(M)$$
en fait un $\mathfrak{Z}(G)$--module de type fini. Puisque $({^\omega{r}_{G^\natural}^{Q^\natural}})^*$ est un 
morphisme de $\mathfrak{Z}(G)$--modules (\ref{actions duales}) et que $\F_{\rm tr}(M^\natural,\omega)$ est un sous--$\mathfrak{Z}(M)$--module 
de $\G_{\Bbb C}(M^\natural,\omega)^*$ (lemme 2 de \ref{actions duales}), $\H_{J_M}^\star(M)$ est un 
sous--$\mathfrak{Z}(G)$--module de $\H_{J_M}(M)$. Par suite il existe un entier $l_0\geq 1$ tel que
$$
(f^J_a)^{*{l_0}}*\phi'\in \H^\star_{J_M}(M),\quad \phi'\in \H_{J_M}(M).
$$
Puisque $f^J_a$ est inversible dans $\H_{J_M}(M)$, on obtient le rŽsultat cherchŽ:
$$\H_{J_M}^\star(M)=\H_{J_M}(M).
$$
Cela achve la dŽmonstration du point (2). \end{proof}

\subsection{Terme constant suivant $P^\natural$; caractres des induites paraboliques.}\label{terme constant}
On a notŽ $G^1$\index{$G^1$, $K^\circ$} 
le sous--groupe de $G$ engendrŽ par les sous--groupes ouverts compacts. 
Il co\"{\i}ncide avec le sous--groupe
$$\bigcap_{\chi\in {\rm X}^*_F(G)} \ker \vert \chi \vert_F\subset G,
$$
o (rappel) ${\rm X}^*_F(G)$ dŽsigne 
le groupe des caractres algŽbriques de $G$ qui sont dŽfinis sur $F$.
Soit $K_\circ$ le stabilisateur dans $G^1$ d'un sommet spŽcial de l'appartement 
de l'immeuble (non Žtendu) de $G$ associŽ ˆ $A_\circ$,  C'est un sous--groupe 
compact maximal spŽcial de $G$, {\it en bonne position rapport ˆ $(P,M_P)$} pour tout $P\in \P(G)$: on a
$$
G=K_\circ P,\quad P\cap K_\circ = (P\cap M_P)(U_P\cap K_\circ).
$$
Notons qu'on ne suppose pas que $\theta(K_\circ)=K_\circ$ (d'ailleurs $K_\circ$ n'est en 
gŽnŽral pas $\theta$--stable, cf. la remarque (3) de \cite[5.2]{L2}). 

Pour $\phi\in \H^\natural$ et $P^\natural\in \P(G^\natural)$, on note 
${^\omega{\phi}_{P^\natural,K_\circ}}\in \H(M^\natural_{\sP})$ le terme constant de $\phi$ suivant 
$P^\natural$ relativement ˆ $(K_\circ,\omega)$, dŽfini par \cite[5.9]{L2}\index{${^\omega{\phi}_{P^\natural,K_\circ}}$}
$$
{^\omega{\phi}_{P^\natural,K_\circ}}(\delta)= \delta_{P^\natural}^{1/2}(\delta)\int\!\!\!\int_{U_P\times K_\circ}
\omega(k)\phi(k^{-1}\cdot \delta\cdot uk)du_P dk,\quad \delta\in M^\natural_{\sP},
$$
o les mesures de Haar $d\delta$, $du_P$, $dk$ sur $G^\natural$, $U_P$, $K_\circ$ ont ŽtŽ choisies 
de manire compatible. PrŽcisŽment, on peut supposer que toutes ces mesures sont celles 
{\it normalisŽes par $K_\circ$}, comme suit.

\begin{madefi}
{\rm 
Soit $K$ un sous--groupe ouvert compact de $G$. Pour tout sous--groupe fermŽ unimodulaire 
$H$ de $G$, on appelle {\it mesure de Haar sur $H$ normalisŽe par $K$} l'unique mesure de 
Haar $dh$ sur $H$ telle que ${\rm vol}(H\cap K,dh)=1$. De m\^eme, pour tout sous--espace 
fermŽ $H^\natural$ de $G^\natural$ de groupe sous--jacent $H$ unimodulaire, on appelle 
{\it mesure de Haar sur $H^\natural$ normalisŽe par $K$} l'image de la mesure de Haar $dh$ 
sur $H$ normalisŽe par $K$ par l'isomorphisme topologique $H\rightarrow H^\natural,\,h\mapsto
h\cdot \delta$ pour un (i.e. pour tout) $\delta\in H^\natural$.  
}
\end{madefi}

\begin{monhypo}
On suppose dŽsormais que toutes les mesures de Haar utilisŽes sont celles normalisŽes par $K_\circ$. 
\end{monhypo}

On a 
donc
$$
{\rm vol}(K_\circ\cdot \delta_1,d\delta)= {\rm vol}(K_\circ, dg)= {\rm vol}(K_\circ,dk)=1,
$$
et pour $P^\natural\in \P(G^\natural)$, on a
$$
{\rm vol}(M^\natural_{\sP}\cap (K_\circ \cdot \delta_1),d\delta_{\smash{M^\natural_{\sP}}})=
{\rm vol}(M_P\cap K_\circ,dm_P)=1
$$
et
$$
{\rm vol}(U_P\cap K_\circ ,du_P)=1.
$$

Soit $P^\natural\in \P(G^\natural)$. Pour toute $\omega$--reprŽsentation {\it admissible} $\Sigma$ de $M^\natural_{\sP}$ 
telle $\Sigma^\circ$ est admissible, on note $\Theta_\Sigma$ 
la distribution sur $M^\natural_{\sP}$ dŽfinie comme en \ref{carac} ˆ l'aide de la mesure $d\delta_{\smash{M^\natural_{\sP}}}$. 
On a la formule de descente \cite[5.9, thŽo.]{L2}:

\begin{mapropo}
Soit $P^\natural\in \P(G^\natural)$. Soit $\Sigma$ une $\omega$--reprŽsentation de $M^\natural_{\sP}$ telle que 
$\Sigma^\circ$ est admissible, 
et soit $\Pi={^\omega{i}_{P^\natural}^{G^\natural}(\Sigma)}$. Pour toute fonction $\phi\in \H^\natural$, on a l'ŽgalitŽ
$$
\Theta_\Pi(\phi)= \Theta_\Sigma({^\omega{\phi}_{P^\natural,K_\circ}})
$$
\end{mapropo}

\subsection{Le thŽorme principal sur la partie \og discrte \fg.}\label{ŽnoncŽ discret} 
Une forme linŽaire $\Phi$ sur $\G_{\Bbb C}(G^\natural,\omega)$ 
est dite \og discrte \fg si elle est nulle sur $\G_{{\Bbb C},{\rm ind}}(G^\natural,\omega)$. On note 
$\G_{\Bbb C}^{\rm dis}(G^\natural,\omega)^*$ l'espace des formes linŽaires discrtes sur $\G_{\Bbb C}(G^\natural,\omega)$, et 
l'on pose\index{$\F^{\rm dis}(G^\natural,\omega)$, $\F^{\rm dis}_{\rm tr}(G^\natural,\omega)$}
$$
\F^{\rm dis}(G^\natural,\omega)=\F(G^\natural,\omega)\cap \G_{\Bbb C}^{\rm dis}(G^\natural,\omega)^*,
$$
$$
\F^{\rm dis}_{\rm tr}(G^\natural,\omega)=\F_{\rm tr}(G^\natural,\omega)\cap \G_{\Bbb C}^{\rm dis}(G^\natural,\omega)^*.
$$

Une fonction $\phi\in \H^\natural$ est dite \og $\omega$--cuspidale \fg 
si pour tout $P^\natural\in \P(G^\natural)\smallsetminus \{G^\natural\}$, l'image de $\phi^{K_\circ,\omega}_{P^\natural}\in \H(M^\natural_{\sP})$ 
dans $\overline{\H}(M^\natural_{\sP},\omega)$ est nulle; o (rappel)
$$
\overline{\H}(M^\natural_{\sP},\omega)= \H(M^\natural_{\sP})/[\H(M^\natural_{\sP}),\H(M_P)]_\omega.
$$
On note $\smash{\overline{\H}}^{\rm dis}(G^\natural,\omega)$\index{$\smash{\overline{\H}}^{\rm dis}(G^\natural,\omega)$} 
le sous--espace de $\smash{\overline{\H}}(G^\natural,\omega)$ 
engendrŽ par les images dans $\smash{\overline{\H}}(G^\natural,\omega)$ des fonctions $\omega$--cuspidales.

\begin{montheo}
L'application $\H(G^\natural)\rightarrow \G_{\Bbb C}(G^\natural,\omega)^*,\,\phi\mapsto \Phi_\phi$ induit par restriction 
un isomorphisme de ${\Bbb C}$--espaces vectoriels
$$
\smash{\overline{\H}}^{\rm dis}(G^\natural,\omega)\rightarrow \F^{\rm dis}(G^\natural,\omega).
$$
\end{montheo}

D'aprs la proposition de \ref{terme constant}, la transformŽee de Fourier 
$\H^\natural\rightarrow \F(G^\natural,\omega),\,\phi \mapsto \Phi_\phi$ 
induit bien une application ${\Bbb C}$--linŽaire
$$
\smash{\overline{\H}}^{\rm dis}(G^\natural,\omega)\rightarrow \F^{\rm dis}(G^\natural,\omega).
$$

\begin{marema}
{\rm Si $G=M_\circ$, \cad si $G$ est compact modulo son centre, alors on a $\smash{\overline{\H}}^{\rm dis}(G^\natural,\omega)=
\overline{\H}(G^\natural,\omega)$ et $\F^{\rm dis}(G^\natural,\omega)=\F(G^\natural,\omega)$; en ce cas le thŽorme ci--dessus 
co\"{\i}ncide avec le thŽorme principal (\ref{ŽnoncŽ}). En gŽnŽral, si le thŽorme de \ref{ŽnoncŽ} est vrai pour tous les 
sous--espaces tordus $M^\natural_{\sP}$, $P^\natural\in \P(G^\natural)\smallsetminus \{G^\natural\}$, alors d'aprs la proposition 
de \ref{terme constant}, une fonction $\phi\in \H^\natural$ est $\omega$--cuspidale si et seulement si elle annule toutes les 
traces des reprŽsentations ${^\omega{i}_{P^\natural}^{G^\natural}}(\Sigma)$, $P^\natural\in \P(G^\natural)\smallsetminus 
\{G^\natural\}$, $\Sigma\in \G_{\Bbb C}(M^\natural_{\sP},\omega)$. On en dŽduit que si le thŽorme de \ref{ŽnoncŽ} est vrai en gŽnŽral 
--- donc en particulier pour tous les espaces tordus $M^\natural_{\sP}$, $P^\natural\in \P(G^\natural)$ ---, alors 
le thŽorme ci--dessus l'est aussi.\hfill $\blacksquare$
}
\end{marema}

\subsection{RŽduction du thŽorme principal (\ref{ŽnoncŽ}) au thŽorme de \ref{ŽnoncŽ discret}.}\label{rŽduction} Supposons 
dŽmontrŽ le thŽorme de \ref{ŽnoncŽ discret} (en gŽnŽral, \cad pour tout $G^\natural$) 
et dŽduisons--en le thŽorme de \ref{ŽnoncŽ}. D'aprs la remarque de \ref{ŽnoncŽ discret}, 
le thŽorme de \ref{ŽnoncŽ} est vrai pour $G^\natural = M^\natural_\circ$. Par rŽcurrence sur la dimension de $M_P$ pour 
$P\in \P(G^\natural)$, on peut donc supposer que le thŽorme de \ref{ŽnoncŽ} est vrai pour tout sous--espace $M^\natural_{\sP}$, 
$P^\natural\in \P(G^\natural)\smallsetminus\{G^\natural\}$.

Commen\c{c}ons par le thŽorme de Paley--Wiener, \cad la surjectivitŽ de l'application 
$\H^\natural\rightarrow \F(G^\natural,\omega),\, \phi\mapsto \Phi_\phi$. D'aprs la proposition de 
\ref{les T} pour $d=d(G)$, il existe des nombres rationnels $\lambda(P^\natural)$ pour 
$P^\natural\in \P(G^\natural)\smallsetminus \{G^\natural\}$ tels que le ${\Bbb C}$--endomorphisme 
de $\G_{\Bbb C}(G^\natural,\omega)$
$$
{\boldsymbol{A}}_d= {\rm id}+\sum_{P^\natural\neq G^\natural}\lambda(P^\natural){^\omega{T}_{P^\natural,{\Bbb C}}}
$$
vŽrifie
$$
\ker {\boldsymbol{A}}_d = \G_{{\Bbb C},{\rm ind}}(G^\natural,\omega),\quad {\boldsymbol{A}}_d\circ 
{\boldsymbol{A}}_d={\boldsymbol{A}}_d.
$$
Le morphisme adjoint
$$
{\boldsymbol{A}}_d^*= {\rm id} + \sum_{P^\natural\neq G^\natural} \lambda(P^\natural) ({^\omega{r}_{G^\natural}^{P^\natural}})^*
\circ ({^\omega{i}_{P^\natural}^{G^\natural}})^*
$$
envoie $\G_{\Bbb C}(G^\natural,\omega)^*$ dans $\G_{\Bbb C}^{\rm dis}(G^\natural,\omega)^*$. Soit $\Phi\in \F(G^\natural,\omega)$. 
Posons $\Phi'= {\boldsymbol{A}}_d^*(\Phi)\in \G_{\Bbb C}^{\rm dis}(G^\natural,\omega)^*$. Pour $P^\natural\in \P(G^\natural)
\smallsetminus \{G^\natural)$, d'aprs la proposition de \ref{morphismes duaux} et le thŽorme de Paley--Wiener pour 
$M^\natural_{\sP}$, on a
$$
({^\omega{i}_{P^\natural}^{G^\natural}})^*(\Phi)\in \F(M^\natural_{\sP},\omega)= \F_{\rm tr}(M^\natural_{\sP},\omega),
$$
Ensuite, d'aprs le point (2) de la proposition de \ref{morphismes duaux}, on a
$$
({^\omega{r}_{G^\natural}^{P^\natural}})^*
\circ ({^\omega{i}_{P^\natural}^{G^\natural}})^*(\Phi)\in \F_{\rm tr}(G^\natural,\omega)\subset \F(G^\natural,\omega).
$$
Par consŽquent $\Phi'$ appartient ˆ $\G_{\Bbb C}^{\rm dis}(G^\natural,\omega)^*\cap \F(G^\natural,\omega)= \F^{\rm dis}(G^\natural,\omega)$, et 
d'aprs le thŽorme de \ref{ŽnoncŽ discret}, il existe une fonction $\omega$--cuspidale $\phi'\in \H^\natural$ telle 
que $\Phi'= \Phi_{\phi'}$. D'autre part, on vient de voir qu'il existe une fonction $\phi''\in \H^\natural$ telle que
$$
({^\omega{r}_{G^\natural}^{P^\natural}})^*
\circ ({^\omega{i}_{P^\natural}^{G^\natural}})^*(\Phi)=\Phi_{\phi''}.
$$
On a donc
$$
\Phi = \Phi_{\phi'}-\Phi_{\phi''}=\Phi_{\phi'-\phi''}\in \F_{\rm tr}(G^\natural,\omega).
$$

Passons au thŽorme de densitŽ spectrale, \cad ˆ l'injectivitŽ de l'application 
$\smash{\overline{\H}}^\natural_\omega\rightarrow \F(G^\natural,\omega),\, \phi\mapsto \Phi_\phi$. 
Soit $\phi\in \H^\natural$ une fonction telle que $\Phi_\phi=0$. Pour $P^\natural\in \P(G^\natural)$ tel que $P^\natural\neq G^\natural$, 
d'aprs la proposition de \ref{terme constant} et le thŽorme de densitŽ spectrale 
pour $M^\natural_{\sP}$, le terme constant $\phi_{P^\natural}= \phi_{P^\natural}^{K_\circ,\omega}$ 
appartient au sous--espace $[\H(M^\natural_{\sP}),\H(M_P)]_\omega$ de $\H(M^\natural_{\sP})$. Par consŽquent 
la fonction $\phi$ est $\omega$--cuspidale, et d'aprs le thŽorme de \ref{ŽnoncŽ discret}, elle est dans $[\H^\natural,\H]_\omega$, 
ce qu'il fallait dŽmontrer.

\begin{marema}
{\rm D'aprs ce qui prŽcde, par rŽcurrence sur la dimension de $G$, la surjectivitŽ de l'application du 
thŽorme de \ref{ŽnoncŽ} est impliquŽe par la surjectivitŽ de celle du thŽorme de \ref{ŽnoncŽ discret}; idem 
pour l'injectivitŽ. Pour dŽmontrer le thŽorme principal (\ref{ŽnoncŽ}), il suffit donc de dŽmontrer sa 
variante sur la partie discrte (\ref{ŽnoncŽ discret}). \hfill $\blacksquare$
}
\end{marema}

\section{Le thŽorme de Paley--Wiener sur la partie discrte}\label{PW discret}

D'aprs \ref{rŽduction}, le thŽorme de Paley--Wiener 
est impliquŽ par la surjectivitŽ de l'application 
du thŽorme de \ref{ŽnoncŽ discret}. C'est cette dernire que l'on Žtablit dans 
cette section \ref{PW discret}.

\subsection{Support cuspidal des reprŽsentations discrtes} 
Une $\omega$--reprŽsentation $G$--irrŽductible $\Pi$ de $G^\natural$ est dite \og discrte \fg si son 
image dans $\G_{\Bbb C}^{\rm dis}(G^\natural,\omega)$ 
n'est pas nulle; o (rappel) $\G_{\Bbb C}^{\rm dis}(G^\natural,\omega)=\G_{\Bbb C}(G^\natural,\omega)/\G_{{\Bbb C},{\rm ind}}(G^\natural,\omega)$. 
On note ${\rm Irr}_0^{\rm dis}(G^\natural,\omega)$\index{${\rm Irr}_0^{\rm dis}(G^\natural,\omega)$, 
${\rm Irr}_{\Bbb C}^{\rm dis}(G^\natural, \omega)$} le sous--ensemble de ${\rm Irr}_0(G^\natural,\omega)$ 
formŽ des $\omega$--reprŽsentations discrtes. Il s'identifie ˆ un sous--ensemble de ${\rm Irr}_{\Bbb C}(G^\natural,\omega)$, 
que l'on note ${\rm Irr}_{\Bbb C}^{\rm dis}(G^\natural, \omega)$.

\begin{monlem}\label{supp cusp des discrtes}
Soit $\Pi$ une $\omega$--repr\'esentation $G$--irrŽductible discr\`ete de $G^\natural$. Il existe une 
$\omega_{\rm u}$-repr\'esentation $G$--irr\'eductible temp\'er\'ee $\Pi'$ de $G^\natural$ et un 
ŽlŽment $\Psi$ de $\mathfrak{P}^{\vert \omega \vert}_{\Bbb C}(G^\natural)$ tels que $\theta_{G^\natural}(\Pi)= \theta_{G^\natural}(\Psi\cdot \Pi')$.
\end{monlem}

\begin{proof} Soit $(P^\natural,\Sigma,\Xi)$ un triplet de Langlands associ\'e \`a $\Pi$. 
Posons $\sigma=\Sigma^\circ$, $\xi=\Xi^\circ$, et soit $(M_{P'},\sigma')$ 
une paire cuspidale standard de $G$ telle que $P'\subset P$ et $\xi\cdot \sigma$ 
est isomorphe \`a un sous-quotient (irr\'eductible) de l'induite parabolique ${i}_{P'}^P(\sigma')$. 
D'aprs la dŽmonstration du lemme 1 de \ref{base langlands}, on a une dŽcomposition
$$
\Pi \equiv \sum_{i=1}^m \widetilde{\Pi}_{\mu_i}\quad({\rm mod}\;\G_{>0}(G^\natural,\omega))
$$
o les $\mu_i$ sont des triplets de Langlands pour $(G^\natural,\omega)$ tels que
$\theta_G(\Pi_{\mu_i}^\circ)=[M_{P'},\sigma']$. Rappelons que si $\mu=(P^\natural,\Sigma,\Xi)$ est un triplet de 
Langlands pour $(G^\natural,\omega)$, on a posŽ $\widetilde{\Pi}_\mu = {^\omega{i_{P^\natural}^{G^\natural}}}(\Xi\cdot\Pi)$. 
Puisque $\Pi$ est discrte, l'un au moins de ces triplets, disons $\mu_{i_0}$ est de la forme 
$(G^\natural,\Pi',\Psi)$, et l'on a $\theta_{G^\natural}(\Psi\cdot \Pi')= \theta_G(\Pi_{\mu_{i_0}}^\circ)=\theta_{G^\natural}(\Pi)$.
\end{proof}

\subsection{Un rŽsultat de finitude}\label{finitude}
Pour $\mathfrak{s}\in \mathfrak{B}(G)$, 
on note $\Theta_{G^\natural,\omega}^{\rm dis}(\mathfrak{s})$\index{$\Theta_{G^\natural,\omega}^{\rm dis}(\mathfrak{s})$} 
le sous--ensemble de $\Theta_{G^\natural,\omega}(\mathfrak{s})$ 
form\'e des $\theta_{G^\natural}(\Pi)$ pour un $\Pi\in {\rm Irr}_0^{\rm dis}(G^\natural,\omega)$. 

\begin{mapropo}
Soit $\mathfrak{s}\in \mathfrak{B}(G)$. L'ensemble $\Theta_{G^\natural,\omega}^{\rm dis}(\mathfrak{s})$, s'il est non vide, 
est union finie de $\mathfrak{P}(G^\natural)$--orbites.
\end{mapropo}

\begin{proof}
On reprend la m\'ethode de \cite{BDK} en l'adaptant au cas tordu, 
comme le fait Flicker dans \cite{F}. 
Cette m\'ethode consiste \`a v\'erifier que $\Theta_{G^\natural,\omega}^{\rm dis}(\mathfrak{s})$ est une partie {\it constructible} de $\Theta(\mathfrak{s})$, c'est-\`a-dire 
une union finie de parties localement ferm\'es (pour la topologie de Zariski). Admettons pour l'instant ce rŽsultat --- il sera dŽmontrŽ en \ref{constructible} --- 
et dŽduisons--en la proposition. 

Notons 
$\omega^+$ le caract\`ere $\overline{\omega^{-1}}$ de $G$. 
L'application $\Pi\mapsto \Pi^+=\overline{\check{\Pi}}$ est une bijection de ${\rm Irr}_0(G^\natural,\omega)$ sur 
${\rm Irr}_0(G^\natural,\omega^+)$, v\'erifiant $(\Pi^+)^+=\Pi$. Elle 
commute au foncteur d'oubli $\Pi\mapsto \Pi^\circ$: en notant $\pi\mapsto \pi^+$ l'involution de ${\rm Irr}(G)$ 
d\'efinie de la m\^eme mani\`ere, on a $(\Pi^+)^\circ = (\Pi^\circ)^+$. De plus, l'involution $\pi\mapsto \pi^+$ de 
${\rm Irr}(G)$ induit par lin\'earit\'e 
une involution de $\EuScript{G}(G)$, qui commute au morphisme induction parabolique ${i}_P^G$, 
$P\in \EuScript{P}(G)$. D'o\`u une involution \og $+$ \fg sur la vari\'et\'e alg\'ebrique complexe 
$\Theta(\mathfrak{s})$, qui commute \`a l'application support cuspidal $\theta_G$: si $\theta_G(\pi)=[M,\rho]\in \Theta(\mathfrak{s})$ 
pour une paire cuspidale standard $(M,\rho)$ de $G$, on a
$$
\theta_G(\pi^+)= [M,\rho^+] =\theta_G(\pi)^+.
$$
On a aussi une involution $\psi\mapsto \psi^+$ sur $\mathfrak{P}(M)$, 
donn\'ee par $\psi^+=\overline{\psi^{-1}}$. Les involutions $+$ sur $\mathfrak{P}(M)$ et $\Theta(\mathfrak{s})$ 
sont anti--alg\'ebriques, et compatibles: 
pour $[M,\rho]\in \Theta(\mathfrak{s})$ et $\psi\in \mathfrak{P}(M)$, on a
$$
(\psi\cdot [M,\rho])^+= [M, (\psi\rho)^+]=\psi^+\cdot [M,\rho]^+.
$$
Supposons que $\Theta_{G^\natural,\omega}^{\rm dis}(\mathfrak{s})$ est non vide. D'apr\`es le lemme de \ref{supp cusp des discrtes}, 
tout \'el\'ement $x$ de $\Theta_{G^\natural,\omega}^{\rm dis}(\mathfrak{s})$ est de la forme $x=\theta_G(\psi \Pi'^\circ)$ pour une 
$\omega_{\rm u}$--repr\'esentation $G$--irr\'eductible temp\'er\'ee $\Pi'$ de $G^\natural$ 
et un ŽlŽment $\psi$ de $\mathfrak{P}(G^\natural)$. Puisque $\pi'=\Pi'^\circ$ est temp\'er\'ee, donc 
en particulier unitaire, on a $\pi'^+=\pi'$ et
$$
x^+ = \theta_G(\psi^+\pi') = \psi^+\psi^{-1}\cdot x
$$
appartient ˆ $\mathfrak{P}(G^\natural)\cdot x$. 
Ici $\mathfrak{P}(G^\natural)\;(=\mathfrak{P}(G)^\theta)$ op\`ere sur $\Theta(\mathfrak{s})$ via la restriction des caract\`eres de $G$ \`a $M$. 
L'involution $+$ sur $\Theta(\mathfrak{s})$ induit par passage au quotient une involution anti--alg\'ebrique 
sur la vari\'et\'e alg\'ebrique quotient $\mathfrak{X}= \Theta(\mathfrak{s})/\mathfrak{P}(G^\natural)$, que l'on note encore \og $+$ \fg. 
D'apr\`es le calcul ci--dessus, cette involution $+$ sur $\mathfrak{X}$ fixe les points du sous--ensemble 
$\Theta_{G^\natural,\omega}^{\rm dis}(\mathfrak{s})/\mathfrak{P}(G^\natural)$. Ce dernier est constructible, 
donc fini, ce qu'il fallait d\'emontrer.
\end{proof}

\subsection{DŽcomposition des fonctions rŽgulires}\label{dŽcomposition} 
On est donc ramenŽ ˆ vŽrifier que pour $\mathfrak{s}\in \mathfrak{B}(G)$, 
l'ensemble $\Theta_{G^\natural,\omega}^{\rm dis}(\mathfrak{s})$ est une partie constructible de 
$\Theta(\mathfrak{s})$. Posons $\Theta= \Theta(\mathfrak{s})$ et 
$\Theta'= \Theta_{G^\natural,\omega}^{\rm dis}(\mathfrak{s})$. Comme dans \cite[5.1]{BDK}, il 
suffit de montrer que la paire $\Theta'\subset \Theta$ satisfait au critre suivant (loc.~cit., p.~187):

\begin{enumerate}
\item[(*)]Pour toute sous-vari\'et\'e localement ferm\'ee $\mathfrak{X}$ 
de $\Theta$, 
il existe un morphisme dominant \'etale $\phi: \mathfrak{Y}\rightarrow \mathfrak{X}$ tel que, notant $(\mathfrak{X}\cap \Theta')^\C$ le 
complŽmentaire de $\mathfrak{X}\cap \Theta'$ dans $\mathfrak{X}$, l'ensemble 
$\phi^{-1}((\mathfrak{X}\cap \Theta')^\C)$ est 
vide ou \'egal \`a $\mathfrak{Y}$ tout entier.
\end{enumerate}

\begin{madefi}
{\rm 
Soit $\mathfrak{X}$ une vari\'et\'e alg\'ebrique affine complexe, d'anneau de fonctions r\'eguli\`eres 
$B={\Bbb C}[\mathfrak{X}]$. Une application $\nu:\mathfrak{X}\rightarrow \G(G^\natural,\omega)$ 
est dite {\it rŽgulire} si elle est de la forme $x\mapsto \Pi_x$ 
pour un $(G^\natural,\omega,B)$-module 
admissible $(\Pi,V)$, o $\Pi_x$ dŽsigne la (semisimplifiŽe de la) localisation 
de $\Pi$ en $x$ --- cf. \ref{modules admissibles}. Une application rŽgulire 
$\nu:\mathfrak{X}\rightarrow \G(G^\natural,\omega)$ est dite 
{\it irrŽductible} si $\nu(\mathfrak{X})\subset {\rm Irr}(G^\natural,\omega)$, et 
{\it $G$--irrŽductible} si $\nu(\mathfrak{X})\subset{\rm Irr}_0(G^\natural,\omega)$. 
Deux applications rŽgulires (irrŽductibles) 
$\nu,\,\nu':\mathfrak{X}\rightarrow {\rm Irr}(G^\natural,\omega)$ sont dites {\it disjointes} si $\nu(x)\neq \nu'(x)$ 
pour tout $x\in \mathfrak{X}$, et elles sont dites ${\Bbb C}^\times$--disjointes si $\nu(x)\neq \lambda\cdot\nu'(x)$ 
pour tout $x\in \mathfrak{X}$ et tout $\lambda\in {\Bbb C}^\times$.}
\end{madefi}

\begin{marema}
{\rm 
Pour chaque $x\in \mathfrak{X}$, la localisation $\Pi_x$ de $\Pi$ en $x$ est une $\omega$--reprŽsentation de $G^\natural$ de 
longueur finie, et puisque la reprŽsentation $\Pi_x^\circ$ de $G$ sous--jacente est elle aussi de longueur finie, pour tout 
sous--quotient irrŽductible $\Pi'$ de $\Pi_x$, on a $s(\Pi')<+\infty$. Choisissons une suite de Jordan-H\"{o}lder de $\Pi$:
$$
0=\Pi_{x,0}\subset \Pi_{x,1}\subset \cdots \subset \Pi_{x,n}=\Pi_x.
$$
Pour $i=1,\ldots ,n$, $\Pi_{x,i}$ est une sous--$\omega$--reprŽsentation de $\Pi$, et le sous--quotient 
$\Pi_{x,i}/\Pi_{x,i-1}$ de $\Pi_x$ est irrŽductible. La semisimplifiŽe de $\Pi_x$ est par dŽfinition 
la somme sur $i$ des classes d'isomorphisme des $\Pi_{x,i}/\Pi_{x,i-1}$. \hfill $\blacksquare$
}
\end{marema}

Le lemme suivant est une variante tordue des constructions  de loc.~cit. (step 1 -- step 2). On aurait pu 
se contenter de la version non tordue (cf. la proposition de \ref{consŽquence}), mais comme ces constructions sont au coeur du raisonnement, et qu'il nous 
faut de toutes fa\c{c}ons les reprendre en dŽtail, on prŽfre le faire dans le cadre tordu qui nous 
intŽresse ici.

\begin{monlem}
Soit $\nu:\mathfrak{X}\rightarrow \G(G^\natural,\omega)$  une application r\'eguli\`ere. 
Il existe un morphisme domi\-nant \'etale $\phi:\mathfrak{Y}\rightarrow \mathfrak{X}$, des 
applications r\'eguli\`eres 
$\mu_1,\ldots ,\mu_m:\mathfrak{Y}\rightarrow{\rm Irr}(G^\natural,\omega)$ 
deux--\`a--deux disjointes, et des entiers $a_1,\ldots ,a_m>0$ tels que $\nu\circ \phi = \sum_{i=1}^m a_i\mu_i$. 
\end{monlem}

\begin{proof} Posons $B={\Bbb C}[\mathfrak{X}]$, et soit $(\Pi,V)$ un $(G^\natural,\omega,B)$--module admissible 
tel que $\nu(x)=\Pi_x$. Choisissons un sous--groupe ouvert compact $J$ de 
$G$ tel que $V^J$ engendre $V$ comme $\H$--module, i.e. tel que $\Pi^\circ(G)(V^J)=V$. D'aprs  
\ref{bons sgoc}, on peut supposer que $J$ est \og bon \fg, $\theta(J)=J$ et $\omega\vert_J$. C'est donc un ŽlŽment de 
${\boldsymbol{J}}_{G^\natural,\omega}(G)$, et $J^\natural=J\cdot \delta_1$ est un Žlement de ${\boldsymbol{J}}(G^\natural,\omega)$. 
D'aprs \ref{bons seoc}, l'Žtude du $(\Hn,\omega)$--module non dŽgŽnŽrŽ $V$ se 
ramne ˆ celle du $(\Hn_{\sJ},\omega)$--module non dŽgŽnŽrŽ $V^J=V^{J^\natural}$. En particulier, 
tout sous--quotient 
irrŽductible $\Pi'$ de $\Pi$ vŽrifie $\Pi'^{J^\natural}\neq 0$.

Quitte ˆ remplacer la variŽtŽ $\mathfrak{X}$ par l'une de ses composantes irrŽductibles, on peut la 
supposer irrŽductible. Soit ${\Bbb K}={\Bbb C}(\mathfrak{X})$ le corps des fractions de $B$, et soit 
$\overline{\Bbb K}$ une cl\^oture algŽbrique de ${\Bbb K}$. Le $(\Hn_{\sJ},\omega)$--module 
$V^J$ est aussi un $B$--module de type fini. Par consŽquent $W= {\Bbb K}\otimes_B V^J$ est un ${\Bbb K}$--espace vectoriel 
de dimension finie, $\overline{W}= \overline{\Bbb K}\otimes_{\Bbb K}W$ est un $\overline{\Bbb K}$--espace vectoriel 
de dimension finie, et il existe une suite de $\overline{\Bbb K}$--espaces vectoriels
$$
0= \overline{W}_0\subset \overline{W}_1\subset \cdots \subset\overline{W}_n= \overline{W}
$$
telle que pour $i=1,\ldots ,n$:

\begin{itemize}
\item $\overline{W}_i$ est un sous--$(\Hn_{\sJ},\omega)$--module de $\overline{W}$; 
\item $\overline{W}_i/\overline{W}_{i-1}$ est un $(\Hn_{\sJ},\omega)$--module simple 
sur $\overline{\Bbb K}$. 
\end{itemize}

\ni De plus, 
il existe une sous--extension finie ${\Bbb K}'\!/{\Bbb K}$ de $\overline{\Bbb K}/{\Bbb K}$, un sous--${\Bbb K}'$--espace vectoriel 
$W'$ de $\overline{W}$ de dimension finie, et une suite de ${\Bbb K}'$--espaces 
vectoriels
$$
0=W'_0\subset W'_1\subset \cdots \subset W'_n=W'
$$
telle que pour $i=1,\ldots ,n$:

\begin{itemize}
\item $W'_i$ est un sous--$(\Hn_{\sJ},\omega)$--module de $W'$; 
\item $\overline{W}_i=\overline{\Bbb K}\otimes_{{\Bbb K}'}W'_i$.
\end{itemize}

\ni Ainsi pour $i=1,\ldots ,n$, le quotient $X'_i=W'_i/W'_{i-1}$ est un $(\Hn_{\sJ},\omega)$--module simple 
sur ${\Bbb K}'$. D'aprs la remarque 2 de \ref{H modules}, $X'_i$ est un $\H_J$--module semisimple sur ${\Bbb K}'$. 
PrŽcisŽment, on choisit un sous--$\H_J$--module simple (sur ${\Bbb K}'$) $X'_{i,0}$ de $X'_i$, et pour chaque entier $k\geq 1$, 
on note $X'_{i,k}$ le sous--$\H_J$--module simple de $X'_i$ dŽfini par $X'_{i,k}=e_{J^\natural}\cdot X'_{i,k-1}$. 
Alors il existe un plus petit entier entier $s=s(i)\geq 1$ tel que 
$$
X'_i= X'_{i,0}\oplus X'_{i,1}\oplus \cdots \oplus X'_{i,s-1}.
$$
D'aprs le thŽorme de Burnside, $\H_J$ engendre 
l'espace ${\rm End}_{{\Bbb K}'}(X'_{i,k})$ sur ${\Bbb K}'$.

Le corps ${\Bbb K}'$ est une extension finie sŽparable de ${\Bbb K}$, par suite il existe une 
variŽtŽ algŽbrique affine irrŽductible $\mathfrak{Y}'$ Žtale sur $\mathfrak{X}$ (\cad un morphisme 
dominant Žtale $\mathfrak{Y}'\rightarrow \mathfrak{X}$) telle que ${\Bbb K}'={\Bbb C}(\mathfrak{Y}')$. 
Notons $B'={\Bbb C}[\mathfrak{Y}']$ l'algbre affine de $\mathfrak{Y}'$. 

Soit $i\in \{1,\ldots ,n\}$. Choisissons une ${\Bbb K}'$--base de $X'_{i,0}$, et notons 
$Y'_{i,0}$ le sous--$B'$--module de $X'_{i,0}$ engendrŽ par cette base. On a donc ${\Bbb K}'\otimes_{B'}Y'_{i,0}=X'_{i,0}$. 
Rappelons que $\H_J$ est une ${\Bbb C}$--algbre de type fini. Puisque
$$
{\rm End}_{{\Bbb K}'}(X'_{i,0})= {\Bbb K}'\otimes_{B'}{\rm End}_{B'}(Y'_{i,0})
$$
et que $\H_J$ engendre ${\rm End}_{{\Bbb K}'}(X'_{i,0})$ sur ${\Bbb K}'$, il existe un 
ouvert dense $\mathfrak{Y}_i$ de $\mathfrak{Y}'$ (ce qui revient ˆ inverser certains ŽlŽments de $B'$) tel que, 
notant $B_i={\Bbb C}[\mathfrak{Y}_i]$ l'algbre affine de $\mathfrak{Y}_i$, on a:

\begin{itemize}
\item $Y_{i,0}= B_i\otimes_{B'} Y'_{i,0}$ est libre (de type fini) sur $B_i$;
\item $\H_J\subset {\rm End}_{B_i}(Y_{i,0})$;
\item $\H_J$ engendre ${\rm End}_{B_i}(Y_{i,0})$ sur $B_i$.
\end{itemize}

\ni 
Pour $k=1,\ldots ,s(i)-1$, posons $Y_{i,k}=e_{J^\natural}\cdot Y_{i,0}$. C'est un sous--$(\H_J\otimes_{\Bbb C}B_i)$--module 
de $X'_{i,k}$, qui vŽrifie ${\Bbb K}'\otimes_{B_i}Y_{i,k}= X'_{i,k}$. On obtient ainsi un 
sous--$(\Hn_{\sJ},\omega)$--module (non dŽgŽnŽrŽ)
$$
Y_i=Y_{i,0}\oplus \cdots \oplus Y_{i,s(i)-1}
$$
de $X'_i$, qui est aussi un $B_i$--module libre de type fini, tel que $\Hn_{\sJ}$ engendre ${\rm End}_{B_i}(Y_i)$ sur 
$B_i$. Pour $y\in \mathfrak{Y}_i$, le localisŽ $(Y_{i,0})_y$ de 
$Y_{i,0}$ en $y$ est un $\H_J$--module simple, et le localisŽ $(Y_i)_y$ de $Y_i$ en $y$ est un $(\Hn_{\sJ},\omega)$--module 
simple. Notons $\Pi_i$ la $\omega$--reprŽsentation de $G^\natural$ d'espace $\H*e_J\otimes_{\H_J}Y_i$ 
correspondant au $(\Hn_J,\omega)$--module non dŽgŽnŽrŽ $Y_i$ (\ref{bons seoc}). C'est un $(G^\natural,\omega,B_i)$--module admissible 
tel que 
pour $y\in \mathfrak{Y}_i$, la localisation $\Pi_{i,y}$ de $\Pi_i$ en $y$ est irrŽductible (elle est $G$--irrŽductible si 
et seulement si $s(i)=1$).

Notons $\mathfrak{Y}$ l'intersection des $\mathfrak{Y}_i$, $i=1,\ldots ,n$. 
C'est un ouvert dense de $\mathfrak{Y}'$, Žtale sur $\mathfrak{X}$. La composition des morphismes $\mathfrak{Y}'\rightarrow \mathfrak{X}$ 
et $\mathfrak{Y}\hookrightarrow \mathfrak{Y}'$ est un morphisme dominant Žtale $\nu:\mathfrak{Y}\rightarrow \mathfrak{X}$. 
Pour $y\in \mathfrak{Y}$, on a l'ŽgalitŽ dans $\G(G^\natural,\omega)$: 
$$\Pi_{\nu(y)}=\sum_{i=1}^n\Pi_{i,y}.$$ 
En regroupant les indices $i$ tels que les fonctions 
$\mathfrak{Y}\rightarrow {\rm Irr}(G^\natural,\omega),\,y\mapsto \Pi_{i,x}$ 
sont Žgales, l'ŽgalitŽ ci--dessus s'Žcrit
$$
\Pi_{\nu(y)}= \sum_{i=1}^m a_i\mu_i(y)
$$
pour des applications rŽgulires $\mu_i:\mathfrak{Y}\rightarrow {\rm Irr}(G^\natural,\omega)$ 
deux--ˆ--deux distinctes 
et des entiers $a_1,\ldots ,a_m>0$. Pour $i\neq j$, l'ensemble des 
$y\in \mathfrak{Y}$ tels que $\mu_i(y)=\mu_j(y)$ est fermŽ dans $\mathfrak{Y}$ pour la topologie de Zariski. 
Quitte ˆ remplacer $\mathfrak{Y}$ par un ouvert plus petit, 
on peut supposer les fonctions $\mu_i$ deux--ˆ--deux disjointes. Cela achve la dŽmonstration 
du lemme. 
\end{proof}

\subsection{Une consŽquence du lemme de dŽcomposition}\label{consŽquence}
La dŽfinition de \ref{dŽcomposition} s'applique bien s\^ur au cas non tordu: 
si $\mathfrak{X}$ est une vari\'et\'e alg\'ebrique affine complexe, 
d'anneau de fonctions r\'eguli\`eres $B={\Bbb C}[\mathfrak{X}]$, une application 
$\nu:\mathfrak{X}\rightarrow \G(G)$ est dite {\it rŽgulire} si elle est de la forme 
$x\mapsto \pi_x$ pour un $(G,B)$--module admissible $(\pi,V)$, o $\pi_x$ 
est la (semisimplifiŽe de la) localisation de $\pi$ en $x$. Toute application rŽgulire $\nu:\mathfrak{X}\rightarrow \G(G^\natural,\omega)$ 
induit une application rŽgulire $\nu^\circ:\mathfrak{X}\rightarrow \G(G)$, donnŽe par 
$\nu^\circ(x)=\nu(x)^\circ$, $x\in \mathfrak{X}$.
Notons que deux application rŽgulires ($G$--irrŽductibles) 
$\nu,\,\nu':\mathfrak{X}\rightarrow {\rm Irr}_0(G^\natural,\omega)$ sont 
${\Bbb C}^\times$--disjointes si et seulement si les applications 
rŽgulires sous--jacentes $\nu^\circ,\,\nu'^\circ:\mathfrak{X}\rightarrow {\rm Irr}(G)$ sont disjointes.

Rappelons que le foncteur d'oubli $\Pi\mapsto \Pi^\circ$ induit une 
application injective
$${\rm Irr}_0(G^\natural,\omega)/{\Bbb C}^\times \hookrightarrow {\rm Irr}(G)$$
d'image le sous--ensemble ${\rm Irr}_1(G)={\rm Irr}_{G^\natural,\omega}(G)$\index{${\rm Irr}_1(G)={\rm Irr}_{G^\natural,\omega}(G)$} 
de ${\rm Irr}(G)$ 
formŽ des $\pi$ tels que $\pi(1)=\pi$. 
Soit $\G_1(G)$\index{$\G_1(G)$} le sous-groupe de $\EuScript{G}(G)$ engendr\'e par 
${\rm Irr}_1(G)$ --- c'est aussi un quotient de $\EuScript{G}(G)$ --- et
$$q_1:\EuScript{G}(G)\rightarrow \EuScript{G}_1(G)$$
la projection canonique. Par dŽfinition, le foncteur d'oubli $\Pi\mapsto \Pi^\circ$ induit 
un morphisme de groupes surjectif $\G_0(G^\natural,\omega)\rightarrow \G_1(G)$,
de noyau le sous--goupe de $\G_0(G^\natural,\omega)$ engendrŽ par les $\Pi-\lambda\cdot \Pi$ 
pour $\Pi\in {\rm Irr}_0(G^\natural,\omega)$ et $\lambda\in {\Bbb C}^\times$.

\begin{mapropo}
Soit $\nu:\mathfrak{X}\rightarrow \G(G)$ une application rŽgulire. 
Il existe un morphisme dominant Žtale $\phi:\mathfrak{Y}\rightarrow\mathfrak{X}$, des applications rŽgulires 
$\mu_1,\ldots ,\mu_m:\mathfrak{Y}\rightarrow {\rm Irr}_0(G^\natural,\omega)$ 
deux--ˆ--deux 
${\Bbb C}^\times$--disjointes, et des 
entiers $a_1,\ldots ,a_m>0$, tels que $q_1\circ \nu \circ \phi= \sum_{i=1}^m a_i\mu_i^\circ$.
\end{mapropo}

\begin{proof}
Soit $B={\Bbb C}[\mathfrak{X}]$. Par dŽfinition 
l'application $\nu$ est de la forme $x\mapsto \pi_x$ pour un $(G,B)$--module admissible $(\pi,V)$. 
On choisit un bon sous--groupe ouvert compact $J$ de $G$ comme dans la dŽmonstration du lemme de \ref{dŽcomposition}, i.e. tel que 
$V^J$ engendre $V$ comme $\H$--module, $J$ est en \og bon \fg, $\theta(J)=J$ et 
$\omega\vert_J=1$. 
On peut aussi supposer $\mathfrak{X}$ irrŽductible. 
D'apr\`es le lemme de \ref{dŽcomposition}, il existe un morphisme dominant \'etale $\phi':\mathfrak{Y}'\rightarrow \mathfrak{X}$, 
des applications r\'eguli\`eres 
$\mu'_1,\ldots ,\mu'_n:\mathfrak{Y}'\rightarrow {\rm Irr}(G)$ deux--\`a--deux 
disjointes, et des entiers $a'_1,\ldots ,a'_n>0$, tels que $\nu\circ \phi' =\sum_{i=1}^n a'_i\mu'_i$. 
PrŽcisŽment, $\mathfrak{Y}'$ est une variŽtŽ algŽbrique affine complexe irrŽductible, 
d'anneau de fonctions rŽgulires $B'={\Bbb C}[\mathfrak{Y}']$, et pour $i=1,\ldots ,n$, on a 
$\mu'_i(y)=\pi_{i,y}$ pour un $(G,B')$-module admissible $(\pi_i,V_i)$ tel que:

\begin{itemize}
\item $W_i=(V_i)^J$ engendre $V_i$ comme $G$-module,
\item $\EuScript{H}_J$ engendre ${\rm End}_{B'}(W_i)$ comme $B'$-module.
\end{itemize}

\noindent Pour chaque $y\in \mathfrak{Y}'$, le $\EuScript{H}_J$-module $W_{i,y}=(V_{i,y})^J$ est simple, et pour $i\neq j$, 
les $\EuScript{H}_J$--modules simples $W_{i,y}$ et $W_{j,y}$ sont non isomorphes.

Pour $i=1,\ldots ,n$, notons $(\pi_i(1),V_i(1))$ le $(G,B')$--module admissible d\'efini par 
$V_i(1)=V_i$ et $\pi_i(1)=\omega^{-1}(\pi_i)^\theta$. Pour $y\in \mathfrak{Y}'$, la localisation $\pi_i(1)_y$ de 
$\pi_i(1)$ en $y$ est donn\'ee par
$$
\pi_i(1)_y= \omega^{-1}(\pi_{i,y})^\theta =\pi_{i,y}(1).
$$
L'application
$$
\mu'_i(1):\mathfrak{Y}'\rightarrow \EuScript{G}(G),\, y\mapsto \pi_i(1)_y
$$ est encore r\'eguli\`ere. Notons $W_i(1)$ le $(\EuScript{H}_J\otimes_{\Bbb C} B')$--module $V_i(1)^J=(V_i)^J$ d\'eduit de $\pi_i(1)$ 
par passage aux points fixes sous $J$. Soit ${\Bbb K}'={\Bbb C}(\mathfrak{Y}')$ le corps des fractions de $B'$. 
Pour $i=1,\ldots ,n$, les $\H_J$-modules $W_{i,{\Bbb K}'}={\Bbb K}'\otimes_{B'}W_i$ et 
${\Bbb K}'\otimes_{B'}W_i(1)$ sont simples (sur ${\Bbb K}'$), et l'on note $I_0$ le sous--ensemble de $\{1,\ldots, n\}$ formŽ des indices 
$i$ tels qu'ils sont isomorphes. D'apr\`es \cite[8.2]{L2}, pour $i\in I_0$, il existe un ${\Bbb K}'$-automorphisme 
$A_{i,{\Bbb K}'}$ de $W_{i,{\Bbb K}'}$ tel que
$$
A_{i,{\Bbb K}'}(\omega f\cdot v)= {^{\theta}\!f}\cdot A_{i,{\Bbb K}'}(v),\quad f\in \EuScript{H}_J,\, v\in W_{i,{\Bbb K}'},
$$
o l'on a posŽ ${^{\theta}\!f}= f\circ \theta^{-1}$. 
Puisque $W_{i,{\Bbb K}'}$ est de dimension finie sur ${\Bbb K}'$, il existe un ouvert $\mathfrak{Y}_i$ de $\mathfrak{Y}'$ tel que 
$A_{i,{\Bbb K}'}$ induit par restriction un ${\Bbb C}[\mathfrak{Y}_i]$--automorphisme de $W_i$, disons $A_i$. 
Cela munit $W_i$ d'une structure de $(\H^\natural_J,\omega)$--module non dŽgŽnŽrŽ, l'action de $\H^\natural_J$ 
commutant ˆ celle de ${\Bbb C}[\mathfrak{Y}_i]$. Soit $\Pi_i$ la $\omega$--reprŽsentation de $G^\natural$ 
d'espace $V_i=(\H*e_J)\otimes_{\H_J}W_i$ correspondant au $(\Hn_{\sJ},\omega)$--module $W_i$ 
(\ref{bons seoc}). Notons 
que $A_i$ co\"{\i}ncide avec la restriction de $\Pi_i(\delta_1)$ ˆ $W_i = (V_i)^J$. 
Par construction, $(\Pi_i,V_i)$ est un $(G^\natural,\omega,{\Bbb C}[\mathfrak{Y}_i])$--module admissible tel 
que pour $y\in \mathfrak{Y}_i$, la localisation $\Pi_{i,y}$ de $\Pi_i$ en $y$ est une $\omega$--reprŽsentation 
$G$--irrŽductible de $G^\natural$.

L'intersection $\mathfrak{Y}=\bigcap_{i\in I_0}\mathfrak{Y}_i$ est un ouvert (dense) de $\mathfrak{Y}'$, et 
pour $i\in I_0$, l'application rŽgulire
$$\mathfrak{Y}\rightarrow {\rm Irr}_0(G^\natural,\omega),\,y\mapsto \mu_i(y)=\Pi_{i,y}$$ 
vŽrifie $\mu_i^\circ(y)=\mu'_i(y)$, o $\mu'_i:\mathfrak{Y}'\rightarrow {\rm Irr}(G)$ est l'application rŽgulire introduite en dŽbut de dŽmonstration. 
La composition des morphismes $\phi':\mathfrak{Y}'\rightarrow \mathfrak{X}$ et $\mathfrak{Y}\hookrightarrow \mathfrak{Y}'$ 
est un morphisme dominant Žtale $\phi:\mathfrak{Y}\rightarrow \mathfrak{X}$, et  
pour $y\in \mathfrak{Y}$, on a
$$
q_1\circ\nu\circ \phi(y) =\sum_{i\in I_0}a'_i\mu'_i(y).
$$
Puisque par construction les applications $\mu_i$ (pour $i\in I_0$) sont deux--ˆ--deux ${\Bbb C}^\times$--disjointes, 
la proposition est dŽmontrŽe.
\end{proof}

\begin{moncoro}Soit $\nu_1,\ldots ,\nu_m:\mathfrak{X}\rightarrow \G(G)$ des applications 
rŽgulires, et $\mu_1,\ldots ,\mu_n:\mathfrak{X}\rightarrow {\rm Irr}_0(G^\natural,\omega)$ 
des applications r\'eguli\`eres deux--ˆ--deux ${\Bbb C}^\times$--disjointes. 
Soit $\mathfrak{X}'$ l'ensemble des $x\in \mathfrak{X}$ 
tels que $\{\mu_1^\circ(x),\ldots ,\mu_n^\circ(x)\}$ est contenu dans 
le sous--groupe de $\G_1(G)$ engendr\'e par $q_1\circ\nu_1(x),\ldots ,q_1\circ \nu_m(x)$. 
Il existe un morphisme dominant \'etale $\phi:\mathfrak{Y}\rightarrow \mathfrak{X}$ tel que $\phi^{-1}(\mathfrak{X}')$ est 
soit vide soit \'egal \`a $\mathfrak{Y}$ tout entier.
\end{moncoro}

\begin{proof} D'aprs la proposition, il existe un morphisme dominant Žtale $\phi:\mathfrak{Y}\rightarrow \mathfrak{X}$ 
et des applications rŽgulires $\lambda_1,\ldots ,\lambda_s:\mathfrak{Y}\rightarrow {\rm Irr}_0(G^\natural,\omega)$ 
tels que pour $i=1,\ldots ,m$, l'application $q_1\circ \nu_i\circ \phi:\mathfrak{Y}\rightarrow \G_1(G)$ se dŽcompose en
$$
q_1\circ \nu_i\circ \phi= \sum_{j=1}^{s}a_{i,j}\lambda_j^\circ
$$
pour des entiers $a_{i,j}>0$. 
Quitte ˆ remplacer $\mathfrak{Y}$ par un ouvert plus petit, on peut supposer que les applications 
$\lambda_1^\circ,\ldots ,\lambda_s^\circ:\mathfrak{Y}\rightarrow {\rm Irr}_{G^\natural,\omega}(G)$ sont deux--ˆ--deux disjointes, 
et que pour pour 
tout $k\in \{1,\ldots ,n\}$ et tout $j\in \{1,\ldots ,s\}$, les applications $\mu_k^\circ\circ \phi$ et 
$\lambda_j^\circ$ 
sont soit Žgales soit disjointes. Supposons $\phi^{-1}(\mathfrak{X}')\neq \emptyset$ et soit 
$y\in \phi^{-1}(\mathfrak{X}')$. On a forcŽment $n\leq s$ et quitte ˆ rŽordonner 
les $\lambda_j$, on peut supposer que $\mu_k^\circ \circ \phi=\lambda_k^\circ$ ($k=1,\ldots ,n$). 
Par hypothse, pour $k=1,\ldots ,n$, il existe des entiers $b_{k,i}$ ($i=1,\ldots ,m$) tels que
$$
\mu_k^\circ\circ \phi(y) =\sum_{i=1}^m b_{k,i}\sum_{j=1}^sa_{i,j}\lambda_j^\circ(y).
$$
Pour $k=1,\ldots ,n$ et $j=1,\ldots ,s$, on a donc
$$\sum_{i=1}^mb_{k,i}a_{i,j}=\delta_{k,j}.
$$
Par suite $\phi^{-1}(\mathfrak{X}')=\mathfrak{Y}$ et le lemme est dŽmontrŽ.
\end{proof}

\subsection{La partie $\Theta'=\Theta_{G^\natural,\omega}^{\rm dis}(\mathfrak{s})$ de $\Theta=\Theta(\mathfrak{s})$ 
est constructible}\label{constructible}Montrons que la paire $\Theta'\subset \Theta$ 
satisfait au critre $(*)$ de \ref{dŽcomposition}. On peut supposer que $\Theta'$ est non vide. 
Soit $(M_P,\rho)$ une paire cuspidale standard de $G$ telle que $[M_P,\rho]\in \Theta$. Pour 
$Q\in \EuScript{P}(G)$ tel que $P\subset Q$, notons 
$\eta_Q:\mathfrak{P}(M_P)\rightarrow \EuScript{G}(M_Q)$ 
l'application r\'eguli\`ere d\'efinie par
$$
\eta_Q(\psi)= i_P^Q(\psi \rho).
$$
Pour $Q^\natural\in \EuScript{P}(G^\natural)$, on dŽfinit 
comme en \ref{dŽcomposition} la projection canonique
$q_{Q,1}:\G(M_Q)\rightarrow \G_1(M_Q)$, o $\G_1(M_Q)$ 
est le sous--groupe de $\G(M_Q)$ engendrŽ par les reprŽsentations 
$\sigma\in {\rm Irr}(M_Q)$ telles que $\omega^{-1}\sigma^\theta =\sigma$.

Soit $\mathfrak{X}$ une sous--variŽtŽ localement fermŽe de $\Theta$. 
Puisque le morphisme
$$\mathfrak{P}(M_P)\rightarrow \Theta,\,\psi\mapsto [M_P,\psi\rho]$$
est dominant (et m\^eme fini) 
Žtale, il existe une sous--variŽtŽ $\widetilde{\mathfrak{X}}$ de $\mathfrak{P}(M_P)$ telle que l'ensemble 
$\{[M_P,\psi\rho]:\psi\in \mathfrak{X}'\}$ est contenu dans $\mathfrak{X}$ et le morphisme 
$$\widetilde{\mathfrak{X}}\rightarrow \mathfrak{X},\,\psi \mapsto [M_P,\psi\rho]$$ est dominant Žtale. 
D'aprs la proposition de \ref{consŽquence}, il existe un morphisme dominant Žtale 
$\tilde{\phi}:\mathfrak{Y}\rightarrow \widetilde{\mathfrak{X}}$ tel que pour chaque 
$Q^\natural\in \EuScript{P}(G^\natural)$ contenant 
$P^\natural$, l'application rŽgulire $$\tilde{\eta}_Q= \eta_Q\circ \tilde{\phi}:\mathfrak{Y}\rightarrow \G(M_Q)$$
composŽe avec la projection canonique $q_{Q,1}:\G(M_Q)\rightarrow \G_1(M_Q)$ 
se dŽcompose en
$$
q_{Q,1}\circ \tilde{\eta}_Q=\sum_{i=1}^{m(Q^\natural)}a_{Q^\natural\!,i}\,\mu_{Q^\natural\!,i}^\circ$$ 
o:
\begin{itemize}
\item $\mu_{Q^\natural\!,1},\ldots ,\mu_{Q^\natural,m(Q^\natural)}:\mathfrak{Y}\rightarrow {\rm Irr}_0(M_Q^\natural,\omega)$ 
sont des applications rŽgulires deux--ˆ--deux ${\Bbb C}^\times$--disjointes; 
\item $a_{Q^\natural\!,1},\ldots ,a_{Q^\natural\!,m(Q^\natural)}$ sont des entiers $>0$. 
\end{itemize}
Posons $n=m(G^\natural)$ et 
$\mu_i= \mu_{G^\natural\!,i}$ ($i=1,\ldots ,n$). 

Pour $Q^\natural\in \P(G^\natural)\smallsetminus \{G^\natural\}$ contenant $P^\natural$ 
et $i=1,\ldots ,m(Q^\natural)$, l'application
$$
\nu_{Q,i}={i}_Q^G\circ \mu_{Q^\natural\!,i}^\circ:\mathfrak{Y}\rightarrow \G(G)
$$
est rŽgulire. La famille 
des $\nu_{Q,i}: \mathfrak{Y}\rightarrow \G(G)$ obtenue en faisant varier $Q^\natural\;(\neq G^\natural)$ et $i$ de cette manire, 
est notŽe 
$\{\nu_1,\ldots ,\nu_m\}$. L'ensemble des points $y\in \mathfrak{Y}$ tels que $\{\mu_1^\circ(y),\ldots ,\mu_n^\circ(y)\}$ 
est contenu dans le sous--groupe de $\G_1(G)$ engendrŽ par $q_1\circ\nu_1(y),\ldots ,q_1\circ\nu_m(y)$ est exactement l'image 
rŽciproque du complŽmentaire 
$(\mathfrak{X}\cap \Theta')^{\EuScript{C}}$ de $\mathfrak{X}\cap \Theta'$ dans $\mathfrak{X}$ 
par le morphisme dominant Žtale
$$\phi=(\widetilde{\mathfrak{X}}\rightarrow \mathfrak{X})\circ \tilde{\phi}:\mathfrak{Y}\rightarrow 
\mathfrak{X}.
$$ On conclut 
gr\^ace au corollaire de \ref{consŽquence}. Puisque $\Theta'$ vŽrifie la propriŽtŽ $(*)$ de \ref{dŽcomposition}, 
c'est une partie constructible de $\Theta$. Cela achve la dŽmonstration de la proposition de \ref{finitude}.

\subsection{DŽcomposition des espaces $\F^{\rm dis}(G^\natural,\omega)$ et $\F_{\rm tr}^{\rm dis}(G^\natural,\omega)$.}
\label{dŽcomposition des espaces}
Pour tout sous--ensemble $Y$ de $\Theta(G)$, on note $\G_{\Bbb C}(G^\natural,\omega;Y)$ le sous--espace 
vectoriel de $\G_{\Bbb C}(G^\natural,\omega)$ engendrŽ par les $\Pi\in {\rm Irr}_{\Bbb C}(G^\natural,\omega)$ tels que 
$\theta_{G^\natural}(\Pi)\in Y$. On a la dŽcomposition\index{$\G_{\Bbb C}(G^\natural,\omega;Y)=\bigoplus_{y}\G_{\Bbb C}(G^\natural,\omega;y)$}
$$
\G_{\Bbb C}(G^\natural,\omega;Y)=\bigoplus_{y}\G_{\Bbb C}(G^\natural,\omega;y)
$$
o $y$ parcourt les ŽlŽments de l'ensemble $Y\cap \Theta_{G^\natural,\omega}(G)$. 
Dualement, pour $Y\subset \Theta(G)$, on note $\F_Y(G^\natural,\omega)$\index{$\F_Y(G^\natural,\omega)$, 
$\F_\mathfrak{S}(G^\natural,\omega)$} le sous--espace vectoriel 
de $\G_{\Bbb C}(G^\natural,\omega;Y)^*$ formŽ des restrictions ˆ $\G_{\Bbb C}(G^\natural,\omega;Y)$ des ŽlŽments de 
$\F(G^\natural,\omega)$, et pour $\mathfrak{S}\subset \mathfrak{B}(G)$, on pose
$$\F_\mathfrak{S}(G^\natural,\omega)=
\F_{\Theta(\mathfrak{S})}(G^\natural,\omega),\quad 
\Theta(\mathfrak{S})= \coprod_{\mathfrak{s}\in \mathfrak{S}}\Theta(\mathfrak{s}).
$$ 
On a la dŽcomposition
$$
\F(G^\natural,\omega)= \bigoplus_{\mathfrak{s}}\F_\mathfrak{s}(G^\natural,\omega)
$$
o $\mathfrak{s}$ parcourt les ŽlŽments de $\mathfrak{B}_{G^\natural,\omega}(G^\natural)$.

\begin{marema}
{\rm Pour $\mathfrak{s}\in \mathfrak{B}(G)$, notant $z_\mathfrak{s}$ l'ŽlŽment unitŽ de l'anneau $\mathfrak{Z}_\mathfrak{s}$, on a 
l'ŽgalitŽ $\F_\mathfrak{s}(G^\natural,\omega)=z_\mathfrak{s}\cdot\F(G^\natural,\omega)$.\hfill $\blacksquare$} 
\end{marema}

Posons\index{$\Theta_{G^\natural,\omega}^{\rm dis}(G)$}
$$
\Theta_{G^\natural,\omega}^{\rm dis}(G)= \coprod_{\mathfrak{s}\in \mathfrak{B}(G)}\Theta_{G^\natural,\omega}^{\rm dis}(\mathfrak{s}).
$$ 
Pour $Y\subset \Theta(G)$, on note $\G_{\Bbb C}^{\rm dis}(G^\natural,\omega;Y)$\index{$\G_{\Bbb C}^{\rm dis}(G^\natural,\omega;Y)$} 
la projection de $\G_{\Bbb C}(G^\natural,\omega;Y)$ sur $\G_{\Bbb C}^{\rm dis}(G^\natural,\omega)$, et 
l'on pose\index{$\F_Y^{\rm dis}(G^\natural,\omega)$, $\F_\mathfrak{S}^{\rm dis}(G^\natural,\omega)$} 
$$
\F_Y^{\rm dis}(G^\natural,\omega)= \F_Y(G^\natural,\omega)\cap \F^{\rm dis}(G^\natural,\omega)\subset 
\G_{\Bbb C}^{\rm dis}(G^\natural,\omega;Y)^*.
$$
Pour $\mathfrak{S}\subset \mathfrak{B}(G)$, on pose
$$
\F_\mathfrak{S}^{\rm dis}(G^\natural,\omega)= \F_{\Theta(\mathfrak{S})}^{\rm dis}(G^\natural,\omega).
$$
Puisque 
$\G_{\Bbb C}^{\rm dis}(G^\natural,\omega)^* \subset 
\G_{\Bbb C}(G^\natural, \omega; \Theta_{G^\natural,\omega}^{\rm dis}(G))^*$,
on a la dŽcomposition
$$
\F^{\rm dis}(G^\natural,\omega)= \bigoplus_{\mathfrak{s}}\F_{\mathfrak{s}}^{\rm dis}(G^\natural,\omega)
$$
o $\mathfrak{s}$ parcourt les ŽlŽments de l'ensemble $\mathfrak{B}_{G^\natural,\omega}^{\rm dis}(G)= 
\beta_{G^\natural}({\rm Irr}_0^{\rm dis}(G^\natural,\omega))$ --- le sous--ensemble de 
$\mathfrak{B}(G)$ formŽ des $\mathfrak{s}$ tels que $\Theta_{G^\natural,\omega}^{\rm dis}(\mathfrak{s})$ est non vide. 

En rempla\c{c}ant $\F(G^\natural,\omega)$ par $\F_{\rm tr}(G^\natural,\omega)$, on dŽfinit 
de la m\^eme manire le sous--espace $\F_{{\rm tr},Y}(G^\natural,\omega)$\index{$\F_{{\rm tr},Y}(G^\natural,\omega)$, 
$\F_{{\rm tr},\mathfrak{S}}(G^\natural,\omega)$} de 
$\G_{\Bbb C}(G^\natural,\omega;Y)^*$ (pour $Y\subset \Theta(G)$), et l'on pose 
$\F_{{\rm tr},\mathfrak{S}}(G^\natural,\omega)= \F_{{\rm tr},\Theta(\mathfrak{S})}(G^\natural,\omega)$ 
(pour $\mathfrak{S}\subset \mathfrak{B}(G)$). On a la dŽcomposition
$$
\F_{\rm tr}(G^\natural,\omega)= \bigoplus_{\mathfrak{s}}\F_{{\rm tr},\mathfrak{s}}(G^\natural,\omega)
$$
o $\mathfrak{s}$ parcourt les ŽlŽments de $\mathfrak{B}_{G^\natural,\omega}(G)$. De m\^eme, 
posant\index{$\F_{{\rm tr},Y}^{\rm dis}(G^\natural,\omega)$, $\F_{{\rm tr},\mathfrak{S}}^{\rm dis}(G^\natural,\omega)$}
$$
\F_{{\rm tr},Y}^{\rm dis}(G^\natural,\omega)= \F_{{\rm tr},Y}(G^\natural,\omega)\cap \F^{\rm dis}(G^\natural,\omega),\quad 
Y\subset \Theta(G),
$$
et 
$$\F_{{\rm tr},\mathfrak{S}}^{\rm dis}(G^\natural,\omega)
=\F_{{\rm tr},\Theta(\mathfrak{S})}^{\rm dis}(G^\natural,\omega),\quad \mathfrak{S}\subset \mathfrak{B}(G),
$$
on a la dŽcomposition
$$
\F_{\rm tr}^{\rm dis}(G^\natural,\omega)= \bigoplus_{\mathfrak{s}}\F_{{\rm tr},\mathfrak{s}}^{\rm dis}(G^\natural,\omega)
$$
o $\mathfrak{s}$ parcourt les ŽlŽments de $\mathfrak{B}_{G^\natural,\omega}^{\rm dis}(G)$.

Pour dŽmontrer la surjectivitŽ dans le thŽorme de \ref{ŽnoncŽ discret}, il suffit de montrer que pour chaque $\mathfrak{s}\in 
\mathfrak{B}_{G^\natural,\omega}^{\rm dis}(G)$, l'inclusion
$$
\F^{\rm dis}_{{\rm tr},\mathfrak{s}}(G^\natural,\omega)\subset \F_\mathfrak{s}^{\rm dis}(G^\natural,\omega)
$$
est une ŽgalitŽ. 

Le lemme suivant est impliquŽ par 
la propriŽtŽ d'indŽpendance linŽaire des caractres--distributions des $\omega$--reprŽsentations $G$--irrŽductibles 
de $G^\natural$ \cite[8.5, prop.]{L2}.

\begin{monlem}Pour tout sous--ensemble fini $Y$ de $\Theta(G)$, 
on a l'ŽgalitŽ
$$\F_{{\rm tr},Y}(G^\natural,\omega)= \G_{\Bbb C}(G^\natural,\omega;Y)^*.$$ 
\end{monlem}

Notons que si le groupe $\mathfrak{P}(G^\natural)$ est fini, alors pour chaque $\mathfrak{s}$ l'ensemble 
$\Theta_{G^\natural,\omega}^{\rm dis}(\mathfrak{s})$ est fini 
(proposition de \ref{finitude}), et d'aprs le lemme on a l'ŽgalitŽ cherchŽe: 
$$\F_{{\rm tr},\mathfrak{s}}^{\rm dis}(G^\natural,\omega)=\G_{\Bbb C}^{\rm dis}(G^\natural,\omega;\Theta(\mathfrak{s}))^*
=\F_\mathfrak{s}^{\rm dis}(G^\natural,\omega).$$

\subsection{SurjectivitŽ dans le thŽorme de \ref{ŽnoncŽ discret}.}\label{surjectivitŽ} 
L'idŽe consiste ˆ se ramener au lemme de \ref{dŽcomposition des espaces}, 
comme dans \cite[4.2]{BDK}.

Rappelons que $A_G$ est le tore central dŽployŽ maximal de $G$. 
Choisissons une uniformisante $\varpi$ 
de $F$ et identifions le groupe ${\rm Hom}(A_G^\varpi,{\Bbb C}^\times)$ ˆ $\mathfrak{P}(A_G)$ comme dans l'exemple de \ref{carac nr}. 
Pour $u\in \H(A_G)$, on note $z(u)$ 
l'ŽlŽment de $\mathfrak{Z}(G)$ dŽfini par
$$
z(u)_\pi = \int_{A_G}u(a)\pi(a)da, \quad \pi \in {\rm Irr}(G);
$$
o $da$ est la mesure de Haar sur $A_G$ qui donne le volume $1$ ˆ $A_G^1$. L'algbre ${\Bbb C}[\mathfrak{P}(A_G)]$ 
des fonctions rŽgulires sur la variŽtŽ $\mathfrak{P}(A_G)$, identifiŽe ˆ l'algbre de groupe 
${\Bbb C}[A_G^\varpi]$, est une sous--algbre de $\H(A_G)$; d'o un morphisme d'algbres 
${\Bbb C}[\mathfrak{P}(A_G)]\rightarrow \mathfrak{Z}(G)$. Pour chaque $\mathfrak{s}\in \mathfrak{B}(G)$, 
on peut composer ce morphisme d'algbres avec la projection canonique $\mathfrak{Z}(G)\rightarrow \mathfrak{Z}_\mathfrak{s}$. 
Le morphisme de variŽtŽs correspondant $\eta_\mathfrak{s}:\Theta(\mathfrak{s})\rightarrow \mathfrak{P}(A_G)$ est donnŽ par
$$
\eta_\mathfrak{s}(\theta_G(\pi))= \omega_\pi\vert_{A_G^\varpi},\quad \pi\in \beta_G^{-1}(\mathfrak{s}).
$$
Notons que ce morphisme $\eta_\mathfrak{s}$ est $\mathfrak{P}(G)$--Žquivariant pour l'action (algŽbrique) 
de $\mathfrak{P}(G)$ sur $\mathfrak{P}(A_G)$ donnŽe par 
$(\psi,\chi)\mapsto (\psi\vert_{A_G})\chi$, pour $\psi\in \mathfrak{P}(G)$ et $\chi\in \mathfrak{P}(A_G)$. 

Comme pour $A_G$, on identifie le groupe ${\rm Hom}(A_{G^\natural}^\varpi,{\Bbb C}^\times)$ ˆ $\mathfrak{P}(A_{G^\natural})$; 
o (rappel) $A_{G^\natural}$ est le tore dŽployŽ maximal du centre $Z^\natural$ de $G^\natural$. 
L'action de $\mathfrak{P}(G)$ sur $\mathfrak{P}(A_G)$ induit par restricion une 
action (algŽbrique) de $\mathfrak{P}(G^\natural)$ sur $\mathfrak{P}(A_{G^\natural})$. L'application $\mathfrak{P}(G^\natural)
\rightarrow \mathfrak{P}(A_{G^\natural}),\,\psi\mapsto \psi\vert_{A_{G^\natural}}$ est un morphisme surjectif 
de groupes algŽbriques, de noyau fini (cf. \ref{les espaces bP}). En particulier c'est un morphisme fini, donc Žtale (puisqu'il est lisse).

Fixons $\mathfrak{s}\in \mathfrak{B}_{G^\natural,\omega}^{\rm dis}(G)$ et 
posons $Y= \Theta_{G^\natural,\omega}^{\rm dis}(\mathfrak{s})$. Posons aussi $\mathfrak{X}=\mathfrak{P}(A_{G^\natural})$. 
Par restriction, $\eta_\mathfrak{s}$ induit 
une application $\mathfrak{P}(G^\natural)$--Žquivariante $\eta_Y: Y \rightarrow \mathfrak{X}$, 
donnŽe par
$$\eta_Y(\theta_{G^\natural}(\Pi))= \omega_\Pi\vert_{A_{G^\natural}^\varpi},\quad 
\Pi\in \theta_{G^\natural}^{-1}(Y).
$$
Pour $\Pi\in \theta_{G^\natural}^{-1}(Y)$ et $\psi \in \mathfrak{P}(G^\natural)$, elle vŽrifie
$$
\eta_Y(\psi\cdot \theta_{G^\natural}(\Pi))= \psi\vert_{A_{G^\natural}}\cdot\eta_Y(\theta_{G^\natural}(\Pi)).
$$
Comme $Y$ est union d'un nombre fini de $\mathfrak{P}(G^\natural)$--orbites (proposition de \ref{finitude}), et que pour chacune de 
ces orbites le stabilisateur dans $\mathfrak{P}(G^\natural)$ est fini, 
l'application $\eta_Y$ est un morphisme fini Žtale de variŽtŽs algŽbriques affines lisse. Ainsi 
le comorphisme ${\Bbb C}[\mathfrak{X}]\rightarrow {\Bbb C}[Y]$ 
fait de ${\Bbb C}[Y]$ un ${\Bbb C}[\mathfrak{X}]$--module de type fini. En particulier, 
l'espace $\E_Y=\G_{\Bbb C}^{\rm dis}(G^\natural,\omega;Y)^*$ est un ${\Bbb C}[\mathfrak{X}]$--module de type fini pour 
l'action de ${\Bbb C}[\mathfrak{X}]$ donnŽe par ($\varphi\in {\Bbb C}[\mathfrak{X}]$, 
$\Phi\in \E_Y$, $\Pi\in \theta_{G^\natural}^{-1}(Y)$)
$$
(\varphi\cdot \Phi)(\Pi)= \varphi(\omega_{\Pi}\vert_{A_{G^\natural}^\varpi})\Phi(\Pi).
$$
En d'autres termes, ${\Bbb C}[\mathfrak{X}]$ opre sur $\E_Y$ via le morphisme 
d'algbres ${\Bbb C}[\mathfrak{X}]\rightarrow \mathfrak{Z}_{\mathfrak{s},1},\,\varphi\mapsto z_\varphi$ donnŽ par 
$(z_\varphi)_{\Pi^\circ}= \varphi(\omega_{\Pi}\vert_{A_{G^\natural}^\varpi}){\rm id}_{V_\Pi}$ pour tout objet irrŽductible $\Pi$ 
de $\mathfrak{R}_\mathfrak{s}(G^\natural,\omega)$, cf. \ref{action sur le centre}. On en dŽduit 
que les espaces $\F_Y=\F_Y^{\rm dis}(G^\natural,\omega)$ et $\F'_Y= \F_{{\rm tr},Y}^{\rm dis}(G^\natural,\omega)$ sont des 
sous--${\Bbb C}[\mathfrak{X}]$--modules de $\E_Y$. L'action de $\mathfrak{P}_{\Bbb C}(G^\natural)$ sur $\G_{\Bbb C}(G^\natural,\omega)^*$ 
dŽfinie en \ref{actions duales} induit une action sur $\G_{\Bbb C}^{\rm dis}(G^\natural,\omega)^*$, et les espaces $\E_Y$, $\F_Y$ et $\F'_Y$ sont 
stables pour cette action.

Soit $x\in \mathfrak{X}$ correspondant \`a $u:{\Bbb C}[\mathfrak{X}]\rightarrow {\Bbb C}$. Pour tout 
${\Bbb C}[\mathfrak{X}]$--module $E$, on note $E_x$ la fibre $E\otimes_{{\Bbb C}[\mathfrak{X}],u}{\Bbb C}$ de 
$E$ au--dessus de $x$. Comme le morphisme $\eta_Y:Y\rightarrow \mathfrak{X}$ est fini et lisse, l'ensemble 
$Y_x=\eta_Y^{-1}(x)$ est fini et la fibre ${\Bbb C}[Y]_x$ co\"{\i}ncide avec ${\Bbb C}[Y_x]$, i.e. la \og fibre gŽomŽtrique \fg 
au--dessus de $x$ est rŽgulire. 
D'aprs le lemme 
de \ref{dŽcomposition des espaces}, on a l'ŽgalitŽ $\F'_{Y_x}=\E_{Y_x}$. En d'autres termes, posant 
$\mathfrak{p}_x= \ker u$, on a l'inclusion
$$
\F_{Y}\subset \F'_{Y} + \mathfrak{p}_x \cdot \E_Y.
$$

Posons $\overline{\E}= \E_Y/\F'_{Y}$ et $\overline{\F}= \F_Y/\F'_{Y}\subset \overline{\E}$. 
D'aprs l'inclusion 
ci--dessus, on a $\overline{\F}\subset \mathfrak{p}_x\cdot \overline{\E}$ pour tout 
$x\in \mathfrak{X}$. Puisque $\overline{\E}$ est un ${\Bbb C}[\mathfrak{X}]$--module de type fini, il 
est localement libre en presque tout point de $\mathfrak{X}$. Comme d'aprs le lemme 2 de \ref{actions duales}, 
$\overline{\E}$ est $\mathfrak{P}_{\Bbb C}(G^\natural)$--Žquivariant, 
il est localement libre en tout point de $\mathfrak{X}$, et l'inclusion $\overline{\F}\subset \mathfrak{p}_x\cdot \overline{\E}$ 
($x\in \mathfrak{X}$) implique que $\overline{\F}=0$; ce qu'on voulait dŽmontrer.

\vfill
\eject

\printindex

\vfill
\eject


\begin{thebibliography}{3}

\bibitem[BD]{BD} \textsc{Bernstein J. \& Deligne P.}, \textit{Le \og centre \fg de Bernstein} in ReprŽsentations des groupes rŽductifs sur un 
corps local (ed. Bernstein J., Deligne P., Kazhdan D., VignŽras M.--F.), Travaux en Cours, Hermann, 1984.

\bibitem[BDK]{BDK} \textsc{Bernstein J., Deligne P. \& Kazhdan D.}, \textit{Trace Paley--Wiener theorem for reductive $p$--adic groups}, 
J. Analyse Math. {\bf 47} (1986), pp. 180--192.

\bibitem[BT]{BT} \textsc{Bruhat F., Tits J.}, \textit{Groupes rŽductifs sur un corps local II. SchŽmas en groupes. Existence 
d'une donnŽe radicielle valuŽe}, Pub. Math. IHES {\bf 60} (1984), pp. 1--184.

\bibitem[BZ]{BZ} \textsc{Bernstein J., Zelevinsky A.}, \textit{Induced representations of reductive $p$--adic groups I}, 
Ann. Sci. \'Ecole Norm. Sup. {\bf 10} (1977), pp. 441--472.

\bibitem[BW]{BW}\textsc{Borel A., Wallach N.}, \textit{Continuous cohomology, discrete subgroups, and representations of reductive 
groups}, Ann. Math. Studies {\bf 94}, Princeton U. Press, Princeton, NJ, 1980.

\bibitem[C]{C} \textsc{Casselman R.}, {\it Characters and Jacquet modules}, Math. Ann. {\bf 230} (1977), pp. 101--105.

\bibitem[D]{D} \textsc{Dat J.--F.}, \textit{On the $K_0$ of a $p$--adic group}, Invent. Math. {\bf 140} (2000), pp. 171--226.

\bibitem[F]{F} \textsc{Flicker Y.}, \textit{Bernstein's isomorphism and good forms}, Proc. Symp. Pure Math. {\bf 58} (1995), 
pp. 171--196.

\bibitem[G]{G} \textsc{Gabriel P.}, \textit{Des catŽgories abŽliennes}, Bull. Soc. Math. France {\bf  90} (1962), pp. 323--448.

\bibitem[HL]{HL} \textsc{Henniart G., Lemaire B.}, \textit{La transformŽe de Fourier pour les espaces tordus sur un groupe rŽductif $\mathfrak{p}$--adique, II. 
Le thŽorme de densitŽ spectrale}, mansucript. 

\bibitem[K1]{K1} \textsc{Kazhdan D.}, \textit{Cuspidal geometry of $p$--adic groups}, J. Analyse Math. {\bf 47} (1986), pp. 1--36.

\bibitem[K2]{K2} \textsc{Kazhdan D.}, \textit{Representations of groups over close local fields}, J. Analyse Math. {\bf 47} (1986), 
pp. 175--179.

\bibitem[L1]{L1}  \textsc{Lemaire B.} \textit{ReprŽsentations gŽnŽriques de $GL_N$ et corps locaux proches}, J. Algebra {\bf 236} 
(2001), pp. 549--574.

\bibitem[L2(2010)]{L2(2010)} \textsc{Lemaire B.} \textit{Caractres tordus des reprŽsentations admissibles}, prŽpublication 
arXiv: 1007.3576v1 [math.RT], 
21 juillet 2010.

\bibitem[L2]{L2} \textsc{Lemaire B.} \textit{Caractres tordus des reprŽsentations admissibles}, version remaniŽe et complŽtŽe de \cite{L2(2010)}, 
manuscript.

\bibitem[R]{R} \textsc{Rogawski J.}, \textit{Trace Paley--Wiener theorem in the
twisted case}, Trans. Amer. Math. Soc. {\bf 309} (1998), pp. 215--229.

\bibitem[W]{W} \textsc{Waldspurger J.--L.}, \textit{La formule des traces locale tordue}, 
prŽpubication arXiv: 1205.1100v2 [math. RT], 13 septembre 2012.

\end{thebibliography}
\end{document}